\documentclass[12pt]{amsart}

\usepackage[margin = 1in]{geometry}
\usepackage{amsfonts, amsmath,amscd, amssymb,
latexsym, mathrsfs, mathtools, 
slashed, stmaryrd, verbatim, wasysym }
\usepackage[usenames]{xcolor}
\usepackage[all]{xy}
\usepackage[pdftex]{graphicx}
\usepackage{hyperref}
\usepackage[applemac]{inputenc}
      
\newtheorem{theorem}{Theorem}
\newtheorem{corollary}[theorem]{Corollary}
\newtheorem{definition}[theorem]{Definition}

\newtheorem{lemma}[theorem]{Lemma}
\newtheorem{proposition}[theorem]{Proposition}
\newtheorem{remark}[theorem]{Remark}

\numberwithin{equation}{section}
\numberwithin{theorem}{section}

\let\oldsqrt\sqrt
\def\sqrt{\mathpalette\DHLhksqrt}
\def\DHLhksqrt#1#2{%
\setbox0=\hbox{$#1\oldsqrt{#2\,}$}\dimen0=\ht0
\advance\dimen0-0.2\ht0
\setbox2=\hbox{\vrule height\ht0 depth -\dimen0}%
{\box0\lower0.4pt\box2}}

\newcommand{\fib}{\;\scalebox{1.5}[1.0]{\( - \)}\;}

\usepackage{verbatim} 
\DeclareFontFamily{U}{mathx}{\hyphenchar\font45}
\DeclareFontShape{U}{mathx}{m}{n}{
      <5> <6> <7> <8> <9> <10>
      <10.95> <12> <14.4> <17.28> <20.74> <24.88>
      mathx10
      }{}
\DeclareSymbolFont{mathx}{U}{mathx}{m}{n}
\DeclareFontSubstitution{U}{mathx}{m}{n}
\DeclareMathAccent{\widecheck}{0}{mathx}{"71}

\newcommand{\wc}[1]{\widecheck{#1}}

\newcommand{\df}[1]{\mathfrak{#1}}

\newcommand{\bhs}[1]{\mathfrak B_{#1}}
\newcommand{\bhss}[1]{\mathfrak B^{(1)}_{#1}}
\newcommand{\bhsd}[1]{\mathfrak B^{(2)}_{#1}}
\newcommand{\bhst}[1]{\mathfrak B^{(3)}_{#1}}
\newcommand{\bhsh}[1]{\mathfrak B^{(H)}_{#1}}
\newcommand{\bhsT}[1]{\mathfrak B^{(\mathscr{T})}_{#1}}
\newcommand{\bhsC}[1]{\mathfrak B^{(C)}_{#1}}
\newcommand{\bhsD}[1]{\mathfrak B^{(\Delta)}_{#1}}
\renewcommand{\tilde}{\widetilde}
\renewcommand{\bar}{\overline}
\renewcommand{\Re}{\operatorname{Re}}
\renewcommand{\hat}[1]{\widehat{#1}}

\newcommand{\wt}[1]{\widetilde{#1}}
\newcommand{\rest}[1]{\big\rvert_{#1}} 

\newcommand{\Rint}{\sideset{^R}{}\int}

\newcommand\lra{\longrightarrow}
\newcommand\xlra[1]{\xrightarrow{\phantom{x} #1 \phantom{x}}}

\newcommand\pa{\partial}

\newcommand{\Nop}{{{}^\e \mathcal{N} }}

\newcommand\eps\varepsilon
\renewcommand\epsilon\varepsilon

\newcommand\bh{\mathbf{h}}

\newcommand\bv{\mathbf{v}}
\newcommand{\Cl}{\mathbb{C}l}

\newcommand{\e}{\mathrm{e}}

\newcommand{\w}{\mathrm{w}}
\newcommand{\Tw}{{}^{\mathrm{w}}T}

\newcommand\CI{{\mathcal{C}}^{\infty}}
\newcommand\CIc{{\mathcal{C}}^{\infty}_c}

\newcommand\ang[1]{\langle #1 \rangle}
\newcommand\floor[1]{\lfloor #1 \rfloor}
\newcommand{\lrpar}[1]{\left( #1 \right)}
\newcommand{\lrspar}[1]{\left[ #1 \right]}
\newcommand{\lrbrac}[1]{\left\lbrace #1 \right\rbrace}

\newcommand\Ch{\operatorname{Ch}}

\renewcommand\det{\operatorname{det}}

\newcommand\diag{\operatorname{diag}}
\newcommand\Diff{\operatorname{Diff}}

\newcommand\ef{\operatorname{ef}}
\newcommand\End{\operatorname{End}}
\newcommand\ev{\operatorname{even}}

\newcommand\ff{\operatorname{ff}}

\DeclareMathOperator*{\FP}{\operatorname{FP}}

\newcommand\Hom{\operatorname{Hom}}
\newcommand\id{\operatorname{id}}
\newcommand\Id{\operatorname{Id}}
\renewcommand\Im{\operatorname{Im}}
\newcommand\ind{\operatorname{ind}}
\newcommand\Ind{\operatorname{Ind}}

\newcommand\lf{\operatorname{lf}}

\newcommand\odd{\operatorname{odd}}

\newcommand\phg{\operatorname{phg}}
\newcommand\pt{\operatorname{pt}}

\renewcommand\Re{\operatorname{Re}}
\newcommand{\reg}{ \mathrm{reg} }

\newcommand\rf{\operatorname{rf}}

\newcommand\RTr[1]{{}^R\operatorname{Tr}\left( #1 \right)}

\newcommand\scal{\mathrm{scal}}
\renewcommand\sf{\operatorname{sf}}

\newcommand\Spec{\operatorname{Spec}}
\newcommand\str{\operatorname{str}}
\newcommand\Str{\operatorname{Str}}

\newcommand\Tr{\operatorname{Tr}}

\newcommand\tr{\operatorname{tr}}
\newcommand\Vol{\operatorname{Vol}}

\newcommand\Mand{\text{ and }}
\newcommand\Mas{\text{ as }}
\newcommand\Mat{\text{ at }}

\newcommand\Melse{\text{ else }}
\newcommand\Mev{\operatorname{even}}

\newcommand\Mfor{\text{ for }}
\newcommand\Mforall{\text{ for all }}

\newcommand\Mforeach{\text{ for each }}

\newcommand\Mforsome{\text{ for some }}

\newcommand\Mif{\text{ if }}

\newcommand\Min{\text{ in }}

\newcommand\Mnear{\text{ near }}
\newcommand\Modd{\operatorname{odd}}

\newcommand\Mon{\text{ on }}

\newcommand\Motherwise{\text{ otherwise }}

\newcommand\Mst{\text{ s.t. }}

\newcommand\Mwhere{\text{ where }}
\newcommand\Mwith{\text{ with }}

\newcommand\paperintro%
        {%
         }
\newcommand\paperbody%
        {%
         }

\newcommand\bbA{\mathbb{A}}
\newcommand\bbB{\mathbb{B}}
\newcommand\bbC{\mathbb{C}}

\newcommand\bbE{\mathbb{E}}

\newcommand\bbN{\mathbb{N}}

\newcommand\bbR{\mathbb{R}}
\newcommand\bbS{\mathbb{S}}
\newcommand\bbT{\mathbb{T}}

\newcommand\bbZ{\mathbb{Z}}

\newcommand\cA{\mathcal{A}}
\newcommand\cB{\mathcal{B}}
\newcommand\cC{\mathcal{C}}
\newcommand\cD{\mathcal{D}}
\newcommand\cE{\mathcal{E}}
\newcommand\cF{\mathcal{F}}
\newcommand\cG{\mathcal{G}}
\newcommand\cH{\mathcal{H}}
\newcommand\cI{\mathcal{I}}
\newcommand\cJ{\mathcal{J}}
\newcommand\cK{\mathcal{K}}

\newcommand\cM{\mathcal{M}}
\newcommand\cN{\mathcal{N}}
\newcommand\cO{\mathcal{O}}
\newcommand\cP{\mathcal{P}}

\newcommand\cR{\mathcal{R}}
\newcommand\cS{\mathcal{S}}

\newcommand\cU{\mathcal{U}}
\newcommand\cV{\mathcal{V}}

\newcommand\mf[1]{\mathfrak{ #1}}

\newcommand\sA{\mathscr{A}}
\newcommand\sB{\mathscr{B}}
\newcommand\sC{\mathscr{C}}

\newcommand\sN{\mathscr{N}}

\newcommand\sT{\mathscr{T}}

\newcommand\tH{\operatorname{H}}

\newcommand\tK{\operatorname{K}}

\newcommand\wT {{}^\w T}

\DeclareMathAlphabet{\mathpzc}{OT1}{pzc}{m}{it}
\newcommand{\cl}{\mathpzc{cl}}

\newcommand{\overbar}[1]{\mkern 20mu\overline{\mkern-20mu#1\mkern-20mu}\mkern 20mu}

\newcommand\VAPS{\mathrm{VAPS}}
\usepackage[refpage]{nomencl}

\makenomenclature

\begin{document}

\title{The index formula for families of Dirac type operators on pseudomanifolds}
\author{Pierre Albin}
\address{University of Illinois, Urbana-Champaign}
\email{palbin@illinois.edu}
\author{Jesse Gell-Redman}
\address{University of Melbourne}
\email{j.gell@unimelb.edu.au}

\begin{abstract}
We study families of Dirac-type operators, with compatible perturbations, associated to wedge metrics on stratified spaces. We define a closed domain and, under an assumption of invertible boundary families, prove that the operators are self-adjoint and Fredholm with compact resolvents and trace-class heat kernels. We establish a formula for the Chern character of their index.
\end{abstract}

\maketitle

\paperintro
\section*{Introduction}

In this article we establish an index formula for families of Dirac-type operators on stratified spaces endowed with metrics of iterated conic-type singularities. In the process we construct the resolvent and heat kernel of these operators by extending the b-calculus of Melrose and the edge calculus of Mazzeo to manifolds with corners and fibered boundaries. We establish refined asymptotic expansions of these operators of relevance to the study of analytic and spectral constructions and their connections to topology. In the present paper we use them to carry out the heat equation proof of the index theorem.

Our index theorem seems to be the first example of an index theorem for singular metrics on stratified spaces of arbitrary depth since the 
Gauss-Bonnet and signature theorems for piecewise flat admissible metrics established by Cheeger \cite{C1983}.
Since this seminal paper there has been much interest and progress in index theory on spaces with isolated conic singularities and, more recently, non-isolated conic singularities which we will review below.\\

Most singular spaces arising from smooth objects, such as zero sets of polynomials, orbits of group actions, and many moduli spaces, are Thom-Mather stratified spaces. They can be written as the union of smooth manifolds known as strata. One stratum is dense and referred to as the {\em regular part}, while the others make up the {\em singular part}. Each singular stratum has a tubular neighborhood in the stratified space that fibers over it, the fiber is itself a stratified space and is known as the {\em link} of that stratum. The depth of a stratified space is the length of the longest chain of inclusions among the closures of singular strata. (See, e.g., \cite{Albin:Hodge, Kloeckner} and references therein for more on stratified spaces.)

Every Thom-Mather stratified space $\hat X$ can be `resolved' to a manifold with corners $X$ by a procedure that goes back to Thom \cite{Thom:Ensembles}  and was recently reformulated by Melrose (see \cite{ALMP:Witt, Albin-Melrose:Gps, Albin:Hodge}). 
The boundary hypersurfaces of $X$ are partitioned into collective
(i.e.\ disjoint unions of) boundary hypersurfaces, $\bhs{Y},$ one per singular stratum $Y^{\circ}$ of $\hat X,$ which participate in smooth fiber bundles of manifolds with corners,
\begin{equation*}
	Z \fib \bhs{Y} \xlra{\phi_Y} Y,
\end{equation*}
in which $Y$ is the resolution of the closure of $Y^{\circ}$ in $\hat X$ and $Z$ is the resolution of the link of $Y^{\circ}$ in $\hat X.$
These fiber bundles satisfy compatibility conditions at intersections of collective boundary hypersurfaces, see Definition \ref{def:IFS}.
We refer to this structure on $X$ as an {\em iterated fibration structure} and emphasize that it is equivalent to the Thom-Mather structure on $\hat X.$\\

There is a map $\beta: X \lra \hat X$ relating a stratified space and its resolution that restricts to a diffeomorphism between the interior of $X$ and the regular part of $\hat X,$ $\hat X^{\reg}.$ This relation allows us to define the analogues of smooth objects on $\hat X.$ For example, we can define a smooth function on $\hat X$ to be a continuous function on $\hat X$ whose restriction to $\hat X^{\reg} = X^{\circ}$ extends to a smooth function on $X.$ We can study these functions on $X$ without reference to $\hat X,$
\begin{equation*}
	\CI_{\Phi}(X) = \{ f \in \CI(X): f\rest{\bhs{Y}} \in \phi_Y^*\CI(Y) \Mforall Y\}.
\end{equation*}
The differentials of these functions locally span a vector bundle, known as the {\em wedge cotangent bundle},
\begin{equation*}
	{}^{\w} T^*X \lra X,
\end{equation*}
described in more detail below.

Following the paradigm of \cite{tapsit}, we can think of the present paper as carrying out geometric analysis in the `wedge category' where, e.g., the r\^ole of the contangent bundle is supplanted by the wedge cotangent bundle. For example, our metrics will be Riemannian metrics on the interior that extend to bundle metrics on the wedge (co)tangent bundle, and our Clifford bundles will have Clifford actions over the interior that extend to actions of the Clifford algebra of the wedge cotangent bundle. 

In contrast to the situation in \cite{tapsit}, the dual bundle ${}^{\w}TX,$ known as the wedge tangent bundle, does not have a natural Lie bracket; this leads us to define `wedge differential operators' in terms of `edge differential operators'. 
A related issue is that wedge differential operators generally have multiple extensions from smooth sections of compact support to closed operators on $L^2$-spaces. Indeed, formally self-adjoint operators will often have closed extensions that are not self-adjoint.

Thus our first task in studying Dirac-type wedge operators is to
identify a well-behaved closed extension. We carry this out under the
assumption that certain model operators analogous to the `boundary
families' in \cite{BC1990, BC1990II} are invertible; however as in
\cite{Melrose-Piazza:Even,Melrose-Piazza:Odd} we include appropriate
pseudodifferential perturbations and we will show in a forthcoming
paper that as in {\em loc. cit.} this is a `minimal' assumption in that it corresponds
to the vanishing of a topological obstruction (see Remark \ref{rmk:PertExist}).

\begin{theorem}
Let $X$ be a manifold with corners and an iterated fibration structure endowed with a totally geodesic wedge metric, let $(E, g_E, \nabla^E, \cl)$ be a wedge Clifford module with associated Dirac-type operator $\eth_X$  and let $Q_X$ be a compatible perturbation.
If $\eth_{X,Q} = \eth_X+Q_X$ satisfies the Witt assumption (i.e., has
invertible boundary families, see Section \ref{sec:Wittsec} below, in
particular Definition \ref{def:WittAss}) then $\eth_{X,Q}$ with its vertical APS domain
\begin{equation}
	\cD_{\VAPS}(\eth_{X,Q}) = \text{ graph closure of }
	\{ u \in \rho_X^{1/2}H^1_\e(X;E) : \eth_{X,Q}u \in L^2(X;E)\},
\end{equation}
\nomenclature[D]{$\cD_{\VAPS}(\eth_{X,Q})$}{The vertical APS domain of
a Dirac-type operator}
is a closed operator on $L^2(X;E)$ that is self-adjoint and Fredholm with compact resolvent.

The heat kernel of $\eth_{X,Q}^2$ with the induced domain is trace-class and has a short-time asymptotic expansion of the form
\begin{equation*}
	\Tr(e^{-t\eth_{X,Q}^2}) \sim
	t^{-(\dim X)/2} \sum_{j=0}^{\infty} \sum_{k=0}^{\mathrm{depth}(X)}
	a_{j,k} t^{j/2} (\log t)^{k/2}.
\end{equation*}
\end{theorem}

We refer the reader to the text for precise definitions and statements. For example the trace of the heat kernel has fewer powers of $\log t$ for small values of $j,$ see Corollary \ref{cor:ShortTimeTraceExp}. 
The nomenclature `Witt assumption' stems from the case of the signature operator which satisfies this assumption if and only if $X$ is a Witt space in the sense of \cite{Siegel, C1979}, see, e.g., \cite{ALMP:Witt}. As explained in a forthcoming companion paper, allowing for a compatible perturbation in our theorems means in particular that they apply to the signature operator on the more general class of `Cheeger spaces' studied in \cite{ALMP:Hodge, ALMP:Novikov} (originally introduced by Banagl \cite{BanaglShort} and known as L-spaces, see \cite{ABLMP}).
Below we also discuss the `geometric Witt assumption' assumed in most of these references and analyzed in the recent preprint \cite{Hartmann-Lesch-Vertman}.

There are immediate consequences to the spectral theory of such $\eth_{X,Q},$ for example:
\begin{corollary}
The spectrum of $(\eth_{X,Q}^2, \cD_{\VAPS})$ is a discrete subset of $\bbR^+$ that satisfies Weyl's law
\begin{equation*}
	\#\{\text{eigenvalues, with multiplicity, less than }\Lambda\}
	\sim
	\frac{\Vol(X)}{(4\pi)^{\tfrac12\dim X}\Gamma(1+\tfrac12\dim X)}\Lambda^{\tfrac12\dim X}.
\end{equation*}
The zeta function of $\eth_{X,Q}^2,$ $\zeta(s) 
= \Tr((\eth_{X,Q}\rest{\ker(\eth_{X,Q})^{\perp}})^{-2s}),$ is holomorphic on $\{ s \in \bbC: \Re(s)>\tfrac12\dim X\}$ and extends to a meromorphic function on $\bbC$ with poles of order at most $1+ \mathrm{depth}(X).$
\end{corollary} 

$ $\\
In light of the theorem above, a family of wedge Dirac-type operators $\eth_{M/B}$ on the fibers of a fiber bundle, $X \fib M \xlra{\psi} B,$ of manifolds with corners and iterated fibration structures determines an index in $\tK^{\dim X}(B).$ Our second main result is a formula for the Chern character of this index.

Recall that Bismut and Cheeger, in their study of the families index theorem for manifolds with boundary \cite{BC1990, BC1990II,BC1991}, established a formula for the Chern character of the index of a family of spin Dirac operators on even-dimensional spaces with isolated conic singularities and invertible boundary families. Namely, if $M \lra B$ is the resolution of these spaces to manifolds with boundary,
\begin{equation*}
	\Ch_{\ev}(\Ind(\eth_{M/B}))
	= \int_{M/B} \hat A(M/B) \Ch'(E) - \cJ(\pa M/B) \; \Min \tH^{\ev}(B),
\end{equation*}
where $\cJ$ is a differential form depending globally on the geometry of the boundary fibration $\pa M \lra B.$
Our theorem involves Bismut-Cheeger $\cJ$-forms and $\eta$-forms extended to allow for Dirac-type operators on singular spaces and compatible smoothing perturbations.

\begin{theorem} \label{thm:FamIndInt}
Let $M \xlra\psi B$ be a fiber bundle of manifolds with corners and iterated fibration structures, $E \lra M$ a wedge Clifford bundle with associated Dirac-type operator $\eth_{M/B}$ and $Q$ an admissible perturbation.
If $\eth_{M/B,Q}$ with its vertical APS domain satisfies the Witt assumption, then
\begin{multline*}
	\Ch_{\dim(M/B)}(\Ind(\eth_{M/B,Q}), \nabla^{\Ind})
	= 
	\int_{M/B} \hat A(M/B) \Ch'(E) 
	- \sum_{N \in \cS_{\psi}(M)} 
	\int_{N/B} \hat A(N/B) \cJ_Q(\bhs{N}/N) \\
	+ d\eta_Q(M/B),
\end{multline*}
where $\Ch_{j}$ denotes the even or odd Chern character in accordance
with the parity of $j,$ and $\dim(M/B) = \dim X.$ The sum is taken
over the set $\cS_\psi(M)$ of boundary hypersurfaces of $M$ transverse to the
map $\psi \colon M \lra B$, i.e. those which also fibre over $B$.  (In
the even dimensional case, the left hand side is the Chern character
obtained after stabilizing the null space of $\eth_{M/B,Q}$ to a
bundle over $B$ by compressing the Bismut superconnection; in the odd
case, the left hand side is the differential form obtained by
suspension and integration.)
\end{theorem}

Bismut and Cheeger obtained the index formula for a family of spin Dirac operators on manifolds with boundary with invertible boundary families by deforming a neighborhood of the conic singularities to a cylindrical end. In the process the $\cJ$-form was shown to converge to the Bismut-Cheeger $\eta$-form, introduced in \cite{BC1989}. Melrose and Piazza \cite{Melrose-Piazza:Even, Melrose-Piazza:Odd} used the b-calculus to establish the families index theorem for arbitrary families of Dirac-type operators by allowing appropriate pseudodifferential perturbations. The boundary contributions correspondingly depend on the perturbation.\\

We discuss the definition of the Bismut-Cheeger $\eta$ and $\cJ$ forms for families of compatibly perturbed Dirac-type operators on stratified spaces in \S\ref{sec:BCJeta}. This extends the definition of $\eta$-invariants for spaces with conic singularities in \cite[\S 8]{C1987} and for spaces with non-isolated conic singularities (i.e., stratified spaces of depth one) in \cite{Piazza-Vertman}. Heuristically, given a fiber bundle $\wc X \fib \wc{M} \xlra{\check\psi} \wc B,$ a vertical family of Dirac-type operators $\eth_{\wc M/\wc B},$ and a connection for $\check{\psi},$ the $\eta$-forms and $\cJ$-forms are both related to the heat kernel of a family of Dirac-type operators on $\wc{M} \times \bbR^+,$ but with different extensions of $\check{\psi}.$ Indeed,
\begin{equation*}
	\eta \longleftrightarrow 
	\left(
	\vcenter{
	\xymatrix{
	\wc X \ar@{-}[r] & \wc M \times \bbR^+ \ar[d] \\ & \wc B \times \bbR^+ }}
	\right),
	\quad\quad
	\cJ \longleftrightarrow 
	\left(
	\vcenter{
	\xymatrix{
	\wc X \times \bbR^+ \ar@{-}[r] & \wc M \times \bbR^+ \ar[d] \\ & \wc B }}
	\right)
\end{equation*}
where in the former case the $\bbR^+$ factor results in an `auxiliary Grassman variable' but does not change the fiber, while in the latter the fiber $\wc X \times \bbR^+$ is endowed with an exact conic metric $dr^2 + r^2g_{\wc X}.$

In order to find the relation between the $\eta$ and $\cJ$ forms associated to $\wc M \lra \wc B,$ we attach the cone over $\wc M$ to the boundary of a half-cylinder over $\wc M.$ 
That is, we form a `b-c suspension' (where b-c refers to b-metric and conic-metric) of $\wc M$ of the form
\begin{equation*}
	(\wc M \times [0,1]_s, g_{(\wc M \times [0,1]_s)/B}), \quad
	g_{(\wc M \times [0,1]_s)/B}
	= \begin{cases}
	ds^2 + s^2g_{\wc M/\wc B} & \Mnear s=0\\
	\frac{ds^2}{(1-s)^2} + g_{\wc M/\wc B} & \Mnear s=1
	\end{cases},
\end{equation*}
we extend $\eth_{\wc M/\wc B}$ to a family of Dirac-type operators
acting on the fibers of $\wc M \times [0,1]_s \lra \wc B$ and then,
using an extended families index formula together with some
characteristic form computations, conclude the following in Theorem
\ref{thm:Jeta}.  See Section \ref{sec:extended} for a detailed
description of the transgression forms herein.

\begin{theorem} \label{thm:JetaInt}
For a fiber bundle $\wc M \lra \wc B$ as above, the relation between the $\eta$ and $\cJ$-forms for wedge Dirac-type operators with a compatible perturbation $Q$ is
\begin{multline*}
	\cJ_Q(\wc M/\wc B) 
	-\bar\eta_Q(\wc M/\wc B) \\
	=\int_{\wc M/\wc B} T\hat A(\wc M/\wc B) \Ch'(E) 
	+ \sum_{ \wc N \in \cS_{\wc \psi}(\wc M) } 
	\int_{\wc N/\wc B} T\hat A( \wc N/\wc B) \cJ_Q(\bhs{\wc N}/\wc N)
	+ d\eta_{b-\w, Q},
\end{multline*}
where,  for any family $\wc L \lra \wc B$ as above, $T\hat A(\wc L/\wc
B)$ is the restriction to $\wc L\times \{0\}$ of a transgression
between the $\hat A$-forms on $\wc L \times \bbR^+$ corresponding to a
cylindrical metric and a conic metric, and $\eta_{b-\w,Q}$ denotes a
`b-wedge $\eta$-form'.
\end{theorem}

To prove Theorem \ref{thm:JetaInt} we extend Theorem \ref{thm:FamIndInt} to manifolds with corners with metrics that are of `wedge-type' at all collective boundary hypersurfaces but one, where they are of `b-type' (i.e., asymptotically cylindrical, albeit with a singular cross-section).

The formula for the exterior derivative of the $\eta$-forms is already in Theorem \ref{thm:FamIndInt}. Together with Theorem \ref{thm:JetaInt} this implies a formula for the exterior derivative of the $\cJ$-forms, namely
\begin{equation*}
	d\cJ_Q(\wc M/\wc B)
	=\int_{\wc M/\wc B} \hat A_c(\wc M/\wc B) \Ch'(E) 
	+ \sum_{ \wc N \in \cS(\wc  M)}
	\int_{\wc N/\wc B} \hat A_c( \wc N/ \wc B) \cJ_Q(\bhs{\wc N}/\wc N).
\end{equation*}
Here, for any $\wc L,$ $\hat A_c(\wc L/\wc B)$ is the restriction to $\wc L \times \{0\}$ of the $\hat A$-form on $\wc L \times \bbR^+$ corresponding to a conic metric.

\subsection*{Previous results}

The index theorem of Atiyah-Singer for closed manifolds \cite{Atiyah-Singer:0} was soon generalized to operators on manifolds with boundary admitting local elliptic boundary conditions by Atiyah and Bott \cite{Atiyah-Bott} and later extended to operators admitting global elliptic boundary conditions by Atiyah, Patodi, and Singer \cite{APSI}. The resulting formula for a Dirac-type operator $\eth$ on a Clifford bundle $E \lra M$ takes the form
\begin{equation*}
	\ind(\eth) = \int_M \hat A(M)\Ch'(E) -\eta(\eth_{\pa M})
\end{equation*}
where $\hat A$ is the A-hat genus of $M,$ $\Ch'(E)$ the twisting curvature of the Clifford bundle, and $\eta(\eth_{\pa_M})$ the $\eta$-invariant of the induced Dirac-type operator on the boundary. 
The boundary conditions in this paper involve the projection onto the sum of the eigenspaces corresponding to positive eigenvalues of an induced boundary operator and are now known as Atiyah-Patodi-Singer, or APS, boundary conditions. 

Already in \cite{APSI} it was pointed out that the resulting domain had a natural interpretation as the domain of an operator on a complete non-compact manifold obtained by attaching a cylinder to the boundary. Melrose \cite{tapsit} showed that in fact one can consider the APS index theorem as the index theorem in the `category' of manifolds with asymptotically cylindrical end. The present project is analogous to Melrose's treatment, we obtain the index theorem in the `category' of manifolds with corners and iterated fibration structures endowed with wedge metrics.

Cheeger's proof of the Ray-Singer conjecture connecting analytic and Reidemesiter torsion \cite{C1977, C1979A} (established independently by M\"uller \cite{Muller:AT}), led him to develop analysis on spaces with singularities, particularly what we refer to as stratified spaces with wedge metrics. Cheeger realized that the Atiyah-Patodi-Singer index formula can be obtained as the natural index theorem in the context of manifolds with conic singularities, see \cite{C1979, C1980, C1983} for the signature and Gauss-Bonnet theorem both for isolated singularities and for piecewise flat metrics on stratified spaces. Chou \cite{Ch1985, Ch1989} showed for isolated conic singularities that the same was true for the Dirac operator. 

For a fiber bundle of closed operators $M \xlra{\psi} B,$ Atiyah and Singer \cite{Atiyah-Singer:IV} showed that the index theorem generalizes to families of operators on the fibers of $\psi.$ They showed that such a family has an index in the form of a virtual bundle over $B$ and, among other things, they computed the Chern character of this index bundle. Bismut \cite{Bismut:ASFam} used heat equation methods to establish the formula for the Chern character of the index bundle.

In \cite{W1985}, Witten derived a formula for the $\eta$ invariant of a manifold fibering over a circle. This formula was established rigorously by Bismut and Freed \cite{BFI, BFII} for spin Dirac operators and Cheeger \cite{C1987} for the signature operator. This was generalized by Bismut and Cheeger \cite{BC1989} and Dai \cite{Dai:adiabatic} to fiber bundles of arbitrary closed manifolds. They considered the behavior of the $\eta$ invariant as the metric on the total space of the fiber bundle undergoes an `adiabatic limit' in which the fibers are collapsed to a point. The limit involves a `higher' version of the $\eta$-invariant, known as the Bismut-Cheeger $\eta$-forms. These forms are of even degree if the dimension of the fiber is odd and of odd degree if the dimension of the fiber is even.

The r\^ole played by the $\eta$-invariant in the Atiyah-Patodi-Singer index theorem for manifolds with boundary is played by the $\eta$-forms in the formula for the Chern character of the index bundle of a family of Dirac-type operators on a manifold with boundary. This was established by Bismut and Cheeger for spin Dirac operators on even dimensional manifolds with boundary \cite{BC1990, BC1990II, BC1991} assuming invertibility of an induced boundary family of Dirac operators (referred to as the `Witt assumption' below). Melrose and Piazza \cite{Melrose-Piazza:Even,Melrose-Piazza:Odd} proved a formula conjectured in \cite{BC1991} for the case of odd dimensional manifolds with boundary, extended the Bismut-Cheeger result to Dirac-type operators, and removed the assumption of invertible boundary families. They introduced the notion of a `spectral section' of the boundary family and proved an index theorem for each spectral section.

(In the present paper we make the assumption of invertible boundary families as in the work of Bismut-Cheeger mentioned above, but we allow perturbations as in the approach of Melrose-Piazza. In a subsequent paper we will characterize the existence of these perturbations in terms of spectral sections.)

The approach adopted by Bismut and Cheeger to establish the families index theorem for Dirac operators on manifolds with boundary was to attach a cone to the boundary and consider a family of metrics, parametrized by $\eps \in [0,1],$ interpolating between the conic singularity and an infinite cylindrical end. An intermediate result is an index theorem for families of Dirac operators on spaces with isolated conic singularities. One effect of the $\eps$ degeneration is to `scale away' the small eigenvalues of the boundary family of Dirac operators so that the Dirac operators on the spaces with conical singularities are essentially self-adjoint and there is no need to choose a domain. The Chern character of the index bundle of the family of Dirac operators on spaces with isolated conic singularities involves another differential form invariant, the Bismut-Cheeger $\cJ$-forms. Bismut-Cheeger showed that the $\eps$-limit of the $\cJ$-form in this case is the $\eta$ form.\\

Index theory is now a vast field. Among the many ways in which singular spaces arise in index theory there are spaces arising from foliations (see, e.g., \cite{Connes-Skandalis, Bruning-Kamber-Richardson}) and group actions (see, e.g., \cite{Atiyah:EllOpsCmptGps}). 
Index theorems on complete metrics on manifolds with possibly fibered boundary include
\cite{Carron, Lauter-Moroianu:IndFibCusp, Melrose-Rochon:Cusp, Melrose-Rochon:FCusp, Vaillant, Melrose-Rochon:EtaForms, Albin-Melrose:Fred1, Albin-Melrose:Fred2, Albin-Melrose:RelChern, Albin-Rochon:DFam, Leichtnam-Mazzeo-Piazza, Hunsicker, Piazza:OnInd, Melrose-Nistor:Hom}.
Index theory on manifolds with corners endowed with complete metrics have been studied in, e.g.,  \cite{Melrose-Piazza:K, Melrose-Nistor:K, Lauter-Moroianu:IndCuspCorners, Loya, Bunke:IndThy, Monthubert-Nistor, Stern, Muller:L2Ind, Hassell-Mazzeo-Melrose:Sign}.

There is a powerful groupoid approach to index problems on the interior of  manifolds with corners, see e.g., \cite{AmmanLauterNistor, Debord-Lescure-Rochon, Carvalho-Nistor, CarrilloLescureMonthubert, Bohlen-Schrohe} for complete metrics,  \cite{Debord-Lescure-Nistor} for isolated conic singularities, and \cite{Debord-Skandalis} for a Boutet de Monvel type calculus.

An approach of Nazaikinskii, Yu, Savin, Sternin, and Schulze, see \cite{NSSYS:IndEdges, NSSYS:Book, Savin-Sternin:IndStrat} in the stratified setting proceeds by decomposing the index of a pseudodifferential operator on a stratified space into a sum of contributions from each stratum with the property that each is a homotopy invariant of the symbol. These papers make use of the analytic tools developed by Schulze and his collaborators, see e.g., \cite{Schulze, Schrohe-Schulze:BVPII}. Our treatment benefits from recent advances in parallel analytic tools in, e.g., \cite{Albin-Rochon-Sher:Wedge, Mazzeo-Vertman:Edge, Mazzeo-Vertman:AT, Krainer-Mendoza:BVP, GellRedman-Swoboda, Mazzeo-Witten:KW}.\\

Already in \cite{Atiyah:Global, Singer:Future} Atiyah and Singer called for the development of index theory on stratified spaces such as algebraic varieties. 
Index theory on spaces with isolated conic singularities is now very well understood see, e.g., \cite{Lesch:Fuchs, Fox-Haskell, Bruning-Lesch:Alg, Bruning-Seeley:Index}.
However, there are few explicit index formul\ae{} associated to singular metrics on stratified spaces beyond the case of isolated conic singularities. For a stratified space with a single singular stratum and a wedge metric there are index theorems for the signature operator by Chan, Hunsicker-Mazzeo, Br\"uning, and Cheeger-Dai \cite{Chan,  Hunsciker-Mazzeo, Bruning:Signature, Cheeger-Dai}. The $\eta$ and $\rho$ invariants have been studied by Piazza and Vertman \cite{Piazza-Vertman}. Atiyah and LeBrun \cite{AL2013} obtain an index theorem on a smooth four dimensional manifold endowed with a singular wedge metric. Lock and Viaclovsky \cite{LV2013} prove an index theorem for anti-self-dual orbifold-cone metrics, again in four dimensions. In previous work, the authors \cite{Albin-GellRedman} proved an index theorem for spin Dirac operators satisfying the geometric Witt condition.

$ $\\
{\bf Acknowledgements}
The first author was partially supported by a Simons Foundation grant \#317883 and an NSF grant DMS-1711325, and is grateful to M.I.T. for its hospitality during the completion of this work.
The authors are happy to acknowledge helpful conversations with
Rafe Mazzeo,
Richard Melrose,
Gerardo Mendoza, and
Paolo Piazza.

\tableofcontents

\paperbody

\section{Families of Dirac-type wedge operators}

\subsection{Iterated fibration structures and wedge geometry on manifolds with corners} \label{sec:IFS}

Let $X$ be an $n$ dimensional manifold with corners, by which we mean an $n$ dimensional topological manifold with boundary with a smooth atlas modeled on $(\bbR^+)^n$ whose boundary hypersurfaces are embedded.
We denote the set of boundary hypersurfaces of $X$ by $\cM_1(X).$

A {\textbf{ collective boundary hypersurface}} refers to a finite union of non-intersecting boundary hypersurfaces.

\begin{definition}[{Melrose \cite{Albin-Melrose:Gps, ALMP:Witt, Albin:Hodge}}] \label{def:IFS} An {\em iterated fibration structure} on a manifold with corners $X$ consists of a collection of fiber bundles
\begin{equation*}
	Z_Y \fib \bhs{Y} \xlra{\phi_Y} Y
\end{equation*}
where $\bhs{Y}$ is a collective boundary hypersurface of $X$ with base and fiber manifolds with corners such that:\\
\begin{itemize}
\item[i)] Each boundary hypersurface of $X$ occurs in exactly one collective boundary hypersurface $\bhs{Y}.$
\item[ii)] If $\bhs{Y}$ and $\bhs{\wt Y}$ intersect, then $\dim Y \neq \dim \wt Y,$ and we write
\begin{equation*}
	Y < \wt Y \Mif \dim Y < \dim \wt Y.
\end{equation*}
\item[iii)] If $Y < \wt Y$ then $\wt Y$ has a collective boundary hypersurface $\bhs{Y\wt Y}$ participating in a fiber bundle 
$\phi_{Y\wt Y}:\bhs{Y\wt Y} \lra Y$ such that the diagram
\begin{equation}\label{eq:bfsDiag}
	\xymatrix{
	\bhs{Y} \cap \bhs{\wt Y} \ar[rd]_-{\phi_Y} \ar[rr]^-{\phi_{\wt Y}} & & 
		\bhs{Y\wt Y} \ar[ld]^-{\phi_{Y\wt Y}} \ar@{}[r]|-*[@]{\subseteq} & \wt Y \\
	& Y & }
\end{equation}
commutes.
\end{itemize}
\end{definition}

Unless stated otherwise, we will assume that $X$ is compact and that $\dim  Z_Y >0$ for all $Y.$\\

There is no real loss of generality in assuming that the bases are connected, but the fibers of the boundary fibrations will generally be disconnected.

There is a functorial equivalence between Thom-Mather stratified spaces and manifolds with corners and iterated fibration structures see, e.g., \cite{ALMP:Witt},\cite[Theorem 6.3]{Albin:Hodge}. Under this equivalence, the bases of the boundary fibrations correspond to the different strata. We will denote this set by
\begin{equation*}
	\cS(X) = \{ Y : Y \text{ is the base of a boundary fibration of } X \}.
\end{equation*}
Both the bases and fibers of the boundary fiber bundles themselves are
manifolds with corners and iterated fibration structures, see e.g.,
\cite[Lemma 3.4]{Albin-Melrose:Gps}. The assumption $\dim
Z_Y>0$ corresponds to the category of pseudomanifolds within the
larger category of stratified spaces.

The partial order on $\cS(X)$ gives us a notion of depth
\begin{equation*}
	\mathrm{depth}_X(Y)
	= \max\{ n \in \bbN_0: \exists Y_i \in \cS(X) \Mst Y=Y_0<Y_1 < \ldots <Y_n \}.
\end{equation*}
The depth of $X$ is then the maximum of the integers
$\mathrm{depth}_X(Y)$ over $Y \in \cS(X).$  \\

If $H$ is a boundary hypersurface then, because it is assumed embedded, there is a non-negative function $\rho_H$ such that
\begin{equation*}
	\rho_H^{-1}(0) = H \Mand |d\rho_H| \neq 0 \Mon H,
\end{equation*}
where $| \cdot |$ denote the norm with respect to some smooth
background metric on $M$; we call any such function a {\textbf{ boundary defining function}} for $H.$
It is always possible (see, e.g., \cite[Proposition 1.2]{Albin-Melrose:Gps}) to choose: a boundary defining function $\rho_H$ for each $H \in \cM_1(X),$ an open neighborhood $\cU_H \subseteq X$ of each $H,$ and a smooth vector field $V_H$ defined in  $\cU_H$ such that 
\begin{equation*}
\begin{gathered}
	V_H\rho_K 
	= \begin{cases} 
	1 \Min \cU_H & \Mif K=H\\
	0 \Min \cU_H\cap \cU_K & \Mif K\neq H
	\end{cases} \\
	[V_H, V_K]=0 \Min \cU_H\cap \cU_K
\end{gathered}
\end{equation*}
for all $H,K \in \cM_1(X).$ We refer to these choices as a {\bf boundary product structure,} and will always assume that our boundary defining functions are chosen this way.

For each $Y \in \cS(X)$ we denote a collective boundary defining function by
\begin{equation*}
	\rho_Y = \prod_{H \in \bhs{Y}} \rho_H,
\end{equation*}
we also use the notation
\begin{equation*}
	\rho_X = \prod_{H \in \cM_1(X)} \rho_H
\end{equation*}
for a {\em total boundary defining function.}

A boundary product structure allows us to extend an iterated fibration structure to a {\bf collared iterated fibration structure.} 
Indeed, let us assume for simplicity of notation that the neighborhoods $\cU_H$ coincide with $\rho_H^{-1}([0,1)),$ so that $\cU_H \cong [0,1)_{\rho_H}\times H.$ For each $Y \in \cS(X),$ we write
\begin{equation*}
	\sC(\bhs{Y}) 
	= \bigcup_{H \in \bhs{Y}} \cU_H \cong [0,1)_{\rho_Y} \times \bhs{Y}, \quad
	Y^+ = [0,1)_{\rho_Y} \times Y
\end{equation*}
and we denote the natural extension of $\phi_Y$ by
\begin{equation*}
	\phi_{Y^+}: \sC(\bhs{Y}) \lra Y^+.
\end{equation*}
Choosing compatible boundary product structures on each $Y \in \cS(X)$ (existence is checked by a simple induction on the depth of $X$), the extended boundary fibrations participate in commutative diagrams,
\begin{equation}\label{eq:bfsDiagExt}
	\xymatrix{
	\sC(\bhs{Y}) \cap \sC(\bhs{\wt Y}) 
		\ar[rd]_-{\phi_{Y^+}} \ar[rr]^-{\phi_{\wt Y^+}} & & 
	\sC(\bhs{Y\wt Y}) 
		\ar[ld]^-{\phi_{(Y\wt Y)^+}} \ar@{}[r]|-*[@]{\subseteq} & \wt Y^+ \\
	& Y^+ & }
\end{equation}
whenever $Y<\wt Y.$

This structure will be useful when we discuss Getzler rescaling below
(\S \ref{sec:Getzler}). In that setting, we will have a filtration of
a vector bundle defined on collective boundary hypersurfaces and we
will need to extend it into a neighborhood of the boundary
consistently; a collared iterated fibration structure makes this easy
to do. 
\\

Various differential geometric objects have natural analogues that take the iterated fibration structure into account.
For example, we define 
\begin{equation*}
	\CI_{\Phi}(X) = \{ f \in \CI(X) : f\rest{\bhs{Y}} \in \phi_Y^*\CI(Y) \Mforall Y \in \cS(X) \}.
\end{equation*}
(This corresponds to the smooth functions on $X$ that are continuous on the underlying stratified space. If an open cover of $X$ is the lift of a cover of the underlying stratified space, then there is a compatible partition of unity in $\CI_{\Phi}(X)$ see, e.g., \cite[Lemma 5.2]{ABLMP}.)

The {\em edge vector fields} on $X$ \cite{Mazzeo:Edge} are
\begin{equation*}
	\cV_\e = \{ V \in \CI(X;TX) : 
	V\rest{\bhs{Y}} \text{ is tangent to the fibers of } \phi_Y \Mforall Y \in \cS(X) \},
\end{equation*}
or, equivalently, they are the b-vector fields (vector fields tangent to the boundary) that when applied to $\CI_{\Phi}(X)$ yield functions that vanish at the boundary of $X.$

There is a vector bundle, the {\em edge tangent bundle}, ${}^\e TX,$ together with a natural vector bundle map $i_\e: {}^\e TX \lra TX$ that is an isomorphism over the interior and satisfies
\begin{equation*}
	(i_\e)_*\CI(X;{}^\e TX) = \cV_\e.
\end{equation*}
In local coordinates near a point in $\bhs{Y},$ $(x, y, z),$ where $x$ is a bdf, $y$ denotes coordinates along $Y,$ and $z$ denotes coordinates along $Z,$ a local frame for ${}^\e TX$ is given by
\begin{equation*}
	x\pa_x, \quad x\pa_y, \quad \pa_z.
\end{equation*}
Note that the vector fields $x\pa_x$ and $x\pa_y$ are degenerate as sections of $TX$ but not as sections of ${}^\e TX.$
If $Y= \{\pt\}$ then the edge tangent bundle coincides with the b-tangent bundle discussed in \S\ref{sec:Conormal}.

The universal enveloping algebra of $\cV_{\e}$ is the ring $\Diff^*_\e(X)$ of edge differential operators \cite[\S2]{Mazzeo:Edge}. Thus these are the differential operators on $X$ that can be expressed locally as finite sums of products of elements of $\cV_{\e}.$ They have the usual notion of degree and extension to sections of vector bundles, as well as an edge symbol map defined on the edge cotangent bundle, see \cite{Mazzeo:Edge,ALMP:Witt, ALMP:Hodge}.

\begin{remark}
In \cite{ALMP:Witt, ALMP:Hodge, ALMP:Novikov} the edge tangent bundle was referred to as the `iterated edge tangent bundle'. We prefer to think of it as the edge tangent bundle of the iterated fibration structure. Similarly, the wedge tangent bundle, defined below, was referred to in {\em loc. cit.} as the `iterated incomplete edge tangent bundle'.
\end{remark}

Among the metrics most closely associated to these spaces are metrics
that degenerate conically reflecting the conic degeneration of the
space.  We will define these formally in Section \ref{sec:Metrics},
but roughly speaking their are the geometric object arising from
iterating the process of taking the geometric cones of spaces and
local bundles of such cones.  Metrics of this form, {\em wedge
  metrics}, are singular at the boundary of $X.$ However, they can be
seen as smooth (or more generally $\cI$-smooth or polyhomogeneous) sections of a
{\em rescaled} bundle of symmetric tensors.

Formally, we proceed as follows. Let $X$ be a manifold with corners and iterated fibration structure. Consider the `wedge one-forms'
\begin{equation*}
	\cV^*_\w = \{ 
	\omega \in \CI(X;T^*X): \Mforeach Y \in \cS(X), \;
	i^*_{\bhs{Y}}\omega(V) =0 \Mforall V \in \ker D\phi_Y \}.
\end{equation*}
Just as we have done with the edge tangent bundle, we can identify $\cV_{\w}^*$ with the space of sections of a vector bundle.
That is, there exists a vector bundle ${}^\w T^*X,$ the wedge cotangent bundle of $X,$ together with a bundle map
\begin{equation}\label{eq:wedge-cot}
	i_\w: {}^\w T^*X \lra T^*X
\end{equation}
\nomenclature[T]{${}^\w T^*X$}{The wedge cotangent bundles of $X$, a
  manifold with corners with iterated fibration structure.  Similarly
  ${}^\w TX$ denotes the wedge tangent bundle}
that is an isomorphism over the interior of $X$ and such that 
\begin{equation*}
	(i_\w)_*\CI(X;{}^\w T^*X) = \cV_{\w}^* \subseteq \CI(X;T^*X).
\end{equation*}
In particular, in local coordinates near the collective boundary hypersurface $\bhs{Y}$ the wedge cotangent bundle is spanned by
\begin{equation*}
	dx, \quad x dz, \quad dy
\end{equation*}
where $x$ is a boundary defining function for $\bhs{Y},$ $dz$ represents covectors along the fibers and $dy$ covectors along the base.

The dual bundle to the wedge cotangent bundle is the {\em wedge tangent bundle}, ${}^\w TX.$ It is locally spanned by 
\begin{equation*}
	\pa_x, \quad \tfrac1x\pa_z, \quad \pa_y.
\end{equation*}
A \textbf{wedge metric} is simply a bundle metric on the wedge tangent
bundle.  Below we will make further assumptions on the metric, see
Section \ref{sec:Metrics}.

Wedge differential operators are defined in terms of edge differential operators: $P$ is a wedge differential operator of order $k$ acting on sections of a vector bundle $E$ if $\rho_X^kP$ is an edge differential operator of order $k$ acting on sections of $E,$
\begin{equation}\label{eq:DefWedgeDiff}
	\Diff_\w^k(X;E) = \rho_X^{-k}\Diff_{\e}^k(X;E).
\end{equation}
See, e.g., \cite{Mazzeo-Vertman:Edge, Gil-Krainer-Mendoza:Closure, ALMP:Witt, ALMP:Hodge}.\\

By a smooth family of manifolds with corners and iterated fibration structures we will mean first of all a fiber bundle
\begin{equation*}
	X \fib M \xlra{\psi} B
\end{equation*}
where $X,$ $M,$ and $B$ are manifolds with corners. 
Since $M$ is locally diffeomorphic to $X \times \cU,$ $\cU\subseteq B,$ every boundary hypersurface of $M$ corresponds to either a boundary hypersurface of $B$ or a boundary hypersurface of $X.$ The latter are the boundary hypersurfaces that are transverse to $\psi.$
We want the fibers of $\psi$ to have iterated fibration structures that themselves vary smoothly. 
We formalize this as follows.

\begin{definition}\label{def:FibIFS}
A {\bf locally trivial family of manifolds with corners and iterated fibration structures over $B$} is:\\
\begin{itemize}
\item a fiber bundle of manifolds with corners $X \fib M \xlra{\psi}
  B,$
\item a partition of the boundary hypersurfaces of $M$ transverse to
  $\psi$, which we denote by $\cS_{\psi}(M) \subset \cS(M)$, into collective boundary hypersurfaces $\{ \bhs{N}: N \in
  \cS_{\psi}(M) \},$ where each $N$ is a manifold with corners endowed
  with a fiber bundle map $N \xlra{\psi_N} B,$
\item a collection of fiber bundles $Z_N \fib \bhs{N} \xlra{\phi_N} N$
  satisfying Definition \ref{def:IFS}(ii-iii),
\end{itemize}
satisfying that, for all $N \in \cS_{\psi}(M),$ the diagram
\begin{equation}\label{eq:collective-bhss}
	\xymatrix{
	\bhs{N} \ar@{^(->}[r] \ar[d]_-{\phi_N}  & M \ar[dd]^-{\psi} \\
	N \ar[rd]_-{\psi_N} & \\
	& B  }
\end{equation}
commutes.
\nomenclature[S]{$\cS_\psi(M)$}{The boundary hyperfaces of a manifold
  with corners $M$ that are transverse to a fibration $\psi \colon M
  \lra B$}
\end{definition}

For each $b \in B$ the fiber of $\psi:M \lra B,$ $X = \psi^{-1}(b)$ has a iterated fibration structure with
\begin{equation*}
	\cS(X) = \{ Y_b = \psi_N^{-1}(b) : N \in \cS_{\psi}(M) \}
\end{equation*}
and boundary fiber bundles determined by the diagrams, one for each $N \in \cS_{\psi}(M),$
\begin{equation*}
	\xymatrix{ & & X \ar@{-}[r] & M \ar[d]^-{\psi} \\
	Z \ar@{-}[r] & \bhs{Y} \ar[d]_-{\phi_Y} \ar@{-}[r] \ar@{^(->}[ru] 
		& \bhs{N} \ar[d]_-{\phi_N} \ar@{^(->}[ru] & B\\
	& Y \ar@{-}[r] & N \ar[ru]_-{\psi_N} & }
\end{equation*}
In particular, $\bhs{Y}$ is the typical fiber of $\phi_N \circ
\psi_N$.  We will always use $\bhs{N}$ to denote collective boundary hypersurfaces of $M$ and $\bhs{Y}$ to denote collective boundary hypersurfaces of $X$ and hope the similar notation does not cause confusion. \\

Analogously to what we have done before, the $\psi$-wedge one forms are the covectors on $M$ that vanish on vertical vectors at all boundary hypersurfaces transverse to $\psi,$
\begin{equation*}
	\cV_{\w(\psi)}^* = 
	\{\omega \in \CI(M;T^*M): \Mforeach N \in \cS_{\psi}(M),
	i_{\bhs{N}}^*\omega(V) = 0 \Mforall V \in \ker D\phi_N \},
\end{equation*}
and can be identified with the sections of a vector bundle, ${}^{\w(\psi)}T^*M.$
We refer to this as the `$\psi$-wedge cotangent bundle' and to the dual bundle ${}^{\w(\psi)}TM,$ as the `$\psi$-wedge tangent bundle'.
The latter has a sub-bundle determined by its sections,
\begin{equation*}
	\CI(M; {}^{\w}TM/B)
	 = \{ V \in \CI(M;{}^{\w(\psi)}TM) : (\rho_XV)(\psi^*_Bf)=0 \Mforall f \in \CI(B)\},
\end{equation*}
where $\rho_X V \in \CI(M;TM)$ acts by differentiation,
which we will call the {\em vertical wedge tangent bundle.} The {\em vertical wedge cotangent bundle} is the dual bundle ${}^{\w}T^*M/B \lra M.$ A choice of connection for $M \xlra{\psi} B$ induces splittings
\begin{equation}\label{eq:Splittings}
	{}^{\w(\psi)} TM = {}^{\w} TM/B \oplus \psi^*TB, \quad
	{}^{\w(\psi)} T^*M = {}^{\w} T^*M/B \oplus \psi^* T^*B.
\end{equation}
%

\subsection{Totally geodesic wedge metrics}\label{sec:Metrics}

Let $M \xlra{\psi} B$ be a family of manifolds with corners and iterated fibration structures. A {\em vertical wedge metric} on $M$ will refer to a bundle metric on ${}^{\w}TM/B.$ For simplicity we will work with a subset of these metrics which we call `totally geodesic vertical wedge metrics'. For simplicity of notation, in this section we discuss these metrics on a fixed fiber $X = \psi^{-1}(b)$ of $\psi.$\\

We define totally geodesic wedge metrics inductively by the depth of the space.
If $X$ has depth zero, so is a smooth manifold, a wedge metric is simply a Riemannian metric.
Assuming we have defined totally geodesic wedge metrics at spaces of depth less than $k,$ let $X$ have depth $k.$
A wedge metric $g_{\w}$ on $X$ is a {\em totally geodesic wedge metric} if, for every $Y \in \cS(X)$ of depth $k$  there is a collar neighborhood $\sC(\bhs{Y}) \cong [0,1)_x \times \bhs{Y}$ of $\bhs{Y}$ in $X,$ a metric $g_{\w,\pt}$ of the form
\begin{equation}\label{eq:pt-canonical-form}
	g_{\w,\pt} = dx^2 + x^2 g_{\bhs{Y}/Y} + \phi_Y^*g_Y
\end{equation}
where $g_Y$ is a totally geodesic wedge metric on $Y,$ $g_{Z_Y} + \phi^*g_Y$ is a submersion metric for $\bhs{Y} \xlra{\phi_Y} Y,$ and $g_{Z_Y}$ restricts to each fiber of $\phi_Y$ to be a totally geodesic wedge metric on $Z_Y,$ and
\begin{equation}\label{eq:tg-canonical-form}
	g_{\w} - g_{\w, \pt} \in x^2 \CI(\cC(\bhs{Y}; S^2({}^{\w} T^*X)).
\end{equation}
Off of these collar neighborhoods the form of the metric is fixed by the induction.

If at every step $g_{\w} = g_{\w, \pt}$ we say that $g_{\w}$ is a {\em rigid} or {\em product-type} wedge metric.
If at every step $g_{\w} - g_{\w, \pt} = \cO(x)$ as a symmetric two-tensor on the wedge tangent bundle, we say that $g_{\w}$ is an {\em exact} wedge metric.
We will always work with totally geodesic wedge metrics.\\

It is worthwhile describing a product-type metric near boundary faces of arbitrary depth. As a simple example, the metric on the cone over a cone has the form
\begin{equation*}
	ds^2 + s^2(dr^2 + r^2g_Z).
\end{equation*}
This has the form above near $\{s=0\},$ but not near $\{r=0\}.$
However, from \eqref{eq:bfsDiag}, every point $\zeta \in \pa X$ lies over the interior of a unique $Y \in \cS(X).$ 
We can choose a boundary defining function $x$ for $\bhs{Y},$ and identify a neighborhood of $\zeta$ in $X$ with the form $[0,1)_x \times Z \times \cU$, where $\cU$ is a subset of $Y$ with closure contained in the interior of $Y,$ and then the metric takes the form
\begin{equation*}
	dx^2 + x^2g_Z + \phi_Y^*g_Y
\end{equation*}
in this neighborhood.
This is consistent with the description above since the boundary hypersurface of greatest depth in this neighborhood is $\bhs{Y}.$
A common theme throughout this work is to work over the interior of each $Y \in \cS(X),$ and make use of the fact that this exhausts all of $\pa X.$\\

Let us describe the asymptotics of the Levi-Civita connection of a totally geodesic wedge metric at $\bhs{Y}$ for $Y \in \cS(X).$
First let us start by recalling the behavior of the Levi-Civita connection of a submersion metric.
Endow $\bhs{Y}$ with a submersion metric of the form $g_{\bhs{Y}}= g_{\bhs{Y}/Y} + \phi^*g_Y.$
We denote the associated splitting of the tangent bundle $T\pa X$ by
\begin{equation*}
	T\bhs{Y} = T\bhs{Y}/Y \oplus \phi_Y^*TY
\end{equation*}
and the orthogonal projections onto each summand by
\begin{equation*}
	\bh:T\bhs{Y} \lra \phi^*TY, \quad \bv:T\bhs{Y} \lra T\bhs{Y}/Y.
\end{equation*}
Given a vector field $U$ on $Y,$ we denote its horizontal lift to $\bhs{Y}$ by $\wt U.$ 
The Levi-Civita connection of $(\bhs{Y}, g_{\bhs{Y}}),$
$\nabla^{\bhs{Y}},$ can be written in terms of the Levi-Civita
connections $\nabla^Y$ on the base and the connections
$\nabla^{\bhs{Y}/Y}$ on the fibers using two tensors: 1) the second fundamental form of the fibers, defined by
\begin{equation*}
	\cS^{\phi_Y}:T\bhs{Y}/Y \times T\bhs{Y}/Y \lra \phi^*TY, 
	\quad \cS^{\phi_Y}(V_1, V_2) = \bh(\nabla^{\bhs{Y}/Y}_{V_1}V_2)
\end{equation*}
and, 2) the curvature of the fibration, defined by
\begin{equation*}
	\cR^{\phi_Y}:\phi^*TY \times \phi^*TY \lra T\bhs{Y}/Y, 
	\quad \cR^{\phi_Y}(\wt U_1, \wt U_2) = \bv([\wt U_1, \wt U_2]).
\end{equation*}
The behavior of the Levi-Civita connection (cf. \cite[Proposition 13]{HHM2004}) is then summed up in the table:
\begin{equation*}
\begin{tabular}{|c||c|c|} \hline 
$g_{\bhs{Y}}\lrpar{\nabla^{\bhs{Y}}_{W_1} W_2, W_3}$ & $V_0$ & $\wt U_0$ \\ \hline\hline
$\nabla^{\bhs{Y}}_{V_1}V_2$ & 
	$g_{\pa X/Y}\lrpar{\nabla^{\bhs{Y}/Y}_{ V_1} V_2, V_0}$ & 
	$\phi^*g_Y(\cS^{\phi_Y}( V_1, V_2), \wt U_0)$ \\ \hline
$\nabla^{\bhs{Y}}_{\wt U} V$ & 
	$g_{\bhs{Y}/Y}\lrpar{[\wt U, V], V_0} - \phi_Y^*g_Y(\cS^{\phi_Y}( V,  V_0), \wt U)$ &
	$-\frac12g_{\bhs{Y}/Y}(\cR^{\phi_Y}(\wt U, \wt U_0), V)$ \\ \hline
$\nabla^{\bhs{Y}}_{ V} \wt U$ &
	$-\phi_Y^*g_Y(\cS^{\phi_Y}(V, V_0),\wt U)$ &
	$\frac12g_{\bhs{Y}/Y}(\cR^{\phi_Y}(\wt U, \wt U_0), V)$ \\ \hline
$\nabla^{\bhs{Y}}_{\wt U_1}\wt U_2$ &
	$\frac12g_{\bhs{Y}/Y}(\cR^{\phi_Y}(\wt U_1, \wt U_2), V_0)$ &
	$g_Y( \nabla^Y_{U_1} U_2, U_0)$ \\ \hline
\end{tabular}
\end{equation*}
$ $

We want a similar description of the Levi-Civita connection of a totally geodesic wedge metric.
We define an operator $\nabla$ on sections of the wedge tangent bundle through the usual Koszul formula
\begin{multline*}
	2g_{\w}(\nabla_{W_0}W_1, W_2) = W_0 g_{\w}(W_1, W_2) + W_1 g_{\w}(W_0, W_2) - W_2 g_{\w}(W_0, W_1) \\
	+ g_{\w}(\lrspar{W_0, W_1}, W_2) - g_{\w}(\lrspar{W_0, W_2}, W_1) - g_{\w}(\lrspar{ W_1, W_2}, W_0).
\end{multline*}
We will consider wedge metrics $g_{\w}$ and $g'_{\w}$ that differ by
$g_{\w}-g_{\w}' = \cO_{\w}(x^2)$, where the $\cO_{\w}(x^p)$ notation
reminds the reader that these are tensors on wedge vector fields.  In particular, for any such pair, if $W_i \in
C^\infty(X;\wT X)$, $i = 0, 1, 2$, then $g_{\w}(W_i, W_j) -
g_{\w}'(W_i, W_j) = \cO(x^2)$, and by the Koszul formula and the fact,
seen below, that $x [W_i, W_j]$ is a smooth wedge vector field,
\begin{equation}\label{eq:metric-diff}
	g_{\w}(\nabla_{W_0}W_1, W_2) - g_{\w}'(\nabla'_{W_0}W_1, W_2) = \cO(x).
\end{equation}
Thus to understand the leading asymptotics of the Levi-Civita connection
of a totally geodesic wedge metric, it suffices to understand the
leading asymptotics of a product-type wedge metric. 

Fix a product-type wedge metric $g_{\w,\pt}=dx^2 + x^2g_{\bhs Y/Y}
+\phi^*g_Y$ such that $g_{\w} - g_{\w,\pt} =x^2\wt g,$ the splitting
of the tangent bundle of $X$ associated to $g_{\bhs Y/Y} + \phi^*g_Y$
extends to a splitting of the wedge tangent bundle of $\sC(\bhs{Y})$
and hence induces a splitting
\begin{equation}\label{eq:SplittingIE}
	{}^\w T\sC(\bhs{Y}) = \ang{\pa_x} \oplus \tfrac1x  \wT\bhs{Y}/Y \oplus \phi_Y^*  TY.
\end{equation}
In terms of which a convenient choice of vector fields is
\begin{equation*}
	\pa_x, \quad \tfrac1xV, \quad \wt U
\end{equation*}
where $V$ denotes a $\phi_Y$-vertical wedge vector field at $\{x=0\}$
extended trivially to $\sC(\bhs{Y})$ and $\wt U$ denotes a wedge
vector field on $Y,$ lifted horizontally to $\bhs{Y}$ and then extended trivially to $\sC(\bhs{Y}).$
Note that, with respect to $g_{\w,\pt},$ these three types of vector fields are orthogonal, and that their commutators satisfy
\begin{equation*}
\begin{gathered}
	\lrspar{ \pa_x, \tfrac1xV } = -\tfrac1{x^2}V 
		\in x^{-1}\CI(\sC(\bhs{Y}), \tfrac1x\wT\bhs Y/Y), \quad
	\lrspar{ \pa_x, \wt U } = 0, \\
	\lrspar{ \tfrac1xV_1, \tfrac1xV_2} = \tfrac1{x^2} \lrspar{V_1, V_2} 
		\in x^{-1}\CI(\sC(\bhs{Y}), \tfrac1x\wT\bhs{Y}/Y), \\
	\lrspar{ \tfrac1xV, \wt U} = \tfrac1x\lrspar{ V, \wt U} \in 
		\CI(\sC(\bhs{Y}), \tfrac1x\wT\bhs{Y}/Y), \\
	\lrspar{ \wt U_1, \wt U_2} 
		\in x\CI(\sC(\bhs{Y}), \tfrac1x\wT\bhs{Y}/Y) 
		+ \CI(\sC(\bhs{Y}), \phi_Y^*TY).
\end{gathered}
\end{equation*}

\begin{remark}
It is important to understand that the inductive definition of
  product type and totally geodesic wedge vector fields does
  \emph{not} imply that for a fixed $g_{\w}$, every $\bhs{Y}$ has a
  collar neighborhood with a metric that of the canonical form above.
  Indeed, such a decomposition can only be assumed on the complement
  of those $\bhs{Y'}$ with $Y' < Y$.  Here we compute the asymptotics
  of the connection and curvature on a neighborhood where the
  decomposition holds, and when we work at a corner, i.e.\ an
  intersection of $\bhs{Y} \cap \bhs{Y'}$, we assume the decomposition
  holds only on a neighborhood of the deeper one.
\end{remark}

Below we will work with a local frame of wedge vector fields, orthogonal with respect to $g_{\w,\pt},$
\begin{equation}\label{eq:local-frame}
	\pa_x , \quad 
	\frac 1x V_\alpha, \quad 
	\wt{U}_i, \quad
	\alpha = 1, \dots, v_Y = \dim Z_Y, \quad 
	i = 1, \dots, h_Y = \dim Y,
\end{equation}
where the $V_\alpha$ are a a local frame of ${}^\w \wT\bhs{Y}/Y$, the $\wt{U}_i$ are the
horizontal lifts of a local frame $U_i$ of ${}^\w TY.$

If $W_1 \in \{ \pa_x, V_\alpha, \wt U_i\}$ and $W_2, W_3 \in \{\pa_x, \tfrac1x V_\alpha, \wt U_i \}$ then we find
\begin{multline*}
	g_{\w, \pt}(\nabla_{W_1}W_2, W_3) =0 \Mif \pa_x \in \{  W_1, W_2, W_3 \} \\
	\text{ except for }
	g_{\w,\pt}(\nabla_{V_1}\pa_x, \tfrac1x V_2) 
		= - g_{\w,\pt}(\nabla_{V_1}\tfrac1xV_2, \pa_x) =  g_{\bhs{Y}/Y}(V_1, V_2)
\end{multline*}
and otherwise 
\begin{equation*}
\begin{tabular}{|c||c|c|} \hline 
$g_{\w,\pt}\lrpar{\nabla_{W_1} W_2, W_3}$ & $\tfrac1xV_3$ & $\wt U_3$ \\ \hline\hline
$\nabla_{V_1}\tfrac1x V_2$ & 
	$g_{\bhs{Y}/Y}\lrpar{\nabla^{\bhs{Y}/Y}_{ V_1} V_2, V_3}$ & 
	$x\phi_Y^*g_Y(\cS^{\phi_Y}( V_1, V_2), \wt U_3)$ \\ \hline
$\nabla_{\wt U} \tfrac1xV$ & 
	$g_{\bhs{Y}/Y}\lrpar{[\wt U, V], V_3} - \phi_Y^*g_Y(\cS^{\phi_Y}( V,  V_3), \wt U)$ &
	$-\frac x2 g_{\bhs{Y}/Y}(\cR^{\phi_Y}(\wt U, \wt U_3), V)$ \\ \hline
$\nabla_{ V} \wt U$ &
	$-x\phi_Y^*g_Y(\cS^{\phi_Y}(V, V_3),\wt U)$ &
	$\frac{x^2}2g_{\bhs{Y}/Y}(\cR^{\phi_Y}(\wt U, \wt U_3), V)$ \\ \hline
$\nabla_{\wt U_1}\wt U_2$ &
	$\frac x2g_{\bhs{Y}/Y}(\cR^{\phi_Y}(\wt U_1, \wt U_2), V_3)$ &
	$g_Y( \nabla^Y_{U_1} U_2, U_3)$ \\ \hline
\end{tabular}
\end{equation*}

We point out a few consequences of these computations, valid for an arbitrary totally geodesic wedge metric.
First note that the operator
\begin{equation*}
	\nabla: \CI(X;\Tw X) \lra \CI(X; T^*X \otimes \Tw X)
\end{equation*}
defines a connection on the wedge tangent bundle. Thus, in particular,
the curvature tensor $R_{\w}$ of $\nabla$ is a well-defined 2-form on
all of $X$ with values in endomorphisms of $\wT X$.
Also note that this connection asymptotically preserves the splitting of $\Tw\sC(\bhs{Y})$ into two bundles
\begin{equation}\label{eq:SecondSplittingIE}
	\Tw\sC(\bhs{Y}) 
		= \lrspar{ \ang{\pa_x} \oplus \tfrac1x \wT\bhs{Y}/Y } \oplus \phi_Y^*TY
\end{equation}
in that if $W_1, W_2 \in \cV_{\w}$ are sections of the two different summands then
\begin{equation}\label{eq:moresplitting}
	g_{\w}(\nabla_{W_0} W_1, W_2) = \cO(x) \Mforall W_0 \in \CI(X;TX).
\end{equation}
Let us denote by
\begin{equation*}
	 \bv_+: \Tw\sC(\bhs{Y}) \lra \ang{\pa_x} \oplus \tfrac1x \wT\bhs{Y}/Y, \quad
         \bh_+ :\Tw\sC(\bhs{Y}) \lra \ang{\pa_x} \oplus \phi^*TY,
\end{equation*}
the orthogonal projections onto their images, while $\bv$ and $\bh$ continue to denote projection onto ${}^\w T\bhs{Y}/Y$ and
$\phi^*{}^\w TY$, and define connections
\begin{multline*}
	\nabla^{v_+} =  \bv_+ \circ \nabla \circ \bv_+: \CI(\sC(\bhs Y); \ang{\pa_x} \oplus \tfrac1x\wT\bhs Y/Y) \\
		\lra \CI(\sC(\bhs Y); T^*\sC(\bhs{Y}) \otimes \lrpar{ \ang{\pa_x} \oplus \tfrac1x\wT\bhs{Y}/Y } ), 
\end{multline*}
\begin{equation*}
	\nabla^h = \phi^*\nabla^Y : \CI(\sC(\bhs{Y}); \phi_Y^*TY) 
		\lra \CI(\sC; T^*\sC(\bhs{Y}) \otimes \phi_Y^*TY ) .
\end{equation*}
Denote by $j_{\eps}:\{x = \eps\} \hookrightarrow \sC(\bhs{Y})$ the inclusion, and identify $\{ x=\eps\}$ with $\bhs{Y}=\{x=0\},$
note that the pull-back connections $j_{\eps}^*\nabla^{v_+}$ and $j_{\eps}^*\nabla^h$ are independent of $\eps$ and  
\begin{equation}\label{eq:AsympSplittingConn}
	j_0^*\nabla = j_0^*\nabla^{v_+} \oplus j_0^*\nabla^h.
\end{equation}

Let us describe the asymptotics of the curvature.

\begin{proposition}\label{prop:curvature} $ $
Let $(X, g_{\w})$ be a manifold with corners and an iterated fibration structure endowed with a totally geodesic wedge metric.
Let $Y \in \cS(X)$ and let $x$ be a bdf for $\bhs{Y}$ in which the
canonical metric form 
decomposition
\eqref{eq:pt-canonical-form}--\eqref{eq:tg-canonical-form} holds.
\begin{enumerate}
\item If $W_1$, $W_2$ are vector fields tangent to $\bhs{Y}$ then
$R_{\w}(W_1, W_2)$ asymptotically preserves the splitting
\begin{equation*}
	\Tw\sC(\bhs{Y}) = \lrspar{ \ang{\pa_x} \oplus \tfrac1x\wT\bhs{Y}/Y } \oplus \phi_Y^*TY 
\end{equation*}
\item For $N = N(x) \partial_x, W_0 \in \CI(X;TX),$ with $W_0$ tangent to the fibers of $\phi_Y$, $W_1, W_2 \in \CI(X; {}^\w TX),$ we have
\begin{equation*}
\begin{multlined}
	g_{\w}(R_{\w}(N, W_0)W_1, W_2)
	= g_{\w}(R_{\w}(N, \bv W_0)\bv W_1, \bv W_2) \\
	-N(x)\lrpar{ \phi_Y^*g_Y(\cS^{\phi_Y}(\bv W_0, \bv W_2), \bh W_1) 
	- \phi_Y^*g_Y(\cS^{\phi_Y}(\bv W_0, \bv W_1), \bh W_2) }
	+ \cO(x).
\end{multlined}
\end{equation*}

\item
For $N = N(x) \partial_x$ and $W_i$ as in part (2),
\begin{equation*}
\begin{multlined}
	g_{\w}( \nabla_N(R_{\w}(N,W_0)) \bh_+ W_1, \bh_+ W_2)
	=   (N(x))^2 g_{\bhs{Y}/Y}(\cR^{\phi_Y}(\bh W_2, \bh W_1), \bv W_0) + \cO(x).
\end{multlined}
\end{equation*}
\end{enumerate}
\end{proposition}

\begin{proof}
If a connection asymptotically preserves a splitting of the bundle,
then its curvature evaluated in vector fields tangent to the boundary hypersurface will also
asymptotically preserve that splitting.  By \eqref{eq:metric-diff} and
\eqref{eq:SecondSplittingIE}, the $g_{\w}$ connection preserves
the splitting, so statement (1) above is correct.
	
Now consider $W_0$ tangent to the fibers of $\phi_Y$ and $W_1, W_2 \in \CI(X;{}^\w TX).$
Since $R_{\w}$ is a tensor, its value at $\pa X$ only depends on the values of the vector fields at the boundary, so  if $\bar W_0$ agrees with $W_0$ at $x=0$ as elements of $\CI(X;TX),$ and $\bar W_1, \bar W_2$ agree with $W_1, W_2$ at $x=0$ as elements of $\CI(X; {}^\w TX),$ then 
\begin{equation*}
	g_{\w}(R_{\w}(N, W_0)W_1, W_2)
	= g_{\w}(R_{\w}(N, \bar W_0)\bar W_1, \bar W_2) + \cO(x).
\end{equation*}
Thus we may assume that $W_0 = V_0 \in \{ V_\alpha \}$ (since $R_{\w}(N, x\pa_x)$ and $R_{\w}(N, x\wt U)$ are $\cO(x)$) and $W_1, W_2 \in \{ \pa_x, \tfrac1x V_\alpha, \wt U_i\}$ extended to $\sC(\bhs{Y})$ as above.
The advantage being that $\nabla_{\pa_x}W_i \in x C^\infty(X; \wT X)$.

Thus consider $N \in \ang{\pa_x},$ $N = N(x)\pa_x,$ and $W_0=\bv W_0$ vertical.
Using $[\pa_x, W_0]=0$, 
\begin{equation*}
	g_{\w}(R_{\w}(N,\bv W_0)W_1, W_2) = N(x)
        \pa_x(g_{\w}(\nabla_{\bv W_0}W_1, W_2)) + \cO(x).
\end{equation*}
If $W_1 = \bv W_1, W_2 = \bv W_2$, then we get the first term in part
(2).  If either $W_1, W_2 = \pa_x$, or if $W_1 = \bh W_1$ and $W_2 =
\bh W_2$, then \eqref{eq:metric-diff} and the asymptotics of the
connection above gives $\cO(x)$. The remaining possibilities are
\begin{equation*}
\begin{gathered}
	N(x)\pa_x(\nabla_{\bv W_0} \bv W_1, \bh W_2) = N(x) \phi_Y^*g_Y(\cS^{\phi_Y}(\bv W_0, V_1), \wt U_2), \Mwhere 
	\bv W_1 = \tfrac1x V_1, \; \bh W_2 = \wt U_2, \\
	N(x)\pa_x(\nabla_{\bv W_0} \bh W_1, \bv W_2) = -N(x) \phi_Y^*g_Y(\cS^{\phi_Y}(\bv W_0, V_2), \wt U_1), \Mwhere 
	\bh W_1 = \wt U_2, \; \bv W_2 = \tfrac1x V_2.
\end{gathered}
\end{equation*}
which establishes (2).\\

Next consider $g_{\w}( \nabla_N(R_{\w}(N,W_0)) \bh_+ W_1, \bh_+ W_2).$ 
Here $R_{\w}(N, W_0)$ is a section of $\hom(\Tw X)$ and correspondingly
\begin{equation*}
	\nabla_N(R_{\w}(N,W_0))\bh_+ W_1
	= \nabla_N(R_{\w}(N,W_0)\bh_+ W_1) 
	- R_{\w}(N,W_0)\nabla_N(\bh_+ W_1).
\end{equation*}
The second term in $\cO(x)$, so we have
\begin{equation*}
	g_{\w}( \nabla_N(R_{\w}(N,W_0)) \bh_+ W_1, \bh_+ W_2)
	= N(g_{\w}( R_{\w}(N,W_0) \bh_+ W_1, \bh_+ W_2) ) + \cO(x). 
\end{equation*}

Since $W_0$ is an edge vector field we may write it as $x(\bh_+ W_0) + \bv W_0.$ For the first summand we have
\begin{equation*}
	N(g_{\w}( R_{\w}(N,x \bh_+ W_0) \bh_+ W_1, \bh_+ W_2) )
	= N(x) g_{\w}(R_{\w}( N, \bh_+ W_0) \bh W_1, \bh W_2) = \cO(x).
\end{equation*}
For the second summand we have
\begin{equation*}
\begin{multlined}
	N(g_{\w}( R_{\w}(N,\bv W_0) \bh_+ W_1, \bh_+ W_2) )
	= N  g_{\w}( -R_{\w}(\bv W_0,  \bh_+ W_1)N 
		- R_{\w}( \bh_+ W_1, N)\bv W_0, \bh_+ W_2 ) \\
	= N \lrpar{ - g_{\w}(R_{\w}(N, \bh_+ W_2)\bv W_0, \bh_+ W_1) 
		+ g_{\w}(R_{\w}(N, \bh_+ W_1)\bv W_0, \bh_+ W_2) } \\
	= (Nx) \lrpar{ - g_{\w}(R_{\w}(N, \bh_+ W_2)\tfrac1x \bv W_0, \bh_+ W_1) 
		+ g_{\w}(R_{\w}(N, \bh_+ W_1)\tfrac1x \bv W_0, \bh_+ W_2) }
\end{multlined}
\end{equation*}
Now
\begin{equation*}
\begin{gathered}
	g_{\w}(R_{\w}(N, \bh_+ W_2)\tfrac1x \bv W_0, \bh_+ W_1) 
	= N( g_{\w}(\nabla_{\bh_+ W_2} \tfrac1x \bv W_0, \bh_+ W_1) ) \\
	= -\tfrac12(Nx)g_{\bhs{Y}/Y}(\cR^{\phi_Y}(\bh W_2, \bh W_1), \bv W_0) +\cO(x)
\end{gathered}
\end{equation*}
and so altogether
\begin{equation*}
\begin{multlined}
	g_{\w}( (\nabla_NR_{\w})(N,\bv W_0)  \bh_+ W_1,  \bh_+ W_2) \\
	= \tfrac12(Nx)^2
	\lrpar{
	g_{\bhs{Y}/Y}(\cR^{\phi}(\bh W_2, \bh W_1), \bv W_0)
	- g_{\bhs{Y}/Y}(\cR^{\phi}(\bh W_1, \bh W_2), \bv W_0) } + \cO(x) \\
	= (Nx)^2 g_{\bhs{Y}/Y}(\cR^{\phi}(\bh W_2, \bh W_1), \bv W_0) + \cO(x).
\end{multlined}
\end{equation*}
\end{proof}

This establishes the asymptotics of the curvature at each boundary hypersurface. Let us consider the implications at a corner. Suppose $Y, \wt Y \in \cS(X)$ and $Y<\wt Y$ so that the boundary fiber bundles participate in \eqref{eq:bfsDiag}.
Near $\bhs{Y}\cap \bhs{\wt Y}$ a `boundary product structure' as in \S\ref{sec:IFS} yields a collar neighborhood of the form
\begin{equation*}
	[0,1)_x \times [0,1)_r \times (\bhs{Y} \cap \bhs{\wt Y})
\end{equation*}
with $x$ a boundary defining function for $\bhs{Y}$ and $r$ a boundary defining function for $\bhs{\wt Y}.$ 
In this collar, the vector field $\pa_x$ is $\phi_{\wt Y}$-horizontal, the vector field $\tfrac1x\pa_r$ is $\phi_Y$-vertical, and any wedge vector field that is vertical with respect to $\phi_{\wt Y}$ is also vertical with respect to $\phi_Y.$
We will eventually carry out a Getzler rescaling argument where we rescale in the horizontal directions at each boundary hypersurface, so the interesting expressions at the corner are the ones of the form
\begin{equation*}
	g_\w(R_{\w}(\pa_r, W_0)\bv_{\wt Y}W_1, \pa_x), \quad
	g_\w(\nabla_{\pa_r}R_{\w}(\pa_r, W_0)\bh_{\wt Y}W_1, \pa_x),
\end{equation*}
with $W_0$ a vector field tangent to the fibers of $\phi_{\wt Y}$ and (hence) $\phi_Y$ and $W_1$ a wedge vector field.
The first expression is equal to
\begin{equation*}
	g_\w(R_{\w}(\bv_{\wt Y}W_1, \pa_x)\pa_r, W_0)
	= xg_\w(R_{\w}(xr\bv_{\wt Y}W_1, \pa_x)\tfrac1x \pa_r, \tfrac1{xr}W_0) = \cO(x),
\end{equation*}
while the second expression has leading term at the corner involving 
\begin{equation*}
	\cR^{\phi_{\wt Y}}(\bh_{\wt Y}W_1, \pa_x) = \bv_{\wt Y}([\bh_{\wt Y}W_1, \pa_x]) = 0.
\end{equation*}
The upshot is the vanishing at the corner of every term in these asymptotics in which $\pa_x$ occurs as a horizontal vector field.

\subsection{Dirac-type wedge operators}

Let $X \fib M \xlra{\psi} B$ be a fiber bundle of manifolds with corners and iterated fibration structures in the sense of Definition \ref{def:FibIFS} and fix a choice of splitting as in \eqref{eq:Splittings}.

\begin{definition} \label{def:WedgeCliffMod}
Let $g_{M/B}$ be a totally geodesic vertical wedge metric on $M.$
A {\bf wedge Clifford module} along the fibers of $\psi$ consists of 
\begin{enumerate}
\item a complex vector bundle $E\lra M$
\item a Hermitian bundle metric $g_E$
\item a connection $\nabla^E$ on $E$ compatible with $g_E$
\item a bundle homomorphism from the `vertical wedge Clifford algebra' into the endomorphism bundle of $E,$
\begin{equation*}
	\cl: \Cl_{\w}(M/B) = \bbC \otimes \mathrm{Cl}\lrpar{ {}^\w T^*M/B, g_{M/B} } \lra \End(E)
\end{equation*}
compatible with the metric and connection in that, for all  $\theta \in \CI(M;{}^{\w}T^*M/B),$
\begin{itemize}
	\item $g_{E}(\cl(\theta)\cdot, \cdot) = -g_{E}(\cdot, \cl(\theta)\cdot)$ 
	\item $\nabla^E_W\cl(\theta) = \cl(\theta)\nabla^E_W +\cl(\nabla_W^{M/B}\theta)$ 
		as endomorphisms of $E,$ for all $W \in TM.$
\end{itemize}
\end{enumerate}
This information determines a smooth {\bf family of wedge Dirac-type operators} by
\begin{equation*}
	D_{M/B}:\CIc(M^{\circ};E) \xlra{\nabla^E} \CIc(M^{\circ}; T^*M \otimes E) \xlra{\cl} \CIc(M^{\circ};E)
\end{equation*}
where we have used that $T^*M$ and ${}^{\w} T^*M$ are canonically isomorphic over the interior of $M.$
\end{definition}

If the fibers of $\psi$ are even dimensional we will want for $E$ to admit a $\bbZ_2$ grading
\begin{equation*}
	E = E^+\oplus E^-
\end{equation*}
compatible in that it is orthogonal with respect to $g_E,$ parallel with respect to $\nabla^E,$ and odd with respect to $\cl.$ 

In local coordinates, we can write
\begin{equation*}
	D_{M/B} = \sum_{i=1}^n \cl(\theta^i)\nabla^E_{(\theta^i)^{\sharp}}
\end{equation*}
where $\theta^i$ runs over a $g_{\w}$-orthonormal frame of $T^*M/B.$
If we restrict to a fiber $X$ of $\psi$ and then further to a collar neighborhood of $\bhs{Y},$ $Y \in \cS(X),$ this takes the form
\begin{equation}\label{eq:CollarDirac}
	\cl(dx) \nabla^E_{\pa_x} + \cl(x \; dz^i)\nabla^E_{\tfrac1x\pa_{z_i}} 
	+ \cl(dy^j)\nabla^E_{\pa_{y_j}}
	=
	\cl(dx) \nabla^E_{\pa_x} + \cl(x \; dz^i)\tfrac1x\nabla^E_{\pa_{z_i}} 
	+ \cl(dy^j)\nabla^E_{\pa_{y_j}}
\end{equation}
plus an error in $\Diff^1_{\e}(X;E)$.  Here $x$ is a boundary defining
function for $\bhs{Y},$ and we recognize \eqref{eq:CollarDirac} as a
wedge differential operator of order one.

We are interested in this operator acting on the natural family of vertical $L^2$ spaces associated to the wedge metric $g_{\w}$ and the Hermitian metric $g_E,$ which we denote $L^2_{\w}(M/B;E).$ However it is convenient to work with $L^2$ spaces with respect to a non-degenerate density, so let us define a multiweight on $M,$
\begin{equation}\label{eq:DefMfB}
	\mf b(H) = \tfrac12(\dim \bhs{N}/N) \Mforall H \subseteq \bhs{N} \Mand N \in \cS_{\psi}(M)
\end{equation}
so that 
\begin{equation}\label{eq:wedge-v-reg}
	L^2_{\w}(M/B;E) = \rho_M^{-\mf b}L^2(M/B;E).
\end{equation}
(On each fiber $X$ of $\psi$ we have $\mf b(H) = \tfrac12(\dim \bhs{Y}/Y)$ for all $H\subseteq \bhs{Y}$ and $Y \in \cS(X).$)
Define the unitarily equivalent family of operators $\eth_{M/B}$ by
\begin{equation}\label{eq:ShiftedL2Dirac}
	\eth_{M/B} = \rho_M^{\mf b}D_{M/B}\rho_M^{-\mf b} 
	= D_{M/B} - \sum_{N \in \cS_{\psi}(M)} \frac{\dim \bhs{N}/N}{2\rho_N}\cl(d\rho_N)
\end{equation}
Then $\eth_{M/B}$ is also a vertical family of wedge differential operators of order one, and studying $\eth_{M/B}$ as operators on $L^2(M/B;E)$ is equivalent to studying $D_{M/B}$ as a family of operators on $L^2_{\w}(M/B;E).$\\


\subsection{Bismut superconnection}\label{sec:BismutSup}

We briefly recall the construction of the Bismut superconnection and refer to \cite[Chapter 9-10]{BGV2004}, \cite{Melrose-Piazza:Even}, \cite{Albin-Rochon:DFam} for more details.\\

Let $M \xlra\psi B$ be a family of manifolds with corners and iterated fibration structures as in Definition \ref{def:FibIFS}, endowed with a splitting 
\begin{equation*}
	{}^{\w(\psi)} TM = {}^\w TM/B \oplus \psi^*TB
\end{equation*}
as in \eqref{eq:Splittings}
and a vertical wedge metric $g_{M/B}.$
Denote the projections onto the summand of the splitting by
$\bv_{\psi}$ (left) and $\bh_{\psi}$ (right).
These data determine a connection on the bundle ${}^\w TM/B,$ $\nabla^{M/B},$ as follows.
We choose a Riemannian metric $g_B$ on $B,$ and obtain
\begin{equation*}
	g_M = g_{M/B} \oplus \psi^*g_B,
\end{equation*}
a wedge metric on $M.$ As in \S\ref{sec:Metrics} the Koszul formula defines a Levi-Civita connection on ${}^{\w(\psi)} TM$ which we denote $\nabla^M$ and use to define
\begin{equation*}
	\nabla^{M/B} = \bv_{\psi} \nabla^M \bv_{\psi}.
\end{equation*}
Just as for families of closed manifolds, this defines a connection on ${}^{\w}TM/B$ that is independent of the choice of the metric $g_B.$

We embed $g_M$ in a one-parameter family of wedge metrics on $M,$
\begin{equation*}
	g_{M,\eps}= g_{M/B} + \tfrac1\eps \psi^*g_B,
\end{equation*}
limiting to the degenerate metric on ${}^{\w(\psi)} TM,$
\begin{equation*}
	g_{M,0}(V, W) = g_{M/B}(\bv_{\psi}V, \bv_{\psi}W),
\end{equation*}
in that the dual metrics on the wedge cotangent bundle converge.

Consider the connection 
\begin{equation*}
	\nabla^{\oplus} = \bv_{\psi} \nabla^M \bv_{\psi} \oplus \bh_{\psi} \nabla^M \bh_{\psi}.
\end{equation*}
Following Bismut, we can describe the difference between this connection and $\nabla^M$ in terms of the fundamental tensors of $\psi.$

Define 
\begin{equation}\label{eq:FunTensors}
\begin{gathered}
	\cS^{\psi} \in \CI(M; {}^\w T^*M/B \otimes {}^\w T^*M/B \otimes \psi^*TB),\\
	\hat \cR^{\psi} \in \CI(M; \psi^*TB \otimes \psi^*TB \otimes {}^\w T^*M/B)
\end{gathered}
\end{equation}
by the equations
\begin{equation*}
\begin{gathered}
	\cS^{\psi}(W_1, W_2)(A) = g_{M/B}(\nabla^{M/B}_A W_1 - [A, W_1], W_2) \\
	\hat \cR^{\psi}(A_1, A_2)(W) = g_{M/B}([A_1, A_2], W).
\end{gathered}
\end{equation*}
Using the splitting, we extend these trivially to $\otimes^3({}^{\w(\psi)} TM),$ and then define
\begin{equation*}
	\omega^{\psi} \in \CI(M; {}^{\w(\psi)} T^*M \otimes \Lambda^2({}^{\w(\psi)} T^*M)), 
\end{equation*}
\begin{multline*}
	\omega^{\psi}(X)(Y,Z) = \cS^{\psi}(X,Z)(Y) - \cS^{\psi}(X,Y)(Z)\\
	+ \tfrac12\hat\cR^{\psi}(X,Z)(Y) -\tfrac12\hat\cR^{\psi}(X,Y)(Z) +\tfrac12\hat\cR^{\psi}(Y,Z)(X).
\end{multline*}
This tensor is isomorphic to $\nabla^M-\nabla^{\oplus}$ via $g_M$ 
by \cite[Prop 10.6]{BGV2004} and allows us to define
\begin{equation*}
	\nabla^0 = \nabla^{\oplus} + \tfrac12\tau(\omega^{\psi})
\end{equation*}
where
\begin{equation*}
	\tau:\Lambda^2({}^{\w(\psi)} T^*M) \lra \hom({}^{\w(\psi)} T^*M), \quad
	\tau(\alpha\wedge\beta)\delta = 2(g_{M,0}(\alpha, \delta)\beta - g_{M,0}(\beta, \delta)\alpha).
\end{equation*}
As the notation is meant to indicate, this is the limiting connection of the Levi-Civita connections of the metrics $g_{M,\eps}$ as $\eps\to 0.$

Given a wedge Clifford module $E \lra M,$ we can extend it to the bundle
\begin{equation*}
	\bbE = \psi^*\Lambda^*B \otimes E.
\end{equation*}
This has a natural Clifford action
\begin{equation*}
\begin{gathered}
	\cl_0:\Cl_0({}^{\w(\psi)}T^*M)
	 = \bbC \otimes \mathrm{Cl}\lrpar{ \psi^*T^*B \oplus {}^\w T^*M/B, g_{M,0}} 
	 	\lra \End(\bbE)\\
	\cl_0(\alpha) = \df e(\bh_{\psi} \alpha) + \cl(\bv_{\psi}\alpha) \Mfor \alpha \in \CI(M; {}^{\w(\psi)} T^*M)
\end{gathered}
\end{equation*}
where $\df e$ is exterior multiplication, with connection
\begin{equation*}
	\nabla^{\bbE,0} = \psi^*\nabla^B \oplus \nabla^E + \tfrac12\cl_0(\omega^{\psi})
\end{equation*}
that is compatible with $\nabla^0$ in that
\begin{equation*}
	\nabla^{\bbE,0}_W\cl_0(\theta) = \cl_0(\theta)\nabla^{\bbE,0}_W 
	+\cl_0(\nabla^0_W\theta), \Mforall W \in \CI(M;TM).
\end{equation*}
We will need to know about the curvature of $\nabla^{\bbE,0}.$ It is easy to see that, for any $U, V \in \CI(M;TM),$
\begin{equation*}
	\bv_{\psi}R^{M,\eps}(U,V)\bv_{\psi} = R^{M/B}(U,V), \quad
	\bh_{\psi}R^{M,\eps}(U,V)\bh_{\psi} = \psi^*R^{B}(U,V).
\end{equation*}
It follows from Proposition 10.9 of \cite{BGV2004} that
\begin{equation*}
	R^{M,0}(T_1, T_2)(T_3, T_4) 
	= \lim_{\eps\to0}
	R^{M,\eps}(T_1, T_2)(T_3, T_4) 
\end{equation*}
as long as $\bh_{\psi}T_i=0$ for some $i$ (which will hold in all cases we need to consider). In particular, applying Proposition \ref{prop:curvature} to $R^{M/B,\eps}$ lets us conclude that the corresponding asymptotics hold for $R^{M,0}.$

The Dirac-type operator corresponding to $\bbE$ with Clifford action $\cl_0$ and connection $\nabla^{\bbE,0}$ is known as the {\bf Bismut superconnection} and denoted $\bbA_{M/B}.$ As a map $\CI(M;E) \lra \CI(M;\bbE)$ it is given by
\begin{equation*}
	\bbA_{M/B} = \bbA_{M/B,[0]} + \bbA_{M/B, [1]} + \bbA_{M/B, [2]}
\end{equation*}
where $\bbA_{\psi,[j]}: \CI(M;E) \lra \CI(M; \psi^*\Lambda^jB \otimes E).$
Explicitly, in terms of local orthonormal frames $\{f_\alpha\}$ for $TB$ and $\{e_i\}$ for ${}^\w TM/B,$ with dual coframes $\{f^{\alpha}\},$ $\{e^i\},$ we have
\begin{equation}\label{eq:BSconnection}
	\bbA_{M/B} = \cl(e^i)\nabla^E_{e_i}
	+ \df e^{\alpha}\lrpar{ \nabla^E_{f_{\alpha}} + \tfrac12k_{\psi}(f_{\alpha})}
	- \frac14 \sum_{\alpha<\beta} \hat\cR^{\psi}(f_{\alpha}, f_{\beta})(e_i)\df e^\alpha \df e^{\beta} \cl(e^i),
\end{equation}
where $k_{\psi}$ is the trace of $\cS^{\psi}.$
Note that $\bbA_{M/B,[0]}$ is $D_{M/B},$ the Dirac-type operator associated to $E.$

As $E$ is a wedge Clifford bundle, its homomorphism bundle has a decomposition
\begin{equation}\label{eq:HomClif}
	\hom(E) \cong \Cl({}^{\w}T^*M/B) \otimes \hom'_{\Cl({}^{\w}T^*M/B)}(E)
\end{equation}
where $\hom'_{\Cl({}^{\w}T^*M/B)}(E)$ denotes homomorphisms that commute with $\Cl({}^{\w}T^*M/B),$ see e.g., \cite[Lemma 5.1]{Vaillant}.
The curvature of $\nabla^E$ decomposes as $\tfrac14\cl(R^M) + K_E',$
where $K_E'$ commutes with the Clifford action and is known as the
`twisting curvature' of $\nabla^E,$ see e.g., \cite[Lemma
8.33]{tapsit}.
The square of the wedge Bismut superconnection satisfies a Lichnerowicz formula \cite[Theorem 3.52]{BGV2004},
\begin{equation}\label{eq:SupLichnerowicz}
	\bbA^2_{M/B} = \Delta^{M/B,0} + \tfrac14\scal(g_{\w}) - \tfrac12\sum_{a,b}K'_E(e_a, e_b)\cl_0(e^a)\cl_0(e^b)
\end{equation}
where the sum is over both horizontal and vertical tangent vectors and
$\Delta^{M/B,0}$ is the vertical family of operators which at $M_b$ is
the Bochner Laplacian corresponding to $\nabla^{\bbE,0}\rest{M_b}.$\\

For each $N \in \cS_{\psi}(M),$ we have three related fiber bundles at the corresponding collective boundary hypersurface
\begin{equation*}
	\xymatrix{
	Z_Y \ar@{-}[r] & \bhs{N}\ar[d]_-{\phi_N} \ar@/^20pt/[dd]^-{\psi|_{\bhs{N}}} \\
	Y \ar@{-}[r] & N \ar[d]_-{\psi_N} \\
	& B}
\end{equation*}
and from the asymptotics of wedge connections we see that
\begin{equation}\label{eq:FibAsymp}
	\cS^{\psi}\rest{\bhs{N}} = \cS^{\psi_N}, \quad \hat\cR^{\psi}\rest{\bhs{N}} = \hat\cR^{\psi_N}.
\end{equation}
We will see that the contribution of $\phi_N$ to these tensors can be recovered by passing to `rescaled normal operators'.

\section{Witt condition}\label{sec:Wittsec}

\subsection{Boundary families} \label{sec:VertOp}

Let $M \xlra{\psi} B$ be a fiber bundle of manifolds with corners and
iterated fibration structures and $g_{M/B}$ a totally geodesic
vertical wedge metric on $M$. Given a $\psi$-vertical wedge Clifford
module over $M,$ $E \lra M,$ we will explain how at each $N \in
\cS_{\psi}(M)$ there is an induced $\phi_N$-vertical wedge Clifford module on
$\bhs{N}.$ This has a Clifford action not just by ${}^\w T^*\bhs{N}/N$
but by all of ${}^\w T^*M/B\rest{\bhs{N}},$ we encode this as a
$\Cl(F)$-wedge Clifford bundle for an appropriate bundle $F \lra B.$
We refer to the corresponding family of vertical Dirac-type operators
as the boundary family of $D_{M/B}$ at $N$ and denote it
$D_{\bhs{N}/N}.$\\

Let $N \in \cS_{\psi}(M)$ with corresponding collective boundary hypersurface $\bhs{N}$ and fix a choice of boundary defining function $x.$ Choose a collar neighborhood $\sC(\bhs{N}) \cong [0,1)_x \times \bhs{N}$ and identify
\begin{equation*}
	{}^\w T^*M/B\rest{\sC(\bhs{N})} 
		= N^*_M\bhs{N} \oplus x{}^{\w}T^*\bhs{N}/N \oplus \phi_N^*T^*N,
\end{equation*}
where $N^*_M\bhs{N}$ is the (rank one) conormal bundle of $\bhs{N}$, and then further identify
\begin{equation*}
	N^*_M\bhs{N} = \ang{dx}, \quad
	x{}^{\w}T^*\bhs{N}/N \cong {}^{\w}T^*\bhs{N}/N.
\end{equation*}

The $\psi$-vertical wedge metrics $g_{M/B}$ on ${}^{\w}T^*M/B$ induce a $\phi_N$-vertical wedge metric $g_{\bhs{N}/N}$ on ${}^{\w}T^*\bhs{N}/N.$
Choose a metric $g_B$ on $B$ and let $g_M = g_{M/B} \oplus \psi^*g_B,$
so that $\nabla^{M/B}$ is given, as above, in terms of the Levi-Civita
connection $\nabla^M$ of $g_M$ by
\begin{equation*}
	\nabla^{M/B} = \bv_{\psi} \circ \nabla^{M} \circ \bv_{\psi}.
\end{equation*}
The metric $g_M$ is a totally geodesic wedge metric on $M$ 
and thus in particular there is a corresponding wedge
metric on $\bhs{N}$, with vertical connection
\begin{equation*}
	\nabla^{\bhs{N}/N} = \bv_{\phi_N} \circ \nabla^{\bhs{N}} \circ \bv_{\phi_N}.
\end{equation*}

In order to relate $\nabla^{\bhs{N}/N}$ with the restriction of $\nabla^{M/B}$ to $\bhs{N},$ recall from \S\ref{sec:Metrics} that the restriction of $\nabla^M$ to $\bhs{N}$ will respect the splitting
\begin{equation}\label{eq:NSplitting}
	\Tw^*\sC(\bhs{N}) 
	= \lrspar{ \ang{\pa_x} \oplus \tfrac1x  {}^\w T^*\bhs{N}/N } \oplus \phi_N^*T^*N
\end{equation}
so that $j_0^*\nabla^M 
= \bv_{\phi_N}^+ \circ\nabla^M \circ\bv_{\phi_N}^
+ \oplus \bh_{\phi_N} \circ \nabla^M \circ \bh_{\phi_N}.$
Let us denote the fully diagonal connection by
\begin{equation*}
	\nabla^{\oplus} = 
	\frac{\pa}{\pa x} \; dx \oplus \nabla^{\bhs{N}/N} \oplus \bh_{\phi_N} \nabla^M \bh_{\phi_N}.
\end{equation*}

The difference between this and a direct sum connection with respect
to the splitting \eqref{eq:NSplitting} comes from the fact that in a
frame like \eqref{eq:local-frame}, letting the connection act on
differential forms,
\begin{equation*}
	\nabla_{V_1}^M(\tfrac1x V_2)^\flat(\pa_x) = -g_{\bhs{N}/N}(V_1,V_2), \quad
	\nabla_{V_1}^M dx = (\tfrac1x V_1)^{\flat}.
\end{equation*}
Thus  we have
\begin{equation}\label{eq:LCConCyl}
	\nabla^{\oplus}_{W}\theta = j_0^*\nabla^M_W\theta 
	- \lrpar{ g(dx,\theta)(\tfrac1x\bv W)^{\flat} -g((\tfrac1x\bv W)^{\flat},\theta)dx}.
\end{equation}
The Clifford connection $\nabla^E$ which is by definition compatible
with $\nabla^M$ can be modified in a standard way to obtain a Clifford
connection compatible with $\nabla^\oplus$, namely, writing $S =
\nabla^M - \nabla^\oplus$, following \cite[page 375]{BC1990},
$$
\nabla^{E, \oplus}_W := \nabla^E_W - \frac 14 \langle S(W) e_i, e_j
\rangle_{g_M} \cl(e^i) \cl(e^j)
$$
for orthonormal frame $e_i$ and dual frame $e^i$.  (Indeed, from the
fact that $S(W)$ is anti-symmetric, it follows that $\frac 14 \langle S(W) e_i, e_j
\rangle_{g_M} \cl(e^i) \cl(e^j)X = \cl(S(X))$.)
Restricting to $\bhs{N}$ via $j_0$ from \eqref{eq:AsympSplittingConn}, 
we let $\nabla^{E|_N}$ be given by
\begin{equation}\label{eq:EConCyl}
	\nabla^{E|_N}_W = j_0^*\nabla^E_W
	- \tfrac12 \cl(dx)\cl((\tfrac1x\bv W)^{\flat})
\end{equation}
and this is compatible with the restriction of  $\nabla^\oplus$ to
$\bhs{N}$.

After identifying $x{}^{\w}T^*\bhs{N}/N\rest{x=0}$ with
${}^{\w}T^*\bhs{N}/N$ we see that $\nabla^{E|_N}$ is a metric
connection with respect to the restriction of $g_E$ which restricts to
the fibers of $\bhs{N} / N$ to be compatible with  $g_{\bhs{N}/N}.$
Thus altogether we obtain a wedge Clifford module for the induced
vertical wedge metric $g_{\bhs{N}/N}$ on $\bhs{N} \lra N$ that
moreover \emph{is compatible with the Clifford action of the base and
  normal covectors} i.e.\ the Clifford action by sections of the
bundle $\ang{dx} \oplus \phi_N^* {}^\w T^*N/B.$

As the bundle $N^*_M\bhs{N} = \ang{dx} \lra \bhs{N}$ is trivial there is no loss, and some convenience, in treating it as the pull-back of a trivial bundle over $N.$ We introduce the notation $T^*N^+/B$ for the direct sum of $T^*N/B$ and a trivial bundle over $N$ formally generated by $dx,$ so that $N^*_M\bhs{N} \oplus \phi_N^*T^*N/B = \phi_N^*T^*N^+/B.$

\begin{definition}\label{def:Cl(F)Bdle}
Let $X \fib M \xlra{\psi} B$ be a fiber bundle of manifolds with corners
and iterated fibration structures in the sense of Definition
\ref{def:FibIFS}, with a vertical wedge metric $g_{M/B}$. If $F \lra B$ is a real vector bundle with bundle
metric $g_F,$ a wedge Clifford module $(E, g_E, \nabla^E, \cl)$ is
{\bf a $\Cl(F)$-wedge Clifford module} if there is a bundle
homomorphism (also denoted $\cl$)
\begin{equation*}
	\cl: \Cl(\psi^*F)= \bbC \otimes \mathrm{Cl}\lrpar{ F, g_{F} } \lra \End(E)
\end{equation*}
satisfying, for all $\eta \in \CI(M;\psi^*F),$
\begin{equation*}
\begin{gathered}
	g_E(\cl(\eta)s_1, s_2) = -g_E(s_1, \cl(\eta)s_2) 
		\Mforall s_i \in \CI(M;E) \\
	\nabla^E(\cl(\eta)s) = \cl(\eta)\nabla^Es, \Mforall s \in \CI(M;E)\\
	\cl(\eta)\cl(\theta) + \cl(\theta)\cl(\eta) = 0 \Mforall \theta \in \CI(M; {}^{\w} T^*M/B).
\end{gathered}
\end{equation*}
\end{definition}

Note that if $F$ has rank zero, a $\Cl(F)$-wedge Clifford module is just a wedge Clifford module.

Clearly if $D$ is the Dirac-type operator corresponding to (the
underlying wedge Clifford module $(E, g_E, \nabla^E, \cl)$ of) a $\Cl(F)$-wedge module, and $\theta \in \CI(M;F)$ then $D \cl(\theta) = - \cl(\theta)D.$\\

We have seen that a wedge Clifford bundle $(E, g_E, \nabla^E, \cl)$ along the fibers of $M \xlra{\psi} B$ induces a $\Cl(T^*N^+/B)$-wedge Clifford module, $(E|_N, g_E|_N, \nabla^{E|N}, \cl)$
along the fibers of $\bhs{N} \xlra{\phi_N} N.$
Let us finally consider the relation between the corresponding Dirac-type operators.

To each $N \in \cS_{\psi}(M)$ we can associate a $\phi_N$-vertical family of operators
\begin{equation*}
	\rho_N D_{M/B}\rest{\bhs{N}}.
\end{equation*}
From the local expression \eqref{eq:CollarDirac} we see, letting
$V_i$ be an orthonormal frame of vertical vectors and $V^i$ the dual
frame, that 
\begin{equation*}
	\rho_N D_{M/B}\rest{\bhs{N}} = \cl(x \; V^i) j_0^*\nabla^E_{V_i}
\end{equation*}
where $V_i$ runs over a local frame for the vertical wedge tangent bundle associated to $\bhs{N} \xlra{\phi_N} N.$
Replacing the connection $j_0^*\nabla^E$ by $\nabla^{E|_N}$ yields
\begin{equation*}
	\cl(x \; dz^i) (\nabla^{E|_N}_{\pa_{z_i}} + \tfrac12 \cl(dx)\cl(x\; dz^i) )
	= D_{\bhs{N}/N} + \tfrac v2\cl(dx)
\end{equation*}
where we recall that $v = \dim \bhs{N}/N.$
Thus we can conclude:

\begin{lemma}\label{lem:ModelDirac}
A wedge Clifford module along the fibers of $\psi:M \lra B$ induces,
for each $N \in \cS(N),$ a $\Cl(T^*N^+/B)$-wedge 
Clifford module along the fibers of $\phi_N: \bhs{N} \lra N.$
The vertical operator of a family of wedge Dirac-type operators $D_{M/B}$ defined by the former is equal to the family of wedge Dirac-type operators defined by the latter, which we denote $D_{\bhs{N}/N},$ plus a zero-th order term
\begin{equation*}
	\rho_ND_{M/B}\rest{\bhs{N}} = D_{\bhs{N}/N} + \tfrac v2\cl(dx).
\end{equation*}
\end{lemma}

We will denote the restriction of $D_{\bhs{N}/N}$ to the fiber over $y \in N$ by
\begin{equation*}
	D_{\bhs{N}/N}\rest{Z_y} = D_{Z_y}
\end{equation*}
when it is clear from context.
Note that Clifford multiplication by the global section $dx$ of $T^*N^+$ satisfies 
\begin{equation*}
	D_{\bhs{N}/N} \circ \cl(dx) = -\cl(dx)\circ D_{\bhs{N}/N}.
\end{equation*}

Hence for any choice of closed domain for $D_{Z_y},$ invariant under multiplication by $\cl(dx),$ we have
\begin{multline*}
	s \in \text{ $\lambda$-eigenspace of $D_{Z_y}$}
	\iff
	\cl(dx)s \in \text{ $(-\lambda)$-eigenspace of $D_{Z_y}$}.
	 \\
	\iff
	(\Id + \cl(dx))s \in \text{ $\lambda$-eigenspace of $\cl(dx)D_{Z_y}$}
\end{multline*}
and so the eigenvalues of $D_{Z_y},$ $-D_{Z_y},$ and $\cl(dx)D_{Z_y}$ coincide, including multiplicity.\\

Lemma \ref{lem:ModelDirac} highlights one advantage of working with $\eth_{M/B}$ from \eqref{eq:ShiftedL2Dirac} since the induced vertical family is precisely the boundary family of Dirac operators,
\begin{equation*}
	\rho_N \eth_{M/B}\rest{\bhs{N}} = D_{\bhs{N}/N}.
\end{equation*}
Note that since we are interested in $\eth_{M/B}$ on the space $L^2(M/B;E),$ we are interested in the boundary family $D_{\bhs{N}/N}$ as an operator on $L^2(\bhs{N}/N;E|_N).$

\subsection{Witt assumption and vertical APS domain} \label{sec:Witt}

Let $\psi \colon M \lra B$ be a fibration of manifolds with corners
with iterated fibration structure as in Definition \ref{def:FibIFS}
(so $B$ is closed) with typical fiber $X$, let $g_{M/B}$ be a vertical
wedge metric, and $E \lra M$ a wedge Clifford module as in Definition
\ref{def:WedgeCliffMod}.  Let $D_{M/B}$ be the corresponding family of
vertical wedge Dirac-type operators and $D_X$ the restriction of
$D_{M/B}$ to a fixed fiber $X$.
As an unbounded operator on $L^2_{\w}(X;E),$ for an arbitrary totally geodesic wedge
metric $g_{M/B} \rvert_X$, $D_X$ generally has many closed extensions.
As discussed in, e.g., \cite{ALMP:Hodge}, the two canonical closed domains,
\begin{equation*}
\begin{gathered}
	\cD_{\min}(D_X) = \{ u \in L^2_{\w}(X;E) : \exists (u_n) \subseteq \CIc(X^{\circ};E) \Mst
	u_n \to u \Mand (D_Xu_n) \text{ is $L^2_{\w}$-Cauchy} \}, \\
	\cD_{\max}(D_X) = \{ u \in L^2_{\w}(X;E) : D_X u \in L^2_{\w}(X;E)\},
\end{gathered}
\end{equation*}
where in the latter $D_Xu$ is computed distributionally, satisfy
\begin{equation*}
	\rho_XH^1_\e(X;E) \subseteq \cD_{\min}(D_X) \subseteq \cD_{\max}(D_X) \subseteq H^1_\e(X;E).
\end{equation*}
Here
\begin{equation*}
	H^1_\e(X;E) = \{ u \in L^2(X;E) : Vu \in L^2(X;E) \Mforall V \in \CI(X;{}^\e TX) \}
\end{equation*}
\nomenclature[H]{$H^1_\e(X;E)$}{The $L^2$-based edge Sobolev space}
is the edge Sobolev space introduced in \cite{Mazzeo:Edge}.  We
consider the following domain:

\begin{definition}
The {\bf vertical APS domain} of $D_X$ is the graph closure of $\rho_X^{1/2}H^1_\e(X;E)\cap \cD_{\max}(D_X),$
\begin{multline*}
	\cD_{\VAPS}(D_X) = \{ u \in L^2_{\w}(X;E) : \exists (u_n) \subseteq \rho_X^{1/2}H^1_e(X;E)\cap \cD_{\max}(D_X) \\
	\Mst
	u_n \to u \Mand (D_Xu_n) \text{ is $L^2_{\w}$-Cauchy} \}.
\end{multline*}
\end{definition}

As in \cite{ALMP:Hodge}, this domain induces a domain for each vertical family $D_{\bhs{Y}/Y},$ namely the corresponding vertical APS domain. Note that this domain is invariant under multiplication by $\cl(dx)$ so that the spectrum of each $D_{Z_y}$ has the symmetries mentioned at the end of \S\ref{sec:VertOp}.

\begin{definition} \label{def:WittAss}
The operator $(D_X,\cD_{\VAPS})$ satisfies the {\bf geometric Witt condition} if
\begin{equation*}
	Y \in \cS(X), y \in Y \implies 
	\Spec(D_{Z_y}) \cap (-\tfrac12, \tfrac12) = \emptyset.
\end{equation*}
If, instead, we only require
\begin{equation*}
	Y \in \cS(X), y \in Y \implies 
	\Spec(D_{Z_y}) \cap \{0\} = \emptyset
\end{equation*}
then we say that $(D_X, \cD_{\VAPS})$ satisfies the {\bf Witt
  condition}.  (Tacitly, we take the domain of the links to be the
VAPS domains.  Thus the definition could be stated without reference
to a domain on the whole of $X$, specifically only with reference to
the spectrum of $\Spec(D_{Z_y}, \cD_{VAPS}(D_{Z_y})).$  Since this
notation is cumbersome we speak only of the VAPS domain on $X$ and
think of the domains on the fibres as induced.)
\end{definition}

The analysis in \cite{ALMP:Hodge} can be used to show that the  geometric Witt condition
\begin{equation*}
	Y \in \cS(X), y \in Y \implies 
	\Spec(D_{Z_y}) \cap (-\tfrac12, \tfrac12) = \emptyset
\end{equation*}
implies $\cD_{\min}(D_X) = \cD_{\max}(D_X)$ so that $D_X$ is
essentially self-adjoint. 

\begin{remark}
We use the nomenclature `vertical APS domain' because the different local ideal boundary conditions for $D_X$ involve the spectrum of 
$D_{Z_y}$
in the interval $(-\tfrac 12, \tfrac 12).$ The vertical APS domain corresponds to projecting off of the negative half of this interval, analogous to the Atiyah-Patodi-Singer boundary conditions \cite{APSI}.
\end{remark}

\subsection{Normal operator} \label{sec:NormOp}

As in \S\ref{sec:Witt}, let $D_X$ be the restriction of $D_{M/B}$ to a fiber of $\psi:M \lra B.$
At every $Y \in \cS(X),$ $y \in Y^{\circ},$ there is a {\em normal operator} of $D_X,$ modeling its behavior on a model wedge,
\begin{equation*}
	\bbR^+_s \times \bbR^h_u \times Z_y,
\end{equation*}
where $h = \dim Y,$ acting on sections of the bundle $E\rest{Z_y}$ pulled back along the natural projection.
This operator is given by
\begin{equation*}
	\cN_y(D_X) = \cl(dx)\pa_s + \tfrac1s (D_{Z_y} + \tfrac v2\cl(dx)) + \sum \cl(H_j^{\flat}) H_j
	= \cl(dx)\pa_s + \tfrac1s (D_{Z_y} + \tfrac v2\cl(dx)) + D_{\bbR^h}
\end{equation*}
where the sum ranges over an orthonormal frame for $T_yY,$ and throughout this section we use the notation
\begin{equation*}
	v = \dim Z_y.
\end{equation*}
The normal operator of $\eth_X$ is then
\begin{equation*}
	\cN_y(\eth_X) = \cN_y(D_X) - s^{-1} \tfrac v2 \cl(dx).
\end{equation*}

The vertical APS domain for $D_X$ induces domains for both $D_{Z_y}$ and $\cN_y(D_X),$ which are easily identified as the corresponding vertical APS domains.
The induced domain for the normal operator can be described in terms of $\cI$-smooth (i.e., polyhomogeneous) asymptotic expansions.

Indeed, the operator $D_{Z_y}$ is independent of $s,$ and so we are in
the `constant indicial root' situation studied in, e.g.,
\cite{ALMP:Hodge}. Note that if $f$ is a section of $E\rest{Z_y}$ over
the model wedge and is $\cO(s^0)$ as $s\to0$ then for almost all $\sigma \in \bbC,$ $s\cN_y(D_{X})(s^{\sigma}f) = \cO(s^{\sigma}).$ We say that $\sigma$ is {\bf an indicial root} if there is an $f$ such that $s\cN_y(D_X)(s^{\sigma}f) = \cO(s^{\sigma-1}).$

Let us describe the induced domain for the normal operator in more detail.
Analogously to \cite[\S1]{ALMP:Hodge}, we can consider an intermediate domain where we have imposed the vertical APS `boundary condition' at all of the $\bhs{Y'}$ with $Y<Y',$
\begin{equation*}
	\cD_{\max, Y<}(\eth_X)
	= \text{ graph closure of }
	\cD_{\max}(\eth_X) \cap \lrpar{\prod_{Y<Y'} \rho_{Y'}^{1/2}}L^2(X;E).
\end{equation*}
As explained in {\em loc. cit.}, elements in the induced domain for the normal operator have a partial asymptotic expansion.
(In the setting of \cite{ALMP:Hodge}, under an assumption of constant indicial roots, domains were defined by an inductive process. Elements of the maximal domain have partial asymptotic expansions with distributional coefficients at strata of depth one on which we can impose ideal boundary conditions. Elements of the resulting domain have partial asymptotic expansions with distributional coefficients at strata of depth two and so on.)
\begin{lemma}\label{lem:IndicialRoots}
The indicial roots of $\cN_y(D_X)$ are the eigenvalues of $D_{Z_y}$ shifted by $-\tfrac v2.$ 
Every $f$ in 
\begin{equation*}
	\cD_{\max, Y<}(\cN_y(D_X)) 
	= \text{graph closure of }
	\cD_{\max}(\cN_y(D_X)) \cap \lrpar{\prod_{Y<Y'} \rho_{Y'}^{1/2}}
	L^2_{\w}(\bbR^+_s \times \bbR^h_u \times Z_y;E\rest{Z_y}),
\end{equation*}
has a partial asymptotic expansion as $s\to0,$
\begin{equation*}
	f \sim \sum_{\substack{\lambda \in \Spec(D_{Z_y})\\ \lambda \in (-\tfrac12, \tfrac12) }} 
		f_{\lambda} s^{-\tfrac v2+\lambda} + \wt f
\end{equation*}
in which each $f_\lambda$ is a distributional element of the $\lambda$ eigenspace of $\cl(dx)D_{Z_y}$ and $\wt f \in x^{1-}H^{-1}_e(\bbR^+_s \times \bbR^h_u \times Z_y).$
The vertical APS domain of $\cN_y(D_X)$ consists of those $f$ such that $f_{\lambda}=0$ whenever $\lambda\leq 0.$
\end{lemma}

\begin{proof}
Applying $\cN_y(D_X)$ to $s^{\sigma}f$ yields
\begin{equation*}
	s\cN_y(D_X)(s^{\sigma}f) = 
	s^{\sigma} ((\sigma+\tfrac v2) \cl(dx)+ D_{Z_y})f + \cO(s^{\sigma+1})
\end{equation*}
so $\sigma$ will be an indicial root when $(\sigma+\tfrac v2 - \cl(dx) D_{Z_y})$ is not invertible, i.e.,
\begin{equation*}
	\sigma \in -\tfrac v2+ \Spec(\cl(dx)D_{Z_y}) = -\tfrac v2+ \Spec(D_{Z_y}).
\end{equation*}
The existence of the asymptotic expansion is established in \cite[Lemma 2.2]{ALMP:Hodge}.
\end{proof}

Note that the translation invariance of $\cN_y(D_X)$ in $\bbR^h$ allows us to Fourier transform in $u$ and obtain the family of model operators,
\begin{equation*}
	Y \times \bbR^h \ni (y,\eta) \mapsto \cN_{(y,\eta)}(D_X) = \cl(dx)(\pa_s+\tfrac v{2s}) + \tfrac1s D_{Z_y} + i \cl(\eta).
\end{equation*}
By a standard computation (cf. \cite[Lemma 2.10]{Albin-GellRedman}, \cite[Lemma 5.5]{ALMP:Witt}, \cite[Proposition 4.1]{Lesch:Def}, \cite[Proposition 2.25]{Ch1985}) we can establish injectivity and self-adjointness as follows.

\begin{proposition}\label{prop:NyInjSA}
For each $(y,\eta) \in Y \times \bbR^h,$ the operator
$N_{(y,\eta)}(D_X),$ together with the corresponding domain $\cD_{\max, Y<}(N_{(y,\eta)}(D_X)),$
is injective if $\eta =0$ and
otherwise has null space spanned by
\begin{equation*}
	\bigcup_{\lambda \in \Spec(D_{Z_y})\cap [0,1/2)}
	\lrbrac{
	s^{1/2} K_{\lambda-\tfrac12}(s|\eta|) \phi_{\lambda}
	- \cl(\tfrac{\eta}{|\eta|}) s^{1/2}K_{\lambda+\tfrac12}(s|\eta|) \cl(dx)\phi_{\lambda}
	: \cl(dx)D_{Z_y}\phi_{\lambda} = \lambda \phi_{\lambda} }.
\end{equation*}
It follows that, if $D_X$ satisfies the Witt condition at $Y,$ $\cN_y$ is injective and self-adjoint with its vertical APS domain.
\end{proposition}
 
\begin{proof}
Since $L^2(s^v ds) = s^{-v/2}L^2(ds),$ the operator $\cN_{(y,\eta)}(D_X)$ acting on $L^2_{\w}$ is equivalent to the operator
\begin{equation*}
	\cN_{(y,\eta)}(\eth_X) = s^{v/2} \cN_{(y,\eta)}(D_X) s^{-v/2} 
	=  \cl(dx)\pa_s + \tfrac1s D_{Z_y} + i \cl(\eta) 
\end{equation*}
acting on $L^2(ds).$

First note that, with $A_y = \cl(dx)D_{Z_y},$ we have
\begin{equation*}
	N_{(y,\eta)}(\eth_X)^2 
	= -\pa_s^2 + \tfrac1{s^2}(D_{Z_y}^2 - \cl(dx)D_{Z_y}) + |\eta|^2
	= -\pa_s^2 + \tfrac1{s^2}(A_y^2 - A_y) + |\eta|^2.
\end{equation*}
It will be useful to recall that the null space of 
\begin{equation*}
\begin{gathered}
	-\pa_s^2+ \tfrac1{s^2}(\lambda^2-\lambda) + |\eta|^2
	= s^{1/2}(-\pa_s^2-\tfrac1s\pa_s+\tfrac1{s^2}(\lambda-\tfrac12)^2+|\eta|^2)s^{-1/2} \\
	= -s^{1/2-2}\lrpar{s^2\pa_s^2+s\pa_s-((\lambda-\tfrac12)^2+s^2|\eta|^2)}s^{-1/2}
\end{gathered}
\end{equation*}
is spanned by $s^{\lambda}$ and $s^{1-\lambda}$ if $\eta=0,$ and otherwise by $s^{1/2}I_{\lambda-\tfrac12}(s|\eta|)$ and $s^{1/2}K_{\lambda-\tfrac12}(s|\eta|).$
In view of the asymptotics
\begin{equation*}
\begin{gathered}
	I_{\alpha}(z) = \cO(z^{|\alpha|}) \Mas z\to 0, \quad 
	I_{\alpha}(z) = \cO(\tfrac1z e^{z})  \Mas z \to \infty \\
	K_{\alpha}(z) = \cO(z^{-|\alpha|}) \Mas z\to 0 \; (\alpha\neq 0), \quad 
	K_{\alpha}(z) = \cO(\tfrac1z e^{-z})  \Mas z \to \infty,
\end{gathered}
\end{equation*}
none of these are in $L^2(ds)$ except for $s^{1/2}K_{\lambda-\tfrac12}(s|\eta|)$ when $|\lambda-\tfrac12|<1.$\\

Now since
\begin{multline*}
	\cl(dx)D_{Z_y} = -D_{Z_y} \cl(dx), \quad
	\cl(\eta) D_{Z_y} =- D_{Z_y} \cl(\eta) \\
	\implies
	\cl(dx) A_y = -\cl(dx) D_{Z_y} \cl(dx) = -A_y \cl(dx), \quad
	\cl(\eta) A_y 
	= A_y \cl(\eta),
\end{multline*}
the operator $\cN_{(y,\eta)}(\eth_X) = \cl(dx)(\pa_s - \tfrac1s A_y)+i \cl(\eta)$ preserves the space $F_{\lambda} = E_{\lambda}(A_y) \oplus E_{-\lambda}(A_y),$ on which it acts by
\begin{equation*}
	\cN_{(y,\eta,\lambda)} = 
	\begin{pmatrix}
	i\cl(\eta) & \cl(dx)(\pa_s + \tfrac1s \lambda) \\
	\cl(dx)(\pa_s - \tfrac1s\lambda) & i\cl(\eta)
	\end{pmatrix}.
\end{equation*}
If we further identify
\begin{equation*}
	\xymatrix @R=1pt
	{ F_{\lambda} \ar[r] & E_{\lambda}^2 = E_{\lambda}(A_y)^2 \\
	(a,b) \ar@{|->}[r] & (a, \cl(dx)b) }
\end{equation*}
we end up with the map
\begin{equation*}
\begin{gathered}
	\wt \cN_{y,\eta,\lambda} = 
	\begin{pmatrix}
	i\cl(\eta) & -\cl(dx)(\pa_s + \tfrac1s \lambda)\cl(dx) \\
	\cl(dx)^2(\pa_s - \tfrac1s\lambda) & -\cl(dx)i\cl(\eta)\cl(dx)
	\end{pmatrix}
	=\begin{pmatrix}
	i\cl(\eta) & \pa_s + \tfrac1s \lambda \\
	-(\pa_s - \tfrac1s\lambda) & -i\cl(\eta)
	\end{pmatrix} \\
	\text{ satisfying }
	\wt \cN_{y,\eta,\lambda}^2 = 
	\begin{pmatrix}
	|\eta|^2 -\pa_s^2 + \tfrac1{s^2}(\lambda^2-\lambda) & 0 \\
	0 & |\eta|^2 -\pa_s^2 + \tfrac1{s^2}(\lambda^2+\lambda) 
	\end{pmatrix}.
\end{gathered}
\end{equation*}
Thus any element of the $L^2$-null space of $\wt \cN_{y,\eta,\lambda}^2$ has the form
\begin{equation*}
	(as^{1/2}K_{\lambda-\tfrac12}(s|\eta|), bs^{1/2}K_{\lambda+\tfrac12}(s|\eta|)).
\end{equation*}
Recall that \cite[(9.6.26)]{AS1964}
\begin{equation*}
	\pa_{z}K_{\nu}(z) 
	= -K_{\nu-1}(z) - \tfrac\nu{z}K_{\nu}(z) 
	= -K_{\nu+1}(z)+\tfrac\nu{z}K_{\nu}(z)
\end{equation*}
hence applying $\wt \cN_{y,\eta,\lambda}$ to the putative element of the null space results in
\begin{equation*}
	\lrpar{ (\cl(\eta)a - b|\eta|)s^{1/2}K_{\lambda-\tfrac12}(s|\eta|),
	(-a|\eta| - \cl(\eta)b)s^{1/2}K_{\lambda+\tfrac12}(s|\eta|) } = 0 \iff
	\cl(\eta) a = |\eta| b,
\end{equation*}
and significantly $a=0\iff b=0.$ This means that to get an element in
$L^2$ we need both $|\lambda-\tfrac12|<1$ and
$|\lambda+\tfrac12|<1,$ i.e., $|\lambda|<\tfrac12.$
This establishes the first part of the proposition.

Recall that, for $\nu\in\bbR^*,$ $K_{\nu}(z) \sim C_{\nu} z^{-|\nu|}$ as $z\to 0.$
Thus the elements of the null space of $\cN_y$ are spanned by elements with non-trivial asymptotics at both exponents $\lambda$ and $-\lambda,$ for $\lambda \in \Spec(A_y)\cap[0,1/2).$ There are no such elements in the vertical APS domain. 

Finally we show that the vertical APS domain is self-adjoint by showing the vanishing of its deficiency indices by a similar argument. Indeed, the null space of $\cN_y\pm i$ is contained in the null space of $\cN_y^2+1.$ Analyzing this as above shows that solutions are built up from elements of the form
\begin{equation*}
	(as^{1/2}K_{\lambda-\tfrac12}(s\ang{\eta}), bs^{1/2}K_{\lambda+\tfrac12}(s\ang{\eta})).
\end{equation*}
where $\ang{\eta} = \sqrt{|\eta|^2+1}$ and solutions coming from the null space of $\cN_y\pm i$ further satisfy that $a=0\iff b=0.$ Since this requires non-trivial asymptotics at exponents $\lambda$ and $-\lambda,$ there are no such solutions in the vertical APS domain.

\end{proof}

\section{Edge calculus with bounds and wedge heat calculus}

\subsection{Conormal distributions on manifolds with corners}\label{sec:Conormal}

We briefly recall some of the results of \cite{Melrose:Conormal} that we will use in our constructions and refer the reader to {\em loc. cit.} for details. (See also, e.g., \cite{Grieser:Basics}, \cite[\S2A]{Mazzeo:Edge}.)

Recall that we use the notation 
\begin{equation*}
	\rho_X = \prod_{H \in \cM_1(X)} \rho_H
\end{equation*}
for a `total boundary defining function'. A \textbf{multiweight} for $X$ is a map
\begin{equation*}
	\mf s: \cM_1(X) \lra \bbR \cup \{\infty\}
\end{equation*}
and we denote the corresponding product of boundary defining functions by
\begin{equation*}
	\rho_X^{\mf s} = \prod_{H \in \cM_1(X)} \rho_H^{\mf s(H)}.
\end{equation*}
We write $\mf s \leq \mf s'$ if $\mf s(H) \leq \mf s'(H)$ for all $H \in \cM_1(X).$

$ $\\
A smooth map between manifolds with corners $f:X \lra Y$ is a \textbf{b-map} if, for each $H \in \cM_1(Y),$ and some choice of boundary defining functions, we have
\begin{equation*}
	f^*\rho_H = \prod_{H \in \cM_1(X)} \rho_G^{e_f(H,G)}
\end{equation*}
where $e_f(H,G)$ is a non-negative integer. (These are called
`interior b-maps' in \cite{Melrose:Conormal} because they map the
interior of the domain into the interior of the target.)
The map
\begin{equation*}
	e_f:\cM_1(X) \times \cM_1(Y) \lra \bbN_0
\end{equation*}
is known as the {\em exponent matrix} of the b-map $f$ and we write
\begin{equation*}
	\ker(e_f) = \{ H \in \cM_1(X): e_f(H,G)=0 \Mforall G \in \cM_1(Y) \}.
\end{equation*}

The vector fields tangent to the boundary hypersurfaces of $X$ are known as the {\em b-vector fields} and are denoted
\begin{equation*}
	\cV_b = \{ V \in \CI(X;TX) : V \text{ is tangent to each } H \in \cM_1(X) \}.
\end{equation*}
There is a vector bundle, the {\em b-vector bundle}, ${}^bTX,$ together with a natural vector bundle map $i_b:{}^bTX \lra TX$ that is an isomorphism over the interior of $X$ and satisfies
\begin{equation*}
	(i_b)_*\CI(X;{}^bTX) = \cV_b.
\end{equation*}
Thus, for example, if $x$ is a boundary defining function for a
boundary hypersurface $H$ of $X$ then, near $H,$ the vector field
$x\pa_x$ is non-degenerate at $H$ as a section of ${}^bTX.$ Indeed, it
does not vanish at $H$ because it is not an element of $x\cV_b.$ We
refer to any such vector field as a {\em radial vector field} for $H.$
It is determined up to an element of $x\cV_b,$ and its restriction to
the boundary generates a canonical trivialization of the null space of
$i_b$ over $H,$ known as the {\em b-normal bundle}, ${}^bNH.$

The differential of a b-map $f: X \lra Y$ extends to a bundle map
between b-tangent bundles and b-normal bundles. If both of these
induced maps are surjective, $f$ is a {\em b-fibration}.

$ $\\
{\bf Conormal functions.}
Let $\mu$ denote a positive section of the density bundle $\Omega(X).$
Denote
\begin{equation*}
	L^2(X) = L^2(X,\mu) = \{ u:X\lra \bbC \text{ measurable }: \int_X |u|^2 \; \mu <\infty \}
\end{equation*}
and, for $n \in \bbN_0$ and $\mf s$ a multiweight, the {\em weighted b-Sobolev spaces} corresponding to $\mu$ are
\begin{equation*}
	\rho_X^{\mf s}H^n_b(X) = \rho_X^{\mf s}H^n_b(X,\mu) 
	= \{ u: X \lra \bbC \text{ measurable }: \cV_b(X)^n(\rho_X^{-\mf s}u) \subseteq L^2(X)\}.
\end{equation*}
The $L^2$-based {\em conormal spaces} are
\begin{equation*}
	\rho_X^{\mf s}H_b^{\infty}(X) = \bigcap_{n \in \bbN_0} \rho_X^{\mf s}H^n_b(X)
\end{equation*}
though we shall usually use
\begin{equation*}
	\sA^{\mf s}_-(X) = \bigcap_{\mf s' < \mf s}\rho_X^{\mf s'}H_b^{\infty}(X).
\end{equation*}
We refer to these as conormal functions with multiweight $\mf s-.$
We denote the union over all multiweights by
\begin{equation*}
	\sA^*(X) = \bigcup_{\mf s} \rho_X^{\mf s}H_b^{\infty}(X).
\end{equation*}

By Sobolev embedding, any function in $\sA^{*}(X)$ is smooth in the interior of $X,$ and indeed the individual $\rho_X^{\mf s}H_b^{\infty}(X)$ are preserved by the action of $\cV_b(X).$ They are also $\CI(X)$-modules, so it makes sense to talk about conormal sections of a vector bundle $E \lra X,$ e.g.,
\begin{equation*}
	\sA^{\mf s}_-(X;E) = \sA^{\mf s}_-(X) \otimes_{\CI(X)} \CI(X;E).
\end{equation*}

$ $\\
{\bf $\cI$-smooth (or polyhomogeneous) expansions.}
Regularity at the boundary hypersurfaces is often manifest in an asymptotic expansion reminiscent of the Taylor series but with exponents that are not necessarily integers and with the presence of powers of logarithms,
\begin{equation*}
	u \sim \sum u_{s_j,p} x^{s_j}(\log x)^p
\end{equation*}
with coefficients $u_{s_j,p}$ themselves conormal functions.
We keep track of the allowed exponents in index sets and refer to this class of functions as $\cI$-smooth ($\cI$ for index set) or as polyhomogeneous.

An {\em index set} $E$ is a discrete subset of $\bbC \times \bbN_0$ such that 
\begin{equation*}
	\{ (s_j,p_j) \}\subseteq E, 
	|(s_j,p_j)| \to \infty
	\implies
	\Re s_j \to \infty.
\end{equation*}
To ensure independence from the choice of bdf $x$ we also require
\begin{equation*}
\begin{gathered}
	(z,p) \in E, \; p \geq 1 \implies (z,p-1) \in E \\
	(z,p) \in E \implies (z+k,p) \in E \Mforall k \in \bbN.
\end{gathered}
\end{equation*}

We often denote the index set $\{ (\alpha+n, 0) \in \bbC \times \bbN_0 : n \in \bbN_0 \}$ simply as $\alpha.$
The {\em extended union} of two index sets is
\begin{equation*}
	E \bar \cup F = E \cup F \cup \{ (z,p) \in \bbC \times \bbN_0 : (z,q) \in E \Mand (z,p-q-1) \in F \Mforsome q \in \bbN_0 \}.
\end{equation*}
Given an index set $E$ we define
\begin{equation*}
	\Re E = \{ \Re(z) : (z,0) \in E \}, \quad \inf E = \min \Re E.
\end{equation*}
We allow the empty set as an index set and define $\inf \emptyset = \infty.$

To each index set $E$ and $w \in \bbR$ we assign the polynomial
\begin{equation*}
	b(E,w;\lambda) = \prod_{ \substack{ (z,p)\in E \\ \Re z < w } }(\lambda - z).
\end{equation*}
Note that if $r_H$ is a radial vector field for $H$ then the null space of the differential operator $b(E,s;r_H)$ is spanned by 
\begin{equation*}
	\{ x^{z}(\log x)^p : (z,p) \in E, \Re z <w \}.
\end{equation*}

An {\em index family} $\cE$ on a manifold with corners is an assignment of an index set $\cE(H)$ to each boundary hypersurface $H.$
To each index family $\cE$ we associate a multiweight $\inf \cE.$
Given an index family, a multiweight $\mf w,$ a choice of radial vector field $r_H$ for each boundary hypersurface, and an ordering of the boundary hypersurfaces we define the differential operator
\begin{equation*}
	b(\cE,\mf w) = \prod_{H \in \cM_1(X)} b(\cE(H), \mf w(H); r_H)
\end{equation*}
and the spaces of {\em partially $\cI$-smooth conormal functions} by
\begin{equation}\label{eq:pphg}
	\sB_{phg}^{\cE/\mf w}\sA^{\mf s}_- (X)
	= \{ u \in \sA^{\mf s}_-(X): b(\cE,\mf r)u \in \sA^{\mf r}_-(X) \Mforall \mf s\leq \mf r\leq \mf w\}.
\end{equation}
\nomenclature[B]{$\sB_{phg}^{\cE/\mf w}\sA^{\mf s}_- (X)$}{The partially $\cI$-smooth conormal functions}
Thus these are conormal functions with multiweight $\mf s-$ that have an asymptotic expansion at each boundary hypersurface $H$ with exponents in $\cE(H)$ with real part less than $\mf w(H)$ and with remainder a conormal function with multiweight $\mf w-$.
The space of (totally) $\cI$-smooth conormal functions with index
family $\cE$ is 
\begin{equation*}
	\sA_{phg}^{\cE}(X) = \bigcap_{\mf w} \sB_{phg}^{\cE/\mf w}\sA^{\mf s}_-(X)
\end{equation*}
\nomenclature[A]{$\sA_{phg}^{\cE}(X)$}{The space of (totally)
  $\cI$-smooth conormal functions with index family $\cE$}
where $\mf s$ is any multiweight satisfying $\mf s < \inf \cE.$

When the empty set is used as an index set we interpret $\sB_{phg}^{\emptyset/\mf w}\sA^{\mf s}_-(X) = \sA^{\mf w}_-(X)$ whenever $\mf s\leq \mf w.$

$ $\\
{\bf Pull back and push forward.}
If $f:X \lra Y$ is a b-map, then to each multiweight $\mf r$ on $Y$ we associate a multiweight on $X,$
\begin{equation*}
	\cM_1(X) \ni H \mapsto
	f^{\sharp}\mf r(H) = \sum_{G \in \cM_1(Y)}e_f(H,G)\mf r(G).
\end{equation*}
Note that $f^{\sharp}\mf r(H) =0$ for any $H \in \ker(e_f).$
Let $\mf n_f$ be the multiweight on $X,$
\begin{equation*}
	\mf n_f(H) = \begin{cases}
	\infty & \Mif H \in \ker(e_f) \\
	0 & \Melse
	\end{cases}
\end{equation*}
To an index family $\cF$ on $Y$ we associate an index family on $X,$
\begin{multline*}
	\cM_1(X) \setminus \ker(e_f) \ni H \mapsto
	f^{\sharp}\cF(H) = \Big\lbrace (S,P) : \exists \{(s_G, p_G) \in \cF(G) : e_f(H,G) \neq 0\} \\ \Mst
	S = \sum e_f(H,G) s_G, P = \sum p_G \Big\rbrace
\end{multline*}
and $f^{\sharp}\cF(H) = 0$ for all $H \in \ker(e_f).$
For any multiweights $\mf r,$ $\mf r'$ and index family $\cF$ on $Y,$ pull-back along $f$ gives a map \cite[Theorem 3]{Melrose:Conormal}
\begin{equation*}
	f^*: \sB^{\cF/\mf r'}\sA^{\mf r}_-(Y) \lra
	\sB^{f^{\sharp}\cF/(f^{\sharp}\mf r' + \mf n_f)}\sA^{f^{\sharp} r}_-(X).
\end{equation*}

Similarly, if $f:X \lra Y$ is a b-fibration (defined above) we can associate to each multiweight $\mf s$ on $X$ a multiweight on $Y,$
\begin{equation*}
	\cM_1(Y) \ni G \mapsto f_{\sharp}\mf s(G) = \min \{ \mf s(H)/e_f(H,G) : H \in \cM_1(X), e_f(H,G) \neq 0 \},
\end{equation*}
and to each index family $\cE$ on $X$ an index family on $Y,$
\begin{equation*}
	\cM_1(Y) \ni G \mapsto f_{\sharp}\cE(G) = \overbar{\bigcup_{ \substack{ H \in \cM_1(X) \\ e_f(H,G)\neq 0 } } }
	\{ (z/e_f(H,G), p) : (z,p) \in \cE(H) \}.
\end{equation*}
For any multiweights $\mf s,$ $\mf s'$ and index family $\cE$ on $X$ satisfying
\begin{equation*}
	H \in \ker(e_f) \implies \inf \cE(H) >0,
\end{equation*}
push-forward along $f$ gives a map \cite[Theorem 5]{Melrose:Conormal}
\begin{equation*}
	f_*: \sB^{\cE/\mf s'}\sA^{\mf s}_-(X;\rho_X^{-1}\Omega) \lra 
	\sB^{f_{\sharp}\cE/f_{\sharp}\mf s'}\sA^{f_{\sharp}\mf s}(Y; \rho_Y^{-1}\Omega).
\end{equation*}

$ $\\
These theorems hold with functions replaced by sections of a vector bundle with only notational differences. 
Another useful extension is to sections that are also conormal with respect to an interior \textbf{p-submanifold}. A submanifold $W \subseteq X$ is a p-submanifold if every point in $W$ has a neighborhood $\cU$ in $X$ such that 
\begin{equation}\label{eq:DefPSub}
\begin{gathered}
	X \cap \cU = X' \times X'', \Mwhere \pa X'' = \emptyset,\\
	W \cap \cU = X' \times \{ p'' \} \Mforsome p''\in X''.
\end{gathered}
\end{equation}
We will not detail this extension but refer the reader to e.g., \cite[Appendix B]{Epstein-Melrose-Mendoza}.

\subsection{Edge double space}\label{sec:DoubleSpace}

Given a manifold with corners and an iterated fibration structure $X,$ we follow \cite{Mazzeo:Edge, Mazzeo-Witten:KW} and define edge pseudodifferential operators by describing their integral kernels on a replacement of $X^2$ that takes the iterated fibration structure into account.

Recall that the radial blow-up of a manifold with corners $X$ along a p-submanifold $W$ (as in \eqref{eq:DefPSub}) is the manifold with corners $[X;W]$ obtained by replacing $W$ with the inward-pointing part of its spherical normal bundle, see e.g., \cite[\S4.2]{tapsit}, \cite[\S2.2]{Mazzeo-Melrose:Surgery}, \cite[Chapter 5]{damwc}. \\

Recall that there is a partial order on $\cS(X),$ $Y < Y'$ iff
$\bhs{Y}\cap \bhs{Y'}\neq\emptyset$ and $\dim Y < \dim Y'.$
The edge double space associated to $X$ is obtained from $X^2$ by
blowing-up certain p-submanifolds.
 For each $Y \in \cS(X)$ we denote the the fiber diagonal of $\phi_Y$ in $X^2$ by
\begin{equation*}
	\bhs{Y} \times_{\phi_Y} \bhs{Y} 
	= \{ (\zeta, \zeta') \in (\bhs{Y})^2 : \phi_Y(\zeta) = \phi_Y(\zeta') \}.
\end{equation*}

\begin{definition}\label{def:EdgeDouble}
Let $X$ be a manifold with corners and an iterated fibration structure. Let $\cS(X) = \{Y_1, Y_2, \ldots, Y_\ell\}$ be a listing of $\cS(X)$ such that $Y_i < Y_j \implies i < j,$ i.e., such that the list is non-decreasing in depth. The {\bf edge double space of $X$} is 
\begin{equation}\label{eq:EdgeDoubleSpace}
	X^2_\e = [X^2; 
	\bhs{Y_1} \times_{\phi_{Y_1}} \bhs{Y_1};
	\bhs{Y_2} \times_{\phi_{Y_2}} \bhs{Y_2};
	\ldots ;
	\bhs{Y_{\ell}} \times_{\phi_{Y_\ell}} \bhs{Y_\ell}].
\end{equation}
\end{definition}

As in, e.g.,  \cite{Dai-Melrose, Melrose-Piazza:Even}, there is an analogous construction for families of manifolds with corners.

\begin{definition}
Given a fiber bundle $M \xlra\psi B$ of manifolds with corners and iterated fibration structures as in Definition \ref{def:FibIFS}, fix a non-decreasing list of $\cS_{\psi}(M),$ $\{N_1, N_2, \ldots, N_{\ell}\},$ let the {\bf families edge double space} be
\begin{equation*}
	(M/B)^2_\e =
	\Big[
	M \times_{\psi} M; 
	\bhs{N_1} \times_{\phi_{N_1}} \bhs{N_1}; \ldots;
	\bhs{N_\ell} \times_{\phi_{N_\ell}} \bhs{N_\ell} \Big].
\end{equation*}
The map $\psi$ induces a fiber bundle 
\begin{equation*}
	X^2_\e \fib (M/B)^2_\e \xlra{\psi_{(2)}} B.
\end{equation*}
\end{definition}

Let us check, as is implicit in Definition \ref{def:EdgeDouble}, that after performing the first $k-1$ blow-ups, the lift of $\bhs{Y_k} \times_{\phi_{Y_k}} \bhs{Y_k}$ to the blown-up space is a p-submanifold.
Local coordinates near $\bhs{Y},$ say $x, y, z$ where $x$ is a bdf for $\bhs{Y},$ $y$ are coordinates along $Y$ and $z$ coordinates along the fiber $Z$ of $\phi_Y,$  induce coordinates $x,y,z, x', y', z'$ near $\bhs{Y} \times \bhs{Y},$ in which 
\begin{equation*}
	\bhs{Y} \times_{\phi_Y} \bhs{Y} = \{x=x'=0, y=y'\}
\end{equation*}
and so this is a p-submanifold whenever $Y$ is a closed manifold, e.g., for $Y_1.$

If $Y<\wt Y,$ so that $\wt Y$ has a collective boundary hypersurface $\bhs{Y\wt Y}$ as in \eqref{eq:bfsDiag}, let us label the fibers of these fiber bundles,
\begin{equation}\label{eq:IFSDetail'}
	\xymatrix{
	Z \ar@{}[r]|-*[@]{\supseteq} & \bhs{\wt Y Z} \ar@{-}[ld] \ar@{-}[rd] \ar[rrrr]^{\phi_{\wt Y Z}} & & & &  W \ar@{-}[ld] \\
	\wt Z \ar@{-}[rr] & & \bhs{Y} \cap \bhs{\wt Y} \ar[rr]^{\phi_{\wt Y}} \ar[rd]_{\phi_{\wt Y}} & & 
		\bhs{Y\wt Y} \ar[ld]^{\phi_{Y\wt Y}} \ar@{}[r]|-*[@]{\subseteq} & \wt Y \\
	&&& Y }
\end{equation}
and choose coordinates near $\bhs{Y} \cap \bhs{\wt Y}$ of the form
\begin{equation}\label{eq:IFSDetail''}
	x,\quad y, \quad w, \quad r, \quad \wt z,
\end{equation}
in which $x$ is a bdf for $Y$ and $r$ is a bdf for $\wt Y,$ $y$ are coordinates along $Y,$ $w$ coordinates along $W$ and $\wt z$ coordinates along $\wt Z,$ so that $(x,y,w)$ are coordinates along $\wt Y$ and $(w,r, \wt z)$ are coordinates along $Z.$
In the induced coordinates $x, y,  w,  r,  \wt z, x', y',  w',  r',  \wt z',$ we have
\begin{equation*}
	\bhs{Y} \times_{\phi_Y} \bhs{Y} = \{x=x'=0, y=y'\}, \quad
	\bhs{\wt Y} \times_{\phi_{\wt{Y}}} \bhs{\wt Y} = \{r=r'=0, (x,y,w) = (x',y',w') \}
\end{equation*}
which shows that the latter is not a p-submanifold.
After blowing-up the former, projective coordinates with respect to $x'$ are given by
\begin{equation}\label{eq:ProjCoords}
	s = \frac x{x'}, \quad
	u = \frac{y-y'}{x'}, \quad
	w, \quad 
	r, \quad
	\wt z, \quad
	x', \quad
	y', \quad
	w', \quad
	r', \quad
	\wt z'
\end{equation}
and the interior lift of $\bhs{\wt Y} \times_{\phi} \bhs{\wt Y}$ is given by 
\begin{equation}\label{eq:PSubMfd}
	\{ r=r'=0, \;  s=1, \; u=0,\;  w=w'\},
\end{equation}
which is a p-submanifold.

Thus the manifolds blown-up in \eqref{eq:EdgeDoubleSpace} are p-submanifolds.
We denote the blow-down map by 
\begin{equation*}
	\beta_{(2)}: X^2_\e \lra X^2
\end{equation*}
and the compositions with the projections onto the left and right factors by
\begin{equation*}
	\xymatrix{
	& X^2_\e  \ar@/_1.5pc/[ldd]_-{\beta_{(2),L}} \ar[d]^-{\beta_{(2)}} \ar@/^1.5pc/[ddr]^-{\beta_{(2),R}} & \\
	& X^2
	\ar[ld]^-{\pi_L} \ar[rd]_-{\pi_R} & \\
	X & & X } 
\end{equation*}
Since the collective boundary hypersurfaces may contain more that one connected component, this is an `overblown' version of the double space in \cite[\S2]{Mazzeo:Edge} and an edge version of the double space in \cite[Appendix]{Melrose-Piazza:K}.

The edge double space has collective boundary hypersurfaces, for each $Y \in \cS(X),$
\begin{equation*}
	\bhss{Y} \times X \leftrightarrow \bhsd{10}(Y), \quad
	X \times \bhss{Y} \leftrightarrow \bhsd{01}(Y), \quad
	\bhss{Y} \times_{\phi_Y} \bhss{Y} \leftrightarrow \bhsd{\phi\phi}(Y),
\end{equation*}
where the notation indicates that, e.g., the interior lift of $\bhss{Y} \times X$ is the boundary hypersurface $\bhsd{10}(Y)$ of $X^2_\e.$ We denote the family of collective boundary hypersurfaces produced by the blow-ups by $\ff(X^2_\e)$ (the `front faces') and the other collective boundary hypersurfaces by $\sf(X^2_\e)$ (the `side faces'), thus
\begin{equation}\label{eq:side-and-front}
\begin{gathered}
	\ff(X^2_\e) = \{ \bhsd{\phi\phi}(Y) : Y \in \cS(X) \}, \\
	\sf(X^2_\e) = \{ \bhsd{10}(Y) , \bhsd{01}(Y) : Y \in \cS(X) \}.
\end{gathered}
\end{equation}
We use similar notations for the family edge double space $(M/B)^2_{\e},$ e.g., $\bhsd{\phi\phi}(N).$\\

It will be useful to describe the structure of these collective hypersurfaces in more detail.
If $\cS(X) = \{Y \},$ the case treated in \cite{Mazzeo:Edge}, then the restriction of the blow-down map 
\begin{equation*}
	\bhsd{\phi\phi}(Y) \lra \bhs{Y} \times_{\phi_Y} \bhs{Y}
\end{equation*}
is the fiber bundle map of the inward pointing spherical normal bundle of $\bhs{Y} \times_{\phi_Y} \bhs{Y}$ in $X^2.$ The fiber is a quarter sphere $\bbS^{h+1}_{++},$ where $h = \dim Y.$
Invariantly the spherical normal bundle at a point $q \in \bhs{Y} \times_{\phi_Y} \bhs{Y}$ is obtained from $T_qX^2$ by moding out by tangent vectors to $\bhs{Y} \times_{\phi_Y} \bhs{Y},$ removing zero, and taking the $\bbR^+$-orbit space of the dilation action,
\begin{equation*}
	\bbS(N_{X^2}(\bhs{Y} \times_{\phi_Y} \bhs{Y})_q)
	= \bbR^+ \diagdown 
	[(T_qX^2 / T_q(\bhs{Y} \times_{\phi_Y} \bhs{Y})) \setminus \{0\}].
\end{equation*}
and so every vector field on $X^2$ transverse to $\bhs{Y}
\times_{\phi_Y}\bhs{Y}$ (i.e.\ not in the image of its tangent space
via the inclusion) defines a section of the spherical normal bundle,
and every inward-pointing vector field defines a section of the fiber
bundle $\bhsd{\phi\phi}(Y) \lra \bhs{Y} \times_{\phi_Y} \bhs{Y}.$ 
As pointed out in \cite[(3.10)]{Mazzeo-Melrose:Zero}, there is a canonical section: 
let $\nu'$ be any vector field on $X$ that is inward pointing at
$\bhs{Y},$ denote the corresponding vector fields acting on the left,
respectively right, factor of $X$ in $X^2$ by $\nu_L'$, respectively
$\nu_R',$ set $\nu = \nu_L' + \nu_R'$ and let $[\nu_{\phi_Y}]$ be the  induced section of $\bhsd{\phi\phi}(Y) \lra \bhs{Y} \times_{\phi_Y} \bhs{Y},$
\begin{equation*}
	[\nu_{\phi_Y}]: \bhs{Y} \times_{\phi_Y} \bhs{Y} \lra
	\bhsd{\phi\phi}(Y).
\end{equation*}
A different choice of $\nu'$ would change the value $\nu'(y,z)$ at a point $(y,z) \in \bhs{Y}$ by multiplication by a positive constant and addition of a vector tangent to $T_{(y,z)}\bhs{Y},$ and correspondingly change $\nu(y,z,z')$ by multiplication by a positive constant and addition of a vector in $T_{(y,z,z')}(\bhs{Y} \times_{\phi_Y} \bhs{Y}),$ and hence would not change $[\nu](y,z,z').$ We denote the image of $[\nu]$ by
\begin{equation*}
	\nu_{\phi_Y}(\bhs{Y}) 
	= [\nu_{\phi_Y}](\bhs{Y} \times_{\phi_Y} \bhs{Y}) \subseteq \bhsd{\phi\phi}(Y).
\end{equation*}
For reasons described below, this will be referred to as the {\bf identity section} of the $\bbS^{h+1}_{++}$-bundle, in analogy to the zero section of a vector bundle. 

A choice of connection for $\bhs{Y} \times_{\phi_Y} \bhs{Y} \lra Y$
lets us identify the normal bundle to $\bhs{Y} \times_{\phi_Y}
\bhs{Y}$ in $X^2$
with the pull-back of $TY$ (as the normal bundle to the diagonal of $Y$ in $Y^2$), times two copies of $N^+_X\bhs{Y},$ the inward-pointing normal bundle to $\bhs{Y}$ in $X,$ one for each factor of $X^2,$
\begin{equation*}
	\bbS(TY \times (N^+_X\bhs{Y})^2) 
	=  \bhsd{\phi\phi}(Y) \lra \bhs{Y} \times_{\phi_Y} \bhs{Y}.
\end{equation*}
With this identification, the identity section $\nu_{\phi_Y}(\bhs{Y})$ is the subset $\pi(\{0\}\times \{N_L\} \times \{N_R\})$ where $\pi$ is
the projection from $TY \times (N^+_X\bhs{Y})^2$ minus its zero section onto its sphere bundle, and $N_L,$ $N_R,$ denote the pull-back along the left and right of an inward-pointing vector field transverse to $\bhs{Y}$ in $X.$
We can compose the blow-down map with the fiber bundle map $\bhs{Y} \times_{\phi_Y} \bhs{Y} \lra Y$ to obtain a fiber bundle
\begin{equation*}
	\bbS^{h+1}_{++} \times Z^2 \fib \bhsd{\phi\phi}(Y) \lra Y.
\end{equation*}
Thus near its front face, $X^2_e$ is locally diffeomorphic to $\bbR^+ \times \bbS^{h+1}_{++} \times Z^2 \times \cU_Y$ with $\cU_Y$ an open set in $Y$ over which $\phi_Y$ is trivial.

Finally, in local projective coordinates analogous to \eqref{eq:ProjCoords}, 
\begin{equation*}
	s = \frac x{x'}, \quad u = \frac{y-y'}{x'}, \quad z, \quad x', \quad y', \quad z',
\end{equation*}
$X^2_e$ is locally diffeomorphic to 
\begin{equation*}
	\bbR^+_s \times \bbR^h_u \times [0,1)_{x'} \times \cU_Y \times Z^2.
\end{equation*}
The submanifold $\nu_{\phi_Y}(\bhs{Y})$ in these coordinates is $\{s=1, u=0\}.$
Notice that if we view $\bbR^+_s \times \bbR^h_u$ as the `ax+b' group, $\bbR^+\ltimes \bbR^h,$ with product
\begin{equation*}
	(s,u) \cdot (s',u') = (ss', s'u + u'),
\end{equation*}
then $(1,0)$ is the identity element of the group.
This is the reason why we refer to  $\nu_{\phi_Y}(\bhs{Y})$ as the identity section.
The action of edge pseudodifferential operators is by convolution with respect to this group (see, e.g., \cite[(3.5)]{Mazzeo:Edge}).\\

As we discuss now, this structure persists in a modified way in the setting of manifolds with iterated fibration structures.

\begin{remark}
It may be useful to consider a 
`toy case' with underlying stratified space
\begin{equation*}
	\hat X = Y \times C_{[0,1)}(W \times C_{[0,1)}(\wt Z)),
\end{equation*}
where $C_{[0,1)}$ denotes the truncated cone, so that
\begin{equation*}
	X = Y \times [0,1)_x \times W \times [0,1)_r \times \wt Z, \qquad
	\cS(X) = \{Y, \quad \wt Y = Y \times [0,1)_x \times W \} 
\end{equation*}
with collective boundary hypersurfaces
\begin{equation*}
	\bhs{Y} = \{x=0\} = Y \times W \times [0,1)_r \times \wt Z, \qquad
	\bhs{\wt Y} = \{r=0\} = Y \times [0,1)_x \times W \times \wt Z 
\end{equation*}
participating in the (trivial) fiber bundles
\begin{equation*}
	Z = W \times [0,1)_r \times \wt Z \fib \bhs{Y} \xlra{\phi_Y} Y, \qquad
	\wt Z \fib \bhs{\wt Y} \xlra{\phi_{\wt Y}} \wt Y 
\end{equation*}
whose compatibility diagram takes the form
{\small
\begin{equation*}
	\xymatrix @C=5pt{
	Z = W \times [0,1)_r \times \wt Z \ar@{}[r]|-*[@]{\supseteq} & 
	\bhs{\wt Y Z} = W \times \wt Z \ar@{-}[ld] \ar@{-}[rd] \ar[rrrr]^{\phi_{\wt Y Z}} &
	 & & &  W \ar@{-}[ld] \\
	\wt Z \ar@{-}[rr] & & 
	\bhs{Y} \cap \bhs{\wt Y} = Y \times W \times \wt Z \ar[rr]^-{\phi_{\wt Y}}
	 \ar[rd]_{\phi_{\wt Y}} & & 
	\bhs{Y\wt Y} = Y \times W \ar[ld]^{\phi_{Y\wt Y}} 
	\ar@{}[r]|-*[@]{\subseteq} & \wt Y  \\
	&&& Y }
\end{equation*}
}
To construct $X^2_\e,$ we start with $X^2$ and blow-up 
\begin{equation*}
	\bhs{Y}\times_{\phi_Y}\bhs{Y} = \{x=x'=0, y=y'\}
	= \diag(Y^2) \times \{x=x'=0\} \times W^2 
	\times [0,1)_r^2 \times \wt Z^2.
\end{equation*}
Near the resulting front face the blown-up space is locally diffeomorphic to 
\begin{equation*}
	\bbR^+_R \times \bbS^{h_Y+1}_{++} \times  Y 
	\times W^2 \times [0,1)_r^2 \times \wt Z^2
\end{equation*}
where $h_Y$ is the dimension of $Y$ and $R$ is a defining function for
the front face.
In this local description, the interior lift of the submanifold
\begin{equation*}
	\bhs{\wt Y}\times_{\phi_{\wt Y}}\bhs{\wt Y} 
	= \{r=r'=0, (y,x,w)=(y',x',w')\}
	= \diag(Y^2 \times [0,1)_x^2 \times W^2)
	\times \{r=r'=0\} \times \wt Z^2
\end{equation*}
is equal to (cf.  \eqref{eq:PSubMfd})
\begin{equation*}
\begin{gathered}
	\bbR^+_R \times 
	\{[(1,0)]\} \times Y
	\times 
	\diag(W^2) \times \{r=r'=0\} \times \wt Z^2 \\
	=\bbR^+_R \times 
	\lrpar{ \nu_{\phi_Y}(\bhs{Y}) \cap
	(\diag(W^2) \times \{r=r'=0\} \times \wt Z^2)},
\end{gathered}
\end{equation*}
where $\{[(1,0)]\}$ denotes the identity element of the `ax+b group', as discussed above, and $\nu_{\phi_Y}(\bhs{Y})$ is the identity section of $\bbS(TY \times \mathbb{R}_+^2).$

Blowing-up the interior lift of $\bhs{\wt Y}\times_{\phi_{\wt
    Y}}\bhs{\wt Y}$ produces the manifold $X^2_{\e}.$ The normal
bundle of the interior lift of $\bhs{\wt Y}\times_{\phi_{\wt
    Y}}\bhs{\wt Y}$ fibres over $[0, 1)_R \times Y \times W$ with
fiber $\mathbb{R}^{h_Y + h_W + 1}$ where $h_Y,$ $h_W$ are the
dimensions of $Y$ and $W,$ respectively; indeed the normal bundle is exactly $\mathbb{R}_+ \times \mathbb{R}
\times TY \times TW$.  So near the intersection of the front faces this space
is locally diffeomorphic to 
\begin{equation*}
	\bbR^+_R \times \bbR^+_{\wt R} \times \bbS(\mathbb{R} \times
        TY \times TW \times \mathbb{R}_+^2) \times \wt Z^2,
\end{equation*}
where $R$, $\wt R$ are defining functions for $\bhsd{\phi\phi}(Y)$ and
$\bhsd{\phi\phi}(\wt Y),$ respectively.  The third term is a fibre bundle
$\bbS^{h_Y + h_W + 2}_{++} \fib \bbS(\mathbb{R} \times
        TY \times TW \times \mathbb{R}_+^2) \to [0,1)_R \times Y
        \times W \simeq \wt{Y}$.  

With an eye to the case of non-trivial fiber bundles, note that the normal bundle of 
$\bhs{\wt Y}\times_{\phi_{\wt Y}}\bhs{\wt Y}$ is naturally isomorphic to 
${}^{\e}T\tilde{Y}\times \mathbb{R}_+^2$,
where ${}^{\e}T\tilde{Y} $ is the edge tangent bundle, generated by
the vector fields $x \pa_x, x\pa_y, \pa_w$, as can be seen by lifting
these vector fields to $X^2$ from the left projection and restricting
to the interior lift of $\bhs{\wt Y}\times_{\phi_{\wt Y}}\bhs{\wt Y}$.  So in
fact this neighborhood can be expressed as
\begin{equation*}
	 \bbR^+_{\wt R} \times \bbS({}^{\e} T\wt{Y} 
	 \times \mathbb{R}_+^2) \times \wt Z^2,
\end{equation*}
where $\bbS^{h_Y + h_W + 2}_{++} \fib \bbS({}^{\e} T\wt{Y} \times
\mathbb{R}_+^2) \to \wt{Y}.$  

In particular, $\bhsd{\phi\phi}(Y)$ is locally diffeomorphic to 
\begin{equation*}
	\bbR^+_{\wt R} \times \bbS^{h_Y + 1 +  h_W + 1}_{++}
	\times Y \times W \times \wt Z^2.
\end{equation*}
To compare this to 
\begin{equation*}
	Z^2_{\e} 
	= [Z^2;
	\bhs{\wt Y Z} \times_{\phi_{\wt Y Z}} \bhs{\wt Y Z}]
	= [(W \times [0,1)_r \times \wt Z)^2;
	\diag(W^2) \times \{r=r'=0\} \times \wt Z^2]
\end{equation*}
note that this space has a similar description near its front face, namely
\begin{equation*}
	\bbR^+_{\wt R} \times \bbS^{h_W +1}_{++} \times W \times \wt Z^2.
\end{equation*}
Thus the local description of $\bhsd{\phi\phi}(Y)$ fibers over $Y$ and we can think of the fiber as a `suspended' version of the edge double space of $Z,$ 
\begin{equation*}
	[\bbS_{++}^{h_Y+1} \times Z^2; \nu_{\phi_Y}(\bhs{Y}) \cap 
	(\bhs{\wt Y Z} \times_{\phi_{\wt Y Z}} \bhs{\wt Y Z})].
\end{equation*}
\end{remark}

The front faces of the edge double space of $X$ are related to the
edge double spaces of the fibers of its boundary fibrations, but as
pointed out in the remark, they are `suspended' versions. To define
this structure in general we momentarily replace a boundary fiber bundle $Z \fib
\bhs{Y} \lra Y$ with an arbitrary fiber bundle of manifolds with
corners with iterated fibration structures:
\begin{definition}
Let $\wc X \fib \wc M \xlra{\wc \psi} \wc B$ be a fiber bundle of
manifolds with corners and iterated fibration structures and let
$\cS_{\wc \psi}(\wc M) = \{\wc N_1, \ldots, \wc N_{\ell}\}$ be a non-decreasing
listing of $\cS_{\wc \psi}(\wc M).$ 
Let 
$\pi \colon \cS_{++}(\wc M \times_{\wc \psi} \wc M) \lra \wc M
\times_{\wc \psi} \wc M$ be the pull-back of the fiber bundle
$\bbS({}^\e T\wc B \times \bbR_+^2)$ from $B$ to $\wc M \times_{\wc \psi} \wc M$ and let $\nu_{\wc \psi}(\wc M)$ denote the identity section.
The {\bf suspended edge double space of $\wc M/\wc B$} is 
\begin{equation*}
\begin{multlined}
    (\wc M/\wc B)^2_{Sus(\e)} = [\cS_{++}(\wc M \times_{\wc \psi} \wc
    M); \; \nu_{\wc \psi}(\wc M) \cap \pi^{-1}(\bhs{\wc N_1}
    \times_{\phi_{\wc N_1}} \bhs{\wc N_1});\;\ldots;\; \\
\qquad \qquad \nu_{\wc
      \psi}(\wc M) \cap \pi^{-1}(\bhs{\wc N_\ell} \times_{\phi_{\wc
        N_\ell}} \bhs{\wc N_\ell})].
\end{multlined}
  \end{equation*}

This fibers over $\wc B$ and we denote the typical fiber by $\wc X^2_{Sus_{\wc B}(\e)},$
\begin{equation*}
	\wc X^2_{Sus_{\wc B}(\e)} \fib (\wc M/\wc B)^2_{Sus(\e)} \lra \wc B.
\end{equation*}
\end{definition}

\begin{proposition}[Structure of the front faces of $X^2_\e$]\label{prop:ffX2}
Let $X$ be a manifold with corners and an iterated fibration structure. For each $Y \in \cS(X),$ let 
\begin{equation*}
	\phi_Y^{(2)}: \bhsd{\phi\phi}(Y) \lra Y
\end{equation*}
denote the composition of $\beta_{(2)}:X^2_\e \lra X^2$ with the
fibration $\bhs{Y}\times_{\phi_Y} \bhs{Y} \lra Y,$ restricted to $\bhsd{\phi\phi}(Y)$.  This map participates in a fiber bundle with fiber the suspended edge double space of $Z,$
\begin{equation*}
	Z^2_{Sus_Y(\e)} 
	\fib \bhsd{\phi\phi}(Y) = (\bhs{Y}/Y)^2_{Sus(\e)} \xlra{\phi_Y^{(2)}} Y.
\end{equation*}
\end{proposition}

Note that if $Y$ is a maximal element of $\cS(X)$ (and hence the fiber $Z$ of $\bhs{Y} \xlra{\phi_Y} Y$ is a closed manifold),
then the fiber bundle in the proposition over $Y$ is 
\begin{equation*}
	\bbS( \bbR^{\dim Y} \times \bbR_+^2) \times Z^2 \fib
	\bhsd{\phi\phi}(Y) \xlra{ \phi_Y^{(2)} } Y,
\end{equation*}
just as in \cite{Mazzeo:Edge}.

\begin{proof}

If $Y$ has $\mathrm{depth}_X(Y) = k$ then the fiber $Z$ of $\phi_Y: \bhs{Y} \lra Y$ is a manifold with corners and an iterated fibration structure of depth equal to $k-1.$ $Z$ has one collective boundary hypersurface for each $\wt Y \in \cS(X)$ such that $Y<\wt Y.$
Indeed following diagram \eqref{eq:IFSDetail'}, this boundary hypersurface, which we denote $\bhs{\wt Y Z},$ is the fiber of the restriction of $\phi_Y$ to $\bhs{Y}\cap \bhs{\wt Y}$ and its boundary fibration is the restriction of $\phi_{\wt Y}.$
Consequently we have
\begin{equation}\label{eq:IFSDetail2}
	\xymatrix @C=10pt {
	Z^2 \ar@{}[r]|-*[@]{\supseteq}  
	& \bhs{\wt Y Z} \times_{\phi_{\wt Y Z}} \bhs{\wt Y Z} \ar@{-}[ld] \ar@{-}[rd] \ar[rrrr]^-{\phi_{\wt Y Z}} 
		& & & &  W \ar@{-}[ld] \\
	\wt Z^2 \ar@{-}[rr] & & (\bhs{Y} \cap \bhs{\wt Y}) \times_{\phi_{\wt Y}} (\bhs{Y} \cap \bhs{\wt Y})
		 \ar[rr]^-{\phi_{\wt Y}} \ar[rd]_-{\phi_{Y}} & & 
		\bhs{Y\wt Y} \ar[ld]^-{\phi_{Y\wt Y}} \\
	&&& Y }
\end{equation}

Let $X^2_\e(k+1)$ denote the blow up in $X^2$ of all the fibre diagonals of strata of depth greater than or equal to $k+1.$
Let $\sN_{+}(Y)$ denote the inward pointing normal bundle of the interior lift of $\bhs{Y} \times_{\phi_{Y}} \bhs{Y}$ to $X^2_\e(k+1)$.  
Composing the projection down to $\bhs{ Y} \times_{\phi_{{Y}}} \bhs{ Y}$ with the projection of this space down to $Y$, we see, e.g.\ by using \eqref{eq:ProjCoords}  and lifting vector fields from the left projection and restricting to the interior lift, that we have a diagram of fiber bundles
\begin{equation*}
	\xymatrix{
	{Z}^2 \ar@{-}[r] & \sN_{+}( Y) \ar[d] \\
	\bbR^{h_{{Y}}} \ar@{-}[r] & {}^\e T Y \times \mathbb{R}_+^2  \ar[d]\\
	& Y
	}
\end{equation*}
where ${}^\e T{Y}$ is the edge tangent bundle.

Thus when we blow up the interior lift of $\bhs{Y} \times_{\phi_{Y}} \bhs{Y}$ to $X^2_\e(k+1),$
we produce a 
boundary hypersurface $\bar{\bhsd{\phi\phi}}(Y) = \bbS(\sN_{+}( Y) )$
which we can identify with the pull-back to $\bbS({}^\e TY \times
\bbR_+^2) \lra Y$ of two copies of the $Z$ bundle over $Y$.
In particular we have a fiber bundle map
\begin{equation*}
	\bbS^{h+1}_{++} \times Z^2 \fib 
	\bar{\bhsd{\phi\phi}}(Y)
	\xlra{\bar{\phi}_Y^{(2)}} Y.
\end{equation*}

If $\wt Y> Y,$ the interior lift of $\bhs{\wt Y} \times_{\phi_{\wt Y}}
\bhs{\wt Y}$ intersects $\bar{\bhsd{\phi\phi}}(Y)$ at the identity
section of $\bbS({}^\e TY \times \bbR_+^2)$ over the submanifold $\bhs{\wt Y Z} \times_{\phi_{\wt Y Z}} \bhs{\wt Y Z}$ of $Z$ (e.g., by the computation \eqref{eq:PSubMfd}). The restriction of $\bar{\phi}_Y^{(2)}$ fibers over $Y,$
\begin{equation}\label{eq:intersection-double-space}
	[(0,1, 1)] \times \bhs{\wt Y Z} \times_{\phi_{\wt Y Z}} \bhs{\wt Y Z} 
	\fib \bar{\bhsd{\phi\phi}}(Y) \cap 
		\bar\beta^{\sharp}(\bhs{\wt Y} \times_{\phi_{\wt Y}} \bhs{\wt Y})
	\xlra{\bar{\phi}_Y^{(2)}} Y.
\end{equation}
Hence blowing these submanifolds up in the appropriate order produces
$(\bhs{Y}/Y)^2_{Sus(\e)}$ as required.
\end{proof}

We point out that, just as when $\cS(X) = \{Y\},$ the maps
\begin{equation*}
	\xymatrix{
	& X^2_\e \ar[ld]_-{\beta_{(2),L}} \ar[rd]^-{\beta_{(2),R}} & \\
	X & & X }
\end{equation*}
obtained by blowing-down and then projecting onto the left or right factor of $X,$ are b-fibrations. 

The double edge space has a distinguished submanifold, the interior lift of the diagonal, which is known as the {\bf edge diagonal} and denoted $\diag_\e.$ It is a p-submanifold of $X^2_\e$ and its normal bundle is canonically identified with the edge tangent bundle.\\

\subsection{Edge pseudodifferential operators}

Let $X$ be a manifold with an iterated fibration structure. We define the edge pseudodifferential operators as the natural analogue of the operators defined by Mazzeo in \cite{Mazzeo:Edge} by specifying the structure of their integral kernels. These will be conormal distributions as in \S\ref{sec:Conormal} on the manifold with corners $X^2_\e$ defined in \S\ref{sec:DoubleSpace}. We will first define the `small calculus' which includes edge differential operators and then the `large calculus' which can be shown to include the inverse of invertible edge differential operators when they have constant indicial roots. Elements in this calculus are very well behaved but since the hypothesis of constant indicial roots is very restrictive, we also define a `calculus with bounds.'

Our convention is that the integral kernels of operators acting on functions will be weighted sections of the density bundle of $X,$ pulled back along the projection onto the second factor of $X^2.$ 
We introduce the multiweight
\begin{equation*}
	\mf d:\cM_1(X^2_\e) \lra \bbR, \quad
	\mf d(H) =
	\begin{cases}
	-(\dim(Y) + 1) & \Mif H \subseteq \bhsd{\phi\phi}(Y) \Mforsome Y \in \cS(X) \\
	0 & \Motherwise
	\end{cases}
\end{equation*}
and the weighted right density bundle over $X^2_\e,$
\begin{equation*}
	\Omega_{\mf d,R} = \rho_{X^2_\e}^{\mf d}\beta_{(2),R}^*\Omega(X).
\end{equation*}

Following \cite[Definition 3.3]{Mazzeo:Edge} in the simple edge case, the \textbf{small edge calculus of a manifold with corners and an
  iterated fibration structure} is the filtred algebra of
pseudodifferential operators consisting of the union, over $r\in
\bbR,$ of the operators defined by the integral kernels:
\begin{equation*}
	\Psi^r_\e(X) 
	= \rho_{\sf(X^2_\e)}^{\infty} I^{r}(X^2_\e,\diag_\e;\Omega_{\mf d,R}),
\end{equation*}
where $\rho_{\sf(X^2_\e)}$ is a total boundary defining function for
the `side faces' defined in \eqref{eq:side-and-front}, i.e.\ 
$$
\rho_{\sf(X^2_\e)} = \prod_{H \in \sf(X^2_\e)} \rho_H
$$
(We use classical, one-step distributions conormal to the diagonal, see {\em loc. cit.}.)
Note that our convention is that the integral kernels are right-densities so that they will map functions to functions.

If $E$ and $F$ are vector bundles over $X,$ we define the vector bundle $\Hom(E,F)$ over $X^2_\e$ by
\begin{equation*}
	\Hom(E,F) = \beta_{(2),L}^*F \otimes \beta_{(2),R}^*E'
\end{equation*}
where $E'$ denotes the dual bundle to $E,$ 
and then the edge pseudodifferential operators acting between sections of $E$ and $F$ are given by
\begin{equation*}
	\Psi^r_\e(X;E,F) 
	= \rho_{\sf(X^2_\e)}^{\infty} 
	I^{r}(X^2_\e,\diag_\e;\Hom(E,F) \otimes \Omega_{\mf d,R})
\end{equation*}
for each $r\in\bbR.$ We abbreviate $\Psi^r_\e(X;E) = \Psi^r_\e(X;E,E).$

The edge smoothing operators in the small calculus are
\begin{equation*}
	\Psi^{-\infty}_\e(X;E,F) = \bigcap_{r\in\bbR}\Psi^r_\e(X;E,F)
	= \rho_{\sf(X^2_\e)}^{\infty} 
	\CI(X^2_\e;\Hom(E,F) \otimes \Omega_{\mf d,R}).
\end{equation*}

The integral kernels of edge differential operators lifted to $X^2_\e$
are supported on the edge diagonal and identifying the operators with
their kernels (and multiplying by a section of the weighted density
bundle, on which the operators act trivially) we have
\begin{equation*}
	\Diff^k_\e(X;E,F) \subseteq \Psi^k_\e(X;E,F), \Mforall k \in \bbN_0.
\end{equation*}

The conormal singularity at the diagonal means \cite[Definition 18.2.6]{Hvol3} that elements of the small calculus have a symbol map defined on the conormal bundle to the diagonal, i.e., the edge cotangent bundle,
\begin{equation*}
	\bar\sigma_r: \Psi^r_\e(X;E,F) \lra \rho_{RC}^{-r}\CI(RC({}^{\e}T^*X), \pi^*\hom(E,F))
\end{equation*}
where $RC({}^{\e}T^*X)$ denotes the radial compactification of the edge cotangent bundle, $\pi: RC({}^{\e}T^*X) \lra X$ denotes the projection, and $\rho_{RC}$ denotes a boundary defining for the boundary at radial infinity.
Multiplying by $\rho_{RC}^r$ and restricting to the boundary maps into $\CI({}^{\e}\bbS^*X, \pi^*\hom(E,F)),$ and we denote the resulting map by $\sigma_r.$
The symbol fits into a short exact sequence,
\begin{equation*}
	0 \lra
	\Psi^{r-1}_\e(X;E,F) \lra \Psi^r_\e(X;E,F) \xlra{\sigma_r} \CI({}^{\e}\bbS^*X, \pi^*\hom(E,F))
	\lra 0.
\end{equation*}

In Appendix A we construct a triple edge space $X^3_\e$ such that composition of edge pseudodifferential operators is given by pull-back, multiplication, and push-forward along b-fibrations; the behavior of distributions conormal along the lifted diagonal is essentially the same as in, e.g, \cite{Mazzeo:Edge}, and hence, for any $r_A, r_B \in \bbR,$
\begin{equation*}
\begin{gathered}
	A \in \Psi^{r_A}_\e(X;G,F), \quad B \in \Psi^{r_B}_\e(X; E,G) \\ \implies
	A \circ B \in \Psi^{r_A+r_B}_\e(X;E,F) \Mand
	\sigma_{r_A+r_B}(A \circ B) = \sigma_{r_A}(A) \circ \sigma_{r_B}(B).
\end{gathered}
\end{equation*}

If $A \in \Psi^r_\e(X;E,F)$ has invertible symbol, we say that $A$ is {\em elliptic} (or {\em edge elliptic}). If $A$ is elliptic then we can find $B \in \Psi^{-r}_\e(X;F,E)$ satisfying
\begin{equation*}
	\sigma_{-r}(B) = \sigma_r(A)^{-1}
\end{equation*}
and any such is known as a {\em symbolic parametrix} of $A.$ These satisfy
\begin{equation*}
	A \circ B - \Id \in \Psi^{-1}_\e(X;F), \quad B\circ A -\Id \in \Psi^{-1}_\e(X;E).
\end{equation*}

$ $\\
The {\bf large edge calculus of a manifold with corners and an iterated fibration structure} consists of, for any $r \in \bbR,$ and $\cE$ an index family for $X^2_\e,$
\begin{equation}\label{eq:DefLargeCalc}
	\Psi^{r,\cE}_{\e,phg}(X;E,F) 
	= \Psi^r_\e(X;E,F) 
	+ \sA_{phg}^{\cE}(X^2_\e; \Hom(E,F)\otimes \Omega_{\mf d,R}).
\end{equation}

\begin{definition}\label{def:BddCalc}
Let $\cE_{ff}$ be the index set for $X^2_\e$ given by
\begin{equation*}
	\cE_{ff}(\bhsd{\phi\phi}(Y)) = \bbN_0, \quad
	\cE_{ff}(\bhsd{10}(Y)) 
	=\cE_{ff}(\bhsd{01}(Y)) =\emptyset, \;
	\Mforall Y \in \cS(X).
\end{equation*}
The {\bf edge calculus with bounds of a manifold with corners and an iterated fibration structure} 
consists of, for any $r\in \bbR,$ multiweight $\mf w$ for $X^2_\e,$ 
\begin{equation*}
\begin{gathered}
	\Psi^{-\infty,\mf w}_\e(X;E,F) 
	= \sB_{phg}^{\cE_{ff}/\mf w}\sA^{-N}_-(X; \Hom(E,F)\otimes \Omega_{\mf d,R}), \\
	\Psi^{r,\mf w}_\e(X;E,F)
	= \Psi^r_\e(X;E,F) + \Psi^{-\infty, \mf w}_\e(X;E,F).
\end{gathered}
\end{equation*}
with notation as in \eqref{eq:pphg} and $-N<\min(0, \mf w).$

(We will always implicitly assume that the multiweight used in the edge calculus with bounds, $\mf w$ above, is positive on $\ff(X^2_\e)$ so that one can restrict to $\ff(X^2_\e).$)
\end{definition}

As in \cite[\S5]{Mazzeo:Edge}, for each $Y \in \cS(X)$ we have a
restriction map (with $\sf(\bhsd{\phi\phi}(Y))$ the collective
boundary hypersurface given by intersections of
$\bhsd{\phi\phi}(Y)$ with the side faces of $X^2_{\e}$) 
\begin{multline}\label{eq:CompEdgeNormal}
	\Nop_Y: \Psi^{r, \mf w}_\e(X;E) \lra \Psi^{r, \mf w_Y}_{Nsus({}^{\e}TY^+)}(\bhs{Y}/Y;E), 
	\Mwhere \\ 
	\Psi^{r, \mf w_Y}_{Nsus(TY^+)}(\bhs{Y}/Y;E) = 
	\rho_{\sf(\bhsd{\phi\phi}(Y))}^{\infty} 
	I^r(\bhsd{\phi\phi}(Y), \diag_\e \cap \bhsd{\phi\phi}(Y); \Hom(E) \otimes \Omega_{\mf d, R}) \\
	+ 
	\sB_{phg}^{(\cE_{ff}/\mf w)|_{\bhsd{\phi\phi}(Y)}}
		\sA^{-N}_-(\bhsd{\phi\phi}(Y); \Hom(E) \otimes \Omega_{\mf d, R}).
\end{multline}
The notation indicates that the latter space is a `non-commutative
suspension' (cf. \cite[\S1]{Albin-Melrose:Fred2},
\cite[\S4]{Mazzeo-Melrose:Phi}) which refers to the following.
We can identify the fibers of the bundle ${}^{\e}TY^+,$ over $Y,$ with
the Lie group $G = \bbR^+ \ltimes \bbR^h,$ i.e.,
\begin{equation*}
	(s,u)\cdot (s',u') = (ss',u+su'),
\end{equation*}
and the composition of edge pseudodifferential operators induces
convolution with respect to this action for the normal operators.

Moreover, $\Psi^{r, \mf g_Y}_{Nsus({}^{\e}TY^+)}(\bhs{Y}/Y;E)$ is
naturally a bundle of operators over $Y$, and the normal operators of
differential operator form a special sub-bundle, namely the product
lie algebra
$U(\mathfrak{g}) \times \mathbb{\Diff}^*_\e(Z)$ where $U(\mathfrak{g})$ is
the universal enveloping algebra of $\mathfrak{g}$, the Lie algebra of
the Lie Group $G$ above.  This goes also for zeroeth order operators,
including the identity operator, in particular
\begin{equation}
	\Nop_Y(Id) = \delta_{\nu_{\phi_Y}(\bhs{Y})} Id_Z,\label{eq:normal-identity}
\end{equation}
with $\nu_{\phi_Y}(\bhs{Y})$ the identity section.

With $\rho_X$ a total boundary defining function for $X$, for $Y \in \cS(X)$ we now have a
composite map $\eth \in \Diff_\w(X)$ to $\Nop_Y(\rho_X \eth)$, which is
related to the `wedge' normal operator above by
\begin{equation}\label{eq:edge-v-wedge-norm}
  \Nop_Y(\rho_X \eth) =  (\rho_X / \prod_{Y' \le Y} \rho_{Y'}')\cN_Y(\eth),
\end{equation}
where $\rho'_{Y'}$ is the pullback of $\rho_{Y'}$ to $X^2_\e$ to the
right factor.  $\cN_Y(\eth)$ is not the normal operator of an edge
operator; it is a \emph{wedge} operator on $\bhsd{\phi\phi}(Y)$.

Consider the normal operators $\{ \Nop_Y A \}_{Y
  \in \cS(X)}$ of an edge pseudodifferential operator with bounds $A
\in \Psi^{r,\mf w}_\e(X;E,F)$.  Since these are defined by restriction
of the integral kernel to the front faces of $X^2_\e$, they
automatically agree on the intersections of the front faces $\bhsd{\phi\phi}(Y)$, but it
will be useful in the parametrix construction below to have a concrete
understanding of these intersections.
Let $Y <\wt{Y}$, so we have a diagram
\eqref{eq:IFSDetail2}.  The intersection of $\bhsd{\phi\phi}(Y)$ with
$\phi_{\wt{Y}}^{(2)} \colon \bhsd{\phi\phi}(\wt{Y}) \lra \wt{Y}$ takes place `in the base' of
$\phi_{\wt{Y}}^{(2)}$, i.e.\ exactly over $\bhs{Y\wt{Y}} \in
\cM_1(\wt{Y})$.  Thus 
\begin{equation}\label{eq:normal-int-1}
\Nop_{\wt{Y}}(A)\rvert_{\bhsd{\phi\phi}(Y) \cap \bhsd{\phi\phi}(\wt{Y})}
\in \Psi^{r, \mf w_{\wt{Y}}}_{Nsus({}^{\e}T\wt{Y}^+)}(\bhs{\wt{Y}}/\wt{Y};E) \rvert_{Y},
\end{equation}
where the restriction on the right hand side comes from fibration of
$\Psi^{r, \mf w_{\wt{Y}}}_{Nsus({}^{\e}T\wt{Y}^+)}(\bhs{\wt{Y}}/\wt{Y} ;E)$ over
$\wt{Y}$ discussed in the previous paragraph.
On the other hand, the intersection $\bhsd{\phi\phi}(Y) \cap
\bhsd{\phi\phi}(\wt{Y})$ is the front face obtained from blow up of  $
		\bar\beta^{\sharp}(\bhs{\wt Y} \times_{\phi_{\wt Y}}
                \bhs{\wt Y})$ in $\bar{\bhsd{\phi\phi}}(Y)$ (see
                \eqref{eq:intersection-double-space}), and is
                trivially equal to
                $(\phi_{\wt{Y}}^{(2)})^{-1}(\bhs{Y\wt{Y}})$.  Thus
                restriction of $\Nop_{Y}(A)$ to this front face is really
                taking the normal operator of an (albeit suspended)
                edge pseudodifferential operator.  This is summarized
                in the diagram
                \begin{equation}
                  \xymatrix{\Psi^{r, \mf w}_{e}(X; E)
                    \ar[r]^-{\Nop_{\wt{Y}}} \ar[d]^{\Nop_Y} &  \Psi^{r, \mf
                      w_{\wt{Y}}}_{Nsus({}^{\e}T\wt{Y}^+)}(\bhs{\wt{Y}}/\wt{Y};E)
                    \ar[d]^-{res}  \\
 \Psi^{r, \mf w_{Y}}_{Nsus({}^{\e}TY^+)}(\bhs{Y}/Y;E) \ar[r]^-{res}& \Psi^{r, \mf
   w_{\wt{Y}}}_{Nsus({}^{\e}T\wt{Y}^+)}(\bhs{\wt{Y}}/\wt{Y};E)\rvert_{Y}
} 
                \end{equation}
where $res$ means `restriction'. \\

To each index family $\cE$ we assign a multiweight $\mf w(\cE)$ such that
\begin{equation*}
	\Psi^{r, \cE}_{\e,phg}(X;E,F) \subseteq \Psi^{r,\mf w(\cE)}_\e(X;E,F) 
\end{equation*}
by defining, for each $Y \in \cS(X),$
\begin{equation*}
\begin{gathered}
	\mf w(\cE)(\bhsd{10}(Y)) = \inf \cE(\bhsd{10}(Y)), \quad
	\mf w(\cE)(\bhsd{01}(Y)) = \inf \cE(\bhsd{01}(Y)), \\
	\mf w(\cE)(\bhsd{\phi\phi}(Y)) = \inf (\Re \cE(\bhsd{\phi\phi}(Y)) \setminus \bbN_0)
\end{gathered}
\end{equation*}
with the convention that the infimum of the empty set is $\infty.$

These operators act on Sobolev sections of vector bundles.
In view of the inclusion of the large calculus into the calculus with
bounds, it suffices to describe the mapping properties of the
latter. We prove the following theorem in Appendix A, where as usual
we identify operators with their integral kernels, so for a multiweight
$\mf g$,
$\rho_{X}^{\mf g}\Psi^{r, \mf g}_\e(X;E,F)$ is the space of
operators with integral kernel in $\rho_X^{\mf g}(\Psi^r_\e(X;E,F) +
\Psi^{-\infty, \mf w}_\e(X;E,F)).$ Moreover we will use multiweights
for the front faces, specifically multiweights $\mf f \colon
\ff(X^2_\e) \lra \mathbb{R}$, notation as in
\eqref{eq:side-and-front}, and corresponding weight functions
$$
\rho_{\ff(X^2_\e)}^{\mf f} = \prod_{H \in \ff(X^2_\e)} \rho_H^{\mf f (H)}
$$

\begin{theorem}[Action on edge Sobolev spaces] \label{thm:EdgeSobAct}
Let $\mf f$ be a multiweight for $\ff(X^2_\e).$ 
Any $A \in \rho_{\ff(X^2_\e)}^{\mf f}\Psi^{r, \mf g}_\e(X;E,F)$ defines a bounded map, for any $t \in \bbR,$ 
\begin{equation}\label{eq:EdgeSobAct}
	\rho^{\mathfrak s}H^t_\e(X;E) \lra \rho^{\mathfrak s'}H^{t'}_\e(X;F)
\end{equation}
as long as $t \geq t'+r$ and, for each $Y \in \cS(X),$
\begin{equation}\label{eq:EdgeMapWeights}
\begin{gathered}
	\mf g( \bhsd{01}(Y)) ) + \mathfrak s(Y) > -\tfrac12 \\
	\mf g( \bhsd{10}(Y) ) > \mathfrak s'(Y) -\tfrac12\\
	\mf f( \bhsd{\phi\phi}(Y))  + \mathfrak s(Y) \geq \mathfrak s'(Y).
\end{gathered}
\end{equation}
\end{theorem}

Essentially by Arzela-Ascoli we can see that the inclusion
\begin{equation*}
	\rho^{\mf s}H^t_\e(X;E) \lra \rho^{\mf s'}H^{t'}_\e(X;E)
\end{equation*}
is compact if (and only if) $\mf s>\mf s'$ and $t>t'.$
Combining with the mapping properties, we can identify the edge pseudodifferential operators that act as compact operators.

\begin{corollary} \label{cor:CompactAct}
If $A$ is as in Theorem \ref{thm:EdgeSobAct} 
then the operator \eqref{eq:EdgeSobAct} is compact if and only if $t>t'+r$ and the inequalities in \eqref{eq:EdgeMapWeights} are strict.
\end{corollary}

In Appendix A we study the composition of these pseudodifferential operators
at the level of their integral kernels. One advantage of studying composition at this level is that one can then deduce composition results for functions spaces (Sobolev spaces, H\"older spaces, etc.) see, e.g., \cite{Mazzeo:Edge}.

\begin{theorem}[Composition of edge pseudodifferential operators]\label{thm:CompositionEdge} $ $

\begin{enumerate}
\item
Let $r_A, r_B \in \bbR$ and let $\cE_A,$ $\cE_B$ be index families for $X^2_\e$ such that
\begin{equation*}
	\Re(\cE_A(\bhsd{01}(Y))) + \Re(\cE_B(\bhsd{10}(Y))) > -1 \Mforall Y \in \cS(X).
\end{equation*}
If $A \in \Psi^{r_A,\cE_A}_{e,phg}(X^2_\e;G,F)$ and 
$B \in \Psi^{r_B, \cE_B}_{e,phg}(X^2_\e;E,G)$ then 
\begin{equation*}
	C = A \circ B \in \Psi^{r_A + r_B, \cE_C}_{\e,phg}(X^2_\e;E,F)
\end{equation*}
where $\cE_C$ is the index family on $X^2_\e$ given by, for each $Y \in \cS(X),$
\begin{equation*}
\begin{aligned}
	\cE_C(\bhsd{10}(Y)) &= \cE_A(\bhsd{10}(Y)) \; \bar\cup \; 
		\lrpar{ \cE_A(\bhsd{\phi\phi}(Y)) + \cE_B(\bhsd{10}(Y))}, \\
	\cE_C(\bhsd{01}(Y)) &= \cE_B(\bhsd{01}(Y)) \; \bar\cup \;
		\lrpar{ \cE_A(\bhsd{01}(Y)) + \cE_B(\bhsd{\phi\phi}(Y)) }, \\
	\cE_C(\bhsd{\phi\phi}(Y)) &= \lrpar{\cE_A(\bhsd{10}(Y)) + \cE_B(\bhsd{01}(Y)) + \dim(Y) + 1} \\
		&\phantom{xxxxxx}
		\; \bar\cup\;  
		\lrpar{ \cE_A(\bhsd{\phi\phi}(Y)) + \cE_B(\bhsd{\phi\phi}(Y))}
\end{aligned}
\end{equation*}

\item
Let $r_A, r_B \in \bbR$ and let $\mf g_A,$ $\mf g_B$ be multiweights for $X^2_\e$
such that
\begin{equation*}
	\mf g_A(\bhsd{01}(Y)) + \mf g_B(\bhsd{10}(Y)) > -1 \Mforall Y \in \cS(X).
\end{equation*}
If $A \in \Psi^{r_A,\mf g_A}_{\e}(X^2_\e;G,F)$ and 
$B \in \Psi^{r_B, \mf g_B}_{\e}(X^2_\e;E,G)$ then 
\begin{equation*}
	C = A \circ B \in \Psi^{r_A + r_B, \mf g_C}_{e}(X^2_\e;E,F)
\end{equation*}
where $\mf g_C$ is the multiweight on $X^2_\e$ given by, for each $Y \in \cS(X),$
\begin{equation*}
\begin{aligned}
	\mf g_C(\bhsd{10}(Y)) &= \min\lrpar{
	\mf g_A(\bhsd{10}(Y)), \; 
	\mf g_B(\bhsd{10}(Y))}, \\
	\mf g_C(\bhsd{01}(Y)) &= \min\lrpar{
	\mf g_B(\bhsd{01}(Y)), \;
	\mf g_A(\bhsd{01}(Y)) }, \\
	\mf g_C(\bhsd{\phi\phi}(Y)) &= \min\lrpar{
	\mf g_A(\bhsd{10}(Y)) + \mf g_B(\bhsd{01}(Y)) + \dim(Y) + 1, \;
	\mf g_A(\bhsd{\phi\phi}(Y)), \;
	\mf g_B(\bhsd{\phi\phi}(Y)) }
\end{aligned}
\end{equation*}
\end{enumerate}
\end{theorem}

\begin{proof}
The proof of (1) is carried out in Appendix \ref{sec:CompEdge}  following \cite{Mazzeo:Edge} by constructing a `triple edge space' and analyzing the integral kernel of the composite geometrically via the push-forward and pull-back theorems.
As explained in \S\ref{sec:Conormal}, these same theorems apply to partially polyhomogenous distributions with conormal errors. Once we recall that the multiweight $\mf g_C$ denotes the order of the conormal error, we can deduce the behavior of the multiweights in (2) from the behavior of the index sets in (1).
\end{proof}

$ $\\
We formalize the notion of smooth family of edge operators using the space $(M/B)^2_\e,$ e.g.,
\begin{equation*}
\begin{gathered}
	\Psi^r_\e(M/B;E,F) = \rho_{\sf((M/B)^2_\e)}^{\infty}I^r((M/B)^2_\e, \diag_M; \Hom(E,F) \otimes \Omega_{\mf d, R}), \Mand \\
	\Psi_\e^{-\infty,\mf w}(M/B;E,F) = \sB_{phg}^{\cE_{ff}/\mf w}\cA^{-N}_{-}(M/B;\Hom(E,F) \otimes \Omega_{\mf d, R}).
\end{gathered}
\end{equation*}
The composition results in Appendix \ref{sec:CompEdge} are established in the setting of families.

\subsection{Bi-ideal} \label{sec:BiIdeal}

As in, e.g., \cite[Proposition 5.38]{tapsit},
\cite[\S4.12]{Mazzeo-Melrose:Surgery}, we point out a useful bi-ideal property of some edge pseudodifferential operators.

For each $a \in \bbR^+,$ define the residual edge pseudodifferential
operators of weight $a$ to be 
\begin{equation*}
	\Psi^{-\infty,a}_{\e,res}(X;E,F) = 
	\rho_{\ff(X^2_\e)}^{a} \Psi^{-\infty,a}_{\e}(X;E,F).
\end{equation*}

\begin{theorem}\label{thm:BiIdeal}
For $a \in \bbR^+$ 
\begin{equation*}
	\Psi^{-\infty,a}_{\e,res}(X;G,H) \circ 
	\cB(L^2(X;F,G)) \circ
	\Psi^{-\infty,a}_{\e,res}(X;E,F)
	\subseteq
	\Psi^{-\infty,a}_{\e,res}(X;E,H),
\end{equation*}
where $\cB(L^2(X;F,G))$ is the space of bounded operators on $L^2(X;F,G).$
\end{theorem}

\begin{proof}
This result is by now standard, see e.g.\ \cite[\S4.12]{Mazzeo-Melrose:Surgery}, but we
sketch a proof for the convenience of the reader.  For simplicity of notation, let us assume that $E, F, G, H$ are trivial line bundles.
Let $A, C \in \Psi^{-\infty, a}_{\e,res}(X)$ and $B \in \cB(L^2(X)).$
In terms of their distributional kernels on $X^2,$ the composition is given by
\begin{equation*}
	\cK_{ABC}(\zeta, \zeta') =
	\int \int \cK_A(\zeta, \zeta'') \cK_B(\zeta'', \zeta''') 
	\cK_C(\zeta''', \zeta') \; d\zeta'' \; d\zeta''',
\end{equation*}
and so smoothness of $\cK_{ABC}$ in $\zeta$ is inherited from smoothness of $\cK_A$ in $\zeta,$ while smoothness of $\cK_{ABC}$ in $\zeta'$ is inherited from the corresponding smoothness of $\cK_C.$

Next we lift this smooth function from $X^2$ to $X^2_\e$
and check that it is conormal to the boundary hypersurfaces. Indeed, the $b$-vector fields on $X^2_\e$ are spanned by the lifts of the b-vector fields on $X$ along $\beta_{(2),L}$ and $\beta_{(2),R}.$ The kernel $\beta_{(2)}^*\cK_{ABC}$ has stable regularity with respect to the left lift of $b$-vector fields because $\cK_A$ does, and with respect to the right lift of $b$-vector fields because $\cK_B$ does.
\end{proof}

\subsection{Wedge heat space} \label{sec:Heat}

Recall that the edge double space was defined in \S\ref{sec:DoubleSpace} as
\begin{equation*}
	X^2_\e = \Big[ X^2; \bhss{Y_1} \times_{\phi_{Y_1}} \bhss{Y_1}; \ldots \bhss{Y_\ell} \times_{\phi_{Y_{\ell}}} \bhss{Y_{\ell}} ],
\end{equation*}
where  $\{ Y_1, \ldots, Y_\ell\}$ is a non-decreasing listing of $\cS(X),$
and has collective boundary hypersurfaces, for each $Y \in \cS(X),$
\begin{equation*}
	\bhss{Y} \times X \leftrightarrow \bhsd{10}(Y), \quad
	X \times \bhss{Y} \leftrightarrow \bhsd{01}(Y), \quad
	\bhss{Y} \times_{\phi_Y} \bhss{Y} \leftrightarrow \bhsd{\phi\phi}(Y).
\end{equation*}

Now we construct the wedge heat space.
Starting with the space $X^2 \times \bbR^+_t$ we blow-up $\{t=0\}$ parabolically so that $\tau=\sqrt t$ is a smooth function. We will not include this blow-up explicitly but simply change the notation to $X^2 \times \bbR^+_{\tau}.$

\begin{definition}\label{def:WedgeHeat}
Let $X$ be a manifold with corners and an iterated fibration structure and $\{ Y_1, \ldots, Y_\ell\}$ a non-decreasing listing of $\cS(X).$ The {\bf wedge heat space of $X$} is the space $HX_{\w}$ defined by
\begin{gather}\label{eq:intermediate-heat}
	HX_{\w,0} = 
	\Big[ X^2 \times \bbR^+_{\tau}; 
	\bhs{Y_1} \times_{\phi_{Y_1}} \bhs{Y_1} \times \{ 0\}; \ldots ; 
	\bhs{Y_{\ell}} \times_{\phi_{Y_\ell}} \bhs{Y_\ell} \times \{ 0\}; 
	\diag_X \times \{0\} \Big], 
	\\ \label{eq:HeatSpace}
	HX_\w = 
	\Big[ HX_{\w,0};
	\bhs{Y_1} \times_{\phi_{Y_1}} \bhs{Y_1} \times \bbR^+_{\tau}; \ldots ; 
	\bhs{Y_{\ell}} \times_{\phi_{Y_\ell}} \bhs{Y_\ell} \times \bbR^+_{\tau}
	\Big].
\end{gather}
\end{definition}

\begin{remark}
In order to describe the heat kernel of $\eth_{X}^2$ as a conormal distribution with bounds the intermediate space $HX_{\w,0}$ would suffice, 
see, e.g., \cite{Mazzeo-Vertman:AT}. However below we will allow for perturbations of $\eth_{X}$ by smoothing edge pseudodifferential operators and this requires the slightly more complicated space $HX_{\w}.$
\end{remark}

To deal smoothly with a family of wedge heat operators we construct a families wedge heat space to carry their integral kernels.

\begin{definition}
Given a fiber bundle $M \xlra\psi B$ of manifolds with corners and iterated fibration structures as in Definition \ref{def:FibIFS}, fix a non-decreasing list of $\cS_{\psi}(M),$ $\{N_1, N_2, \ldots, N_{\ell}\},$ and let the {\bf families wedge heat space} be
\begin{multline*}
	H(M/B)_{\w} =
	\Big[ M\times_\psi M \times \bbR^+_{\tau}; 
	\bhs{N_1} \times_{\phi_{N_1}} \bhs{N_1} \times \{ 0\}; \ldots ; 
	\bhs{N_{\ell}} \times_{\phi_{N_\ell}} \bhs{N_\ell} \times \{ 0\}; \\
	\bhs{N_1} \times_{\phi_{N_1}} \bhs{N_1} \times \bbR^+_{\tau}; \ldots ; 
	\bhs{N_{\ell}} \times_{\phi_{N_\ell}} \bhs{N_\ell} \times \bbR^+_{\tau}; 
	\diag_M \times \{0\} \Big].
\end{multline*}
The map $\psi$ induces a fiber bundle 
\begin{equation*}
	HX_\w \fib H(M/B)_\w \xlra{\psi_{(H)}} B.
\end{equation*}
\end{definition}

As with the double space, implicit in the definition of $HX_{\w}$ is
the fact that for $\wt{Y} \in \cS(X)$ of depth $k$, the interior lift of
$\bhs{\wt{Y}} \times_{\phi_{\wt{Y}}} \bhs{\wt{Y}} \times \{ 0 \}$ to the space in
which the $\bhs{Y} \times_{\phi_Y} \bhs{Y} \times \{ 0 \}$
have been blown up for all $Y < \wt{Y}$ is a p-submanifold.  It is
helpful to see this
explicitly.  If $Y<\wt Y,$ we have a diagram as in
\eqref{eq:IFSDetail'} and attendant coordinates $x, y,  w,  r,  \wt
z,$ together with their primed versions on the right factor on $X^2$.
Working in the interior of $Y$, after blowing-up $\bhs{Y}
\times_{\phi_Y} \bhs{Y} \times \{ \tau = 0 \}$, projective coordinates
with respect to $x'$ are given by $T = \tau/x'$ and the other
coordinates in \eqref{eq:ProjCoords}, in which the interior lift of
$\bhs{\wt Y} \times_{\phi} \bhs{\wt Y} \times \{ \tau = 0 \}$ is given by 
\begin{equation}\label{eq:PSubMfd-2}
	\{ T = 0, \; r=r'=0, \;  s=1, \; u=0,\;  w=w'\},
\end{equation}
again a p-submanifold.

We denote the blow-down map by
\begin{equation*}
	\beta_{(H)}:HX_\w \lra X^2 \times \bbR^+_{\tau}
\end{equation*}
and its composition with the projections onto the left or right factor of $X$ by $\beta_{(H),L},$ $\beta_{(H),R}$ respectively.
There are boundary hypersurfaces
\begin{equation*}
	X^2 \times \{0\} \leftrightarrow \bhsh{00,1}, \quad
	\diag_X \times \{0\} \leftrightarrow \bhsh{dd,1}
\end{equation*}
and collective boundary hypersurfaces, one for each $Y \in \cS(X),$
\begin{equation*}
\begin{gathered}
	\bhs{Y} \times X \times \bbR^+ \leftrightarrow \bhsh{10,0}(Y), \quad
	X \times \bhs{Y} \times \bbR^+ \leftrightarrow \bhsh{01,0}(Y), \\
	\bhs{Y} \times_{\phi_Y} \bhs{Y} \times \{0\} \leftrightarrow 
		\bhsh{\phi\phi,1}(Y), \quad
	\bhs{Y} \times_{\phi_Y} \bhs{Y} \times \bbR^+ \leftrightarrow 
		\bhsh{\phi\phi,0}(Y).
\end{gathered}
\end{equation*}
We denote the collective boundary hypersurfaces of $H(M/B)_{\w}$ analogously to those of $HX_{\w},$ e.g., $\bhsh{\phi\phi,1}(N).$ \\

We introduce the abbreviations
\begin{equation*}
\begin{gathered}
	\ff(HX_{\w}) = \bigcup_{Y \in \cS(X)} \bhsh{\phi\phi,1}(Y), \quad
	\lf(HX_{\w}) = \bigcup_{Y \in \cS(X)} \bhsh{10,0}(Y), \\
	\rf(HX_{\w}) = \bigcup_{Y \in \cS(X)} \bhsh{01,0}(Y), \quad
	\ef(HX_{\w}) = \bigcup_{Y \in \cS(X)} \bhsh{\phi\phi,0}(Y).
\end{gathered}
\end{equation*}
so that, e.g., $\rho_{\lf(HX_{\w})}$ refers to the product of boundary
defining functions over all $\bhsh{10,0}(Y)$ for $Y \in \cS(X).$  The `edge
faces' making up $\ef(HX_{\w})$ do not intersect the lower depth front faces, nor the $\tau = 0$ diagonal
$$
\diag_X \times \{ 0 \} \cap \ef(HX_{\w}) = \varnothing =
\bhsh{\phi\phi,1}(\wt{Y}) \cap \bhsh{\phi\phi,0}(Y),\ Y < \wt{Y}.
$$

Also as with the double space, the faces created by the blow ups are fibre
bundles whose fibers are suspended versions of wedge heat spaces.
This will be a wedge heat space where the time $[0, \infty)_\tau$ is
compactified along with other normal directions, and we need analogues
of the identity section above.  
Given a fiber bundle $\wc M \lra \wc B$ and a vector bundle $E \lra \wc B,$
the pull-back of $\mathbb{S}(E  \times \mathbb{R}_+^3) = (E  \times
\mathbb{R}_+^3)/\mathbb{R}_+$ to $\wc M$ has two subbundles, $\wt{\nu}_0$ and $\wt{\nu}_t$, given,
respectively, by the inclusion of $\mathbb{R}_+^2 \hookrightarrow
\mathbb{R}_+^2 \times \mathbb{R}_+$ into the right factor and the
projection $\mathbb{R}_+^2 \times \mathbb{R}_+ \lra \mathbb{R}_+^2$
off the right factor, of the identity section.  Concretely,
$\wt{\nu}_0$ is given by the lift of the subbundle $[\{ 0 \} \times
(1,1,0)] \subset \mathbb{S}(E \times \mathbb{R}_+^3)$ to $\wc M$, and
  $\wt{\nu}_t$ is given by the lift of $[\{ 0 \} \times
(x, x, \sqrt{1 + x^2})]$.  For trivial fibrations $\wc M = \pt = \wc B$ with $E =
\mathbb{R}^h$ we denote these by
\begin{align*}
\wt{\nu}_0(\mathbb{S}^{h + 2}) &= [(0, 1, 1, 0) ] \in
  (\mathbb{R}^{h} \times \mathbb{R}_+^3 \setminus (0,0,0,0)) / \mathbb{R}_+, \\
\wt{\nu}_t(\mathbb{S}^{h + 2}) &= [(0, x, x, \sqrt{1 + x^2}) ] \in
  (\mathbb{R}^{h} \times \mathbb{R}_+^3 \setminus (0,0,0,0) ) / \mathbb{R}_+,
\end{align*}
the $\mathbb{R}_+$ acting by dilation.

\begin{definition}\label{thm:sus-heat-def-0}
Let $\wc X \fib \wc M \xlra{\wc \psi} \wc B$ be a fiber bundle of
manifolds with corners and iterated fibration structures and let
$ \{\wc N_1, \ldots, \wc N_{\ell}\}$ be a non-decreasing
listing of $\cS_{\wc \psi}(\wc M).$ Let 
$\cS_{+++}(\wc M \times_{\wc \psi} \wc M)$ be the pull-back of the
fiber bundle $\bbS(T\wc B \times \bbR_+^3)$ from $B$ to $\wc M
\times_{\wc \psi} \wc M$ and let $\wt{\nu}_{\wc \psi,0}(\wc M)$ denote
the $\tau = 0$ identity section.
The {\bf intermediate suspended wedge heat space $H(\wc M/\wc B)_{Sus(\w),0}$} is 
\begin{multline*}
	H(\wc M/\wc B)_{Sus(\w),0} 
    = [\cS_{+++}(\wc M \times_{\wc \psi} \wc
    M); \; \nu_{\wc \psi}(\wc M) \cap \pi^{-1}(\bhs{\wc N_1}
    \times_{\phi_{\wc N_1}} \bhs{\wc N_1});\;\ldots;\; \\
	\nu_{\wc \psi}(\wc M) \cap \pi^{-1}(\bhs{\wc N_\ell} \times_{\phi_{\wc
        N_\ell}} \bhs{\wc N_\ell}) ; \nu_{\wc
      \psi}(\wc M) \cap \pi^{-1}(\diag(\wc M))].
\end{multline*}
This fibers over $\wc B$ and we denote the typical fiber by $H\wc X_{Sus_{\wc B}(\w), 0}$ so that
\begin{equation*}
	H\wc X_{Sus_{\wc B}(\w), 0} \fib H(\wc M/\wc B)_{Sus(\w), 0} \lra \wc B.
\end{equation*}

The {\bf suspended wedge heat space $H(\wc M/\wc B)_{Sus(\w)}$} is 
\begin{multline*}
	H(\wc M/\wc B)_{Sus(\w)} 
	= [H(\wc M/\wc B)_{Sus(\w),0}; \; 
	\wt{\nu}_{\wc \psi,t}(\wc M) \cap \pi^{-1}(\bhs{\wc N_1}
	\times_{\phi_{\wc N_1}} \bhs{\wc N_1});\;\ldots;\; \\
	\wt{\nu}_{\wc \psi,t}(\wc M) \cap 
	\pi^{-1}(\bhs{\wc N_\ell} \times_{\phi_{\wc N_\ell}} \bhs{\wc N_\ell})]
\end{multline*}
and participates in the fiber bundle 
\begin{equation*}
	H\wc X_{Sus_{\wc B}(\w)} \fib H(\wc M/\wc B)_{Sus(\w)} \lra \wc B.
\end{equation*}
\end{definition}

As anticipated, the suspended wedge heat spaces describe the structure of the front faces of the wedge heat space.

\begin{proposition}[Structure of the front faces of $HX_\w$]
Let $X$ be a manifold with corners and an iterated fibration structure.

For each $Y \in \cS(X),$ let
\begin{equation*}
	\phi_Y^{(H)}: \bhsh{\phi\phi,1}(Y) \lra Y
\end{equation*}
denote the composition of $\beta_{(H)}:HX_{\w} \lra X^2 \times
\bbR^+_{\tau}$ with the fibration $\bhs{Y} \times_{\phi_Y} \bhs{Y}
\times \{0\} \lra Y.$ Then $\bhsh{\phi\phi,1}(Y) = H(\bhs{Y}/Y) _{Sus(\w)}$ and $\phi_Y^{(H)}$ is the fiber bundle map
\begin{equation*}
	HZ_{Sus(\w)} \fib \bhsh{\phi\phi,1}(Y) \xlra{\phi_Y^{(H)} } Y.
\end{equation*}
For the intermediate heat space, $HX_{\w,0},$ the corresponding front face is the total space of the fiber bundle
\begin{equation*}
	HZ_{Sus(\w), 0} \fib H(\bhs{Y}/Y) _{Sus(\w),0} \xlra{\phi_Y^{(H)} } Y.
\end{equation*}

The edge face corresponding to $Y,$ $\bhsh{\phi\phi,0}(Y),$ participates in a fiber bundle with typical fiber the suspended edge double space of $Z,$
\begin{equation*}
	Z^2_{Sus_Y(\e)} \fib \bhsh{\phi\phi,0}(Y) \lra (Y \times \mathbb{R}_+)_{res}
\end{equation*}
and base given by
$$
(Y \times \mathbb{R}_+)_{res} = [Y \times \mathbb{R}_+ ; \bhs{Y'_1Y} \times
\{ 0 \} ; \; ...; \bhs{Y'_rY} \times \{ 0 \}],
$$
where $\{ Y'_1, \dots, Y'_r \}$ are the strata with $Y'_i < Y$
indexed in non-increasing order of depth.

Finally, the interior of the boundary hypersurface $\bhsh{dd,1}$
is naturally identified with the edge tangent bundle ${}^\e TX.$
\end{proposition}

\begin{proof}
Let $HX_{\w, 0}(k + 1)$ be the intermediate space obtained by blow-up of all the interior lifts of the $\bhs{Y'} \times_{\phi_{Y'}} \bhs{Y'} \times \{\tau = 0\}$ with $Y'$ of depth not less than $k + 1$, and let $Y \in \cS(X)$ have depth $k$.  Then the normal bundle of $\bhs{Y} \times_{\phi_{Y}} \bhs{Y} \times \{0 \}$ fibers over ${}^\e TY \times (\mathbb{R}_+)^3$ with fibre $Z^2$.  Thus blow up of the interior lift of $\bhs{Y} \times_{\phi_{Y}} \bhs{Y} \times \{0 \}$ gives a front face $\ff$ that fibres
$$
\bbS_{+++}^{h_Y + 2} \times Z^2 \fib  \ff \lra Y.
$$
The section $\wt{\nu}_0(\ff)$ is the subbundle of $\ff$ over $Y$ given by 
$\wt{\nu}_0(\bbS_{+++}^{h_Y+ 2}) \times Z^2$.  
From this we see that for $\wt Y > Y$,  the interior lift of
  $\bhs{{\wt{Y}}} \times_{\phi_{{\wt{Y}}}} \bhs{{\wt{Y}}}
  \times \{0 \}$ intersects $\ff$ exactly in the fibres at
  $\wt{\nu}_0(\bbS_{+++}^{h_Y
    + 2}) \times \bhs{\wt Y Z} \times_{\phi_{\wt Y Z}} \bhs{\wt Y Z}$,
  and the diagonal intersects it at $\wt{\nu}_0(\bbS_{+++}^{h_Y
    + 2}) \times Z^2$.  This yields the structure of the front faces of the intermediate wedge heat space.
For the wedge heat space it suffices to note that the blow ups of the $\bhs{Y_i}
  \times_{\phi_{Y_i}} \bhs{Y_i} \times \bbR^+_{\tau}$ intersect the
  intermediate front faces exactly at the $\wt{\nu}_t$.  

To see that the statement for the $\bhsh{\phi\phi,0}(Y)$ holds, note that the
lifts of $\bhs{Y'} \times_{\phi_{Y'}} \bhs{Y'} \times \{ 0 \} \lra Y' \times
\{ 0 \}$ for
$Y' < Y$ intersect the blowdown, $\bhs{Y} \times_{\phi_Y} \bhs{Y}
\times \mathbb{R}_+ \lra Y \times \mathbb{R}_+$ exactly in the
bases over the $\bhs{Y'Y} \times \{ 0 \} \subset Y \times
\mathbb{R}_+$

Finally the boundary hypersurface $\bhsh{dd,1}$ is the inward-pointing part of the spherical normal bundle to $\diag_e \times \{0\}$ and, as the normal bundle to the edge diagonal in the edge double space is the edge tangent bundle, $\bhsh{dd,1}$ is its radial compactification.
\end{proof}

\subsection{Wedge heat operators} \label{sec:WedgeHeatOp}

Let us specify the weighted density bundle we will use for operators.
Define a multi-weight for $HX_\w$ by
\begin{equation*}
\begin{gathered}
	\mf h: \cM_1(HX_{\w}) \lra \bbR, \\
	\mf h(J) =
	\begin{cases}
	-(\dim Y + 3) & \Mif J \subseteq \bhsh{\phi\phi,1}(Y) \\
	-(\dim Y + 1) & \Mif J \subseteq \bhsh{\phi\phi,0}(Y) \\
	-(\dim X +2) & \Mif J \subseteq \bhsh{dd,1}\\
	0 & \Motherwise
	\end{cases}
\end{gathered}
\end{equation*}
and then 
\begin{equation}\label{eq:HeatDensityBdle}
	\Omega_{\mf h, R}= \rho^{\mf h}\beta_{(H),R}^*\Omega(X).
\end{equation}
We will often denote a nowhere-vanishing section of $\beta_{(H),R}^*\Omega(X)$ by $\mu_R.$\\

By a {\em wedge heat operator} we will mean an element of
\begin{equation*}
	\sB_{phg}^{\cE/\mf w}\sA^{-m-1}_- (HX_{\w};\Hom(E)\otimes\Omega_{\df h,R})
\end{equation*}
where $\cE$ and $\mf w$ are, respectively, an index set and multiweight for $HX_{\w}.$

Recall that, e.g., on a smooth manifold $L$ the composition of two heat operators is given by the formula
\begin{equation*}
	\cK_{A \circ B}(\zeta, \zeta', t) = \int_0^t \int_L \cK_A(\zeta, \zeta'', t-t') \cK_B(\zeta'',\zeta',t') \; d\zeta''\; dt'.
\end{equation*}

In Appendix \ref{sec:HeatComp} we define the composition of two wedge heat operators by a version of this formula and then analyze it using the geometric microlocal approach of Melrose, cf. \cite[Appendix]{Melrose-Piazza:Even}.

\begin{theorem}[Composition of wedge heat operators] $ $

\begin{enumerate}
\item
Let $\cE_A,$ $\cE_B$ be index families for $HX_{\w}$ such that
\begin{equation*}
\begin{gathered}
	\Re(\cE_A(\bhsh{dd,1})) > 0, \quad
	\Re(\cE_B(\bhsh{dd,1})) > 0, \Mand \\
	\Re(\cE_A(\bhsh{01,0}(Y))) + \Re(\cE_B(\bhsh{10,0}(Y)))+1 > 0 \Mforall Y \in \cS(X),
\end{gathered}
\end{equation*}
and let 
\begin{equation*}
	A \in \sA_{phg}^{\cE_A}(HX_{\w};\Hom(E)\otimes\Omega_{\df h,R}), \quad
	B \in \sA_{phg}^{\cE_B}(HX_{\w};\Hom(E)\otimes\Omega_{\df h,R}),
\end{equation*}
then the composition is defined and satisfies
\begin{equation*}
	C = A \circ B \in \sA_{phg}^{\cE_C}(HX_{\w};\Hom(E)\otimes\Omega_{\df h,R}),
\end{equation*}
with $\cE_C(\bhsh{dd,1}) = \cE_A(\bhsh{dd,1}) + \cE_B(\bhsh{dd,1})$
and, for each $Y \in \cS(X),$
\begin{align*}
	\cE_C(\bhsh{10,0}(Y)) &=
	\cE_A(\bhsh{10,0}(Y)) \bar\cup
	\big( \cE_A(\bhsh{\phi\phi,1}(Y)) + \cE_B(\bhsh{10,0}(Y) )\big) \\
	\cE_C(\bhsh{01,0}(Y)) &=
	\cE_B(\bhs{01,0}(Y)) \bar\cup
	\big(\cE_A(\bhsh{01,0}(Y)) + \cE_B(\bhsh{\phi\phi,1}(Y)) \big) \\
	\cE_C(\bhsh{\phi\phi,1}(Y)) &=
	\cE_A(\bhsh{\phi\phi,1}(Y)) + \cE_B(\bhsh{\phi\phi,1}(Y))\\
	\cE_C(\bhsh{\phi\phi,0}(Y)) &=
	(\cE_A(\bhsh{\phi\phi,0}(Y)) + \cE_B(\bhsh{\phi\phi,0}(Y)) ) \\
	& \phantom{xxxxx} \bar\cup
	\big(\cE_A(\bhsh{10,0}(Y)) + \cE_B(\bhsh{01,0}(Y)) + \dim Y+1\big).
\end{align*}

\item
Let $\mf w_A,$ $\mf w_B$ be multiweights for $HX_{\w}$ such that
\begin{equation*}
	\{\mf w_{\cdot}(\bhsh{dd,1}) \} \cup 
	\{\mf w_{\cdot}(\bhsh{\phi\phi,1}(Y)): Y \in \cS(X)\} 
	\subseteq (0,\infty) \cup \{\infty\},
\end{equation*}
and let $\cE_A$ and $\cE_B$ be index sets as above.
If we have
\begin{equation*}
	\mf w_A(\bhsh{01,0}(Y)) + \mf w_B(\bhsh{10,0}(Y))+1 > 0 \Mforall Y \in \cS(X),
\end{equation*}
then for any
\begin{equation*}
\begin{gathered}
	A \in \sB_{phg}^{\cE_A/\mf w_A}\sA^{-m-1}_-(HX_{\w};\Hom(E)\otimes\Omega_{\df h,R}), \\
	B \in \sB_{phg}^{\cE_B/\mf w_B}\sA^{-m-1}_-(HX_{\w};\Hom(E)\otimes\Omega_{\df h,R}),
\end{gathered}
\end{equation*}
the composition is defined and satisfies
\begin{equation*}
	C = A \circ B \in \sB_{phg}^{\cE_C/\mf w_C}\sA^{-m-1}_-(HX_{\w};\Hom(E)\otimes\Omega_{\df h,R}),
\end{equation*}
where $\cE_C$ is as above and $\mf w_C$ is the multiweight on $HX_{\w}$ given by
\begin{equation*}
	\mf w_C(\bhsh{dd,1}) 
	= \min(\mf w(\cE_A)(\bhsh{dd,1}) + \mf w_B(\bhsh{dd,1}),
		\mf w_A(\bhsh{dd,1}) + \mf w(\cE_B)(\bhsh{dd,1}) ),
\end{equation*}
and, for each $Y \in \cS(X),$
\begin{align*}
	\mf w_C(\bhsh{10,0}(Y)) 
	&= \min(\mf w_A(\bhsh{10,0}(Y)), \\
	&\phantom{xxx}
	\mf w(\cE_A)(\bhsh{\phi\phi,1}(Y)) + \mf w_B(\bhsh{10,0}(Y)),
	\mf w_A(\bhsh{\phi\phi,1}(Y)) + \mf w(\cE_B)(\bhsh{10,0}(Y))), \\
	\mf w_C(\bhsh{01,0}(Y)) 
	&= \min(\mf w_B(\bhsh{01,0}(Y)), \\
	&\phantom{xxx}
	\mf w(\cE_A)(\bhsh{01,0}(Y)) + \mf w_B(\bhsh{\phi\phi,1}(Y)),
	\mf w_A(\bhsh{01,0}(Y)) + \mf w(\cE_B)(\bhsh{\phi\phi,1}(Y))), \\
	\mf w_C(\bhsh{\phi\phi,1}(Y)) 
	&= \min(
	\mf w(\cE_A)(\bhsh{\phi\phi,1}(Y)) + \mf w_B(\bhsh{\phi\phi,1}(Y)),\\
	&\phantom{xxx}
	\mf w_A(\bhsh{\phi\phi,1}(Y)) + \mf w(\cE_B)(\bhsh{\phi\phi,1}(Y))),\\
	\mf w_C(\bhsh{01,0}(Y)) 
	&= \min(
	\dim Y + 1 + \mf w_A(\bhsh{10,0}(Y)) + \mf w(\cE_B)(\bhsh{01,0}(Y)), \\
	&\phantom{xxx}
	\dim Y+1+\mf w(\cE_A)(\bhsh{10,0}(Y)) + \mf w_B(\bhsh{01,0}(Y)), \\
	&\phantom{xxx}
	\mf w(\cE_A)(\bhsh{\phi\phi,0}(Y)) + \mf w_B(\bhsh{\phi\phi,0}(Y)),
	\mf w_A(\bhsh{\phi\phi,0}(Y)) + \mf w(\cE_B)(\bhsh{\phi\phi,0}(Y))).
\end{align*}

\end{enumerate}

\end{theorem}

The restriction on the multiweights in the second part of the theorem, which holds for all of the multiweights that we will make use of, is made only to simplify the statement of the theorem.

\begin{proof}
The proof of (1) is carried out in Appendix \ref{sec:HeatComp} following Melrose's geometric microlocal approach, see, e.g., \cite{Melrose-Piazza:Even, Dai-Melrose, Albin:RenInd, Mazzeo-Melrose:Surgery}. As explained in \S\ref{sec:Conormal}, the same pull-back and push-forward theorems used to prove (1) establish (2).
\end{proof}

$ $\\
We formalize the notion of smooth family of wedge heat operators using the space $H(M/B)_\w,$ as elements of 
\begin{equation*}
	\sB_{phg}^{\cE/\mf w}\sA^{-m-1}_- (H(M/B)_{\w};\Hom(E)\otimes\Omega_{\df h,R})
\end{equation*}
where $\cE$ and $\mf w$ are, respectively, an index set and multiweight for $H(M/B)_{\w},$ and $\df h$ is the multiweight above extended to $H(M/B)_{\w},$ i.e.,
\begin{equation}\label{eq:HeatWeight}
\begin{gathered}
	\mathfrak h: \cM_1(H(M/B)_\w) \lra \bbR, \\
	\mathfrak h(H) =
	\begin{cases}
	-(\dim(N/B) + 3) & \Mif H \subseteq \bhsh{\phi\phi,1}(N) 
		\Mforsome N \in \cS_{\psi}(M) \\
	-(\dim(N/B) + 1) & \Mif H \subseteq \bhsh{\phi\phi,0}(N) 
		\Mforsome N \in \cS_{\psi}(M) \\
	-(\dim(M/B) + 2) & \Mif H = \bhsh{dd,1} \\
	\infty & \Mif H = \bhsh{00,1} \\
	0 & \Motherwise
	\end{cases}
\end{gathered}
\end{equation}
The composition results in Appendix \ref{sec:HeatComp} are established in the setting of families of operators.

\section{Resolvent and heat kernel of $\eth_{M/B}$} 

Let us return to our usual setting with $M \xlra{\psi} B$ a locally
trivial family of manifolds with corners and iterated fibration
structures with a totally geodesic
vertical wedge metric $g_{M/B}$, and $E \lra M$ wedge Clifford module along the fibers of
$\psi.$ Recall that $D_{M/B}$ denotes the associated Dirac-type
operator acting on $L^2_{\w}(M/B;E),$ the wedge $L^2$-space and
$\eth_{M/B}$ denotes the unitarily equivalent operator acting on
$L^2(M/B;E)$, see \eqref{eq:wedge-v-reg}. We assigned a `vertical APS domain' to this operator,
$\cD_{\VAPS}(\eth_{M/B}),$ and in this section we will describe the
structure of the resolvent and heat kernel of $\eth_{M/B}$ under the
Witt assumption from Definition \ref{def:WittAss}.\\

\subsection{Compatible perturbations}
Before carrying out these constructions, we will generalize the operators under consideration by allowing certain perturbations by smoothing operators. The perturbations we will use in the main result of this paper will be compactly supported in the interior of $(M/B)^2_e,$ but we allow more general perturbations that share sufficiently many properties of $\eth_{M/B}$ so as to not seriously affect the analysis. In a future publication we will make use of these more general perturbations.

Let
\begin{equation*}
	Q \in \rho_{\ff((M/B)_\e)}^{-1}\Psi_{\e}^{-\infty}(M/B;E)
\end{equation*}
so that $\eth_{M/B} +Q$ has a model operator at every $N \in \cS_{\psi}(M),$ $y \in N^{\circ},$ modeling its behavior on the model wedge $\bbR^+_s \times \bbR^h_u \times Z_y,$ acting on sections of the pull-back of $E\rest{Z_Y},$ given by
\begin{equation*}
	\cN_y(\eth_{M/B}+Q) = \cl(dx)\pa_s + \tfrac1s(D_{Z_y}+ \cN_y(\rho_YQ)) + D_{\bbR^h}.
\end{equation*}
Here $\Nop_y(\rho_Y Q) = Q_{Z_y}$ is the restriction of the integral kernel of $\rho_YQ$ to the fiber of $\bhsd{\phi\phi,1}(N)$ over $y \in N^{\circ}.$ Thus, as in \eqref{eq:CompEdgeNormal}, $Q_{Z_y}$ is a non-commutative suspension operator, 
translation invariant with respect to the Lie group $\bbR^+\ltimes \bbR^h.$

Our analysis of $\cN_y(\eth_{M/B})$ in \S\ref{sec:NormOp} made use of two convenient facts, first that $D_{Z_y}$ is independent of the variables $(s,u),$ and secondly that it anti-commutes with Clifford multiplication by covectors in $T^*(N/B)^+.$ Our methods are insensitive to perturbations that maintain these two properties.
\begin{definition}\label{def:CompPert}
Let $M \xlra\psi B$ be a family of manifolds with corners and iterated fibration structures with a vertical wedge metric $g_{M/B}$ and $E \lra M$ a wedge Clifford module along the fibers of $\psi.$
By a {\bf compatible perturbation} (of the associated $\eth_{M/B}$) we will mean a self-adjoint family of operators $Q=Q_{M/B}$ satisfying two properties:\\
i) The integral kernel of $Q_{M/B}$ is an element of
\begin{equation*}
	\rho_{\ff((M/B)^2_e)}^{-1}\beta_{(2)}^*\CI(M \times_{\psi} M; \Hom(E)),
\end{equation*}
ii) At every $N \in \cS_{\psi}(M),$ $y \in N,$ 
\begin{equation*}
	\cl(\theta)Q_{Z_y} + Q_{Z_y}\cl(\theta) = 0,
\end{equation*}
for every covector $\theta$ in $T^*_y(N/B)^+,$ where $Q_{Z_y}$ is the operator on $Z_y$ whose integral kernel is the restriction of $\rho_NQ_{M/B}$ to the fiber of $\bhs{N}\times_{\phi_N} \bhs{N} \subseteq M \times_{\psi} M$ over $y.$\\
\end{definition}

\begin{remark}\label{rmk:PertExist}
In a subsequent paper we will study the existence of compatible perturbations. For the purpose of this paper we restrict ourselves to an example of how these will arise.

Consider a single manifold with boundary, $X,$ whose boundary $\bhs{Y}$ participates in a fiber bundle of closed manifolds
\begin{equation*}
	Z \fib \bhs{Y} \xlra{\phi_Y} Y,
\end{equation*}
together with a wedge Clifford module $(E, g_E, \nabla^E, \cl)$ and associated Dirac-type operator $\eth_X.$ The boundary family $D_{\bhs{Y}/Y}$ is a family of Dirac-type operators that anti-commute with Clifford multiplication in the $T^*Y^+$-directions and so determine an index class in the $C^*$-K-theory group $K_*(\Cl(T^*Y^+)).$ This index vanishes if and only if there is a family of smoothing operators 
$Q_{\bhs{ Y}/ Y} \in \Psi^{-\infty}(\bhs{ Y}/ Y;E)$ such that $D_{\bhs{ Y}/ Y} + Q_{\bhs{ Y}/ Y}$ is a family of invertible operators with the same anti-commutation property. If $q$ is any smooth function on $X^2$ that is equal to the Schwartz kernel of $Q_{\bhs{Y}/Y}$ on $\diag_{ Y} \times  Z^2 \subseteq \{x=x'=0\} \subseteq X^2,$ and we set $Q_{X} = \rho_{\bhsd{\phi\phi}( Y)}^{-1}\beta_{(2)}^*q,$ then $Q_X$ is a compatible perturbation of $\eth_X.$
\end{remark}

We will use the notation 
\begin{equation*}
	\eth_{M/B,Q} = \eth_{M/B}+Q_{M/B}, \quad
	D_{Z_y, Q} = D_{Z_y} + Q_{Z_y}, \quad \text{etc.,}
\end{equation*}
with the understanding that $Q_{M/B}$ is a compatible perturbation.
We define the vertical APS domain of $\eth_{M/B,Q}$ as the graph closure of $\rho_X^{1/2}H^1_e(X;E)\cap \cD_{\max}(\eth_{M/B,Q})$ and say that the Witt condition is satisfied if
\begin{equation*}
	0 \notin \Spec( D_{Z_y,Q} ),
\end{equation*}
where the spectrum refers to $D_{Z_y}+Q_{Z_y}$ with its vertical APS domain in $L^2(Z_y;E|_{Z_y}).$
The compatibility conditions are chosen so that Proposition \ref{prop:NyInjSA} holds after replacing $\eth_{M/B}$ with $\eth_{M/B,Q}$ with the same proof.\\

From \S\ref{sec:NormOp} we know that the indicial roots of $\eth_{M/B,Q}$ at $y \in N,$ $N \in \cS_{\psi}(M)$ are equal to the positive eigenvalues of the induced Dirac-type operator $D_{Z_y,Q},$ acting on $L^2(\bhs{N}/N;E|_N)$ with its vertical APS domain.
Define an `indicial multiweight', $\mf I,$ for $M$ by
\begin{equation}\label{eq:IndicialMw1}
	\mf I (\bhs{N}) = \min \{ \lambda \in \Spec(D_{Z_y,Q}) \cap \bbR^+ : y \in N \} \Mforall N \in \cS_{\psi}(M)
\end{equation}
and a corresponding multiweight $\mf I^{(2)}$ for $(M/B)^2_\e$ by
\begin{equation*}
\begin{gathered}
	\mf I^{(2)}(\bhsd{10}(N) ) = 
	\mf I^{(2)}(\bhsd{01}(N) ) = \mf I (\bhs{N}) \\
	\mf I^{(2)}(\bhsd{\phi\phi}(N) ) = 2\mf I(\bhs{N}) + \dim (N/B) +1 \Mforall N \in \cS_{\psi}(M).
\end{gathered}
\end{equation*}
The weight at the front faces $\bhsd{\phi\phi}(N)$ is explained by the composition formula for edge pseudodifferential operators: composing an operator with the given weights at the side faces produces this weight at the front faces.
We use the same notation for the indicial multiweights of $X$ and $X^2_\e.$

\begin{theorem}\label{thm:FredholmResolvent}
Let $\eth_{M/B,Q}$ be a family of compatibly perturbed Dirac-type wedge operators endowed with its vertical APS domain and satisfying the Witt assumption.
Then $(\eth_{M/B,Q}, \cD_{\VAPS})$ is a family of self-adjoint, Fredholm operators with compact resolvent. The generalized inverse of $\eth_{M/B,Q}$ is a family valued in $\rho_{\ff((M/B)^2_e)}\Psi_\e^{-1, \mf I^{(2)}}(M/B;E).$

For each fiber $X$ of $\psi,$ the eigenfunctions of $\eth_{X,Q}$ are elements of $\rho_X^{\mf I}H_\e^{\infty}(X;E),$ the resolvent is a meromorphic function on $\bbC$ with values in the edge calculus with bounds,
\begin{equation*}
	(\eth_{X,Q} - \lambda)^{-1} \in \rho_{\ff(X^2_e)}\Psi_\e^{-1, \mf I^{(2)}}(X;E),
\end{equation*}
and the projection onto the $\lambda$-eigenspace of $\eth_{X,Q}$ satisfies
\begin{equation}\label{eq:EigProjSpace}
	\Pi_{\lambda} \in \rho_{\ff(X^2_e)}\Psi_\e^{-\infty, \mf I^{(2)}}(X;E).
\end{equation}
\end{theorem}
$ $\\

Define the indicial multiweight for the heat space, in terms of \eqref{eq:IndicialMw1}, by
\begin{equation}\label{eq:IndicialMwH}
\begin{gathered}
	\mf I^{(H)}(\bhsh{10,0}(N) ) = 
	\mf I^{(H)}(\bhsh{01,0}(N) ) = \mf I (\bhs{N}), \\
	\mf I^{(H)}(\bhsd{\phi\phi,0}(N) ) 
	= 2\mf I(\bhs{N}) + \dim (N/B) +1, \quad
	\mf I^{(H)}(\bhsh{\phi\phi,1}(N) ) = \infty
	\; \; \forall N \in \cS_{\psi}(M), \\
	\Mand \mf I^{(H)}(\bhsh{00,1}) = \mf I^{(H)}(\bhsh{dd,1}) = \infty.
\end{gathered}
\end{equation}
We also define an index set for the heat space by 
\begin{equation}\label{eq:HeatIndSets}
\begin{gathered}
	\cH(\bhsh{dd,1}) = 2, \quad
	\cH(\bhsh{00,1}) = \emptyset, 
	\Mand \\
	\cH(\bhsh{10,0}(N))
	=\cH(\bhsh{01,0}(N))
	=\cH(\bhsh{\phi\phi,0}(N))
	= \emptyset, \;\;
	\cH(\bhsh{\phi\phi,1}(N)) = 2 \;\; \forall N \in \cS_{\psi}(M).
\end{gathered}
\end{equation}

\begin{theorem}\label{thm:WedgeHeatKer}
Let $\eth_{M/B,Q}$ be a compatibly perturbed family of Dirac-type wedge operators endowed with its vertical APS domain and satisfying the Witt assumption.
The heat kernel of $\eth_{M/B,Q}^2$ satisfies
\begin{equation*}
	e^{-t\eth_{M/B,Q}^2} \in
	\sB^{\cH/\mf I^{(H)}}_{phg}\sA^{-m-1}_-(H(M/B)_{\w}; \Hom(E)\otimes \Omega_{\mf h,R})
\end{equation*}
where $\mf I^{(H)}$ and $\cH$ are given by \eqref{eq:IndicialMwH}, \eqref{eq:HeatIndSets} and $\Omega_{\mf h, R}$ is the density bundle from \eqref{eq:HeatDensityBdle}.
\end{theorem}

The rest of this section consists of a proof of Theorems \ref{thm:FredholmResolvent} and \ref{thm:WedgeHeatKer} by induction on the depth of $M.$ Our base case consists of closed manifolds, for which these theorems are well-known, even with a smoothing perturbation (e.g., \cite[Proposition 9.46]{BGV2004}, \cite[Appendix]{Melrose-Piazza:Even}, \cite{Albin-Rochon:DFam, Albin-Rochon:DBar, Albin-Rochon:Some}). Thus we now assume that these theorems are known for all spaces of depth less than $k,$ and that $X$ has depth $k.$\\

\subsection{The model wedge} \label{sec:ModelWedge}

In the situation above, choose $N \in \cS_{\psi}(M),$ $y\in N^{\circ},$ let $Z = \phi_N^{-1}(y)$ and let 
\begin{equation*}
	\cN_y(\eth_{M/B,Q}) 
	= \cl(dx)\pa_s + \tfrac1sD_{Z,Q} + D_{\bbR^h} 
	= \cN_y(D_{M/B,Q}) - \cl(dx) \tfrac{\dim Z}{2s}
\end{equation*}
be the normal operator of $\eth_{M/B,Q}$ on the model wedge at $y,$ $\bbR^+_s \times \bbR^h \times Z,$ from \S\ref{sec:NormOp}.
In this section, we make use of the inductive hypothesis that Theorems
\ref{thm:FredholmResolvent} and \ref{thm:WedgeHeatKer} hold for
$D_{Z,Q}$ to describe the Green's function and heat kernel of
$\cN_y(\eth_{M/B,Q}).$

Our assumptions on the perturbation and inductive hypothesis on the link invite us to analyze $\cN_y(\eth_{M/B,Q})$ by using the Fourier transform on $\bbR^h$ and the Hankel transform on each eigenspace in a spectral decomposition on $Z$ (as is done in, e.g., \cite{C1979}, \cite{Cheeger-TaylorI}, \cite{C1983}, \cite{Ch1985}, \cite[\S2.3]{Lesch:Fuchs}, \cite[\S8.8]{taylor:vol2},\cite[\S3.2]{Mazzeo-Vertman:AT}).

Thus we consider 
\begin{equation*}
	(\cN_y(D, Q))^2 =  -\pa_s^2 
	+ \tfrac1{s^2}( (\cl(dx)D_{Z,Q}-\tfrac12)^2-\tfrac14) + \Delta_{\bbR^h}
\end{equation*}
as an operator on $L^2(ds\ d\eta\ dz)$ and using the inductive hypothesis of discrete spectrum of $D_{Z,Q}$ with its vertical APS domain, denote the eigenvalues of the self-adjoint operator $(\cl(dx)D_{Z, Q} - 1/2)^2$ by $\{\ell_i\}_{i = 1}^\infty$ and the corresponding eigensections by $\{\phi_i\}.$
As in \cite[\S8.8]{taylor:vol2}, by writing
\begin{equation*}
	F(s, z) = \sum f_i(s)\phi_i(z)
\end{equation*}
for appropriate coefficients $f_i,$ we have
\begin{equation*}
	(-\pa_s^2 
	+ \tfrac1{s^2}( (\cl(dx)D_{Z,Q}-\tfrac12)^2-\tfrac14) )
	= \sum 
	(-\pa_s^2 
	+ \tfrac{\ell_i-\tfrac14}{s^2})(f_i(s)\phi_i(z))
\end{equation*}
and this will equal $\mu^2F(s,z)$ if we take
\begin{equation*}
	f_i(s) = \sqrt{s}J_{\nu_i}(\mu s)
\end{equation*}
with $\nu_i^2 = \ell_i.$

Note that there is a potential sign ambiguity in $\nu_i.$ 
As $s \to 0^+,$ $f_i(s)= \cO(s^{\tfrac12+\nu_i})$ will be in $L^2_{\mathrm{loc}}(ds)$ for $\nu_i>-1,$ so the ambiguity is only for the small eigenvalues $\ell_i<1.$ The choice of square root corresponds to different domains for  $\cN_y(\eth_{M/B,Q})^2.$
In terms of the eigenvalues $\{\lambda_i\}$ of $\cl(dx)D_{Z,Q},$  we have $\ell_i = (\lambda_i -\tfrac12)^2$ so this ambiguity corresponds to $\lambda$ in $(-\tfrac12, \tfrac32).$ If we restrict attention to domains of $\cN_y(\eth_{M/B,Q})^2$ induced from domains of $\cN_y(\eth_{M/B,Q})$ then, arguing as in the proof of Proposition \ref{prop:NyInjSA}, the ambiguity corresponds to $\lambda$ in 
$(-\tfrac12, \tfrac32) \cap (-\tfrac32, \tfrac12) = (-\tfrac12, \tfrac12).$ (Thus, as is well-known, these are the small eigenvalues that distinguish domains, see for example the discussion in \cite[top of page 37]{Ch1985} in terms of which we are taking $\Delta_b.$)

In particular, as we are interested in the vertical APS domain for $\cN_y(\eth_{M/B,Q}),$ which induces
\begin{equation*}
	\cD_{\VAPS}(\cN_y(\eth_{M/B,Q})^2)
	= \{ F \in \cD_{\VAPS}(\cN_y(\eth_{M/B,Q})): 
	\cN_y(\eth_{M/B,Q})(F) \in \cD_{\VAPS}(\cN_y(\eth_{M/B,Q}))\},
\end{equation*}
we define $\nu_{APS}$ by
\begin{equation}\label{eq:nuAPS}
	\nu_{APS}(\lambda) =
	\begin{cases}
	-|\lambda-\tfrac12| & \Mif \lambda \in \Spec(D_{Z_y,Q}) \cap (0,\tfrac12) \\
	|\lambda-\tfrac12| & \Mif \lambda \in \Spec(D_{Z_y,Q}) \setminus (0,\tfrac12) 
	\end{cases}
\end{equation}
Then we can, as in \cite[Sec.\ 8.8]{taylor:vol2}, diagonalize $\cN_y(\eth_{M/B,Q})^2$ by combining first the map
\begin{equation*}
	\cH(g) 
	= \bigoplus_{\lambda \in \Spec(D_{Z_y,Q})} 
	( H_{\nu_{APS}(\lambda)}(s^{-1/2} g_\lambda) ),
\end{equation*}
where $H_{\nu_{APS}(\lambda)}$ denotes the Hankel transform, $g_{\lambda}$ denotes the projection of $g$ onto the corresponding (i.e., $(\lambda-\tfrac12)^2-\tfrac14$) eigenspace of $(\cl(dx)D_{Z, Q} - 1/2)^2 -1/4,$
and then the Fourier transform in $\bbR^h.$ This yields a unitary map onto $L^2(\lambda \; d\lambda \; d\xi, \ell^2)$ which replaces $\cN_y(\eth_{M/B,Q}^2)$ with multiplication by $\lambda^2 + |\xi|^2.$

The operator $\cN_y(\eth_{M/B},Q)^2$ is injective on its vertical APS domain and has an (unbounded) inverse $\cG$, determined by multiplication by $\cH(\cG g) = (\lambda^2 + |\xi|^2)^{-1} \cH(g),$
and satisfying
\begin{equation}\label{eq:DefUnbddInv}
	(\cN_y(\eth_{M/B}, Q))^2 \cG = \Id.
\end{equation}
Below we will analyze the integral kernels of the heat kernel and of
$\cG$, in particular we will use the integral kernel of the unbounded operator
$$
G = \cN_y(\eth_{M/B}, Q)) \cG,
$$
the Green's function for $\cN_y(\eth_{M/B, Q})$, to construct a
parametrix for $\eth_{M/B, Q}$.

{\bf The heat kernel on the model wedge.}
The heat kernel of a product is the product of the heat kernels, so to begin with let us disregard the factor of $\bbR^h$ and focus instead on the exact Riemannian cone
\begin{equation*}
	Z^+ = \bbR^+_s \times Z, \quad g_{Z^+} = ds^2 + s^2g_Z.
\end{equation*}
We denote $E$ pulled-back to $\bbR^+_s \times Z$ by the same symbol, and the corresponding Dirac-type operator by
\begin{equation}\label{eq:CDZ}
	\eth_{Z^+, Q} = \cl(dx)\pa_s + \tfrac1sD_{Z,Q}.
\end{equation}

It follows from Proposition \ref{prop:NyInjSA} that $\eth_{Z^+, Q}$ is injective and self-adjoint with its vertical APS domain. Thus the corresponding domain for
\begin{equation*}
	\eth_{Z^+, Q}^2 
	= -\pa_s^2 + \tfrac1{s^2}( ( \cl(dx)D_{Z,Q}-\tfrac12)^2-\tfrac14),
\end{equation*}
namely
\begin{equation*}
	\cD_{\VAPS}(\eth_{Z^+, Q}^2)
	= \{ u \in \cD_{\VAPS}(\eth_{Z^+, Q}): 
	\eth_{Z^+, Q}u \in \cD_{APS}(\eth_{Z^+, Q}) \},
\end{equation*}
is also a self-adjoint domain.

Let $e^{-t\eth_{Z^+, Q}^2}$ be the heat kernel of $(\eth_{Z^+, Q}, \cD_{APS}(\eth_{Z^+, Q}))$ considered as a density, 
\begin{equation}\label{eq:HeatKerDensity}
	e^{-t\eth_{Z^+, Q}^2} = \cK \; \mu_R.
\end{equation}
From the spectral theorem we know that $\cK$ is a distribution on the space
\begin{equation*}
	(Z^+)^2 \times \bbR^+_{t}
\end{equation*}
such that:
\begin{itemize}
\item $\lim_{t\to0} e^{-t\eth_{Z^+, Q}^2}  = \Id,$
\item For every $t>0,$ the map 
\begin{equation*}
	\cD_{\VAPS}(\eth_{Z^+, Q}^2) \ni s \mapsto e^{-t\eth_{Z^+, Q}^2}s \in L^2(Z^+;E)
\end{equation*}
is valued in $\cD_{\VAPS}(\eth_{Z^+, Q}^{\infty}) = \bigcap\cD_{\VAPS}(\eth_{Z^+, Q}^{\ell}).$ In particular, $\cK(r, z, r', z', t)$ is smooth in all of its variables in the interior of $(Z^+)^2 \times \bbR^+_{t},$
\item For each $t,$ $e^{-t\eth_{Z^+, Q}^2}$ is a self-adjoint operator, and hence 
$\cK(r, z, r', z', t) = \cK(r', z', r, z, t).$
\end{itemize} 
We will improve these properties by showing that $e^{-t\eth_{Z^+, Q}^2},$
viewed as a distribution on a different compactification of the
interior of $(Z^+)^2 \times \bbR^+_{t},$ extends nicely to the
boundary.

Recall, e.g., from \cite[Proposition 2.3.9]{Lesch:Fuchs}, that given $a>-1$ and $f \in \CIc(\bbR^+),$ the solution to 
\begin{equation*}
	\begin{cases}
	(\pa_t + (-\pa_s^2 + s^{-2}(a^2-\tfrac14))) u(s,t) =0\\
	\displaystyle \lim_{t\to0} u(s,t)=f(s)
	\end{cases}
\end{equation*}
is given by
\begin{equation*}
	u(s,t) 
	= \int_0^{\infty} \frac{\sqrt{s\wt
            s}}{2t}I_{p(a^2)}\lrpar{\frac{s\wt s}{2t}}
        \exp\lrpar{-\frac{s^2+\wt s^2}{4t}} 
	f(\wt s) \; d\wt s
\end{equation*}
where $I_{p(a^2)}$ denotes the modified Bessel function of the first kind,
$p(a^2) = a$ if $a\geq 1$ and otherwise satisfies $p(a^2) \in \{\pm a\}$ with different choices corresponding to different domains of $(-\pa_s^2 + s^{-2}(a^2-\tfrac14))$ as an unbounded operator on $L^2(\bbR^+).$ 
In projective coordinates as above, this shows that this heat kernel is a right density times the function 
\begin{equation*}
	\frac{\sqrt{s}}{2\sigma} I_{p(a^2)}\lrpar{\frac{s}{2\sigma^2}} 
		\exp\lrpar{-\frac{s^2+s'}{4\sigma^2}}. 
\end{equation*}
Hence we can write the heat kernel on the exact cone as
\begin{equation}\label{eq:spectral-model-heat}
	\sum_{\lambda \in \Spec(D_{Z_y,Q})} 
	\frac{\sqrt{ss'}}{2\sigma} I_{\nu_{APS}(\lambda)}\lrpar{\frac{ss'}{2\sigma^2}}
		\exp\lrpar{-\frac{s^2+(s')^2}{4\sigma^2}}
	\Phi_{\lambda}(z, \wt z)
\end{equation}
where $\nu_{APS}$ is given by \eqref{eq:nuAPS} and $\Phi_{\lambda}(z, \wt z)$ is the projection onto the $\lambda$-eigenspace of $\cl(dx)D_{Z_y,Q}.$
Convergence of this sum in the space of polyhomogeneous conormal
distributions is used in \cite[Proposition
3.2]{Mazzeo-Vertman:AT}. See also \cite[Example 3.1]{C1983}.
\\

To establish the asymptotics of this kernel our strategy, following Mooers \cite{Mooers} and others, e.g., \cite[\S2]{C1983}, \cite[\S2.2]{Lesch:Fuchs}, is to exploit the homogeneity of the cone.
For each $c>0,$ we set
\begin{equation*}
	\Upsilon_c: 
	(Z^+)^2 \times \bbR^+_{t} \lra (Z^+)^2 \times \bbR^+_{t}, \quad
	\Upsilon_c(s,z,s',z',t) = (cs, z, cs', z', c^2t);
\end{equation*}
we use the same symbol to denote the corresponding scalings on $(Z^+)^2$ and $Z^+$  and trust that this will not lead to confusion.

As $E$ is pulled-back from $Z,$ it makes sense to pull-back a section of $E$ over $Z^+$ along $\Upsilon_c$ and it is easy to see that
\begin{equation*}
	\Upsilon_c^*:\CIc( (Z^+)^{\circ};E) \lra \CIc((Z^+)^{\circ};E)
\end{equation*}
extends to a bounded map on $L^2(Z^+;E)$ and satisfies
\begin{equation*}
	\Upsilon_c^*\circ r\pa_r = r\pa_r \circ \Upsilon_c^*, \quad
	\Upsilon_c^*\circ \pa_z = \pa_z \circ \Upsilon_c^*.
\end{equation*}
It follows that $\Upsilon_c^*$ preserves $\cD_{\VAPS}(\eth_{Z^+, Q}^\ell)$ for any $\ell \in \bbN$ and satisfies
\begin{equation*}
	\Upsilon_c^* \circ \eth_{Z^+, Q} 
	= c^{-1} \eth_{Z^+, Q} \circ\Upsilon_c^*, \quad
	\Upsilon_c^* \circ (t\pa_t + t\eth_{Z^+, Q}^2)  
	= (t\pa_t + t\eth_{Z^+, Q}^2) \circ\Upsilon_c^*.
\end{equation*}
In particular, if $u$ is a solution of the heat equation with initial data $f,$ then $\Upsilon_c^*u$ solves the heat equation with initial data $\Upsilon_c^*f.$
However,
\begin{multline*}
	\Upsilon_c^*u(\zeta, t)
	= \int_0^t \int_{Z^+} \cK(cs, z, s', z', c^2t)f(s',z') 
		\; ds'\;dz'\;dt \\
	\xlra{s'=cr'} c
	\int_0^t \int_{Z^+} \cK(cs, z, cr', z', c^2t) f(cr',z') 
		\; dr'\; dz'\; dt,
\end{multline*}
so uniqueness for the heat equation shows that $\Upsilon_c^*\cK= c^{-1}\cK$ and hence, as a right density, the heat kernel is invariant under this dilation,
\begin{equation}\label{eq:DilationInv}
	\Upsilon_c^*e^{-t\eth_{Z^+, Q}^2} = e^{-t\eth_{Z^+, Q}^2}.
\end{equation}
This is the dilation invariance that we will exploit.\\

To do so, we first blow-up $\{t=0\}$ parabolically so that $\tau = \sqrt t$ is a generator of the smooth structure. Secondly we blow-up the cone-tip at time zero, to obtain the space
\begin{equation*}
	H_1Z^+ 
	= [\bbR^+_s \times Z \times \bbR^+_{s'} \times Z' \times \bbR^+_{\tau}; 
	\{s=s'=\tau=0\}]
	\cong Z \times Z' \times \bbS^2_{+} \times \bbR^+_{R}
\end{equation*}
where $\bbS^2_{+} = \{(\omega_s,\omega_{s'},\omega_\tau) \in [0,\infty)^3: 
\omega_s^2+\omega_{s'}^2+\omega_\tau^2=1\}.$
The blow-down map is
\begin{equation*}
	\xymatrix @R=1pt{
	H_1Z^+ \ar[r]^-{\beta} & Z^+ \times Z^+ \times \bbR^+_{t} \\
	(z,z',(\omega_s,\omega_{s'}, \omega_{\tau}), R) \ar@{|->}[r] &
	((R\omega_s, z), (R\omega_{s'}, z'), R^2\tau^2)}
\end{equation*}
and we note that $\Upsilon_c$ lifts to be simply $R \mapsto cR.$
We denote $\{R=0\}$ by $\bhs{R}(H_1Z^+).$

Instead of using polar coordinates, we can use projective coordinates
\begin{equation*}
	r = \frac s{s'}, \quad
	z, \quad
	s', \quad
	z', \quad
	\sigma = \frac{\tau}{s'},
\end{equation*}
valid away from $\omega_{s'}=0,$ in which $s'$ is a boundary defining function for $\bhs{R}(H_1Z^+)$ and $\Upsilon_c$ is $s' \mapsto cs'.$

Let $\pi_L: Z^+ \times Z^+ \times \bbR^+_t \lra Z^+$ be the projection onto the left factor of $Z^+,$ and let $\beta_L = \pi_L \circ \beta.$ We have
\begin{equation*}
	\beta_L^*D^2_{Z^+,Q}
	=(s')^{-2}\lrpar{ -\pa_s^2 
		+ \tfrac1{s^2}( ( \cl(dx)D_{Z,Q}-\tfrac12)^2-\tfrac14)}.
\end{equation*}
The plan is to identify the heat kernel at $s'=1$ and then use dilation invariance.\\

To this end,  let $\chi: \bbR \lra \bbR^+$ satisfy
\begin{equation*}
	\chi(r) =
	\begin{cases}
	1 & \Mif | r - 1 | < 1/4 \\
	0 & \Mif |r - 1 | > 1/2 
	\end{cases},
\end{equation*}
and define the operator
\begin{equation}\label{eq:push-forward-D}
\wt{D} = \cl(dx)\pa_s + ((1 - \chi(s)) + \tfrac{\chi(s)}{s}) D_{Z,Q},
\end{equation}
which we can interpret as an operator on $\mathbb{S}^1 \times Z$ where
$\mathbb{S}^1 = [0, 2]/0 \sim 2$ and
coincides (thought of as an operator on $\mathbb{R}_+ \times Z$) with $D_{Z^+, Q}$
on $3/4 \le r \le 5/4$

Let $e^{-t\wt{D}^2}$ denote the heat kernel of $\wt{D}^2$ on $\bbS^1 \times Z$ endowed with its vertical APS domain and note that the inductive hypothesis applies to it.
We consider $e^{-t\wt{D}^2}$ as a right density and denote it as $\wt{\cK} \mu_R.$

We will multiply the model wedge heat kernel $\cK$ by cut-off functions to obtain a kernel on $\bbS^1 \times Z$ that we can compare to $e^{-t\wt{D}^2}=\wt{\cK} \mu_R.$
Let $\wt{\chi} \colon \mathbb{R} \lra \mathbb{R}_+$ satisfy $\chi
\wt{\chi} = \wt{\chi}$ and $\wt{\chi}(r) = 1$ for $|r - 1| < \delta$
for some $\delta > 0$.  Set
\begin{equation*}
	F_1(r,z, r', z', \tau) = \wt{\chi}(r) \wt{\chi}(r')\cK(r,z,r',z',\tau)
\end{equation*}
so that
\begin{equation*}
		(\pa_t + \wt{D}^2)F_1(r,z,r', z', \tau)
	= E(r,z,r', z',\tau) \\
	\end{equation*}
Note that $E \equiv 0$ for $ |r - 1| < \delta$, indeed
$$
E(r,z,r', z, \tau)
	=( -\wt{\chi}''(r) \cK(r,z,r',z',\tau) 
	- 2\wt{\chi}'(r) \pa_r(\cK(r,z,r',z',\tau)))\wt{\chi}(r').
$$
Letting $F_2(r,z, r', z', \tau) = \wt{\cK} \circ E$, so $(\pa_t +
\wt{D}^2)(F_1 - F_2) = 0$, we have
\begin{itemize}
\item $F_2(r,z, r', z', \tau) = O(\tau^\infty)$ for $r \in [1-\delta, 1 +
 \delta]$.
  Indeed, in 
$$
\int_0^t \int_{\bbS^1 \times Z}\wt{\cK}(r,z, r'', z'', \tau - s)
\circ E(r'',z'', r', z', s)) dsdr''dz'',
$$
where $r'' \in [0, 2]$ is the variable on $\mathbb{S}^1$, we can take
the spatial integral over the set $V = U^c$ where $U = \{ r'' \in [1 -
\delta, 1 + \delta]\}$ since on $U$, $E \equiv 0$, but on $V$, if the left
variable $r \in U$ we have $\wt{\cK} = O(\tau^\infty)$ by our inductive
hypothesis.
\item The restiction of $F_1 - F_2$ to $r' = 1$ satisfies
$$
\lim_{\tau \to 0} (F_1 - F_2) (r, z, 1, z', \tau)  = \wt{\chi}(r,
  z)\delta_{(r,z)}(1,z') = \delta_{(r,z)}(1,z').
$$
\end{itemize}
By uniqueness, it follows that
\begin{equation*}
	\wt{\cK}(r, z, 1, z', \tau) = (F_1-F_2)(r,z,1,z', \tau),\ 
	\wt{\cK}(r, z, s', z', \tau) =
        (s')^{-1}(F_1-F_2)(r,z,1,z',\tau),
\end{equation*}
in particular
\begin{equation}\label{eq:pushforward-difference}
\wt{\cK} - \cK = O(\tau^\infty) \mbox{ for } r' = 1, \; r \in [1 - \delta, 1
+ \delta],
\end{equation}
and this goes most of the way toward proving:

\begin{proposition}\label{prop:WedgeHeatNormal}
The heat kernel of the normal operator, $e^{-t\cN_y(\eth_{M/B,Q}^2)},$ is a conormal distribution in the space
\begin{equation}\label{eq:DistSpaceModelHeat}
	\sB^{\cR_N/\mf I^{(H)}_N}_{phg}
	\sA^{-m-1}_-(\bhsh{\phi\phi,1}(N);
		\Hom(E)\otimes \rho_{dd,1}^{-n-2}\beta_{(H),R}^*\Omega(M/B))
\end{equation}
where $\mf I^{(H)}_N$ is the restriction of $\mf I^{(H)}$ to $\bhsh{\phi\phi,1}(N),$ and
\begin{equation*}
\begin{gathered}
	\cR_N(\bhsh{10,0}\cap \bhsh{\phi\phi,1}) 
	=\cR_N(\bhsh{01,0}\cap \bhsh{\phi\phi,1}) 
	=\cR_N(\bhsh{00,1}\cap \bhsh{\phi\phi,1}) 
	=\cR_N(\bhsh{\phi\phi,0}\cap \bhsh{\phi\phi,1}) 
	= \emptyset, \\
	\cR_N(\bhsh{dd,1}\cap \bhsh{\phi\phi,1}) = 2, \\
\end{gathered}\end{equation*}
that inverts the vertical heat operator $\tfrac12\sigma\pa_\sigma + \sigma^2(\cN_y(\eth_{M/B,Q}))^2.$
\end{proposition}

\begin{proof}[Proof of Proposition \ref{prop:WedgeHeatNormal} under
  inductive hypothesis]
Without loss of generality we assume that $B = \pt$, so $X = M$ is a
manifold with corners with iterated fibration structure of depth $k$ carrying a single
wedge Dirac type operator;
thus the base of a boundary hypersurface $\bhs{Y}$ corresponds to a
heat front face $\bhsh{\phi\phi,1}(Y)$ that is a bundle with fibre
$HZ_{Sus(\w)}$ with $Z$ of depth less than $k$, and again without loss
of generality we assume that $Y$ is a point since the heat kernel is
product type, so $\bhsh{\phi\phi,1}(Y) = HZ_{Sus(\w)}$.  
By our inductive hypothesis, Theorem \ref{thm:WedgeHeatKer}
applies to $\wt{\cK}$ above, and thus $\wt{\cK}$ lies in $\sB^{\cH/\mf
  I^{(H)}}_{phg}\sA^{-m-1}_-(H(\mathbb{S}^1 \times Z)_{\w};
\Hom(E)\otimes \Omega_{\mf h,R})$.  The lift of the set $s' = 1$ to
$H(\mathbb{S}^1 \times Z)_{\w}$ is a suspended heat space itself, and
a neighborhood of its diagonal, say $s \in [1 - \delta, 1 + \delta]$,
may be identified with the same neighborhood in $HZ_{Sus(\w)}$. 

From the expression for the heat kernel in
\eqref{eq:spectral-model-heat} , the high order asymptotics of the modified
Bessel functions, and the inductive hypothesis, we see that for any $\epsilon > 0$, with $s
\in [0, 1 - \epsilon)$, $s' = 1$, the heat kernel has the appropriate
asymptotics at the side faces.
This together with
\eqref{eq:pushforward-difference} gives the proposition.
\end{proof}

{\bf Green's function on the model wedge.}
We now prove the analogous statement above for the right inverse
$G(N)$ of $\cN_y(\eth_{M/B,Q}).$
\begin{proposition}\label{thm:normal-green's}
Let $N \in \cS_{\psi}(M)$ and denote the Green's function of
$\cN_y(\eth_{M/B,Q})$ constructed above by $G(N).$ 
The integral kernel of $G(N)$ is an element of 
\begin{equation*}
	\Psi^{-1, \mf I^{(2)}_N}_{Nsus(T(N/B)^+)}(\bhs{N}/N;E)
\end{equation*}
where $\mf I^{(2)}_N$ denotes the restriction of the indicial multiweight to $\bhsd{\phi\phi}(N).$
\end{proposition}
\begin{proof}[Proof of Proposition \ref{thm:normal-green's} under
  inductive hypothesis]
  The proof is directly analogous to that of Proposition
  \ref{prop:WedgeHeatNormal}, and we use the notation from that proof.
  First we study $\cG$, the Green's
  function of the self-adjoint extension of  $\cN_y(\eth_{M/B, Q})^2 = \eth_{Z^+,
    Q}^2 + \Delta_{\mathbb{R}^h}$  from \eqref{eq:DefUnbddInv}.
    Letting $\wt{D}$ be as in \eqref{eq:push-forward-D}, the
  operator $\wt{L} := \wt{D}^2 +
  \Delta_{\mathbb{R}^h}$ on sections of $E$ over $\mathbb{S}^1 \times Z
  \times \mathbb{R}^h$ is the square of a Dirac-type operator on a
  lower depth space, and thus the inductive hypothesis applies to it.
  Writing $F_1(s, z, y, s', z', y') = \wt{\chi}(s) \cG \wt{\chi}(s')$ and thinking of this
  as the Schwartz kernel of an operator on the $\bbS^1 \times Z \times \bbR^h,$ we have $\wt{L} F_1 -
  \wt{\chi} = E$ is a smooth distribution, in fact vanishing for $|s -
  1|, |s' - 1| < \delta$.  Writing $\wt{G} \circ E = F_2$ and noting
  that all operators here can be taken symmetric, using $\wt{G} \wt{L}
  = I - \Pi$, we see that in the region $|s -
  1|, |s' - 1| < \delta$, $F_1$ differs from $\wt{\cG}$ by a smooth
  function, and thus restricted to $s' = 1$ is a polyhomogeneous
  conormal distribution on the suspended double space.  Homogeneity in
  $s, s'$ then gives the result. 
\end{proof}

\begin{remark}\label{rm:match}
  Recall from our discussion of the normal operators above that
  $\bhsd{\phi\phi}(Y) \cap \bhsd{\phi\phi}(\wt{Y})$ is a front face of
  the resolved (suspended) double space on which the normal operators'
  Schwartz kernels live.  It follows from the proof and the inductive
  hypothesis that for $N'
  < N$, $G(N) \rvert_{\bhsd{\phi\phi}(N) \cap
    \bhsd{\phi\phi}(\wt{N'})}$ is equal to the Green's function for the
  normal operator on that face. 
\end{remark}

 \begin{remark}
 The computation of the null space of the normal operator in Proposition \ref{prop:NyInjSA} can be used to write down integral kernels of the inverses of these operators as in, e.g., \cite[\S XIII.3, Theorem 16]{Dunford-Schwartz:2}.
 For the vertical APS domain we have, see e.g.,  \cite[Lemma
 4.1]{Bruning-Seeley:Index} 
 \begin{multline*}
 	(  -\pa_s^2 
	+ \tfrac1{s^2}( (\cl(dx)D_{Z,Q}-\tfrac12)^2-\tfrac14) + z^2
	)^{-1}(r, \wt r) \\
 	= \bigoplus_{\lambda \in \Spec(A_y)}
 	\sqrt{r\wt r} I_{\nu_{APS}(\lambda)}(rz) K_{\nu_{APS}(\lambda)}(\wt rz) \Pi_{\lambda}, \quad
 	\Mfor r < \wt r
 \end{multline*}
 (and interchanging $r$ with $\wt r$ for $\wt r<r$)
 where $\nu_{APS}$ is defined in \eqref{eq:nuAPS} and $\Im(z^2)\neq 0.$
 It is a bit delicate to establish the asymptotic behavior of the
 inverse from this explicit expression.
 \end{remark}
%

\subsection{Resolvent of $\eth_{M/B,Q}$} \label{sec:Res}

We have described, for each $N \in \cS_{\psi}(M),$ $y \in N^\circ,$ the inverse of the normal operator $\cN_y(\eth_{M/B,Q}).$ 
Putting these together, we specified the integral kernel 
\begin{equation*}
	G(N) = \{ G_y(N) : N \in \cS_{\psi}(M), \; y \in N \}. 
\end{equation*}
In fact this also determines the integral kernel over the boundary of each $N \in \cS_{\psi}(M).$ Indeed, recall that the structure of $\bhsd{\phi\phi}(N)$ is described in Proposition \ref{prop:ffX2}.
At a boundary hypersurface $\bhs{N'N}$ of $N,$ fibering over $N'<N,$ the fibers of $\bhsd{\phi\phi}(N)$ comprise one of the front faces of $(Z')^2_e.$ Namely, they comprise the front face corresponding to the boundary hypersurface $\bhs{NZ'}$ of $Z'$ in the notation of \eqref{eq:IFSDetail'}.
Since the integral kernel $G(N')$ has been specified over $N',$ taking
its normal operator over the front face corresponding to $\bhs{NZ'},$
yields the extension of $G(N)$ to points over the boundary
hypersurface $\bhs{N'N}$ of $N.$  By Remark \ref{rm:match}, these
normal operators fit together smoothly since at each intersection
$\bhsd{\phi\phi}(N) \cap \bhsd{\phi\phi}(N')$ the restriction of the
integral kernel is the Green's function of the model operator induced
by $\eth_{M/B,Q}$ with its vertical APS domain.

We now proceed as in \cite[\S4]{ALMP:Hodge} and obtain a parametrix for $\eth_{M/B,Q}$ from the integral kernels $G(N).$

\begin{proposition} \label{prop:parametrix}
Let $\eth_{M/B,Q}$ be a Dirac-type wedge operator endowed with its vertical APS domain and satisfying the Witt assumption.
There is an edge pseudodifferential operator in the calculus with bounds
\begin{equation*}
	G(M) \in \rho_{\ff((M/B)^2_\e)}\Psi^{-1, \mf I^{(2)}}_\e(M/B;E)
\end{equation*}
such that $\cN_y(G(M)\eth_{M/B,Q}) = \Id$ for each $N \in \cS_{\psi}(M),$ $y \in N^{\circ},$ and hence
\begin{equation}\label{eq:TwoSided}
	G(M)\eth_{M/B,Q} - \Id, 
	\quad
	\eth_{M/B,Q}G(M) - \Id 
	\in 
	\rho_{\ff((M/B)^2_\e)}\Psi^{-1, \mf I^{(2)}}_\e(M/B;E).
\end{equation}
\end{proposition}

\begin{proof}
At each $N \in \cS_{\psi}(M),$ recall that $\cN_y(\eth_{M/B,Q})$ is not the restriction to $\bhsd{\phi\phi}(N)$ of the lift of $\eth_{M/B,Q}$ to $(M/B)^2_\e,$ because this lift is not tangent to that boundary hypersurface. Instead we have 
\begin{equation*}
	\cN_y(\eth_{M/B,Q}) 
	= \tfrac1s \Nop_y(\rho_N \eth_{M/B,Q})
\end{equation*}
(recall that $\Nop_y,$ the edge normal operator, corresponds to restriction of the integral kernel to the front face over $y$ while $\cN_y,$ the wedge normal operator, is obtained by this equality and uses our fixed choice of boundary product structure).
Thus in order to have $\cN_y(G(N)\eth_{M/B,Q})=\Nop_y(G(N)\eth_{M/B,Q})=\Id,$ we see that near $\bhsd{\phi\phi}(N)$ we should first extend 
$
G(N)\tfrac1s$ 
off of $\bhsd{\phi\phi}(N)$ and then multiply on the right by $\rho_N$ to obtain $\wt G(N).$ This will then satisfy
\begin{equation*}
	\Nop_y(\wt G(N)\eth_{M/B,Q}) = G(N)\tfrac1s \Nop_y(\rho_N\eth_{M/B,Q}) = \Id.
\end{equation*}
As mentioned above, since the restriction to $\bhsd{\phi\phi}(N)$ of $\wt G(N)\eth_{M/B,Q}$ is the kernel of the identity, we can patch together the $\wt G(N)$ for each $\bhsd{\phi\phi}(N)$ as they will match at corners.
We proceed as in \cite[Proof of Proposition 5.43]{tapsit} to choose an extension of these kernels to obtain $G(M).$ 
Note that without this extension we would expect $\eth_{M/B,Q}G(M)$ to be $\cO(\rho_{\bhsd{10}(N)}^{-1})$ at each $\bhsd{10}(N),$ but by extending carefully off of the front face this singularity is avoided.
This deals with the first operator in \eqref{eq:TwoSided}, for the second we point out that
\begin{equation*}
	\Nop_y(\eth_{M/B,Q}G(M)) = \Nop_y(\rho_N \eth_{M/B,Q})\Nop_y(G\rho_N^{-1})
	= s \cN_y(\eth_{M/B,Q})G_y(N)\tfrac1s
	= s \Id \tfrac1s = \Id.
\end{equation*}

By construction, $G(M)$ has weight 
$\min \{\Spec_b(I_y(\eth_{M/B,Q}))\cap \bbR^+: y \in N\}$ at each of $\bhsd{10}(N),$ $\bhsd{01}(N)$ for each $N \in \cS_{\psi}(M).$ It has a smooth expansion at each $\bhsd{\phi\phi}(N),$ where it also vanishes to first order.
\end{proof}

We now have all the tools we need to proof Theorem
  \ref{thm:FredholmResolvent}.
  \begin{proof}[{\bf Proof of  Theorem \ref{thm:FredholmResolvent}}]
To simplify notation we assume that $B = \pt, \ M = X$.

    The equation $\langle \eth_{X, Q}\phi, \psi \rangle_{L^2_\w} =
    \langle \phi, \eth_{X, Q}\psi \rangle_{L^2_\w}$ holds for
    $C_c^\infty$ sections and both sides define continuous bilinear
    forms on $\rho_X^{1/2}H^1_\e(X;E)\cap \cD_{\max}(\eth_{X,Q})$.  Thus
    this is a symmetric domain
    for $\eth_{X, Q}$ in $L^2_\w$.

    Next note that the mapping properties in Theorem
    $\ref{thm:EdgeSobAct}$ show that 
$$
G(X),\ R = \eth_{X, Q} \circ G(X) - \Id \colon L^2_\w = H^0_\e(X; E)
    \lra \rho_X^{1/2} H^1_\e(X; E),
$$
are bounded operators, as are $G(X)^*$ and $R^*$ (since as is well-known their integral kernels are obtained from those of $G(X)$ and $R$ by interchanging the two factors of $X^2$).
Taking adjoints and specifying domains when necessary for clarity,
$$
  (\eth_{X, Q}, \cD_{\VAPS}) \circ G(X) = \Id - R \implies  
( \eth_{X, Q} \circ G(X))^*  = \Id - R^*,
$$
and since $G(X)^* \circ (\eth_{X, Q},\cD_{\VAPS})^* \subseteq ( \eth_{X, Q} \circ G(X))^*,$
we have that
\begin{multline*}
	u \in \cD \lrpar{ (\eth_{X, Q},\cD_{\VAPS})^* } \\
	\implies
	u = G(X)^*\eth_{X, Q}^*u +R^*u \in \cD_{\max}(\eth_{X,Q})\cap \rho_X^{1/2} H^1_\e(X; E)
	\subseteq \cD_{\VAPS}(\eth_{X,Q})
\end{multline*}
so 
$$
\cD_{\VAPS}(\eth_{X,Q})^* \subset \cD_{\max}(\eth_{X,Q})\cap \rho_X^{1/2} H^1_\e(X; E)\subset \cD_{\VAPS}(\eth_{X,Q}),
$$
forcing equality of these three spaces.

Since 
$\rho_X^{1/2} H^1_\e(X; E)$ is compactly contained in $L^2_\w$, the resolvent and
the errors $G(X) \eth_{X, Q} - \Id$ and $\eth_{X, Q} G(X) 
- \Id$ are compact on $L^2_\w$, so the Fredholm and discrete
spectrum properties follow. 

Now note that, since $G(X)$ is a compact operator, it is simultaneously a parametrix for $(\eth_{X,Q}-\lambda)$ for all $\lambda.$ It follows that 
the space of eigensections of a given eigenvalue is an element of $\rho_{\ff(X^2_\e)}\Psi^{-1, \mf I^{(2)}}_\e(X;E),$ and hence eigensections are in $\rho_X^{\mf I}H^{\infty}_\e(X;E).$

For $\lambda \in \bbC$ define $R_i(\lambda)$ by
\begin{equation*}
	G(X)(\eth_{X,Q}-\lambda) = \Id - R_1(\lambda), \quad
	(\eth_{X,Q}-\lambda)G(X) = \Id - R_2(\lambda)
\end{equation*}
and note that $R_i(\lambda) \in \rho_{\ff(X^2_\e)}\Psi^{-1, \mf I^{(2)}}_\e(X;E).$ 
These errors can be improved to 
$R_i(\lambda) \in \rho_{\ff(X^2_\e)}\Psi^{-\infty, \mf I^{(2)}}_\e(X;E)$ by the standard symbolic construction. 
For each $\lambda \in \bbC\setminus \Spec(\eth_X, \cD_{\VAPS})$ we have (cf. \cite[(4.25)]{Mazzeo:Edge})
\begin{equation*}
	(\eth_{X,Q}-\lambda)^{-1} 
	= G(X) + R_1(\lambda)(\eth_{X,Q}-\lambda)^{-1}R_2(\lambda)
	+ R_1(\lambda)G(X)
\end{equation*}
which, by virtue of Theorems \ref{thm:CompositionEdge} and \ref{thm:BiIdeal},
is an element of $\rho_{\ff(X^2_\e)}\Psi^{-1, \mf I^{(2)}}_\e(X;E).$ (As in the proof of \cite[Theorem 4.20]{Mazzeo:Edge}, the weights at the various faces of $X^2_\e$ follow from the fact that this inverts $(\eth_{X,Q}-\lambda).$) This formula also shows that the resolvent is holomorphic as a map from $\bbC\setminus \Spec(\eth_{X,Q}, \cD_{\VAPS})$ into $\rho_{\ff(X^2_\e)}\Psi^{-1, \mf I^{(2)}}_e(X;E)$ and analytic Fredholm theory shows that it extends to a meromorphic function on $\bbC.$

Writing the projection onto the $\lambda$-eigenspace as a contour integral, together with elliptic regularity, then proves \eqref{eq:EigProjSpace} and completes the induction and the proof of Theorem \ref{thm:FredholmResolvent}.
\end{proof}

\begin{remark}
It is easy to see from this construction that every family of wedge Dirac-type operators whose vertical APS domain satisfies the Witt condition can be connected smoothly through wedge Dirac-type operators to a family whose vertical APS domains satisfy the geometric Witt condition.

Indeed recall, e.g., from \cite[\S1.4]{Hitchin:Harm}, \cite[\S A.2]{Vaillant}, that for a conformal change of metric $g_\omega = \omega^2 g$ there is a Clifford bundle adapted to this metric with Dirac-type operator $D_{\omega} = \omega^{-1} (\omega^{-(n-1)/2}D\omega^{(n-1)/2}).$
Thus if we scale each of the metrics $g_Z$ we can vary the operators, through Dirac-type operators, and push away any small indicial roots while maintaining the Witt condition. We lose special structures, e.g., this variation will take the signature operator through Dirac-type operators not equal to the signature operators of the varying metrics.
This yields a family over $B \times [0,1]$ with the original family at $B \times \{0\},$ and such that the family over $B \times \{1\}$ satisfies the geometric Witt condition.
The wedge Dirac-type operators over $B \times [0,1]$ with their vertical APS domains form a smooth family of Fredholm operators. (For a discussion of smoothness of this family of operators, see \cite{Melrose-Piazza:Even}; the vertical APS domain corresponds to the spectral section coming from a spectral gap at zero.)
\end{remark}

\subsection{Heat kernel of $\eth_{M/B,Q}^2$} \label{subsec:HeatKer}
The construction of the heat kernel proceeds by solving model problems at the critical boundary hypersurfaces of the wedge heat space. In Proposition \ref{prop:WedgeHeatNormal} we have described the solution of the model problem at each $\bhsh{\phi\phi,1}(N),$ $N \in \cS_{\psi}(M).$

We can similarly solve the model problem at $\bhsh{dd,1}.$
This proceeds exactly as in, e.g., \cite[Chapter 7]{tapsit}, \cite{Albin-Mazzeo}.
The result is naturally compatible with $e^{-t\cN_y(\eth_{M/B,Q}^2)}$ at $\bhsh{\phi\phi,1}(N)\cap \bhsh{dd,1}$ and combining these we find a conormal density $H_{1,1}$ with Schwartz kernel
\begin{equation*}
	H_{1,1}\in 
	\sB^{\cH/\mf I^{(H)}}_{phg}\sA^{-m-1}_-
	(H(M/B)_{\w}; \Hom(E)\otimes \Omega_{\mf h,R})
\end{equation*}
with $\cH$ the index set from \eqref{eq:HeatIndSets}
and
\begin{multline*}
	\beta_{(H),L}^*(t(\pa_t + \eth_{M/B,Q}^2))H_{1,1} \\
	\in
	\rho_{\lf(H(M/B)_{\w})}^{-1}
	\rho_{\ef(H(M/B)_{\w})}^{-1}
	\rho_{\ff(H(M/B)_{\w})}
	\rho_{dd,1}
	\sB^{\cH/\mf I^{(H)}}_{phg}\sA^{-m-1}_-(H(M/B)_{\w}; \Hom(E)\otimes \Omega_{\mf h, R}).
\end{multline*}
Indeed, the vanishing at $\ff(H(M/B)_{\w})$ and $\bhsh{dd,1}$ comes from solving the model problems at these faces, while the singular power of $\rho_{\lf(H(M/B)_{\w})}$ comes from the singularity of $\eth_{M/B,Q}$ in $x$ as $x \to 0,$ at each boundary hypersurface. This singular term would {\em a priori} be of order $-2,$ but proceeding as in \cite[Proof of Proposition 5.43]{tapsit} the extension can be carried out so that the leading term vanishes.\\

We can improve this parametrix by removing the Taylor expansion at $\bhsh{dd,1}$ exactly as in \cite[Chapter 7]{tapsit}. This results in 
\begin{equation*}
\begin{gathered}
	H_{1,\infty}\in 
	\sB^{\cH/\mf I^{(H)}}_{phg}\sA^{-m-1}_-
	(H(M/B)_{\w}; \Hom(E)\otimes \Omega_{\mf h,R}), \\
\begin{multlined}
	\beta_{(H),L}^*(t(\pa_t + \eth_{M/B,Q}^2))H_{1,\infty} \\
	\in
	\rho_{\lf(H(M/B)_{\w})}^{-1}
	\rho_{\ef(H(M/B)_{\w})}^{-1}
	\rho_{\ff(H(M/B)_{\w})}
	\rho_{dd,1}^{\infty}
	\sB^{\cH/\mf I^{(H)}}_{phg}\sA^{-m-1}_-
	(H(M/B)_{\w}; \Hom(E)\otimes \Omega_{\mf h, R}).
\end{multlined}
\end{gathered}\end{equation*}
Similarly, we can solve away the expansion at $\ff(H(M/B)_{\w})$ by proceeding as in Proposition 7.28 of \cite{tapsit} which in this context takes the following form.

\begin{proposition}
For each $N \in \cS_{\psi}(M),$ let $\rho_{\ff(H<N)} = 1$ if $N$ is minimal in $\cS_{\psi}(M)$ and otherwise
\begin{equation*}
	\rho_{\ff(H<N)} =
	\prod_{\substack{N' \in \cS_{\psi}(M)\\ N' < N} } \rho_{\bhsh{\phi\phi,1}(N')}.
\end{equation*}
Given 
\begin{equation*}
	f \in 
	\rho_{dd,1}^{\infty}
	\rho_{\ff(H<N)}^{\infty}
	\sB^{(\cH/\mf r)|_{\bhsh{\phi\phi,1}(N)}}_{phg}
	\sA^{-m-1}_-(\bhsh{\phi\phi,1}(N);
		\Hom(E)\otimes \rho_{dd,1}^{-n-2}\beta_{(H),R}^*\Omega(M/B))
\end{equation*}
the equation
\begin{equation*}
	\beta_{(H),L}^*(t(\pa_t + \eth_{M/B,Q}^2))\rest{\bhsh{\phi\phi,1}(N)}u = f
\end{equation*}
has a unique solution
\begin{equation*}
	u \in 
	\rho_{dd,1}^{\infty}
	\rho_{\ff(H<N)}^{\infty}
	\sB^{(\cH/\mf r)|_{\bhsh{\phi\phi,1}(N)}}_{phg}
	\sA^{-m-1}_-(\bhsh{\phi\phi,1}(N);
		\Hom(E)\otimes \rho_{dd,1}^{-n-2}\beta_{(H),R}^*\Omega(M/B)).
\end{equation*}
\end{proposition}

\begin{proof}
As usual, the solution to the equation is given by
\begin{equation*}
	u(t, \zeta, \zeta') 
	= \int_0^t \int_M 
	e^{-t\cN_{\bhsh{\phi\phi,1}(N)}(\eth_{M/B,Q}^2)}(t-s, \zeta,\zeta'') 
	f(s, \zeta'',\zeta') 
	\; ds d\zeta''
\end{equation*}
but in order to understand the structure of $u$ it is best to express this in terms of pull-back and push-forward. The asymptotics of the first factor are given by Proposition \ref{prop:WedgeHeatNormal} and the composition is a particular case of Proposition \ref{prop:WedgeHeatComp} (e.g., by extending off of $\bhsh{\phi\phi,1}(N),$ composing, and then restricting back), which gives the asymptotics of the result.
\end{proof}

We can use this proposition to solve away the expansion at $\ff(H(M/B)_{\w}),$ one face at a time.
Let $\{N_1, \ldots, N_{\ell}\}$ be the list of $\cS_{\psi}(M)$ used in the construction of the heat space.
Using this proposition to solve away successive terms at $\bhsh{\phi\phi,1}(N_1)$ we can construct, for any $\ell\geq 1,$ an improved parametrix
\begin{equation*}
\begin{gathered}
	H_{\ell,\infty}^{N_1} \in 
	\sB^{\cH/\mf I^{(H)}}_{phg}\sA^{-m-1}_-(H(M/B)_{\w}; \Hom(E)\otimes \Omega_{\mf h,R}), \\
\begin{multlined}
	\beta_{(H),L}^*(t(\pa_t + \eth_{M/B,Q}^2))H_{\ell,\infty}^{N_1} 
	\in
	\rho_{\lf(H(M/B)_{\w})}^{-1}
	\rho_{\ef(H(M/B)_{\w})}^{-1}
	\rho_{\ff(H(M/B)_{\w})} \cdot \\
	\cdot \rho_{\bhs{\phi\phi,1}(N_1)}^{\ell}
	\rho_{dd,1}^{\infty}
	\sB^{\cH/\mf I^{(H)}}_{phg}\sA^{-m-1}_-(H(M/B)_{\w}; \Hom(E)\otimes \Omega_{\mf h, R}).
\end{multlined}
\end{gathered}
\end{equation*}
Asymptotically summing successive differences we can remove the error at $\bhsh{\phi\phi,1}(N)$ altogether,
\begin{equation*}
\begin{gathered}
	H_{\infty,\infty}^{N_1} \in 
	\sB^{\cH/\mf I^{(H)}}_{phg}\sA^{-m-1}_-
	(H(M/B)_{\w}; \Hom(E)\otimes \Omega_{\mf h,R}), \\
\begin{multlined}
	\beta_{(H),L}^*(t(\pa_t + \eth_{M/B,Q}^2))H_{\infty,\infty}^{N_1} 
	\in
	\rho_{\lf(H(M/B)_{\w})}^{-1}
	\rho_{\ef(H(M/B)_{\w})}^{-1}
	\rho_{\ff(H(M/B)_{\w})}\cdot \\
	\cdot \rho_{\bhs{\phi\phi,1}(N_1)}^{\infty}
	\rho_{dd,1}^{\infty}
	\sB^{\cH/\mf I^{(H)}}_{phg}\sA^{-m-1}_-
	(H(M/B)_{\w}; \Hom(E)\otimes \Omega_{\mf h, R}).
\end{multlined}
\end{gathered}
\end{equation*}
Relabeling $H_{\infty,\infty}^{N_1} = H_{1,\infty}^{N_2},$ we now proceed in the same way at $\bhsh{\phi\phi,1}(N_2).$
After carrying this out at $\bhsh{\phi\phi,1}(N_2), \ldots, \bhsh{\phi\phi,1}(N_{\ell}),$ we end up with
\begin{equation*}
\begin{gathered}
	H_{\infty,\infty} \in 
	\sB^{\cH/\mf I^{(H)}}_{phg}\sA^{-m-1}_-
	(H(M/B)_{\w}; \Hom(E)\otimes \Omega_{\mf h,R}), \\
\begin{multlined}
	\beta_{(H),L}^*(t(\pa_t + \eth_{M/B,Q}^2))H_{\infty,\infty} \\
	\in
	\rho_{\lf(H(M/B)_{\w})}^{-1}
	\rho_{\ef(H(M/B)_{\w})}^{-1}
	\rho_{\ff(H(M/B)_{\w})}^{\infty}
	\rho_{dd,1}^{\infty}
	\sB^{\cH/\mf I^{(H)}}_{phg}\sA^{-m-1}_-(H(M/B)_{\w}; \Hom(E)\otimes \Omega_{\mf h, R}).
\end{multlined}
\end{gathered}
\end{equation*}

Note that the error now vanishes to infinite order at all boundary hypersurfaces lying over $\{t=0\},$ so we can just as well view it as a distribution on a simpler space, $\bbR^+ \times (M/B)^2_{\e},$
\begin{multline*}
	\pi_L^*(t\pa_t + t\eth_{M/B,Q}^2)H_{\infty,\infty} \\
	\in 
	\rho_{\lf(H(M/B)_{\w})}^{-1}
	\rho_{\ef(H(M/B)_{\w})}^{-1}
	t^{\infty}
	\sB^{\cH/\mf I^{(H)}}_{phg}\sA^{-m-1}_-
		(\bbR^+ \times (M/B)^2_\e;\Hom(E) \otimes \pi_R^*\Omega(M/B)).
\end{multline*}

A natural next step (as in \cite[Proposition 7.17]{tapsit}) is to interpret the heat kernel as an operator with respect to convolution in $t$, so that the error term satisfies
\begin{multline*}
	\pi_L^*(\pa_t + \eth_{M/B,Q}^2)H_{\infty,\infty} = \Id - A \\
	\Mwith
	A \in 
	\rho_{\lf(H(M/B)_{\w})}^{-1}
	\rho_{\ef(H(M/B)_{\w})}^{-1}t^{\infty}
	\sB^{\cH/\mf I^{(H)}}_{phg}\sA^{-m-1}_-
		(\bbR^+ \times (M/B)^2_\e;\Hom(E) \otimes \pi_R^*\Omega(M/B))
\end{multline*}
and use a Volterra series to invert $(\Id-A).$ 

\begin{proposition}
Let $A$ be the operator above, then 
\begin{multline*}
	(\Id-A)^{-1} = \Id  + S, \\
	\Mwith
	S \in 
	\rho_{\lf(H(M/B)_{\w})}^{-1}
	\rho_{\ef(H(M/B)_{\w})}^{-1}
	t^{\infty}
	\sB^{\cH/\mf I^{(H)}}_{phg}\sA^{-m-1}_-
	(\bbR^+ \times (M/B)^2_\e;\Hom(E) \otimes \pi_R^*\Omega(M/B)).
\end{multline*}
\end{proposition}

\begin{proof}
Fix $t_0>0.$
Let us write the Schwartz kernel of $A$ as $\cK_A (\rho_{\lf(H(M/B)_{\w})}\rho_{\ef(H(M/B)_{\w})})^{-1+\eps}\mu_R$ for any fixed
\begin{equation*}
	\eps \in (0, \min_{N \in \cS_{\psi}(M)} \mf I(\bhs{N}) ),
\end{equation*}
and point out that $|\cK_A|\rest{t\leq t_0}$ is uniformly bounded on $(M/B)^2_{\e},$ say by $C,$ and that 
\begin{equation*}
	V_{\eps} = \max_B\int_{M/B} \rho_{M}^{-1+\eps}\mu <\infty.
\end{equation*}
During this proof, let us write $\rho_{\lf(H(M/B)_{\w})}\rho_{\ef(H(M/B)_{\w})}$ as $x.$\\

If we similarly write the Schwartz kernel of $A^k$ as $\cK_{A^k} x^{-1+\eps}\mu_R$ then assuming that $|\cK_{A^k}|\rest{t\leq t_0}$ is bounded by $C_k t_0^k/k!$ we see from
\begin{equation*}
	\cK_{A^{k+1}}(t,\zeta, \zeta'') x^{-1+\eps} \mu_{\zeta''}
	= \int_0^t \int_{\zeta' \in \psi^{-1}(\psi(\zeta))}
	(\cK_{A^k}(s, \zeta, \zeta') x^{-1+\eps} \mu_{\zeta'} )
	(\cK_{A}(t-s, \zeta', \zeta'') (x')^{-1+\eps} \mu_{\zeta''} ) \; ds
\end{equation*}
that
\begin{equation*}
	|\cK_{A^{k+1}}|\rest{t \leq t_0} \leq
	\int_0^{t_0} \int_{\zeta' \in \psi^{-1}(\psi(\zeta))}
	(C_k s^k/k!) C (x')^{-1+\eps}\mu_{\zeta'}\; ds
	\leq C C_k V_{\eps} t_0^{k+1}/(k+1)!
\end{equation*}
Hence we see that the Volterra series $\sum A^k$ converges uniformly for $t\leq t_0$ and arbitrary $t_0.$

We may run the same argument after differentiating by any vector field on $(M/B)^2_\e\times \bbR^+$ that is tangent to the boundary hypersurfaces, and so we can conclude that the Volterra series converges in the space of conormal sections of $\Hom(E).$
\end{proof}

\begin{theorem}\label{thm:WedgeHeatKer'}
Let $\eth_{M/B,Q}$ be a family of  compatibly perturbed Dirac-type wedge operators acting on a Clifford bundle $E$ on a family of manifolds with corners and iterated fibration structures, $M \xlra\psi B.$
The heat kernel of $\eth_{M/B,Q}^2$ satisfies
\begin{equation*}
	e^{-t\eth_{M/B,Q}^2} \in
	\sB^{\cH/\mf I^{(H)}}_{phg}\sA^{-m-1}_-(H(M/B)_{\w}; \Hom(E)\otimes \Omega_{\mf h,R})
\end{equation*}
where $\mf I^{(H)}$ and $\cH$ are given by \eqref{eq:IndicialMwH}, \eqref{eq:HeatIndSets} and $\Omega_{\mf h, R}$ is the density bundle from \eqref{eq:HeatDensityBdle}.
The leading terms at $\bhsh{dd,1}$ and each $\bhsh{\phi\phi,1}(N)$ are given by
\begin{equation*}
\begin{gathered}
	\cN_{\bhsh{dd,1}}(e^{-t\eth_{M/B,Q}^2}) = e^{-\Delta_{T(M/B)/M}} \\
	\cN_{\bhsh{\phi\phi,1}(N)}(e^{-t\eth_{M/B,Q}^2}) 
		= e^{-\sigma^2 \Delta_{T(N/B)/N}} e^{-\sigma^2 D_{C(\phi_N)/N,Q}^2}
\end{gathered}
\end{equation*}
where $e^{-\Delta_{T(M/B)/M}}$ denotes the Euclidean heat kernel on the fibers of $T(M/B)$ (at time one) and similarly $e^{-\sigma^2 \Delta_{T(N/B)/N}},$
$C(\phi_N)$ denotes the mapping cylinder of $\phi_N,$
\begin{equation}\label{eq:MapCyl}
	Z^+ \fib C(\phi_N) \xlra{ \phi_N^+ } N
\end{equation}
and $D_{C(\phi_N)/N,Q}$ is the family of operators $N \ni y \mapsto D_{Z_y^+,Q}$ from \eqref{eq:CDZ}.
\end{theorem}

\begin{proof}
The operator $H_{\infty,\infty}(\Id-S)$ satisfies the wedge surgery heat equation with initial condition given by the (lift of the) identity since
\begin{equation*}
	\beta_{H,L}(\pa_t + \eth_{M/B,Q}^2)(G_{\infty}(\Id-S)) = (\Id-A)(\Id-S) = \Id.
\end{equation*}
The composition result, Proposition \ref{prop:WedgeHeatComp}, yields the asymptotics of this composition.
The leading terms are immediate from the construction above.
\end{proof}

\section{Getzler rescaling and the trace of the heat kernel} \label{sec:Getzler}

\subsection{Getzler rescaling}

The heat kernel construction above did not significantly use that
$\eth_{M/B}^2$ is the square of a Dirac-type operator rather than an
arbitrary Laplace-type operator. 
We now refine the construction to
take advantage of the Clifford action on $E$ and its compatibility
with $\eth_{M/B,Q}.$ Specifically, we proceed as in \cite[Chapter
8]{tapsit} to carry out Getzler rescaling geometrically (see also
\cite{Dai-Melrose, Vaillant, Albin-Rochon:DFam})
by constructing `rescaled homomorphism bundles' on the resolved
    heat space which capture the relationship between
    the heat kernel and the Clifford action. The second property
in Definition \ref{def:CompPert} ensures that compatible perturbations
will not affect the discussion. After carrying out the rescaling for
the family of Dirac-type operators we carry out the analogous
rescaling for the Bismut superconnection.\\

Recall the decomposition of the homomorphism bundle of $E$ from\eqref{eq:HomClif},
\begin{equation*}
	\hom(E) \cong \Cl({}^{\w}T^*M/B) \otimes \hom'_{\Cl({}^{\w}T^*M/B)}(E).
\end{equation*}
The heat kernel is a section of $\Hom(E) \lra H(M/B)_{\w}$, the lift of $\pi_L^*E^* \otimes \pi_R^*E$ from $M \times_{\psi} M$ to $H(M/B)_{\w}.$
The restriction of this bundle to $\diag_M$ is (canonically isomorphic to) $\hom(E)$ and hence inherits the decomposition above.
In this way we see that $\Hom(E)\lra H(M/B)_{\w}$ has compatible
filtrations 
at $\bhsh{dd,1}$ and each $\bhsh{\phi\phi,1}(N),$
\begin{equation*}
\begin{gathered}
	\Hom(E)\rest{\bhsh{dd,1}} 
		= \Cl^*({}^{\w}T^*M/B) \otimes \Hom'_{\Cl({}^{\w}T^*M/B)}(E), \\
	\Hom(E)\rest{\bhsh{\phi\phi,1}(N)} 
		= \Cl^*({}^{\w} T^*N/B) \otimes \Hom'_{\Cl({}^{\w}T^*N/B)}(E),
\end{gathered}\end{equation*}
where $\Hom'_{\Cl^*({}^{\w}T^*M/B)}(E)$ denotes the elements of $\Hom(E)$ that commute with $\Cl({}^{\w}T^*M/B),$ and similarly $\Hom'_{\Cl({}^{\w}T^*N/B)}(E).$
In fact we can extend the filtration of $\Hom(E)\rest{\bhs{\phi\phi,1}(N)}$ even further by including Clifford multiplication by $d\rho_{\bhs{N}},$ but it will be convenient not to do so.
It is easy to see from \eqref{eq:bfsDiag} that these filtrations are compatible.

We define a connection on $\Hom(E) \lra H(M/B)_{\w}$ by
\begin{equation*}
	\nabla^{\Hom(E)} 
		= \pa_t \; dt \otimes 
		\beta_L^*\nabla^{E^*} \otimes \beta_R^*\nabla^{E}
\end{equation*}
and then choose
vector fields $\nu$ and, for each $N \in \cS_{\psi}(M),$ $\nu_N$
transverse to $\bhsh{dd,1}$ and $\bhsh{\phi\phi,1}(N),$ respectively,
and tangent to all other boundary hypersurfaces (e.g., by modifying
the vector fields in a boundary product structure as in
\S\ref{sec:IFS}). 
We define the space of rescaled sections of $\Hom(\bbE)$ by 
\begin{multline*}
	\Gamma = \Big\{ s \in \CI(H(M/B)_{\w};\Hom(E)) : \\
	\Mfor j \in \{ 0, \ldots, \dim M/B\}, \quad
	\lrpar{ \nabla^{\Hom(\bbE)}_\nu }^j s \rest{\bhsh{dd,1}} 
		\in  \Cl^j({}^{\w}T^*M/B) \otimes \Hom'_{\Cl({}^{\w}T^*M/B)}(E), \\
	\text{ for each } N \in \cS_{\psi}(M) \Mand  k \in \{ 0, \ldots, \dim N/B\},   \\
	 \lrpar{ \nabla^{\Hom(\bbE)}_{\nu_N} }^k s \rest{\bhsh{\phi\phi,1}(N)} 
	 	\in  \Cl^k({}^{\w}T^*N/B) \otimes \Hom'_{\Cl({}^{\w}T^*N/B)}(E) \Big\}.
\end{multline*}
As in \cite[Chapter 8]{tapsit}, there is a vector bundle $\Hom_G(E)
\lra H(M/B)_{\w}$ with a bundle map $j:\Hom_G(E) \lra \Hom(\bbE)$ such
that 
\begin{equation*}
	j\circ \CI(H(M/B)_{\w}; \Hom_G(E)) = \Gamma \subseteq \Hom(E).
\end{equation*}
We will refine the heat kernel construction to show that
$\exp\lrpar{-t \eth_{M/B,Q}^2}$ is a section of $\Hom_G(E).$\\

Given a vector field $W$ on $H(M/B)_{\w}$ it will be useful to know
when $\nabla^{\Hom(E)}_W$ preserves $\cD.$ 
As shown in \cite[Proposition 8.12]{tapsit} directly from the
definition of $\cD,$ this requires $\nabla^{\Hom(E)}_W$ to preserve
the filtrations at each boundary hypersurface, for the curvature
$K^{\Hom(E)}(\wt\nu, W)$ to increase filtration degree by at most one
(where $\wt\nu$ is the appropriate transverse vector field), and for 
$(\nabla^{\Hom(E)}_{\wt\nu})^j(K^{\Hom(E)}(\wt\nu,W))$ to increase filtration
degree by at most two for all $j \ge 1.$ 

\begin{lemma}
If $W \in \CI(H(M/B)_{\w};T(H(M/B)_{\w}))$ is tangent to the fibers of the fiber bundles
\begin{equation*}
	\bhsh{dd,1} \lra \diag_M \Mand \bhsh{\phi\phi,1}(N) \lra N \Mforall N \in \cS_{\psi}(M),
\end{equation*}
then $\nabla^{\Hom(E)}_W$ acts on sections of $\Hom_G(E).$

In particular, if $W'$ is any vector field on $\bbR^+ \times M \times_{\psi} M,$ vertical with respect to the projection onto $B,$  then $\nabla^{\Hom(E)}_{\beta^*(\sqrt t W')}$ acts on sections of $\Hom_G(E).$
\end{lemma}

\begin{proof}
Since $\nabla^E$ is a Clifford connection, it satisfies
\begin{equation*}
	\lrspar{\nabla^E, \cl(\theta)} = \cl(\nabla\theta).
\end{equation*}
Hence it is immediate that $\nabla^{\Hom(E)}$ preserves the filtration at $\bhsh{dd,1},$ while checking that it preserves the filtration at $\bhsh{\phi\phi,1}(N)$ comes down to checking that in a local frame as in \eqref{eq:local-frame}
\begin{equation*}
	g(\nabla_{W_1}W_2, \wt U) = \cO(x) \Mforall W_1 \in \cV, W_2 \in \{ \pa_x, \tfrac1x V\},
\end{equation*}
which follows from \eqref{eq:AsympSplittingConn}.
(Incidentally, this is why we do not rescale at $\bhsh{\phi\phi,1}(N)$ by $\Cl(T^*(N/B)^+),$ as the connection would not preserve this filtration.)

Next recall that 
\begin{equation*}
	K^{\Hom(E)}(W_1, W_2) 
	= K^E((\beta_L)_*W_1, (\beta_L)_*W_2) \circ \cdot
	- \cdot \circ K^E((\beta_R)_*W_1, (\beta_R)_*W_2)
\end{equation*}
and that, since $\nabla^E$ is a Clifford connection,
\begin{equation*}
	K^E(S_1, S_2) = \tfrac14 \cl(R(S_1, S_2)) + K^{E'}(S_1, S_2), 
	\Mwith K^{E'}(S_1, S_2) \in \CI(M,\hom_{\Cl({}^{\w}T^*M/B)}(E)).
\end{equation*}
Thus the covariant derivatives of the curvature of $\Hom(E)$ involve at most two Clifford multiplications and hence can move the filtrations at $\bhsh{dd,1}$ and each $\bhsh{\phi\phi,1}(N)$ by at most two.

It follows that $\nabla^{\Hom(E)}_W$ will act on $\cD$ whenever
\begin{equation*}
\begin{gathered}
	R((\beta_{\cdot})_*\nu\rest{\bhsh{dd,1}}, (\beta_{\cdot})_*W\rest{\bhsh{dd,1}}) = 0, \Mand \\
	R((\beta_{\cdot})_*\nu_N\rest{\bhsh{\phi\phi,1}(N)}, 
		(\beta_{\cdot})_*W\rest{\bhsh{\phi\phi,1}(N)})\rest{\bhs{N}} 
		\in \CI(\bhs{N},  \Lambda^1(T^*N/B) \wedge \Lambda^1(\ang{dx}\oplus x T^*\bhs{N}/N) )
\end{gathered}
\end{equation*}
where $\beta_{\cdot}$ can be either $\beta_L$ or $\beta_R.$
If $W\rest{\bhsh{dd,1}}$ is tangent to the fibers of $\bhsh{dd,1}\lra \diag_M$ then $(\beta_{\cdot})_*W\rest{\bhsh{dd,1}}=0,$ so the first condition holds.
From Proposition \ref{prop:curvature} we know that the second condition will hold as long as $(\beta_{\cdot})_*W\rest{\bhsh{\phi\phi,1}(N)}$ is tangent to the fibers over $N^\circ.$
\end{proof}

Now that we know that the connection acts on $\cD,$ we can use the
Lichnerowicz formula (see, e.g., \cite[\S8.8]{tapsit})
\begin{equation}\label{eq:DiracLich}
	\eth^2_{M/B} 
	= \Delta^{M/B} + \tfrac14\scal_M + \tfrac12\sum_{a,b} K_E'(e_a, e_b)\cl(e^a)\cl(e^b)
\end{equation}
in which $\Delta^{M/B}$ is the Bochner Laplacian of $\nabla^E$ and the sum runs over an orthonormal frame of $TM/B,$
to see that $t\eth^2_{M/B}$ acts on sections of $\Hom_G(E).$

For vector fields as in \eqref{eq:local-frame},
we can read off from \cite[(8.36)]{tapsit} and Proposition \ref{prop:curvature} the rescaled normal operators, for each $N \in \cS_{\psi}(M)$ we have
\begin{equation*}
	\cN_{\bhsh{\phi\phi,1}(N)}^G(\nabla^E_{\tau\pa_x}) = \sigma\pa_s, \quad
	\cN_{\bhsh{\phi\phi,1}(N)}^G(\nabla^E_{\tau\wt {\pa_{y_i}}}) 
	= \sigma(\pa_{u_i} + \tfrac14 \df e\lrpar{ R^{N/B}(\cR, \pa_{y_i}) } ),
\end{equation*}
where $\cR$ denotes the radial vector field in ${}^{\e}TN/B$ and the appearance of the curvature $R^{N/B}$ can be traced back to Proposition \ref{prop:curvature} (1), and 
\begin{multline*}
	\cN_{\bhsh{\phi\phi,1}(N)}^G(\nabla^E_{\tau \tfrac1x V}) 
	= \frac{\sigma}s\Big(\nabla^{E, \pa M/Y}_V 
	+ \tfrac14\cl\lrpar{ R(\pa_x, V) }
	+ \tfrac12\cl\lrpar{ \nabla_{\pa_x}R(\pa_x, V) } \Big)\\
	= \frac{\sigma}s\Big( \nabla^{E, \pa M/Y}_V 
	+ \tfrac14\sum_{\wt U, V'} g_{N/B}(\cS^{\phi_N}(V, V'),\wt U) 
		\cl((\tfrac1xV')^{\flat})\df e(\wt U^{\flat}) \\
	+ \tfrac12\sum_{\wt U, \wt U'} g_{\bhs{N}/N}(\cR^{\phi_N}(\wt U, \wt U'), V) 
		\df e ( (\wt U)^{\flat} \wedge (\wt U')^{\flat} ) \Big)
\end{multline*}
where the appearance of $\cS^{\phi_N},$ $\cR^{\phi_N}$ traces back to
(2)-(3) in Proposition \ref{prop:curvature} 
We have similar behavior at $\bhsh{dd,1},$ save that all vector fields are horizontal.

Combining this with the Lichnerowicz formula we obtain their rescaled normal operator.

\begin{lemma}
Let $\eth_{M/B}$ be a family of Dirac-type wedge operators on the fibers of $M \xlra{\psi} B.$
The rescaled normal operators of $\tau^2\eth^2_{M/B}$ are
\begin{equation*}
\begin{gathered}
	\cN_{\bhsh{dd,1}}^G(\tau^2\eth_{M/B}^2) 
		= -\sum_{\ang{\pa_j} = {}^\e TM/B} (\pa_j + \tfrac14 \df e(R(\pa_j, \cR_{{}^{\e}TM/B})))^2 + \df e(K_E')\\
	\cN_{\bhsh{\phi\phi,1}(N)}^G(\tau^2\eth_{M/B}^2) 
		= -\sigma^2\sum_{\ang{\pa_j} = {}^{\e}TN/B} (\pa_j + \tfrac14 \df e(R(\pa_j, \cR_{{}^{\e}TN/B})))^2 
		+ \sigma^2 \bbB^2_{C(\phi_N)/N}
\end{gathered}
\end{equation*}
where $\bbB_{C(\phi_N)/N}$ assigns to each $b \in B$ the Bismut
superconnection on the mapping cylinder of $\phi_{N_b} (=\phi_N \rvert_{\phi_N^{-1}(N_b)}),$ $Z^+ \fib
C(\phi_{N_b}) \lra N_b.$ 
\end{lemma}

(Note that the rescaled normal operator of $\tau^2\eth^2_{M/B}$ corresponding to $N \in \cS_{\psi}(M)$ is a family of superconnections; for the rescaled superconnection of $\tau^2\bbA^2_{M/B}$ we will instead obtain the superconnection of a family.) 

\begin{proof}
For the rescaled normal operator at $\bhsd{\phi\phi,1}(N)$ for some $N \in \cS_{\psi}(M),$ we start by considering $\tau^2$ times the Bochner Laplacian in \eqref{eq:DiracLich},
\begin{equation*}
	\tau^2\Delta^{M/B} = -\sum_{e_a} (\nabla^E_{\tau e_a})^*(\nabla^E_{\tau e_a})
\end{equation*}
where the sum runs over an orthonormal frame of $TM/B.$ We can write this as a sum over a frame of ${}^{\w}TN/B$ plus a sum over a frame of the orthogonal complement; the former  has rescaled normal operator equal to the harmonic oscillator $-\sigma^2\sum_{\ang{\pa_j} = {}^{\e}TN/B} (\pa_j + \tfrac14 \df e(R(\pa_j, \cR_{{}^{\e}TN/B})))^2,$ 
while the latter has rescaled normal operator equal to $\sigma^2$ times the Bochner Laplacian term in the Lichnerowicz formula \eqref{eq:SupLichnerowicz} for the Bismut superconnection of $C(\phi_N)/N.$ The twisting curvature term in \eqref{eq:DiracLich} gives rise to the twisting curvature term in \eqref{eq:SupLichnerowicz} and similarly for the scalar curvature terms, since, at each $\bhs{N},$
$\scal(X,g_{\w}) \sim \rho_N^{-2}\scal( C(Z), d\rho_N^2+\rho_N^2g_Z) +
\cO(\rho_N^{-1}).$  

The rescaled normal operator at $\bhsd{dd,1}$ is similar but simpler.
\end{proof}

This same analysis applies to the Bismut superconnection $\bbA_{M/B}$ from Section \ref{sec:BismutSup}.
Indeed, one can either repeat the analysis above or consider the parameter $\eps$ from  section \ref{sec:BismutSup} and note that these results for $\eps>0$ imply the analogous results for $\eps=0.$
First, if $Q$ is a compatible perturbation of $\eth_{M/B}$ then let us define
\begin{equation*}
	\bbA_{M/B,Q} = \bbA_{M/B}+Q_{M/B}, \quad
	\bbA_{C(\phi_N)/N,Q} = \bbA_{C(\phi_N)/N} + Q_{\bhs{N}/N}
\end{equation*}
where $Q_{\bhs{N}/N}$ is the family $N \ni y \mapsto Q_{Z_y}.$

In this case the bundle $E$ is replaced by $\bbE = \psi^*\Lambda^*T^*B \otimes E,$ the connection $\nabla^E$ by $\nabla^{\bbE,0}$ and the Clifford action by $\cl_0.$ 
For the bundle $\bbE,$ we set 
\begin{equation*}
	\Hom(\bbE) = \psi^*\Lambda^*T^*B\otimes \Hom(E)
\end{equation*}
and 
\begin{equation*}
\begin{gathered}
	\Cl^*_0({}^{\w}T^*M)
	= \bbC \otimes \mathrm{Cl}(\psi^*T^*B \oplus {}^{\w}T^*M/B, g_{M,0}) \\
	\Cl^*_0({}^{\w}T^*N)
	= \bbC \otimes \mathrm{Cl}(\psi^*T^*B \oplus {}^{\w}T^*N/B, g_{N,0})
\end{gathered}
\end{equation*}
so that 
\begin{equation*}
\begin{gathered}
	\Hom(\bbE)\rest{\bhsh{dd,1}} 
		= \Cl_0^*({}^{\w}T^*M) \otimes \Hom'_{\Cl({}^{\w}T^*M/B)}(E), \\
	\Hom(\bbE)\rest{\bhsh{\phi\phi,1}(N)} 
		= \Cl_0^*({}^{\w} T^*N) \otimes \Hom'_{\Cl({}^{\w}T^*N/B)}(E),
\end{gathered}\end{equation*}
for each $N \in \cS_{\psi}(M).$

We define a connection on $\Hom(\bbE) \lra H(M/B)_{\w}$ by
\begin{equation*}
	\nabla^{\Hom(\bbE)} 
		= \pa_t \; dt \otimes 
		\beta_L^*\nabla^{\bbE,0} \otimes \beta_R^*\nabla^{\bbE^*,0}
\end{equation*}
and the space of rescaled sections of $\Hom(\bbE)$ by
\begin{multline*}
	\cD = \Big\{ s \in \CI(H(M/B)_{\w};\Hom(\bbE)) : \\
	\Mfor j \in \{ 0, \ldots, \dim M\}, \quad
	\lrpar{ \nabla^{\Hom(\bbE)}_\nu }^j s \rest{\bhsh{dd,1}} 
		\in  \Cl_0^j({}^{\w}T^*M) \otimes \Hom'_{\Cl({}^{\w}T^*M/B)}(E), \\
	\text{ for each } N \in \cS_{\psi}(M) \Mand  k \in \{ 0, \ldots, \dim N\},   \\
	 \lrpar{ \nabla^{\Hom(\bbE)}_{\nu_N} }^k s \rest{\bhsh{\phi\phi,1}(N)} 
	 	\in  \Cl_0^k({}^{\w}T^*N) \otimes \Hom'_{\Cl({}^{\w}T^*N/B)}(E) \Big\}.
\end{multline*}
The corresponding rescaled bundle is denoted $\Hom_G(\bbE).$

The rescaled normal operators of $\bbA^2_{M/B}$ are similar to those of $\eth^2_{M/B}$ but valued in differential forms in $M$ instead of differential forms in $M/B.$
Indeed, the Lichnerowicz formula \eqref{eq:SupLichnerowicz} combined with the above yields the following.

\begin{lemma}\label{lem:ResNormOpsA}
Let $\eth_{M/B}$ be a family of  Dirac-type wedge operators on the fibers of $M \xlra{\psi} B,$ and let $ \bbA_{M/B}$ be a  Bismut superconnection extending $\eth_{M/B}.$
The rescaled normal operators of $\tau^2\bbA^2_{M/B}$ are
\begin{equation*}
\begin{gathered}
	\cN_{\bhsh{dd,1}}^G(\tau^2\bbA_{M/B}^2) 
		= -\sum_{\ang{\pa_j} = {}^\e TM/B} 
			(\pa_j + \tfrac14 \df e(R(\pa_j, \cR_{{}^{\e}TM/B})))^2 + \df e(K_E') = \cH_{M/B}\\
\begin{multlined}
	\cN_{\bhsh{\phi\phi,1}(N)}^G(\tau^2\bbA_{M/B}^2) 
		= -\sigma^2\sum_{\ang{\pa_j} = {}^{\e}TN/B} 
		(\pa_j + \tfrac14 \df e(R(\pa_j, \cR_{{}^{\e}TN/B})))^2 
		+ \sigma^2 \bbA^2_{C(\phi_N)/N} \\
		= \sigma^2(\cH_{N/B} + \bbA^2_{C(\phi_N)/N})
\end{multlined}
\end{gathered}
\end{equation*}
where $\bbA_{C(\phi_N)/N}$ is the induced superconnection for the family of cones given by the mapping cylinder of $\phi_N,$ $C(Z) \fib C(\phi_N) \lra N$ from \eqref{eq:MapCyl}.
\end{lemma}

\begin{remark}
Note from \eqref{eq:FibAsymp} that the unrescaled normal operators of
$\tau^2\bbA_{M/B}^2$ would involve the tensors of $\psi_N$ but not the
tensors of $\phi_N.$  
The rescaling makes explicit the contribution of
the tensors of $\phi_N$ to the expansion at $\bhsh{\phi\phi,1}(N).$
\end{remark}

\begin{theorem}\label{thm:WedgeRescaledHeatKer}
Let $\bbA_{M/B}$ be the Bismut superconnection associated to a family of perturbed Dirac-type wedge operators acting on a Clifford bundle $E$ on a family of manifolds with corners and iterated fibration structures, $M \xlra\psi B,$ and let $Q$ be a compatible perturbation.
The heat kernel of $\bbA_{M/B,Q}^2$ satisfies
\begin{equation*}
	e^{-t\bbA_{M/B,Q}^2} \in
	\sB^{\cH/\mf I^{(H)}}_{phg}\sA^{-m-1}_-
		(H(M/B)_{\w}; \Hom_G(\bbE)\otimes \Omega_{\mf h,R}).
\end{equation*}
where $\mf I^{(H)}$ and $\cH$ are given by \eqref{eq:IndicialMwH}, \eqref{eq:HeatIndSets} and $\Omega_{\mf h, R}$ is the density bundle from \eqref{eq:HeatDensityBdle}.
The leading terms at $\bhsh{dd,1}$ and each $\bhsh{\phi\phi,1}(N)$ are given by
\begin{equation*}
\begin{gathered}
	\cN^G_{\bhsh{dd,1}}(e^{-t\bbA_{M/B}^2}) = e^{-\cH_{M/B}^2} \\
	\cN^G_{\bhsh{\phi\phi,1}(N)}(e^{-t\bbA_{M/B}^2}) 
		= e^{-\sigma^2 \cH_{N/B}^2} e^{-\sigma^2 \bbA_{C(\phi_N)/N,Q}}
\end{gathered}
\end{equation*}
\end{theorem}

\begin{proof}
We proceed as in \cite[\S11]{Melrose-Piazza:Even}.

First we recall that there is an explicit formula, Mehler's formula, for the heat kernel of a harmonic oscillator such as $\cH_{M/B}^2$ or $\cH_{N/B}^2,$ see e.g., \cite[\S 8.9]{tapsit}.
Secondly we recall (e.g., \cite[Chapter 9 Appendix]{BGV2004}) that in
a situation like ours where
$\cF = H + \cF_{[+]}$ with $\cF_{[+]}$ nilpotent we have
\begin{equation*}
\begin{gathered}
	\exp(-t\cF) = \exp(-tH) + \sum (-1)^k \cI_k, \\
	\cI_k = \int_{\Delta_k} e^{-i_0 tH}\cF_{[+]}e^{-i_1 tH} \cF_{[+]} 
		\cdots e^{-i_{k-1} tH}\cF_{[+]} e^{-i_k tH} \; di_0 \cdots di_k
\end{gathered}
\end{equation*}
with the nilpotence of $\cF_{[+]}$ guaranteeing that the $\cI_k$ are eventually zero.

We apply this at $\bhsh{\phi\phi,1}(N)$ to see that heat kernel of the
rescaled normal operator is a section with the same asymptotics as
those of the heat kernel of the normal operator. (Since $\cF_{[+]}$ is
a tensor and so its integral kernel is supported on the
diagonal.)
Mehler's formula directly yields a solution of the model heat problem
at $\bhsh{dd,1}.$ These models are compatible at the corner because
the restriction to the corner solves the corresponding model problem.

The rest of the construction proceeds as in \S\ref{subsec:HeatKer}.
\end{proof}

\subsection{Trace of the heat kernel}\label{sec:Trace}

We discuss the trace of integral kernels on the heat space (extending \cite[Theorem 4.2]{Mazzeo-Vertman:AT}) and then specialize to the case of the heat kernel. \\

The trace of an operator is intimately related to the integral of its Schwartz kernel along the diagonal (see \cite{Brislawn} for a general discussion). The interior lift of the diagonal of $M$ from $M \times_{\psi} M \times \bbR^+_t$ to $H(M/B)_{\w}$ can be identified with
\begin{equation}\label{eq:LiftedDiag}
	\diag_{\w}^{(H)}(M)
	= 
	\Big[ M \times \bbR^+_{\tau}; 
	\bhs{N_1} \times \{ 0 \}; \ldots; 
	\bhs{N_{\ell}} \times \{ 0\} \Big],
\end{equation}
where $\tau = t^{1/2}$ and  $\cS_{\psi}(M)= \{N_1, \ldots, N_{\ell}\}$ is listed in a non-decreasing order. The fiber bundle $X \fib M \xlra\psi B$ induces a fiber bundle
\begin{equation*}
	\diag_{\w}^{(H)}(X) \fib 
	\diag_{\w}^{(H)}(M)
	\xlra{\psi} B,
\end{equation*}
which we continue to denote $\psi.$

We will denote the blow-down map by
\begin{equation*}
	\beta_{(\Delta)}: \diag_{\w}^{(H)}(M) \lra M \times \bbR^+_{\tau},
\end{equation*}
and the collective boundary hypersurfaces of $\diag_{\w}^{(H)}(M)$ by
\begin{equation*}
	M \times \{0 \} \leftrightarrow \bhsD{0,1} 
\end{equation*}
and, for each $N \in \cS_{\psi}(M),$
\begin{equation*}
	\bhs{N} \times \bbR^+_{\tau} \leftrightarrow \bhsD{1,0}(N), \quad
	\bhs{N} \times \{ 0\} \leftrightarrow \bhsD{1,1}(N).
\end{equation*}

Assume that the kernel of $A$ has the form $\cK_A \rho^{\mathfrak h} \mu_R$ 
with $\cK_A \in \cA_{phg}^{\cE_A}(H(M/B)_\w;\Hom(E))$ and $\mathfrak h$ the multiweight from \eqref{eq:HeatWeight}. Ultimately we are interested in kernels that are merely conormal with bounds acting on sections of a vector bundle, but the corresponding trace result will follow easily from this one.

For appropriate index sets, this kernel is trace-class and its trace is the integral of its restriction to the diagonal.
In terms of the composition of the blow-down map with the projection onto $B \times \bbR^+_{\tau},$
\begin{equation*}
	\beta_{(\Delta),t}: \diag_{\w}^{(H)}(M) \lra M \times \bbR^+_{\tau} 
		\lra B \times \bbR^+_{\tau},
\end{equation*}
this means that
\begin{equation*}
	\Tr(A) = (\beta_{(\Delta),t})_* 
	\lrpar{ \cK_A \rho^{\mathfrak h} \mu_R\rest{\diag_{\w}^{(H)}(M) } }.
\end{equation*}

\begin{theorem}\label{thm:TraceAs}
If $A$ has integral kernel 
$\cK_A \rho^{\mathfrak h} \mu_R$ 
with $\cK_A \in \cA_{phg}^{\cE_A}(H(M/B)_\w;\Hom(E))$
and
\begin{equation*}
	\Re( \cE_A(\bhsh{\phi\phi,0}(N)) -\dim(N/B) > -1 \Mforall N \in \cS_{\psi}(M)
\end{equation*}
then it is trace-class  with trace polyhomogeneous in $\tau$ satisfying
\begin{equation*}
\begin{gathered}
	\Tr A \in \cA_{phg}^{\cE_{A,\tau}}(\bbR^+_{\tau};\CI(B)), \\
	\cE_{A,\tau} = (\cE_A(\bhsh{\diag,1}) -\dim(M/B) -2) \bar\cup
	\bar\bigcup_{N \in \cS_{\psi}(M)} (\cE_A(\bhsh{\phi\phi,1}(N)) -\dim(N/B)-2)
\end{gathered}
\end{equation*}
\end{theorem}

\begin{proof}
The operator is trace-class because its integral kernel is a bounded smooth function times a measure of finite volume.

Next let us discuss the effect of restricting the kernel to the diagonal. 
Let $\mu_{\Delta} = \beta_{(\Delta)}^*\mu(M/B).$
The restriction of the weight \eqref{eq:HeatWeight} to the diagonal is 
\begin{equation*}
\begin{gathered}
	\mathfrak h^{\Delta}: \cM_1(\diag_{\w}^{(H)}) \lra \bbR, \\
	\mathfrak h^{\Delta}(H) =
	\begin{cases}
	-(\dim(N/B) + 3) & \Mif H \subseteq \bhsD{1,1}(N) 
		\Mforsome N \in \cS_{\psi}(M) \\
	-(\dim(N/B) + 1) & \Mif H \subseteq \bhsD{1,0}(N)
		\Mforsome N \in \cS_{\psi}(M) \\
	-(\dim(M/B) + 2) & \Mif H = \bhsD{0,1}(N) \\
	0 & \Motherwise
	\end{cases}
\end{gathered}
\end{equation*}
Let $\cE_A^{\Delta}$ be the index sets given by
\begin{equation*}
	\cE_A^{\Delta}(\bhsD{0,1}) = \cE_A(\bhsh{dd,1})
\end{equation*}
and, for each $N \in \cS_{\psi}(M),$
\begin{equation*}
	\cE_A^{\Delta}(\bhsD{1,0}(N)) = \cE_A(\bhsh{\phi\phi,0}(N)), \quad
	\cE_A^{\Delta}(\bhsD{1,1}(N)) = \cE_A(\bhsh{\phi\phi,1}(N)).
\end{equation*}
Then we have
\begin{equation*}
	\cK_A \rho^{\mathfrak h} \mu_R\rest{\diag_{\w}^{(H)}}
	= \kappa_A \rho^{\mathfrak h^{\Delta}}\mu_{\Delta} \Mwith
	\kappa_A \in \cA_{phg}^{\cE_A^{\Delta}}(\diag_\w^{(H)}(M)).
\end{equation*}

Next note that the map $\beta_{(\Delta),t}$ is a b-fibration which sends $\{\bhsD{1,0}(N): N \in \cS_{\psi}(M)\}$ to the interior of $\bbR^+_{\tau}$ and the other boundary hypersurfaces to $\{\tau=0\},$ so we can apply the push-forward theorem once we pass to b-densities.
In this setting, we start with
\begin{equation*}
	(\beta_{(\Delta),t})_*(\kappa_A \rho^{\mathfrak h^{\Delta}}\mu_{\Delta}) = (\Tr A),
\end{equation*}
and multiply both sides by $\tfrac{d\tau}\tau$ to obtain
\begin{equation*}
	(\beta_{(\Delta),t})_*(\kappa_A \rho^{\mathfrak h^{\Delta}}(\beta_{(\Delta),t}^*(\tfrac1\tau\mu(M/B\times\bbR^+))) 
	= (\Tr A) \tfrac{d\tau}\tau.
\end{equation*}
Now,
\begin{equation*}
\begin{gathered}
\begin{multlined}
	\beta_{(\Delta),t}^*\mu(M/B\times\bbR^+) = 
	\prod_{N \in \cS_{\psi}(M)} \rho_{\bhsD{1,1}(N)} \mu(\diag_{\w}^{(H)}(M)/B) \\
	= \rho_{\bhsD{0,1}}
	\prod_{N \in \cS_{\psi}(M)} \rho_{\bhsD{1,0}(N)}\rho_{\bhsD{1,1}(N)}^2
		\mu_b(\diag_{\w}^{(H)}(M)/B) 
\end{multlined}\\
	\beta_{(\Delta),t}^*(\tau^{-1}) 
	= \rho_{\bhsD{0,1}}^{-1}
	\prod_{N \in \cS_{\psi}(M)} \rho_{\bhsD{1,1}(N)}^{-1},
\end{gathered}
\end{equation*}
so that we need to push-forward
\begin{equation*}
	\kappa_A \rho^{\mathfrak h^{\Delta}}
	\prod_{N \in \cS_{\psi}(M)} \rho_{\bhsD{1,0}(N)}\rho_{\bhsD{1,1}(N)}
		\mu_b(\diag_{\w}^{(H)}(M)/B).
\end{equation*}
This is a polyhomogeneous b-density with index sets
\begin{equation*}
	\cE_A(\bhsh{dd,1}) -\dim (M/B) - 2 \Mat \bhsD{0,1} 
\end{equation*}
and, for each $N \in \cS_{\psi}(M),$
\begin{equation*}
\begin{gathered}
	\cE_A(\bhsh{\phi\phi,0}(N)) + 1 - \dim(N/B)-1 \Mat \bhsD{1,0}(N) \\
	\cE_A(\bhsh{\phi\phi,1}(N)) +1 -\dim(N/B)-3 \Mat \bhsD{1,1}(N).
\end{gathered}
\end{equation*}
Applying the push-forward theorem yields the index sets for $(\Tr A)\tfrac{d\tau}\tau,$ and finally we cancel the factor of $\tfrac{d\tau}\tau.$
\end{proof}

\begin{corollary}\label{cor:ShortTimeTraceExp}
Let $M \xlra{\psi} B$ be a family of even dimensional manifolds with iterated fibration structures and let $E \lra M$ be a $\bbZ_2$-graded wedge vertical Clifford module.
If $A$ has integral kernel satisfying 
$\cK_A \rho^{\mathfrak h} \mu_R$ 
with 
\begin{equation*}
	\cK_A \in 
	\sB^{\cH/\mf I^{(H)}}_{phg}\sA^{-m-1}_-
		(H(M/B)_{\w}; \Hom_G(\bbE)\otimes \Omega_{\mf h,R}).
\end{equation*}
where $\mf I^{(H)}$ and $\cH$ are given by \eqref{eq:IndicialMwH}, \eqref{eq:HeatIndSets} and $\Omega_{\mf h, R}$ is the density bundle from \eqref{eq:HeatDensityBdle}
then $A$ is trace-class and 
\begin{equation*}
\begin{gathered}
	\Tr(A) \in \cA_{phg}^{\cH_{\tau}}(\bbR^+_{\tau};\CI(B; \Lambda^*T^*B)), \\
	\cH_{\tau}
	= (\bbN_0-\dim M/B) \bar\cup \bar\bigcup_{N \in \cS_{\psi}(M)}(\bbN_0-\dim N/B),
\end{gathered}
\end{equation*}
and 
\begin{equation*}
	\Str(A) \in \cA_{phg}^{\bbN_0 \bar\cup \ldots \bar\cup \bbN_0}
		(\bbR^+_{\tau}; \CI(B;\Lambda^*T^*B)),
\end{equation*}
where $\bbN_0$ is repeated $1+ \mathrm{depth}(M)$ times.
That is, the short-time expansion of $\Str(A)$ has the form
\begin{equation*}
	\Str(A) \sim
	\sum_{j\geq 0} \sum_{k=0}^{\mathrm{depth}(M)}
	\alpha_{j,k} \; \tau^j (\log \tau)^k,
\end{equation*}
where the coefficients $\alpha_{j,k}$ are smooth differential forms on $B.$

For $e^{-t\bbA_{M/B,Q}^2}$ we have the improvement
\begin{equation*}
	\Str(e^{-t\bbA_{M/B,Q}^2}) 
	\in \cA_{phg}^{\bbN_0 \bar\cup \bbN \bar\cup \ldots \bar\cup \bbN}
		(\bbR^+_{\tau}; \CI(B; \Lambda^*T^*B)),
\end{equation*}
where the index set is the extended union of one copy of $\bbN_0$ and $\mathrm{depth}(M)$ copies of $\bbN.$
That is, the short-time expansion of $\Str(e^{-t\bbA_{M/B,Q}^2})$ has the form
\begin{equation}\label{eq:ShortTimeAsympStr}
	\Str(e^{-t\bbA_{M/B,Q}^2}) \sim
	\alpha_{0,0} + 
	\sum_{j\geq 1} \sum_{k=0}^{\mathrm{depth}(M)}
	\alpha_{j,k} \; \tau^j (\log \tau)^k,
\end{equation}
where the coefficients $\alpha_{j,k}$ are smooth differential forms on $B.$
\end{corollary}

\begin{proof}
The statement about the trace of $A$ follows directly from Theorem \ref{thm:TraceAs}.
For the statement about the supertrace, recall Patodi's observation that the supertrace on the Clifford algebra vanishes on homomorphisms whose Clifford degree is less than the maximum Clifford degree. Thus the index set of $\cK_A|_{\diag_{\w}^{(H)}(M)}$ at $\bhsD{0,1}$ is shifted by $\dim (M/B)$ and its index set at $\bhsD{1,1}(N),$ for each $N \in \cS_{\psi}(M),$ is shifted by $\dim(N/B).$ The result is that each of these index sets contributes an $\bbN_0$ to the asymptotics as $\tau \to 0.$ 

In principle the index set of $\Str(A)$ is then $\bbN_0 \bar\cup \ldots \bar\cup \bbN_0$ with $\bbN_0$ repeated as many times as there are boundary hypersurfaces of $H(M/B)_{\w}$ over $\{\tau =0\}.$ However only actual intersections of boundary hypersurfaces produce accidental multiplicities and so it suffices to take the extended product over $1+\mathrm{depth}(M)$ copies of $\bbN_0.$ 

Finally, for $A = e^{-t\eth_{M/B,Q}^2},$ the improvement is that
$\alpha_{0,k}=0$ for all $k>0.$ To establish this, it suffices to show
that the pointwise supertrace of the heat kernel vanishes at corners
(cf. \cite[Lemma 5.27]{Vaillant}). The discussion at the end of
\S\ref{sec:Metrics} shows that the supertrace vanishes at any
intersection $\bhsD{1,1}(N) \cap \bhsD{1,1}(\wt N)$ since there can
not be a term of full Clifford degree. 
Similarly at an intersection of the form $\bhsD{1,1}(N)\cap \bhsD{0,1}$ 
the supertrace vanishes as the can not be a term of full Clifford degree; in this 
case this follows from the rescaled normal operator in Lemma \ref{lem:ResNormOpsA}
at $\bhsd{dd,1}.$ Indeed, the curvature of the Levi-Civita connection is evaluated on edge vector fields, so the vector field $\rho_N\pa_{\rho_N}$ does not occur without the $\rho_N$ factor, and hence any term with $\df e(d\rho_N)$ will vanish at $\bhsd{\phi\phi,1}(N).$

\end{proof}

\section{Families index formula} \label{sec:ChernChar}

Let $M \xlra\psi B$ be a fiber bundle of manifolds with corners and iterated fibration structures, let $E\lra M$ a wedge Clifford module with compatibly perturbed Dirac-type operator $\eth_{M/B,Q}$ equipped with its vertical APS domain satisfying the Witt condition, and let $\bbA_{M/B,Q}$ be the perturbed Bismut superconnection.

\subsection{The finite time limit}

Given an arbitrary superconnection on $M \lra B,$
\begin{equation*}
	\bbA = \bbA_{[0]} + \bbA_{[1]} + \ldots + \bbA_{[k]},
\end{equation*}
recall that the {\bf rescaled superconnection} is
\begin{equation*}
	\bbA^{t} = t^{1/2} \lrpar{\bbA_{[0]} + t^{-1/2}\bbA_{[1]} + \ldots +t^{-k/2}\bbA_{[k]}}
	= \tau \delta_{t}^B \bbA (\delta_{t}^B)^{-1}
\end{equation*}
where $\delta_t^B$ multiplies forms on $B$ of degree $k$ by $\tau^{-k/2}.$

    Let us recall the notion of twisted supertrace.
If $\alpha_E$ is a  grading operator on $E$, i.e., the
    operator which is identity on even degree sections and
    multiplication by $-1$ on odd degree sections,  so that
\begin{equation*}
	\str_E(\cdot) = \tr_E(\alpha_E \cdot),
\end{equation*}
then since $\alpha_E$ is a section of 
\begin{equation}\label{eq:homTensor}
	\hom(E) = \Cl({}^\w T^*M/B) \otimes \hom'_{\Cl({}^\w T^*M/B)}(E),
\end{equation}
we have
\begin{equation*}
	\alpha_E = \alpha_{M/B}\otimes \alpha_E'
\end{equation*}
where $\alpha_{M/B}$ is the grading operator on $\Cl({}^\w T^*M/B)$ and $\alpha_E'$ commutes with Clifford multiplication and squares to the identity. The supertrace functional decomposes with respect to \eqref{eq:homTensor} into the product of two supertrace functionals,
\begin{equation*}
	\str(A \otimes A') = \tr_{\Cl({}^\w T^*M/B)}(\alpha_{M/B}A) \tr(\alpha_E'A')
		= \str_{\Cl(M/B)}(A)\str_{\Cl(M/B)}'(A')  
\end{equation*}
for all $A \in \Cl({}^\w T^*M/B),$ $A' \in \hom'_{\Cl({}^\w T^*M/B)}(E),$
and we refer to $\str_{\Cl(M/B)}'$, defined by
    this equation, as the twisted supertrace.

We have similar decompositions at each $\bhsh{\phi\phi,1}(N)$ of $H(M/B)_{\w};$ indeed, we have seen that
\begin{equation*}
	\Hom(E)\rest{\bhsh{\phi\phi,1}(N)} = \Cl({}^\w T^*N/B)\otimes \hom'_{\Cl({}^\w T^*N/B)}(E)
\end{equation*}
and consequently the supertrace functional decomposes as $\str_{\Cl(N/B)} \otimes \str_{\Cl(N/B)}'.$
$ $\\

Let us introduce the notation for the terms appearing in the short-time limit of the supertrace of the heat kernel.
Let
\begin{equation*}
	\hat A(M/B) = \det^{1/2}\lrpar{\frac{R^{M/B}/4\pi}{\sinh(R^{M/B}/4\pi)}} \in \CI(M;\Lambda^*T^*M)
\end{equation*}
and similarly for $\hat A(N/B),$ with $N \in \cS_{\psi}(M),$
and denote the twisted Chern character by
\begin{equation*}
	\Ch'(E) = \str'_{\Cl(M/B)}(\exp(-K'_E/2\pi)) \in \CI(M;\Lambda^*T^*M).
\end{equation*}
where $K'_E$ is the twisting curvature from \eqref{eq:SupLichnerowicz}.

Define, for each $N \in \cS_{\psi}(M),$ the {\bf Bismut-Cheeger $\cJ$-form},
\begin{equation*}
	\cJ_Q(\bhs{N}/N) 
	= \sideset{^R}{_0^{\infty}}\int \int_{\bhs{N}/N} 
		\str'_{\Cl(N/B)}
		\lrpar{\exp ( - (\bbA_{C(\phi_N)/N,Q}^t)^2) }\rest{s=1} \; \frac{dt}{2t}
	\in \CI(N;\Lambda^*T^*N)
\end{equation*}
where $s$ is the radial variable along the cone.
Here $\sideset{^R}{_0^{\infty}}\int$ denotes the {\em renormalized integral} (also known as the b-integral see, e.g., \cite[\S4.19]{tapsit}, \cite{Albin:RenInt}, \cite[\S2.3]{Hassell-Mazzeo-Melrose:Surgery}) we will see below that this integral is convergent, so that no renormalization is necessary.

\begin{proposition}\label{prop:ShortTimeLimit}
Under the Witt assumption,
\begin{equation*}
	\lim_{t\to 0} \; \Str(e^{-(\bbA_{M/B,Q}^t)^2})
	= 
	\int_{M/B} \hat A(M/B) \Ch'(E) 
	- \sum_{N \in \cS_{\psi}(M)} 
	\int_{N/B} \hat A(N/B) \cJ_Q(\bhs{N}/N)
\end{equation*}
\end{proposition}

\begin{proof}
We proceed as in \cite{Albin-Rochon:DFam, BGV2004, Melrose-Piazza:Even}.

Let
\begin{equation*}
	\cK = e^{-t\bbA_{M/B,Q}^2}\rest{\diag_{\w}^{(H)}(M)}
\end{equation*}
where the lifted diagonal, $\diag_{\w}^{(H)}(M),$ is described in \eqref{eq:LiftedDiag}.
Corollary \ref{cor:ShortTimeTraceExp} established the existence of 
the small time limit of the supertrace and that it is given by
\begin{equation*}
	\int_{\bhsh{dd,1}\cap \diag_M} \str(\cK)\rest{\bhsh{dd,1}}
	+ \sum_{N \in \cS_{\psi}(M)} 
	\int_{\bhsh{\phi\phi,1}(N)\cap \diag_M} \str(\cK)\rest{\bhsh{\phi\phi,1}(N)}.
\end{equation*}

Recall from, e.g., \cite[Lemma 10.22]{BGV2004}
\begin{equation*}
	\str(e^{-(\bbA_{M/B,Q}^t)^2}) = \delta_t^B\lrpar{ \str(e^{-t\bbA_{M/B,Q}^2}) }
\end{equation*}
and Patodi's observation that the supertrace in $\Cl(V)$ only depends on terms of top Clifford degree.
Thus if at $\bhsh{dd,1}$ we have an expansion
\begin{equation*}
	\cK\rest{\diag_M} \sim 
		t^{-(\dim M/B)/2} \sum_{\ell \in \bbN_0} \cU_{\ell} t^{\ell/2}, 
\end{equation*}
with each term $\cU_{\ell}$ of Clifford degree at most $\ell$ then, 
just as in \cite{BGV2004}, \cite{Melrose-Piazza:Even}, this implies that
\begin{equation*}
\begin{gathered}
	\str(\cK|_{\diag_{\w}})\rest{\bhsh{dd,1}} 
	= 
	(-2i)^{(\dim M/B)/2}
	\mathrm{Ev}_{M/B}\lrspar{ \str \cN^G_{\bhs{dd,1}}\lrpar{e^{-t\bbA_{M/B,Q}^2}}}\rest{\diag_M} \\
	= 
	(-2i)^{(\dim M/B)/2}
	\mathrm{Ev}_{M/B}\lrspar{ \str e^{-\cH_{M/B}(\zeta)}\rest{\zeta=0} }
	= 
	\mathrm{Ev}_{M/B}\lrpar{ \hat A(M/B) \Ch'(E)}
\end{gathered}\end{equation*}
where $\mathrm{Ev}_{M/B}$ is the projection of differential forms on $M$ onto those with top $\psi$-vertical degree.
Thus the contribution from this face is
\begin{equation*}
	\int_{\bhsh{dd,1}\cap \diag_M} \str(\cK)\rest{\bhsh{dd,1}}
	= 
	\int_{M/B} \hat A(M/B) \Ch'(E).
\end{equation*}

We now consider the situation at a face $\bhsh{\phi\phi,1}(N)$ for
some $N \in \cS_{\psi}(M).$ Let $x$ denote a boundary defining
function for this face, $\sigma = \tfrac{\tau}x$ a rescaled time
parameter, and note that the rescaling operator $\delta_B^t$ becomes
$\delta_B^{x^2}\delta_B^{\sigma^2}.$  
Since the Taylor expansion of $\cK$ at $\bhsh{\phi\phi,1}(N)$ is in powers of $x,$ the rescaling operator $\delta_B^{x^2}$ plays the same r\^ole at this face as $\delta_B^t$ plays at $\bhsh{dd,1},$ it mediates between the degree in $\Cl({}^\w T^*N/B)$ and that in $\Cl({}^\w T^*N/B)\hat\otimes \psi^*\Lambda^*B.$
Proceeding as in \cite[(1.23),(1.24)]{BC1989},
\begin{equation*}
\begin{gathered}
	\str(\cK)\rest{\bhsh{\phi\phi,1}(N)} 
	= 
	(-2i)^{\floor{(1+\dim N/B)/2}}
	\mathrm{Ev}_{N/B}
		\lrspar{ \str \cN^G_{\bhs{\phi\phi,1}(N)}\lrpar{e^{-t\bbA_{M/B,Q}^2}}}\rest{\diag_M} \\
	= 
	(-2i)^{\floor{(1+\dim N/B)/2}}
	\mathrm{Ev}_{N/B}\lrspar{ \delta_B^{\sigma^2}\str e^{-\sigma^2\cH_{N/B}(\zeta)}
		e^{-\sigma^2\bbA^2_{C(\phi_N)/N,Q}}\rest{\zeta=0, s=1} }\\
	= 
	\mathrm{Ev}_{N/B}\lrspar{ \delta_B^{\sigma^2} \hat A(\sigma^2 R^{N/B})
	\str_{\Cl(N/B)}' e^{-\sigma^2\bbA^2_{C(\phi_N)/N,Q}}\rest{s=1} }\\	
	=
	\mathrm{Ev}_{N/B}\lrspar{ \hat A(N/B)
	\str_{\Cl(N/B)}' e^{-(\bbA^{\sigma^2}_{C(\phi_N)/N,Q})^2}\rest{s=1} }.
\end{gathered}
\end{equation*}
Thus the contribution from this face is
\begin{multline*}
	\int_{\bhsh{\phi\phi,1}(N)\cap \diag_M} \str(\cK)\rest{\bhsh{\phi\phi,1}(N)}\\
	=
	\int_{N/B} \hat A(N/B) \sideset{^R}{_0^{\infty}}\int \int_{\bhs{N}/N} 
		\str'_{\Cl(N/B)}\lrpar{\exp ( - (\bbA_{C(\phi_N)/N,Q}^t)^2) }\rest{s=1} \; \frac{d\tau}\tau \\
	=
	\int_{N/B} \hat A(N/B) \cJ_Q(\bhs{N}/N)
\end{multline*}
as required.
\end{proof}

\subsection{Bismut-Cheeger $\eta$ and $\cJ$ forms} \label{sec:BCJeta}

Bismut and Cheeger \cite{BC1989,BC1990, BC1990II, BC1991} defined differential forms on $B,$ $\eta$ and $\cJ$ in the setting of closed manifold fibers. The former were also defined on spaces of depth one by Cheeger \cite[\S 8]{C1987}, for isolated conic singularities, and by Piazza-Vertman \cite{Piazza-Vertman} in general. Melrose and Piazza \cite{Melrose-Piazza:Even,Melrose-Piazza:Odd}  introduced $\eta$-forms for perturbed Dirac-type operators.
In this section we generalize their construction to compatibly
perturbed families of wedge Dirac-type operators with vertical APS
domain satisfying the Witt condition.

We define these forms for a family of manifolds 
with iterated fibration structures without assuming that they come from a boundary fibration. To differentiate the fiber bundle used in this subsection with the one in the main body of the text, we will adopt the notation
\begin{equation*}
	\hat X \fib \hat M \xlra{\hat\psi} B.
\end{equation*}

As we will review, both the $\eta$-forms and $\cJ$-forms can be
thought of as arising from $\hat M \times \bbR^+.$ In the former case
the factor of $\bbR^+$ is added to the base, in the latter it is added
to the fiber. Analogously to Definition \ref{def:Cl(F)Bdle}, we will use the following notation; for a bundle $\hat
X \fib \hat M \xlra{\hat\psi} B$ with vertical metric $g_{\hat M/B}$,
a $\Cl(1)$ bundle is a bundle $E \lra \hat M$ with an action of
$\Cl(T^*\hat M / B \oplus \mathbb{R})$ where the additional
$\mathbb{R}$ factor is orthogonal. (In particular the fibre is the
complexified Clifford algebra on $\mathbb{R}^{\dim \hat X + 1}$.)  We
will denote the additional generator of this Clifford algebra by $\gamma$.  \\

\begin{definition}
Let $E \lra \hat M$ be a $\hat\psi$-vertical wedge Clifford bundle if
$\dim \hat M/B$ is even dimensional and a $\Cl(1)$-bundle if $\dim
\hat M/B$ is odd and $Q$ a compatible perturbation such that the
associated family of Dirac-type operators $\eth_{\hat M/B,Q},$ endowed
with its vertical APS domain, satisfies the Witt assumption and is
such that $\ker \eth_{\hat M/B,Q}$ forms a vector bundle over $B.$

The {\bf Bismut-Cheeger $\eta$-form} of $\eth_{\hat M/B,Q},$
\begin{equation*}
	\eta_Q(\hat M/B) =\eta(\eth_{\hat M/B,Q}) \in \CI(B;\Lambda^*T^*B)
\end{equation*}
is given by
\begin{equation*}
	\eta_Q(\hat M/B) =
	\begin{cases}
	\displaystyle
	\frac1{2\sqrt{\pi}} \int_0^{\infty} \Str\lrpar{ 
	\frac{\pa (\bbA_{\hat M/B,Q}^t)}{\pa t} e^{-(\bbA_{\hat M/B,Q}^t)^2} } \; dt 
			& \Mif \dim (\hat M/B) \Mev \\
	\displaystyle
	\frac1{2\sqrt{\pi}} \int_0^{\infty} \Str_{\Cl(1)}\lrpar{ 
		\frac{\pa (\bbA_{\hat M/B,Q}^t)}{\pa t} e^{-(\bbA_{\hat M/B,Q}^t)^2} } \; dt 
			& \Mif \dim (\hat M/B) \Modd
	\end{cases}
\end{equation*}

The {\bf normalized $\eta$ form}, $\bar\eta_Q(\hat M/B)$ is the form obtained from $\eta_Q(\hat M/B)$ by multiplying the forms of degree $\ell$ by $(2\pi i)^{-\floor{\ell/2}}.$
\end{definition}

The form $\bar\eta_Q(\hat M/B)$ has even degree if $\dim \hat X$ is odd, and odd degree if $\dim \hat X$ is even.

Implicit in this definition is the fact that the integral
converges. As pointed out in, e.g., \cite[Theorem 2.10]{BFII},
\cite[\S 3]{BC1989}, \cite[Theorem 2.11]{Bismut-Gillet-Soule:II},
\cite[Proof of Theorem 10.32]{BGV2004}, this can be established by
considering the heat kernel of the Bismut superconnection in one
dimension higher. Following \cite[(117), Remark A.15]{Vaillant}, we
can treat the even and odd cases uniformly by extending the fiber
bundle $\hat \psi$ to 
\begin{equation*}
	\hat X \fib \hat M^+ = \hat M \times \bbR^+_s \xlra{\hat\psi \times \id} 
		B^+ = B \times \bbR^+_s,
\end{equation*}
and noting that, in terms of the extended Bismut superconnection $\bbA_{(\hat M^+)/(B^+)}$ corresponding to the natural extensions of the vertical covariant derivative and choice of horizontal tangent bundle,
\begin{equation*}
	\hat \eta(\eth_{\hat M/B})
	= \int_0^{\infty}  \df i(\pa_s) 
	\Str'\big(e^{- ( (\bbA_{\hat M^+/B^+,Q}^t )^2 )} \big) 
		\; dt
\end{equation*}
where $\Str'$ denotes the appropriate supertrace on $\Lambda^*T^*(B\times \bbR^+)\otimes E$ corresponding to the parity of $\dim \hat M/B.$
The short-time asymptotic expansion of the supertrace of the heat
kernel together with the long-time limit implies the convergence of
this integral and so the well-definedness of the $\eta$-forms.
    \\

Next we define the Bismut-Cheeger $\cJ$-forms in our context.
Suppose $E \lra \hat M$ is a $\Cl(1)$-bundle so we have an action of
$\cl(\gamma)$ on $E$ (see above) and that $E$ is $\bbZ_2$-graded if $\dim \hat M/B$ is
odd.  
Consider the extension of $\hat\psi$ to
\begin{equation*}
	\hat X^+ = \hat X \times \bbR^+ \fib \hat M^+ = \hat M \times \bbR^+ \xlra{\hat\psi^+} B.
\end{equation*}
and the warped product metric
\begin{equation*}
	g^+ = g_{\hat M^+/B} = ds^2 + s^2 g_{\hat M/B}.
\end{equation*}
Recall from \eqref{eq:LCConCyl} that the Levi-Civita connection of $g^+$ on the wedge cotangent bundle differs from the product Levi-Civita connection by
\begin{equation}\label{eq:LC-prod-cone}
	\nabla^{g^+}_W\theta 
	= \nabla^{\oplus}_{W}
	+ \lrpar{ g^+(dx,\theta)(\tfrac1x\bv W)^{\flat} - g^+((\tfrac1x\bv W)^{\flat},\theta)dx}.	
\end{equation}

Let $E^+,$ denote $E$ pulled-back to $\hat M^+$ via $\pi:\hat M^+ \lra
\hat M$ and $g_{E^+}$ the pull-back metric of $g_E$. Over $\hat M^+$ we have a bundle isomorphism
\begin{equation}\label{eq:warp-cliff-act}
	{}^\w T\hat M^+/B \cong \ang{ds} \oplus s \; T^*\hat M/B
	\xlra{\Xi}
	\ang{\gamma} \oplus T^*\hat M/B, \quad 
	a_0\; ds + a_i \; s\theta^i
	\mapsto 
	a_0\gamma + a_i\theta^i 
\end{equation}
which we use to define a Clifford action of ${}^\w T^*\hat M^+/B$ on $E,$
\begin{equation*}
	\cl^+(\theta) = \cl(\Xi(\theta)).
\end{equation*}
This Clifford action is compatible with the metric $g_{E^+}.$
We modify the connection on $E$ to get a connection on $E^+$ following \eqref{eq:EConCyl},
\begin{equation}\label{eq:E-prod-cone}
	\nabla^{E^+}_W
	=\pi^*\nabla^{E}_W
	+ \tfrac12 \cl^+(dx)\cl^+((\tfrac1x\bv W)^{\flat}).
\end{equation}
This is a $g_{E^+}$ metric connection compatible with the extended Clifford action $\cl^+$ and the Levi-Civita connection of $g^+.$
Hence $E^+$ is a $\hat\psi^+$-vertical wedge Clifford module.

The families of Dirac operators on $\hat M^+ / B$
  produced by the preceding construction do not
  constitute arbitrary families; indeed, those produced here have in
  some sense $s$-independent twisting bundle (though the twisting
  bundle is only defined globally in the spin case).  In fact the twisting
  curvature of the connection $\nabla^{E^+}$ is $s$ and
  $ds$-independent:
  \begin{equation}\label{eq:twisting-s-ind}
K_E'(\pa_s, \cdot) \equiv 0 \equiv \pa_s K_E' 
\end{equation}
Indeed, locally in the bulk of $\hat M$, the Clifford action induces a
splitting $E \simeq \cS(T^*\hat M^+ /B) \otimes W$ with $\cS(T^*\hat
M^+ /B)$ the (locally defined) bundle of spinors, with Clifford
multiplication acting on the left and the connection decomposing as a
tensor product connection ${}^{\cS}\nabla \otimes \id + \id \otimes
{}^{W} \nabla$.  The pullback construction above changes the
connection only by adding a zeroeth order Clifford multiplication
term, i.e.\ only by modifying the connection on the spinor bundle
part.  The twisting curvature depends only on ${}^W\nabla$ and is
therefore \eqref{eq:twisting-s-ind} holds.

A choice of connection for $\hat\psi,$
\begin{equation*}
	{}^\w T\hat M \cong {}^\w T\hat M/B \oplus \hat\psi^*TB,
\end{equation*}
readily extends to a choice of connection for $\hat \psi^+,$
\begin{equation*}
	{}^\w T\hat M^+ \cong {}^\w T\hat M^+/B \oplus (\hat\psi^+)^*TB.
\end{equation*}
The fundamental tensors \eqref{eq:FunTensors} of $\hat \psi$ and $\hat\psi^+$ are essentially unchanged,
\begin{equation*}
	\cS^{\hat\psi^+}(W_1, W_2)(A) = \cS^{\hat\psi}(W_1, W_2)(A), \quad
	\hat\cR^{\hat\psi^+}(A_1, A_2)(W) = \hat\cR^{\hat\psi}(A_1, A_2)(W),
\end{equation*}
where the latter tensors are understood to vanish if any of the vertical vector fields is $\pa_s.$
Thus the Bismut superconnection on $\hat M^+$ is given by (cf. Lemma \ref{lem:ModelDirac} and \cite[Lemma 1.2]{Albin-GellRedman})
\begin{equation*}
	\bbA_{\hat M^+/B} 
	= \lrpar{ \tfrac1s \eth_{\hat M/B} +\cl(ds)\pa_s + \tfrac{\dim \hat M/B}2 \cl(ds)} 
	+ (\bbA_{\hat M/B})_{[1]}
	+ (\bbA_{\hat M/B})_{[2]}
\end{equation*}
and the square of the rescaled Bismut superconnection satisfies \cite[(6.37)]{BC1990II}
\begin{equation*}
	(\bbA_{\hat M^+/B}^t)^2
	= -t(\pa_s + \tfrac{\dim \hat M/B}{2s})^2
	+ (\bbA_{\hat M/B}^{t/ s^2})^2
	+ \tfrac t{s^2} \cl(ds)\bbA_{\hat M/B, [0]}
	+ \tfrac14 \cl(ds) \bbA_{\hat M/B, [2]}.
\end{equation*}

\begin{definition}
With notation as above, the {\bf Bismut-Cheeger $\cJ$-form} of $\eth_{\hat M/B,Q},$
\begin{equation*}
	\cJ_Q(\hat M/B) = \cJ(\eth_{\hat M/B,Q}) \in \CI(B;\Lambda^*T^*B)
\end{equation*}
is given by
\begin{equation*}
	\cJ_Q(\hat M/B) =
	\begin{cases}
	\displaystyle
	-\int_0^{\infty} \int_{\hat M/B} 
	\str\lrpar{ 
	e^{-(\bbA_{\hat M^+/B}^t)^2 }\rest{\diag_{\hat M^+},s=1} }\; \frac{dt}{2t}
			& \Mif \dim (\hat M^+/B) \Mev \\
	\displaystyle
	-\int_0^{\infty} \int_{\hat M/B} 
	\str_{\Cl(1)}\lrpar{ 
	e^{-(\bbA_{\hat M^+/B}^t)^2 }\rest{\diag_{\hat M^+},s=1} }\; \frac{dt}{2t}
			& \Mif \dim (\hat M^+/B) \Modd
	\end{cases}
\end{equation*}
\end{definition}

The convergence of the integral follows as in \cite[\S VI(a)]{BC1990II} from the dilation invariance property \eqref{eq:DilationInv}. Indeed, the short time asymptotic expansion of the supertrace from \eqref{eq:ShortTimeAsympStr} and Proposition \ref{prop:ShortTimeLimit} guarantees the convergence as $t \to 0,$ and the vanishing of the null space established in Proposition \ref{prop:NyInjSA} together with the resulting decay of the heat kernel, establishes the convergence as $t\to \infty.$\\

\subsection{Even dimensional fibers}
Let $M \xlra{\psi} B$ be a fiber bundle of manifolds with corners and iterated fibration structures such that $\dim(M/B)$ is even. Let $E \lra M$ be a wedge Clifford bundle, with associated Dirac-type operator $\eth_{M/B}$ and $Q$ a compatible perturbation such that $\eth_{M/B,Q}$ with its vertical APS domain satisfies the Witt assumption.\\

We may, as in \cite[Lemma 1.1]{Melrose-Rochon:FCusp}, perturb $\eth_{M/B,Q}$ by smoothing operators compactly supported in the interior of $M$ without changing the families index in K-theory and so that the null spaces form a vector bundle over $B.$ Since this perturbation is supported in the interior it will not change the boundary families of $\eth_{M/B,Q}$ and the arguments in \cite[Proposition 9.46]{BGV2004} apply to show that the effect on the short-time asymptotic expansion of the trace of the heat kernel is $\cO(t).$ By incorporating this perturbation into $Q,$ we will {\em assume that $\ker \eth_{M/B,Q}$ forms a smooth vector bundle over $B.$}

Given an arbitrary superconnection on $M \lra B,$ $\bbA,$ the {\bf Chern character} of $\bbA$ is 
\begin{equation*}
	\Ch(\bbA) = \Str(e^{-\bbA^2}).
\end{equation*}

For the Bismut superconnection and a compatible perturbation, $\bbA_{M/B,Q},$ the arguments in \cite[\S9.3]{BGV2004} apply directly and show that
\begin{equation*}
	\frac{\pa}{\pa t}\Ch(\bbA_{M/B,Q}^t) = 
		-d_B\Str\lrpar{ \frac{\pa \bbA_{M/B,Q}^t}{\pa t}e^{-(\bbA_{M/B,Q}^t)^2}}
\end{equation*}
and
\begin{equation*}
	\lim_{t\to\infty} \Ch(\bbA_{M/B,Q}^t) = \Ch(\Ind(\eth_{M/B,Q}), \nabla^{\Ind})
\end{equation*}
where $\Ind(\eth_{M/B,Q})$ is the virtual index bundle of $\eth_{M/B,Q}$ and $\nabla^{\Ind}$ is the contraction,
\begin{equation*}
	\nabla^{\Ind} = \cP_{\Ind}(\bbA_{M/B,Q})_{[1]}\cP_{\Ind}.
\end{equation*}
Proposition \ref{prop:ShortTimeLimit} and integration in $t$ yield the families index formula.
$ $\\

\begin{theorem}
Let $M \xlra\psi B$ be a fiber bundle of manifolds with corners and iterated fibration structures with even-dimensional fibers, $E \lra M$ a $\bbZ_2$-graded wedge Clifford bundle with associated Dirac-type operator $\eth_{M/B}$ and let $Q$ be a compatible perturbation.
If $\eth_{M/B,Q}$ with its vertical APS domain satisfies the Witt assumption, then
\begin{multline*}
	\Ch_{\ev}(\Ind(\eth_{M/B}), \nabla^{\Ind})
	=
	\int_{M/B} \hat A(M/B) \Ch'(E) 
	- \sum_{N \in \cS_{\psi}(M)} 
	\int_{N/B} \hat A(N/B) \cJ_Q(\bhs{N}/N) \\
	+ d\eta_Q(M/B)
\end{multline*}
\end{theorem}

\subsection{Odd dimensional fibers}

A standard argument going back to \cite{Atiyah-Singer:Skew}  reduces the families index for odd dimensional fibers to the families index for even dimensional fibers by suspension. This was carried out for Dirac operators on closed manifolds in \cite{BFII}.
We will follow the treatment of Melrose and Piazza \cite{Melrose-Piazza:Odd}, though note that our Clifford multiplication conventions differ by a factor of $i.$

Lemma 2 of \cite[\S5]{Melrose-Piazza:Odd} says that: If $L_1$ and $L_2$ are Clifford modules, with Clifford actions $\cl_1$ and $\cl_2,$ over Riemannian manifolds $X_1,$ $X_2,$ respectively, then the bundle
\begin{equation*}
	L = L_1 \otimes L_2 \otimes \bbC^2 \lra X_1 \times X_2
\end{equation*}
has a Clifford action compatible with the product metric on $X = X_1 \times X_2$ given by
\begin{equation*}
\begin{gathered}
	\cl(\alpha) 
	= \cl_1(\alpha) \otimes \Id \otimes \Gamma_1 \Mforall \alpha \in \CI(X_1; T^*X_1) \\
	\cl(\beta) 
	= \Id \otimes \cl_2(\beta) \otimes \Gamma_2 \Mforall \beta \in \CI(X_2; T^*X_2)
\end{gathered}
\end{equation*}
for any choice of $\Gamma_i \in \cM_2(\bbC)$ satisfying
\begin{equation*}
	\Gamma_1\Gamma_2 + \Gamma_2 \Gamma_1 = 0, \quad
	\Gamma_1^2 = \Gamma_2^2 = -1.
\end{equation*}
For example we can take
\begin{equation*}
	\Gamma_1 = \begin{pmatrix} 0 & i \\ i & 0 \end{pmatrix}, \quad
	\Gamma_2 = \begin{pmatrix} 0 & -1 \\ 1 & 0 \end{pmatrix}.
\end{equation*}
If $X_1 \times X_2$ is even-dimensional then $L$ can be taken to be $\bbZ_2$-graded as follows:
If $X_1$ and $X_2$ are both even-dimensional, so that $L_1$ and $L_2$ are $\bbZ_2$-graded, then we  take the product grading of $L_1 \otimes L_2$ and then tensor with $\bbC^2.$ If $X_1$ and $X_2$ are both odd-dimensional, so that $L_1$ and $L_2$ are ungraded, then we put a grading on $\bbC^2$ by 
\begin{equation*}
	(\bbC^2)^+ = \bbC \oplus \{0\}, \quad
	(\bbC^2)^- = \{0\} \oplus \bbC,
\end{equation*}
and then tensor with $L_1 \otimes L_2.$
Endowing $\bbC^2$ with the trivial metric and connection, $L$ has the structure of a Clifford module over $X_1 \times X_2.$ \\

Set $\bbT = \bbS^1_\theta \times \bbS^1_\xi,$ where we parametrize the first circle by $\theta \in [0,2\pi]$ and the second by $\xi \in [0,\pi],$ and let $L_{\bbT}$ be the Hermitian line bundle over $\bbT$ given by identifying the points $(\theta, 0, v)$ with $(\theta, \pi, e^{-i\theta}v)$ endowed with the Hermitian connection
\begin{equation*}
	\nabla^{L_\bbT} = d -\tfrac{2\xi-1}{2\pi i} \; d\theta.
\end{equation*}
These data define a family of Dirac operators on the fibers of 
\begin{equation*}
	\bbT \xlra{\psi_{\bbT}} \bbS^1_\xi
\end{equation*}
given by $\eth_{\bbT/\bbS^1} = \tfrac1i\pa_\theta + \tfrac{2\xi-1}{2\pi}$ with spectral flow equal to one.\\

We replace the original fiber bundle $M \xlra{\psi} B$ with 
\begin{equation*}
	X \times \bbS^1_{\theta} \fib S^2M = M \times \bbT \xlra{\psi\times\psi_{\bbT}} 
		SB = B \times \bbS^1_\xi,
\end{equation*}
and each boundary fiber bundle $\bhs{N} \xlra{\phi_N} N$ with
\begin{equation*}
	S^2\bhs{N} = \bhs{N} \times \bbT \xlra{\phi_N \times \psi_{\bbT}} SN = N \times \bbS^1_t,
\end{equation*}
replace $E \lra M$ with $F = E \otimes L \otimes \bbC^2,$ 
and then the extension of the Clifford structure described above yields the family of Dirac-type wedge operators
\begin{equation*}
	\eth_{S^2M/SB}
	= \begin{pmatrix}
	0 & \eth_{M/B}\otimes \Id + \Id \otimes \eth_{\bbT/\bbS^1} \\
	\eth_{M/B}\otimes \Id - \Id \otimes \eth_{\bbT/\bbS^1} & 0 
	\end{pmatrix}.
\end{equation*}

The invertibility of $\eth_{\bbT/\bbS^1}$ on each fiber of $\psi_{\bbT},$ together with the invertibility of the boundary families of $\eth_{M/B,Q}$ easily yields the invertibility of the boundary families of $\eth_{S^2M/SB,Q},$ as well as the invertibility of $\eth_{S^2M/SB,Q}$ on the fibers lying over some neighborhood of the point $\{t=0\} \in \bbS^1_\xi,$ as in \cite[Lemma 3]{Melrose-Piazza:Odd}.
 
This implies that the index class of $\eth_{M/B,Q}$ in the odd K-theory of $B$ is mapped by suspension into the index class of $\eth_{S^2M/SB,Q}$ \cite[Proposition 6]{Melrose-Piazza:Odd}. Thus the odd Chern character of the index class of $\eth_{M/B,Q}$ satisfies
\begin{equation*}
	\Ch_{\odd}(\Ind(\eth_{M/B,Q})) 
	= \frac i{2\pi} \int_{\bbS^1_\xi} \Ch_{\ev}(\Ind(\eth_{S^2M/SB,Q}))
\end{equation*}
and we obtain a formula for the odd Chern character by integrating our formula for the even dimensional Chern character over the circle.
$ $\\

This leaves us to consider the effect of suspension on the $\cJ_Q$-forms.
Let us start by recalling \cite[\S8]{Melrose-Piazza:Odd} the effect of suspension on the Bismut superconnection. The Bismut superconnection depends on the choice of a vertical metric and connection, in this case we take
\begin{equation*}
	T(S^2M/SB) = TM/B \oplus T\bbS^1_{\theta},\quad
	g_{S^2M/SB} = g_{M/B} \oplus d\theta^2,
\end{equation*}
where we leave implicit the pull-back maps.
The corresponding Bismut superconnection satisfies
\begin{equation*}
\begin{gathered}
	\bbA_{S^2M/SB,Q}
	= \lrpar{\eth_{M/B,Q} \otimes \Gamma_2 + \eth_{\bbT/\bbS^1} \otimes \Gamma_1}
	+ \lrpar{\bbA_{[1]} + \df e(d\xi)\pa_{\xi} } + \bbA_{[2]} \otimes \Gamma_2 \\
	\bbA_{S^2M/SB,Q}^2
	= \bbA_{M/B,Q}^2 + \eth_{\bbT/\bbS^1}^2 
	+ \tfrac1{\pi i}\; \df e(d\xi)(\cl(d\theta)\otimes \Gamma_1)
\end{gathered}
\end{equation*}
$ $\\
$ $\\

Let us describe the $\cJ$-forms occurring in the index theorem on $S^2M/SM.$
Let $N \in \cS_{\psi}(M)$ so that $S^2M$ has the boundary fiber bundle $S^2\bhs{N} \lra SN.$
The Bismut superconnection on $S^2\bhs{N}^+/SN$ is given by
\begin{equation*}
\begin{gathered}
	\bbA_{S^2\bhs{N}^+/SN,Q} 
	= \lrpar{ \tfrac1s \eth_{S^2\bhs{N}/SN,Q} +\cl(ds)\pa_s + \tfrac{\dim \hat S^2\bhs{N}/SN}2 \cl(ds)} 
	+ (\bbA_{S^2\bhs{N}/SN})_{[1]}
	+ (\bbA_{S^2\bhs{N}/SN})_{[2]} \\
	= \tfrac1s\lrpar{\eth_{M/B,Q} \otimes \Gamma_2 + \eth_{\bbT/\bbS^1} \otimes \Gamma_1}
	+ \cl(ds)\pa_s + \tfrac{1+\dim\bhs{N}/N}2 \cl(ds)
	+  \lrpar{\bbA_{[1]} + \df e(d\xi)\pa_{\xi} } + \bbA_{[2]} \otimes \Gamma_2
\end{gathered}
\end{equation*}
and hence its square can be written
\begin{equation*}
\begin{gathered}
	(\bbA_{S^2\bhs{N}^+/SN,Q}^t)^2
	= -t(\pa_s + \tfrac{\dim S^2\bhs{N}/SN}{2s})^2
	+ (\bbA_{S^2\bhs{N}/SN,Q}^{t/ s^2})^2
	+ \tfrac t{s^2} \cl(ds)\bbA_{S^2\bhs{N}/SN, Q,[0]} \\
	+ \tfrac14 \cl(ds) \bbA_{S^2\bhs{N}/SN, [2]} \\
	= -t(\pa_s + \tfrac{1 + \dim \bhs{N}/N}{2s})^2
	+ (\bbA_{M/B,Q}^{t/ s^2})^2 + \tfrac t{s^2}\eth_{\bbT/\bbS^1}^2 
	+ \tfrac{\sqrt t}{s\pi i}\; \df e(d\xi)(\cl(d\theta)\otimes \Gamma_1)\\
	+ \tfrac t{s^2} \cl(ds)
	\lrpar{\eth_{M/B,Q} \otimes \Gamma_2 + \eth_{\bbT/\bbS^1} \otimes \Gamma_1}
	+ \tfrac14 \cl(ds) \bbA_{[2]} \otimes \Gamma_2 \\
	=\tfrac{\sqrt t}{s\pi i}\; \df e(d\xi)(\cl(d\theta)\otimes \Gamma_1)	
	+ \tfrac t{s^2}(\eth_{\bbT/\bbS^1}^2 + \cl(ds)(\eth_{\bbT/\bbS^1} \otimes \Gamma_1) )
	+ s^{1/2}(\bbA_{\bhs{N}^+/N,Q}^t)^2  s^{-1/2}.
\end{gathered}
\end{equation*}
Note that in the final formula the three summands commute, and so we have
\begin{multline*}
	\exp( -(\bbA_{S^2\bhs{N}^+/SN,Q}^t)^2 ) 
	= \lrpar{ 1+ 
	\tfrac{\sqrt t}{s\pi i}\; \df e(d\xi)(\cl(d\theta)\otimes \Gamma_1) }
	\exp\lrpar{ 
		-\tfrac t{s^2}(\eth_{\bbT/\bbS^1}^2 + \cl(ds)(\eth_{\bbT/\bbS^1} \otimes \Gamma_1) ) }\\
	\exp\lrpar{-  s^{1/2}(\bbA_{\bhs{N}^+/N,Q}^t)^2  s^{-1/2} }.
\end{multline*}
Now, from the odd families index theorem of Bismut-Freed, we have
\begin{multline*}
	\frac i{2\pi} 
	\int_{\bbS^1_{\xi}} \Tr\lrpar{ 
	\lrpar{ 1+ 
	\tfrac{\sqrt t}{s\pi i}\; \df e(d\xi)(\cl(d\theta)\otimes \Gamma_1) }
	\exp\lrpar{ 
		-\tfrac t{s^2}(\eth_{\bbT/\bbS^1}^2 + \cl(ds)(\eth_{\bbT/\bbS^1} \otimes \Gamma_1) ) }}\\
	=
	\frac i{2\pi} 
	\int_{\bbS^1_{\xi}} \Tr\lrpar{ 
	\tfrac{\sqrt t}{s\pi i}\; \cl(d\theta)
	\exp\lrpar{ 
		-\tfrac t{s^2}( (\cl(ds)\eth_{\bbT/\bbS^1} +\tfrac12)^2 ) }} \\
	= \text{spectral flow}(\cl(ds)\eth_{\bbT/\bbS^1}+1/2) = 1
\end{multline*}
and hence the $\cJ$-forms satisfy
\begin{equation*}
	\frac i{2\pi} \int_{\bbS^1_\xi} \cJ_Q(S^2\bhs{N}/SN)
	= \cJ_Q(\bhs{N}/N).
\end{equation*}

\begin{theorem}
Let $M \xlra\psi B$ be a fiber bundle of manifolds with corners and iterated fibration structures with odd-dimensional fibers, $E \lra M$ a wedge Clifford bundle with associated Dirac-type operator $\eth_{M/B}$ and compatible perturbation $Q.$
If $\eth_{M/B,Q}$ with its vertical APS domain satisfies the Witt assumption, then
\begin{multline*}
	\Ch_{\odd}(\Ind(\eth_{M/B,Q}), \nabla^{\Ind})
	=
	\int_{M/B} \hat A(M/B) \Ch'(E) 
	- \sum_{N \in \cS_{\psi}(M)} 
	\int_{N/B} \hat A(N/B) \cJ_Q(\bhs{N}/N) \\
	+ d\eta_Q(M/B).
\end{multline*}
\end{theorem}

\section{An extended index formula and the relation between $\cJ$ and
  $\eta$} \label{sec:extended}

In \S\ref{sec:ChernChar}, we have found a formula for the Chern character of the index of a family of wedge Dirac-type operators in terms of the Bismut-Cheeger $\cJ$-forms. In this section we establish the relation between the $\cJ$-forms and the $\eta$ forms. In the process we establish a families index theorem on manifolds with corners and an iterated fibration structure endowed with a metric that is of wedge `type' at all boundary hypersurfaces save one, where it is of `b' or asymptotically cylindrical type.\\

Before we start, we give an example to show that $\cJ$ and $\hat \eta$ do not coincide in general. Consider an embedded surface, $Y,$ in a closed spin 4-manifold $L$ and let $X = [L;Y],$ endowed with a wedge metric with constant cone angle $2\pi\beta,$
\begin{equation*}
	dx^2 + x^2\beta^2d\theta^2 + \phi_Y^*g_Y,
\end{equation*}
where $\phi_Y$ is the boundary fiber bundle $\bbS^1 \fib \pa X \xlra{\phi_Y} Y.$
It follows from \cite[Corollary 1.2, \S6.1]{Albin-GellRedman} that, if $\beta\leq 1,$
\begin{equation*}
	\int_Y\hat A(Y)\cJ(\pa X/Y) = \frac1{24}(\beta^2-1)[Y]^2, \quad
	\int_Y \hat A(Y)\eta(\pa X/Y) = \frac1{24}[Y]^2,
\end{equation*}
where $[Y]^2$ denotes the self-intersection number of $Y$ in $L.$

\subsection{b-c suspension}\label{sec:WarpedProd}

The ultimate aim of this section is to determine the relation between
the $\cJ$ and $\eta$ forms. We do not assume that the fiber bundles
treated here arise as boundary fiber bundle so to distinguish this
setting from that above we will use $\wc M$ instead of $M$ or $\bhs{N}$.  Starting with a fiber bundle
\begin{equation*}
	\wc X \fib \wc M \xlra{\wc\psi} \wc B,
\end{equation*}
we consider $\bbR^+_s \times \wc M$ together with a vertical metric that is of wedge type near $\{s=0\}$ and cylindrical away from $\{s=0\},$ which we refer to as a `b,wedge metric'. In this subsection we show that a wedge Clifford bundle on $\wc M$ with a compatible perturbation induces a b,wedge Clifford bundle on $\bbR^+\times \wc M$ with vanishing index. In the following subsection we will find an explicit formula for the Chern character of this index involving both $\eta(\wc M/\wc B)$ and $\cJ(\wc M/\wc B).$ \\

Let 
\begin{equation*}
	\underbrace{\bbR^+_s \times \wc X }_{\wc X^+} 
	\fib
	\underbrace{\bbR^+_s \times \wc M}_{\wc M^+}
	\xlra{\wc\psi^+}
	\wc B
\end{equation*}
be the extended fiber bundle.
Given a $\wc \psi$-vertical wedge metric $g_{\wc M/\wc B},$ let 
\begin{equation}\label{eq:DefMetgh}
	g^+_h = ds^2 + h(s)^2 g_{\wc M/\wc B}
\end{equation}
where $h$ is a smooth function satisfying
\begin{equation}\label{eq:DefTwist}
	h(s) = \begin{cases} s & \Mfor s<1 \\ 1 & \Mfor s>2 \end{cases}, \quad
	h(s)>0 \Mfor s>0.
\end{equation}

Just as in \S \ref{sec:BCJeta}, given a $\Cl(1)$-wedge Clifford bundle
$E \lra \wc M,$ we can endow the pull-back bundle $E \lra \wc M^+$
with an action of $\Cl({}^{\w}T^*\wc M^+/\wc B),$ a compatible
Hermitian metric $g^E$ and connection $\nabla^E.$ In particular the
connection $\nabla^{E^+}$ in \eqref{eq:E-prod-cone} and the Clifford
action \eqref{eq:warp-cliff-act} have obvious analogies when the
warping factor $s^2$ is replaced by an arbitrary warping factor $k(s)^2$.

Given a compatible perturbation $Q$ over $\wc M,$ let us also use $Q$ to denote the trivial extension of $Q$ to $\wc M^+.$\\

Consider a fiber $\wc X$ of $\wc \psi.$
Given a frame $\{V_1,\ldots, V_m\}$ for $\wc X$ over $\cU\subseteq \wc X,$ consider the frame $\{\pa_s, \wt V_i\}$ with $\wt V_i = \tfrac1h V_i,$ over $\bbR^+ \times \cU.$
Using the Koszul formula we can express the Levi-Civita connection of $g^+_h$ in terms of $h$ and the Levi-Civita connection of $g_{\wc X},$
\begin{equation*}
	\begin{tabular}{|c|c|c|} \hline
	$g^+_h(\nabla_{W_0}^+W_1,W_2)$ & $\pa_s $&$\wt V_2$   \\ \hline\hline
	$\nabla^+_{\pa_s}\pa_s $& $0 $& $0$ \\ \hline
	$\nabla^+_{\pa_s} \wt V_0 $&$ 0 $&$ 0$ \\ \hline
	$\nabla^+_{\wt V_0}\pa_s $&$ 0 $&$ \tfrac{h'}h g_{\wc X}(V_0, V_2) $ \\ \hline
	$\nabla^+_{\wt V_0}\wt V_1 $&$ -\tfrac{h'}h g_{\wc X}(V_0, V_1) $&
		$ \tfrac1h g_{\wc X}(\nabla_{V_0}V_1, V_2) $ \\ \hline
	\end{tabular}
\end{equation*}
Thus we have
\begin{equation*}
	\nabla^+_{\pa_s} = 0, \quad
	\nabla^+_{\wt V_0} \pa_s = \tfrac {h'}h \wt V_0, \quad
	\nabla^+_{\wt V_0} \wt V_1 = -\tfrac{h'}h g_{\wc X}(V_0, V_1) \pa_s + \tfrac1h \wt{\nabla_{V_0}V_1}
\end{equation*}
or, equivalently,
\begin{equation*}
	\nabla^+_{\pa_s} = 0, \quad
	\nabla^+_{V_0} \pa_s = h' \wt V_0, \quad
	\nabla^+_{V_0} \wt V_1 = -h' g_{\wc X}(V_0, V_1) \pa_s + \wt{\nabla_{V_0}V_1}.
\end{equation*}
This can be interpreted as saying that the Levi-Civita connection induces a connection on the `rescaled tangent bundle', locally spanned by $\{\pa_s, \wt V_i\}.$

Let $\{\theta^i\}$ be the dual coframe to $\{V_i\}$ on $\wc X,$ so that $\{ds, h\; \theta^i\}$ is a coframe on $\wc X^+,$ and the Dirac-type operator on $E$ is
\begin{equation*}
	D_{\wc X^+} 
	= \cl(ds) \nabla^E_{\pa_s} + \sum \cl(h \; \theta^i) \nabla^E_{\wt V_i}
	= \cl(ds) \pa_s + \sum \cl(h \; \theta^i) \nabla^E_{\wt V_i}.
\end{equation*}

We can use $\cl(ds)$ to split $E,$
\begin{equation*}
	E = E_{i} \oplus E_{-i}, \quad \cl(ds)\rest{E_{\pm i}} = \pm i,
\end{equation*}
and we have natural projections onto each summand,
\begin{equation*}
	\tfrac12(1\pm \tfrac1i \cl(ds)): E \lra E_{\pm i}.
\end{equation*}

Since $\nabla^E$ is a Clifford connection we have, for $\eps \in \{\pm1\},$ $\sigma \in \CI(\wc X^+;E),$ and $W$ a vector field satisfying $ds(W)=0,$
\begin{equation*}
\begin{multlined}
	\nabla_W^E (\tfrac12(1 + \tfrac{\eps}i \cl(ds) )\sigma)
	=\tfrac12(1 + \tfrac{\eps}i \cl(ds) ) \nabla_W^E \sigma
	+ \tfrac{\eps}{2i}\cl(\nabla_W ds) \sigma \\
	=\tfrac12(1 + \tfrac{\eps}i \cl(ds) ) \nabla_W^E \sigma
	+ \tfrac{\eps}{2i} \tfrac{h'}h \cl(W^{\flat}) \sigma,
\end{multlined}
\end{equation*}
or, in terms of the splitting of $E,$
\begin{equation*}
	\nabla^E_W = 
	\begin{pmatrix}
	\nabla^E_W & -\tfrac1{2i}\tfrac{h'}h \cl(W^{\flat}) \\
	\tfrac1{2i}\tfrac{h'}h \cl(W^{\flat}) & \nabla^E_W
	\end{pmatrix}.
\end{equation*}
Hence the associated Dirac-type operator satisfies
\begin{equation*}
\begin{gathered}
	D_{\wc X^+} = 
	\cl(ds)\pa_s + 
	\sum 
	\begin{pmatrix}
	0 & \cl(h\; \theta^i)\\
	\cl(h\; \theta^i) & 0 
	\end{pmatrix}
	\begin{pmatrix}
	\nabla^E_{\wt V_i} & -\tfrac1{2i}\tfrac{h'}h \cl(\wt V_i^{\flat}) \\
	\tfrac1{2i}\tfrac{h'}h \cl(\wt V_i^{\flat}) & \nabla^E_{\wt V_i}
	\end{pmatrix} \\
	= \begin{pmatrix}
	i(\pa_s + \tfrac m2 \tfrac{h'}h) & \tfrac1h\wc\eth_{\wc X} \\
	\tfrac1h\wc\eth_{\wc X} & -i(\pa_s + \tfrac m2 \tfrac{h'}h) 
	\end{pmatrix}
\end{gathered}
\end{equation*}
with $\wc\eth_{\wc X}= \sum \cl(\theta^i)\nabla_{V_i}^E.$
Thus 
\begin{equation*}
	h\eth_{\wc X^+,Q} = h^{m/2}(h(D_{\wc X^+}+Q))h^{-m/2} 
	= \begin{pmatrix}
	ih\pa_s  & \wc \eth_{\wc X,Q} \\
	\wc \eth_{\wc X,Q} & -ih\pa_s 
	\end{pmatrix}
	= \bigoplus_{\mu \in \Spec(\wc \eth_{\wc X,Q})}
	\begin{pmatrix}
	ih\pa_s  & \mu \\
	\mu & -ih\pa_s 
	\end{pmatrix},
\end{equation*}
where we have used that $\pa_s$ and $\wc \eth_{\wc X,Q}$ commute.
Define
\begin{equation*}
	R(s) = \int_1^s \frac{dt}{h(t)}
\end{equation*}
so that $\pa_R = h(s) \pa_s.$
If $\begin{pmatrix} a \\ b \end{pmatrix}$ is in the null space of $D_{\wc X^+}+Q$ and the $\mu$ eigenspace of $\wc\eth_{\wc X,Q},$ $\mu \neq 0,$ we have
\begin{equation*}
	a''(R) = \mu^2 a(R), \quad b''(R) = \mu^2 b(R)
	\implies
	\begin{pmatrix} a(R,\mu) \\ b(R,\mu) \end{pmatrix}
	= 
	a_1(\mu)
	\begin{pmatrix}  1 \\ -i \end{pmatrix} e^{\mu R}
	+
	a_2(\mu)
	\begin{pmatrix} 1 \\ i \end{pmatrix} e^{-\mu R}
\end{equation*}
while if $\mu=0,$ the solution consists of arbitrary constant vectors $\begin{pmatrix} a_0 \\ b_0 \end{pmatrix}.$

\begin{lemma}
If $\cD_{\VAPS}(\eth_{\wc X^+,Q})$ is the graph closure of $\cD_{\max}(\eth_{\wc X^+,Q}) \cap (h(s)^{1/2}H^1_\e(\wc X^+;E))$ then
\begin{equation*}
	(\eth_{\wc X^+,Q}, \cD_{\VAPS})
\end{equation*}
is self-adjoint and invertible with bounded inverse.
\end{lemma}

\begin{proof}
The conditions \eqref{eq:DefTwist} on $h$ imply that $R(s) = \log s$ for all $s<1$ and $R(s) = s + C$ for some constant $C$ and $s\gg0.$
Since $R(s) = \log s$ for all $s<1,$ the elements of the null space of $\eth_{\wc X^+,Q}$ that are in the $\mu$-eigenspace of $\wc\eth_{\wc X,Q}$ are of the form
\begin{equation*}
	a_1(\mu)
	\begin{pmatrix}  1 \\ -i \end{pmatrix} s^{\mu}
	+
	a_2(\mu)
	\begin{pmatrix} 1 \\ i \end{pmatrix} s^{-\mu}
\end{equation*}
Now
\begin{equation*}
	\int_0^1 s^k \; ds 
	< \infty \iff k>-1
\end{equation*}
shows that for a solution to be in $L^2$ for $s<1$ it must be of the form $e^{\mu R}$ for $\mu >-\tfrac12.$

On the other hand, since $R(s) = C + s$ for $s\gg0$ and 
\begin{equation*}
	\int_1^{\infty} e^{k s} \; ds 
	< \infty \iff k <0
\end{equation*}
shows that for a solution to be in $L^2$ for $s>1$ it must be of the form $e^{\mu R}$ for $\mu <0.$

In particular there are no elements in the null space of $\eth_{\wc X^+}$ with the domain $\cD_{\VAPS}(\eth_{\wc X^+}).$ \\

Consider the operator $I_b(\eth_{\wc X^+,Q}) = \cl(ds)\pa_s + \wc\eth_{\wc X,Q}$ over $\bbR_s \times \wc X.$
Since $\wc \eth_{\wc X,Q}$ is self-adjoint on $\wc X,$ $I_b(\eth_{\wc X^+,Q})$ is self-adjoint on $\bbR \times \wc X.$
Note that the square of $I_b(\eth_{\wc X^+,Q})$ is $-\pa_s^2 + (\wc\eth_{\wc X,Q})^2$ and is bounded below by the smallest eigenvalue of $(\wc\eth_{\wc X})^2,$ which is positive by the Witt assumption.\\

Let $v \in \cD_{\max}(\eth_{\wc X^+})$ be such that 
\begin{equation*}
	\ang{\eth_{\wc X^+}u, v} = \ang{u, \eth_{\wc X^+}v} 
\end{equation*}
for all $u \in \cD_{\VAPS}(\eth_{\wc X^+}).$
By choosing $u$ with support in $\{s \leq C\},$ and using the self-adjointness of wedge Dirac-type operators with compatible perturbations on manifolds with corners and iterated fibration structures, we see that $v$ is in the vertical APS domain of $\eth_{\wc X^+}$ at any boundary hypersurface of $\wc X.$
By choosing $u$ with support in $\{s \geq C >1\},$ and using the self-adjointness of $I_b(\eth_{\wc X^+}),$ we see that $v \in \cD_{\VAPS}(\eth_{\wc X^+})$ and hence this is a self-adjoint domain.

Similarly, the fact that wedge Dirac-type operators with compatible perturbations on manifolds with corners and iterated fibration structures have closed range and the lower bound for $I_b(\eth_{\wc X^+})$ combine to show that $(\eth_{\wc X^+}, \cD_{\VAPS})$ has closed range. As we have already shown that this operator is injective, it follows that it has a bounded inverse.
\end{proof}

Since we have shown that the individual operators in the family
$\eth_{\wc M^+/\wc B,Q}$ are invertible, its families index is
identically zero.

\subsection{Extended families index formula}\label{sec:BWInd}

To exploit the vanishing of the families index of $\eth_{\wc M^+/\wc B,Q}$ we will work out an extension of the discussion of the families index above. To distinguish this setting from that above we will use $M'$ instead of $M,$ etc. For simplicity we only consider the case of even dimensional fibers; the odd dimensional case can be established by suspension as above.
For the most part the constructions above extend easily to the case we will consider here, in which case we will simply indicate the changes necessary.\\

Let $M' \lra B'$ be a locally trivial family of manifolds with corners and iterated fibration structures over $B'$ as in Definition \ref{def:FibIFS}. Assume that a minimal element $N_0' \in \cS(M')$ is such that $\dim N_0'/B' = 0$ and let $\rho_{N_0'}$ be a boundary defining function for $\bhs{N_0'}.$ By a {\bf b-wedge metric} on $M'$ (with respect to $N_0'$) we mean a metric conformally related to a totally geodesic wedge metric on $M',$
\begin{equation*}
	g_{M'/B',b-\w} = \rho_{N_0'}^{-2}g_{M'/B'}.
\end{equation*}
In particular this is a metric on $TM^{\circ}/B$ that near $\bhs{N_0'}$ takes the form
\begin{equation*}
	\frac{d\rho_{N_0'}^2}{\rho_{N_0'}^2} + g_{\bhs{N_0'}/N_0'}
\end{equation*}
with $g_{\bhs{N_0'}/N_0'}$ a vertical family of wedge metrics, and, for any other $N' \in \cS(M'),$ near $\bhs{N}$ takes the form 
\begin{equation*}
	dx^2 + x^2 g_{\bhs{N'}/N'} + \phi_{N'}^*g_{N'/B'}
\end{equation*}
with $g_{\bhs{N'}/N'}$ a vertical wedge metric, while $g_{N'/B'}$ is a family of b-wedge metrics if $N_0' < N'$ and a family of wedge metrics if $\bhs{N_0'} \cap \bhs{N'}=\emptyset.$
This is best understood as a non-degenerate bundle metric on the bundle
\begin{equation*}
	{}^{b,\w}TM'/B' = \rho_{N_0'}^2 {}^\w TM'/B'.
\end{equation*}

A Clifford b-wedge bundle (with respect to $N_0'$) is defined just as in Definition \ref{def:WedgeCliffMod} but with an action of 
$\bbC \otimes \mathrm{Cl}\lrpar{ {}^{b,\w} T^*M'/B', g_{M'/B',b-\w} }.$
We denote the corresponding Dirac-type operator by $D_{M'/B'}$ and, if $Q$ is a compatible perturbation, $D_{M'/B'}+Q$ will be denoted $D_{M'/B',Q}.$ We assume that the perturbation has stabilized the index, so that $\ker D_{M'/B',Q}$ is a vector bundle over $B'.$

Define a multiweight on $M'$ by $\mf b'(H) = 0$ if $H \subseteq N_0'$ and $\mf b'(H) = \mf b(H)$ otherwise (where $\mf b$ is defined in \eqref{eq:DefMfB}). Let
\begin{equation*}
	L^2(M'/B';E) = \rho_M^{-\mf b'}L^2_{b,\w}(M'/B';E)
\end{equation*}
where the latter is defined using the $b$-wedge metric on $M'/B'$ and the Hermitian metric on $E.$
Define $\eth_{M'/B',Q}$ to be the operator $\rho_M^{\mf b}D_{M'/B',Q}\rho_M^{-\mf b},$ so that $\eth_{M'/B',Q}$ acting on $L^2(M'/B';E)$ is unitarily equivalent to $D_{M'/B',Q}$ acting on $L^2_{b,\w}(M'/B';E).$
We define the vertical APS domain to be the graph closure of 
\begin{equation*}
	\cD_{\max}(\eth_{M'/B',Q})\cap 
	\lrpar{\prod_{N' \in \cS_{\psi}(M)\setminus \{N_0'\} } 
	\rho_{N'}^{1/2} H^1_\e(M'/B';E)}.
\end{equation*}

For each $N' \in \cS(M'),$ there is a boundary operator of $\eth_{M'/B',Q}$ given by
\begin{equation*}
	D_{N_0'/B',Q} = \eth_{N_0'/B',Q}\rest{\bhs{N_0'}}, \quad
	D_{N'/B'} = \rho_{N'}\eth_{N'/B',Q}\rest{\bhs{N'}} \Mif N'\neq N_0',
\end{equation*}
all of which are families of wedge Dirac-type operators. The Witt assumption in this case is that each of these boundary operators are invertible. Just as in \S \ref{sec:Res}, under the Witt condition we can construct a generalized inverse of $\eth_{M'/B',Q}$ with compact errors within the edge calculus. However in this case, the proof of Proposition \ref{prop:parametrix} should be modified at $\bhsd{\phi\phi}(N_0')$ because $D_{N_0'/B',Q}$ is simply the restriction to $\bhs{N_0'}$ without having to multiply by the boundary defining function. For this reason the generalized inverse has order zero at $\bhsd{\phi\phi}(N_0')$ while having order one at $\bhsd{\phi\phi}(N')$ for $N'\neq N_0'.$
The upshot is that the generalized inverse is not compact and so does not guarantee discrete spectrum. Indeed, one can argue as in \cite{tapsit} and see that the spectrum will not be discrete.
Nevertheless, the Witt assumption does guarantee that $\eth_{M'/B',Q}$ is a smooth family of self-adjoint Fredholm operators.\\

The heat kernel construction is, as we now briefly describe, an easy amalgamation of the heat kernel construction in \cite[Chapter 7]{tapsit} for the `b'-face corresponding to $\bhs{N_0'}$ and the heat kernel construction of \S\ref{sec:Heat} at the other boundary hypersurfaces.

The b-wedge heat space is given by
\begin{multline*}
	H(M'/B')_{b,\w} =
	\Big[ M'\times_{\psi'} M' \times \bbR^+_{\tau}; \bhs{N_0'}
        \times_{\phi_{N_0'}} \bhs{N_0'} \times \bbR^+_{\tau}; \dots \bhs{N_\ell'} \times_{\phi_{N_\ell'}} \bhs{N_\ell'} \times \bbR^+_{\tau}; \\	 
        \bhs{N_1'} \times_{\phi_{N_1'}} \bhs{N_1'} \times \{ 0\}; \ldots ; 
	\bhs{N_{\ell}'} \times_{\phi_{N_\ell'}} \bhs{N_\ell'} \times
        \{ 0\} \Big],
\end{multline*}
where $\{ N_0', N_1'. \ldots, N_{\ell}'\}$ is a non-decreasing list of $\cS(M').$
(The difference between this heat space and the wedge heat space in \S\ref{sec:Heat} is that there is no boundary hypersurface corresponding to $\bhs{N_0'}$ at time zero, as is to be expected from \cite[Chapter 7]{tapsit}.)
The composition heat space is described in Appendix \ref{sec:bwHeatComp} below, where a composition result is established.

Blowing-up $\bhs{N_0'} \times_{\phi_{N_0'}} \bhs{N_0'} \times \bbR^+_{\tau}$ results in a collective boundary hypersurface $\bhsh{11,0}(N_0')$ which we can identify with $\bbR^+_s \times H(\bhs{N_0'}/B)_{\w}$ and at which the model heat operator is 
\begin{equation*}
	\pa_t - (-\pa_s^2 + D_{N_0'/B',Q}^2).
\end{equation*}
Hence the model heat kernel at this face is $e^{-t(-\pa_s^2)}e^{-tD_{N_0'/B',Q}^2}.$
The other blow-ups produce model problems that are identical to the ones in \S\ref{subsec:HeatKer}.
Once the model problems are solved we can solve away the expansion at each face using the composition result from Appendix \ref{sec:bwHeatComp} and obtain the heat kernel as an element of 
\begin{equation*}
	e^{-t\eth_{M'/B',Q}^2} \in
	\sB^{\cH/\mf I^{(H)}}_{phg}\sA^{-m-1}_-(H(M'/B')_{b,\w}; \Hom(E)\otimes \Omega_{\mf h,R})
\end{equation*}
where the index set $\cH$ and multiweights $\mf I^{(H)},$ $\mf h,$ are defined as before for $\bhsh{10,0}(N'), \bhsh{01,0}(N')$ with $N' \neq N_0',$ and are given by
\begin{equation*}
\begin{gathered}
	\cH(\bhsh{11,0}(N_0')) = \bbN_0, \quad
	\cH(\bhsh{10,0}(N_0')) = 
	\cH(\bhsh{01,0}(N_0')) = \emptyset, \\
	\mf I^{(H)}(\bhsh{11,0}(N_0')) = 
	\mf I^{(H)}(\bhsh{10,0}(N_0')) =
	\mf I^{(H)}(\bhsh{01,0}(N_0')) = \infty, \\
	\mf h(\bhsh{11,0}(N_0')) = -1, \quad
	\mf h(\bhsh{10,0}(N_0')) = 
	\mf h(\bhsh{01,0}(N_0')) = 0,
\end{gathered}
\end{equation*}
at the collective boundary hypersurfaces associated to $N_0'.$\\

Just as in \cite[Chapter 7]{tapsit}, the heat kernel is not trace-class because at $\bhsh{11,0}(N_0')$ it is $\cO(\rho_{N_0'}^{-1})$ times a non-degenerate density. However the renormalized (fibrewise) trace of the heat kernel,
\begin{equation*}
	\RTr{e^{-t\eth_{M'/B',Q}}} 
	= \Rint_{M'/B'} \tr(e^{-t\eth_{M'/B',Q}})\rest{\diag_{M'}}
	= \FP_{z=0}\int_{M'/B'} \rho_{N_0'}^z \; \tr(e^{-t\eth_{M'/B',Q}})\rest{\diag_{M'}},
\end{equation*}
will stand in for the trace as it does in, e.g., \cite{tapsit, Melrose-Piazza:Even}. (For more on these renormalizations, see \cite{Albin:RenInt} and \cite{Albin:RenInd} for another application to an index theorem.)
In particular, the renormalized trace converges as $t\to\infty$ to the dimensions of the null spaces of $\eth_{M'/B',Q},$ and the corresponding renormalized supertrace converges to the index, while as $t\to 0,$ the renormalized trace has short-time asymptotics as before.

Thus the renormalized supertrace mediates between the index and the short-time asymptotic expansion of the heat kernel but crucially the renormalized supertrace does depend on $t.$
This dependence can be computed via the trace-defect formula,
\begin{multline*}
	\RTr{[A,B]} 
	= \FP_{z=0} \Tr(\rho_{N_0'}^z[A,B])
	= \FP_{z=0} \Tr( [\rho_{N_0'}^z A, B] - A [\rho_{N_0'}^z,B])
	= \FP_{z=0} \Tr (A (B\rho_{N_0'}^z - \rho_{N_0'}^z B)) \\
	= \FP_{z=0} \Big[z \Tr \lrpar{A \rho_{N_0'}^z \frac{\rho_{N_0'}^{-z}B\rho_{N_0'}^z -  B}z}\Big]
	= \frac1{2\pi i}\int_{\bbR} \Tr( I_{N_0'}(A;\sigma)\pa_{\sigma} I_{N_0'}(B;\sigma)) \; d\sigma
\end{multline*}
where, e.g., $I_{N_0'}(A;\sigma)$ is the restriction of $\rho_{N_0'}^{-z}A\rho_{N_0'}^z$ to $\rho_{N_0'}=0,$ and gives rise to the $\eta$ invariant.\\

The next step in obtaining the families index theorem is to fix a connection for $\psi',$ and define the Bismut superconnection just as in \S\ref{sec:BismutSup}, $\bbA_{M'/B',Q}.$
The construction of the heat kernel of $\bbA_{M'/B',Q}$ follows quickly from that of $\eth_{M'/B',Q}$ and the composition results from Appendix \ref{sec:bwHeatComp} as in Theorem \ref{thm:WedgeRescaledHeatKer}. (The rescaling only takes place at collective boundary hypersurfaces $\bhsh{\phi\phi,1}(N')$ with $N' \neq N_0',$ and so proceeds exactly as before.) If the null spaces of $\eth_{M'/B',Q}$ do not form a vector bundle over $B$ we find a smoothing perturbation as in \cite[Lemma 1.1]{Melrose-Rochon:FCusp} that is compactly supported in the interior of $M'$ (and hence does not affect any of our other arguments) and incorporate this perturbation into $Q$ without further comment.\\

\begin{theorem}\label{thm:extended-thm}
Let $M' \xlra{\psi'} B'$ be a fiber bundle of manifolds with corners and iterated fibration structures, such that $\dim M'/B'$ is even, with a minimal element $N_0'$ such that $\dim N_0'/B'=0,$ $E \lra M'$ a $\bbZ_2$-graded b,wedge Clifford bundle with associated Dirac-type operator $\eth_{M'/B',Q}.$
If $\eth_{M'/B',Q}$ with its vertical APS domain satisfies the Witt assumption, then
\begin{multline*}
	\Ch_{\ev}(\Ind(\eth_{M'/B',Q}), \nabla^{\Ind})
	= \int_{M'/B'} \hat A(M'/B') \Ch'(E) 
	+ \bar\eta_Q(\eth_{\bhs{N_0'}/B'}) \\
	- \sum_{N' \in \cS_{\psi'}(M')\setminus \{N_0'\}} 
	\int_{N'/B'} \hat A(N'/B') \cJ_Q(\bhs{N'}/N') 
	+ d \int_0^{\infty}
	{}^R\Str\lrpar{ \frac{\pa \bbA^t_{M'/B'}}{\pa t} e^{-(\bbA^t_{M'/B'})^2} } \; dt
\end{multline*}
where $\bar\eta_Q(\eth_{\bhs{N_0'}/B'})$ is the normalized Bismut-Cheeger $\eta$-form.
\end{theorem}

\subsection{b-c suspension families index formula}
Let us return to the context of the family $\eth_{\wc M^+/\wc B,Q}$
from \S\ref{sec:WarpedProd}, where we showed that this operators
familes index vanishes.

We compactify $\wc M^+ = \bbR^+_s \times \wc M$ to
\begin{equation*}
	[0,1]_\sigma \times \wc M
\end{equation*}
using the logarithm so that the warped product metric $g^+$ is a
b-metric near $\{1\} \times \wc M.$ 
We continue to denote $[0,1]_\sigma \times \wc M$ by $\wc M^+.$
Note that $\wc M^+$ is naturally a locally trivial fiber bundle of manifolds with corners and iterated fibration structures, with each of $\{0\} \times \wc M,$ $\{1\} \times \wc M$ a collective boundary hypersurface over $\wc B,$ and in sum
\begin{equation*}
	\cS(\wc M^+) 
	= \{ \wc B \}
	\cup \{ [0,1]_\sigma \times \wc N = \wc N^+: \wc N \in \cS_{\wc \psi}(\wc M) \}.
\end{equation*}

By the Theorem \ref{thm:extended-thm} above we have
\begin{multline}\label{eq:Jeta1}
	\cJ_Q(\wc M/\wc B) 
	-\bar\eta_Q(\wc M/\wc B) \\
	=\int_{\wc M^+/\wc B} \hat A(\wc M^+/\wc B) \Ch'(E) 
	+ \sum_{ \wc N \in \cS_{\wc \psi}(\wc M)}
	\int_{\wc N^+/\wt B} \hat A( \wc N^+/\wt B) \cJ_Q(\bhs{\wc N}^+/\wc N^+) 
	+ d \eta_{b-\w,Q}.
\end{multline}
We will simplify this formula by carrying out the integrals over $[0,1]_\sigma.$
$ $\\

More generally for any connection on $T\wc M^+/\wc B,$ $\nabla,$ and
any polynomial $f,$ 
let
\begin{equation*}
	A_f(\nabla) = \Tr(f(\nabla^2)) \in \Omega^*(\wc M^+)
\end{equation*}
where $\nabla^2$ denotes the curvature of $\nabla.$
Following, e.g., \cite[Proposition 1.41]{BGV2004}, 
given two connections $\nabla,$ $\nabla',$ we fix a transgression form $TA_f(\nabla, \nabla')$ satisfying $dTA_f(\nabla, \nabla') = A_f(\nabla')- A_f(\nabla)$ by the formula
\begin{equation*}
	TA_f(\nabla, \nabla') = \int_0^1\Tr( \frac{\pa \nabla_t}{\pa
          t}  f'(\nabla_t^2))\; dt
\end{equation*}
where $\nabla_t = (1-t)\nabla + t \nabla'.$

As we have fixed a connection for $\wc M^+ \lra \wc B,$ we get a connection on $T\wc M^+/\wc B,$ for each choice of vertical metric. We are particularly interested in the connections
\begin{equation*}
	\nabla^h \leftrightarrow ds^2 + h(s)^2 g_{\wc M/\wc B},
	\quad
	\nabla^{con} \leftrightarrow ds^2 + s^2 g_{\wc M/\wc B},
	\quad
	\bar{\nabla}  \leftrightarrow ds^2 + g_{\wc M/\wc B}.
\end{equation*}
Let $\pi: \wc M^+ \lra \wc M$ be the natural projection and $j: \wc M \lra \wc M^+$ the inclusion of the left endpoint,.
\begin{proposition}
For any polynomial $f,$ 
we have
\begin{equation*}
	\pi_*(A_f(\nabla)) = j^*(TA_f(\bar \nabla, \nabla^h))
	= TA_f(\bar{\nabla}, \nabla^h)\rest{s=0}
	= TA_f(\bar{\nabla}, \nabla^{con})\rest{s=0}.
\end{equation*}
\end{proposition}

\begin{proof}
The middle equality holds by definition and the final holds because $h
\equiv s$ near $s = 0$, so we focus on the first equality.
$ $

Let $\{\bar e_i, \bar f_\alpha\}$ be a local frame for $T\wc M,$ in which $\{\bar e_i\}$ constitute an orthonormal frame for $T\wc M/\wc B,$ and let $\{\bar e^i, \bar f^{\alpha}\}$ be the dual coframe. Denote their lifts/pull-backs to $\wc M^+$ by the same symbols.
Let
\begin{equation*}
	\{ V_a \} = \{ \pa_s, \tfrac 1h \bar e_i \}, \quad
	\{ V^a \} = \{ ds, h\; \bar e^i \}
\end{equation*}
be the corresponding frames for $T\wc M^+/\wc B$ on $\wc M^+.$
(We will use $a,b,c$ for indices that begin at $0$ and $i,j,\ell$ for indices corresponding to $T\wc M/\wc B.$)

Let $\omega$ and $\bar\omega$ denote the one-form matrices corresponding to these frames and the connections $\nabla,$ $\bar\nabla,$ by
\begin{equation*}
	\nabla V_a = \omega_a^b V_b, \quad
	\bar\nabla V_a = \bar\omega_a^b V_b.
\end{equation*}
Let $d_{\wc M^+}$ denote the exterior derivative on $\wc M^+,$ $d_{\wc M^+/\wc B}$ the part of the exterior derivative that raises the $\wc\psi^+$-vertical degree by one, and $\hat d_{\wc M}$ the difference between these two, so that
\begin{equation*}
	d_{\wc M^+} = d_{\wc M^+/\wc B} + \hat d_{\wc M}.
\end{equation*}
(See, e.g., \cite[Proposition 10.1]{BGV2004}, \cite[Proposition 14]{HHM2004}, for descriptions of $\hat d_{\wc M}.$)
It is easy to see that the forms $\omega_a^b$  satisfy
\begin{equation*}
	d_{\wc M^+/\wc B} V^a = V^b \wedge \omega_b^a
\end{equation*}
and hence 
\begin{equation*}
	\omega_0^j = -\omega_j^0 = h'(s) \;\bar e^j, \quad
	\omega_i^j = \bar\omega_i^j
\end{equation*}
Then $\theta =\omega-  \bar\omega$ satisfies
\begin{equation*}
	\theta_0^j = -\theta_j^0 = h'(s)\; \bar e^j, \quad
	\theta_i^j = 0 \Motherwise.
\end{equation*}

Let $\nabla(t) = (1-t)\bar\nabla + t \nabla = \bar\nabla + t(\nabla-\bar\nabla)$ so that its connection one-form is $\omega(t) =\bar\omega + t\theta$ and its curvature $\Omega(t) = d\omega(t) -\omega(t)\wedge\omega(t)$ is given by
\begin{equation*}
\begin{gathered}
	\Omega_0^j(t) = -\Omega_j^0(t) 
	= d( t h'(s)\; \bar e^j) - (th'(s)\; \bar e^\ell)\wedge (\bar \omega_\ell^j )
	= t h''(s)\; ds\wedge \bar e^j + th'(s) \; \hat d_W \bar e^j \\
	\Omega_i^j(t) 
	=\bar\Omega_i^j
	+ t^2 (h'(s))^2 \; \bar e^i\wedge\bar e^j
\end{gathered}
\end{equation*}
where $\bar\Omega$ denotes the curvature of $\bar\nabla.$

We can write these expressions succinctly in terms of the one-form valued matrix given by
\begin{equation*}
	\gamma_0^j = -\gamma_j^0 = \bar e^j, \quad \gamma_i^j = 0 \Mforall i, j.
\end{equation*}
Indeed, 
\begin{equation*}
	(\gamma^2)_i^j = \gamma_i^0\wedge\gamma_0^j = -\bar e^i \wedge \bar e^j, \quad
	(\gamma^2)_a^b=0 \Motherwise
\end{equation*}
and hence
\begin{equation*}
	\theta = h'(s) \gamma, \quad
	\Omega(t) 
	= \bar\Omega + t h''(s) \; ds \wedge \gamma + t h'(s) \; \hat d_W \gamma  
		- t^2(h'(s))^2 \; \gamma^2.
\end{equation*}

Now let $f$ be as in the statement of the proposition and note that, since $h'(0)=1,$
\begin{equation*}
	j^*(TA_f) = j^*\lrpar{ 
	\int_0^1 h'(t) \Tr(\gamma f'(\Omega(t))) \; dt }
	= \int_0^1 \Tr(\gamma f'(\bar \Omega + t \; \hat d_W \gamma  
		- t^2\; \gamma^2) )\; dt.
\end{equation*}
On the other hand, if we denote $\Omega = \wt\Omega + h''(s) \; ds \wedge\gamma,$ then as in \cite[pg.48-49]{BGV2004} we have
\begin{equation*}
	\Tr(f(\Omega)) = \Tr(f(\wt\Omega)) + h''(s) \; ds \wedge \Tr(\gamma f'(\wt\Omega)),
\end{equation*}
and so
\begin{multline*}
	\pi_*(A_f) 
	= \int_0^1 h''(s) \Tr(\gamma f'(\wt\Omega)) \; ds
	= \int_0^1 h''(s) \Tr(\gamma f'( 
	\bar\Omega + h'(s) \; \hat d_W \gamma  
		- (h'(s))^2 \; \gamma^2) ) \; ds \\
	\xlra{t = h'(s)}
	\int_0^1 \Tr(\gamma f'( 
	\bar\Omega + t \; \hat d_W \gamma  
		- t^2 \; \gamma^2) ) \; dt
\end{multline*}
which coincides with $j^*(TA_f)$ as required.
\end{proof}

To apply this to simplify formula \eqref{eq:Jeta1}, note that 
\begin{equation*}
	\pa_s (\cJ_Q(\bhs{N'}^+/N'^+)) = 0 \Mand \pa_s\Ch'(E)=0 =
        \iota_{\pa_s}\Ch'(E).
\end{equation*}
For the twisted Chern character this follows from the fact that the twisted curvature is independent of $s.$
On the other hand, the $\cJ$ forms only depend on the vertical metric, and are unchanged by rescaling the metric. The warping factor $h,$ as it only depends on $s,$ has the effect of rescaling the vertical metric at each $s.$

Applying the proposition to simplify the index formula \eqref{eq:Jeta1} yields the first part of the following theorem. Taking the exterior derivative and applying the expression for $d\bar\eta$ yields the second part. We introduce the abbreviation
\begin{equation*}
	\hat A_c(\wc M/\wc B) = \hat A(\nabla^{con})\rest{s=0}, \quad
	T\hat A(\wc M/\wc B) = TA_f(\nabla^{cyl}, \nabla^{con})\rest{s=0}.
\end{equation*}
and similarly for $\wc N/\wc B.$

\begin{theorem}\label{thm:Jeta}
\begin{multline*}
	\cJ_Q(\wc M/\wc B) 
	-\bar\eta_Q(\wc M/\wc B) \\
	=\int_{\wc M/\wc B} T\hat A(\wc M/\wc B) \Ch'(E) 
	+ \sum_{ \wc N \in \cS_{\wc \psi}(\wc M)}
	\int_{\wc N/\wc B} T\hat A( \wc N/\wc B) \cJ_Q(\bhs{\wc N}/\wc N)
	+ d\eta_{b-\w, Q}.
\end{multline*}
Moreover
\begin{equation*}
	d\cJ_Q(\wc M/\wc B)
	=\int_{\wc M/\wc B} \hat A_c(\wc M/\wc B) \Ch'(E) 
	+ \sum_{ \wc N \in \cS_{\wc \psi}(\wc M)}
	\int_{\wc N/\wc B} \hat A_c( \wc N/\wc B) \cJ_Q(\bhs{\wc N}/\wc N).
\end{equation*}
\end{theorem}

\begin{remark}
  In the case $\wc B$ is a point, the deduction of the first part of Theorem
  \ref{thm:Jeta} from \eqref{eq:Jeta1} follows from
  \cite[Section 5]{Albin-GellRedman}.  Since this follows from the the
  preceeding proof also we only briefly sketch the argument.  To
  evaluate the term $\int_{\wc  X^+} \hat A(\wc X^+) \Ch'(E)$ in \eqref{eq:Jeta1},
  using that $\Ch'(E)$ is closed and that $\hat A(\nabla^h) - \hat
  A(\nabla^{cyl}) = dT\hat A(\nabla^h, \nabla^{cyl})$.  Since both
  $A(\nabla^{cyl})$ and $\Ch'(E)$ have no $ds$ component, the integral
  is $\int_{\wc X^+} d T\hat A(\nabla^h, \nabla^{cyl}) \Ch'(E)$.  From
  \emph{loc.\ cit.\ }, at the interior of each boundary hypersurface,
  the transgression splits into according to whether the base or
  vertical connections change; at the $s = 0$ only the vertical metric
  changes and there $T\hat A(\nabla^h, \nabla^{cyl}) \rvert_{s = 0} = T
    \hat{A}(\nabla^{con}, \nabla^{cyl})$ while at the other boundary
    faces (the $\hat Y^+$) the base metric changes from $ds^2 + g_Y$
    to $ds^2 + h(s)^2 g_Y$ and as discussed the transgression of $\hat
    A$ for such metrics is zero.
\end{remark}

When the fibers of $\wc M \lra B$ are closed manifolds this reads
\begin{equation*}
	\cJ_Q(\wc M/\wc B) - \bar\eta_Q(\wc M/\wc B) 
	= \int_{\wc M/\wc B} T\hat A(\wc M/\wc B)\Ch'(E) + d \eta_{b-\w,Q}
\end{equation*}
which is consistent with \cite[Main Theorem]{Albin-GellRedman}.
Note that if $\wc B$ is a point then by conformal invariance $T\hat
A(\nabla^{con}, \nabla^{cyl})=0$, so $\cJ_Q(\wc M) = \eta_Q(\wc M)$
However as noted above this is generally not the case.

\appendix
\section{Composition of edge pseudodifferential operators} \label{sec:CompEdge}

Let $M \xlra{\psi} B$ be a family of manifolds with corners and iterated fibration structures as in Definition \ref{def:FibIFS}. In this section we prove the composition formula for families of edge pseudodifferential operators acting on the fibers of $\psi.$\\

{\bf Edge triple space.}
Our construction of the triple space is the natural combination of \cite[Appendix]{Melrose-Piazza:K} and \cite[\S3]{Mazzeo:Edge}.
Let
\begin{equation*}
\begin{gathered}
	M^2_{\psi} = M \times_{\psi} M = \{ (\zeta, \zeta') \in M^2: 
	\psi(\zeta) = \psi(\zeta') \}, \\
	M^3_{\psi} = \{ (\zeta, \zeta', \zeta'') \in M^3: 
	\psi(\zeta) = \psi(\zeta') = \psi(\zeta'') \}.
\end{gathered}
\end{equation*}
We start with the three natural projections,
\begin{equation*}
	\xymatrix{
	& M^3_{\psi} \ar[ld]^{\pi_{LM}} \ar[d]^{\pi_{LR}} \ar[rd]^{\pi_{MR}} & \\
	M^2_{\psi} & M^2_{\psi} & M^2_{\psi} }
	\quad
	\xymatrix{
	& (\zeta, \zeta', \zeta'') \ar@{|->}[ld]^{\pi_{LM}} \ar@{|->}[d]^{\pi_{LR}} \ar@{|->}[rd]^{\pi_{MR}} & \\
	(\zeta, \zeta') & (\zeta, \zeta'') & (\zeta', \zeta'') }
\end{equation*}
and we will modify $M^3_{\psi}$ so as to end up with $b$-fibrations down to $(M/B)^2_\e.$

For each $N \in \cS_{\psi}(M),$ let
\begin{equation}\label{eq:TripDiag}
\begin{gathered}
	T(N) = \{ (\zeta, \zeta', \zeta'') \in \bhs{N}^3: 
	\phi_N(\zeta) = \phi_N(\zeta') = \phi_N(\zeta'') \}, \quad
	S_{LM}(N) = \pi_{LM}^{-1}(\bhs{N}\times_{\phi_N}\bhs{N}), \\
	S_{LR}(N) = \pi_{LR}^{-1}(\bhs{N}\times_{\phi_N}\bhs{N}), \quad
	S_{MR}(N) = \pi_{MR}^{-1}(\bhs{N}\times_{\phi_N}\bhs{N}).
\end{gathered}
\end{equation}
Let $\cS_{\psi}(M) = \{N_1, N_2, \ldots, N_{\ell}\}$ be a listing of $\cS_{\psi}(M)$ with non-decreasing depth.

We construct $(M/B)^3_{\e}$ in two steps.
In the first step, we blow-up the submanifolds $T(N_i)$ in order,
\begin{equation*}
	M^3_{\psi,T} = 
	\Big[ M^3_{\psi}; T(N_1); T(N_2); \ldots; T(N_{\ell}) \Big].
\end{equation*}
As in \S\ref{sec:DoubleSpace}, $T(N_1)$ is a p-submanifold of $M^3_{\psi}$ and each $T(N_i)$ is a p-submanifold once we have blown-up $T(N_j)$ for all $N_j < N_i.$

The second step is to blow-up the submanifolds 
$\{ S_{\bullet}(N) \}_{\bullet \in \{LM, LR, MR
\}\}}$ as the $N$ range over $\cS_{\psi}(M)$.  As we will see below, these are separated by the blow-ups
performed in the first step.
Thus,
\begin{equation*}
	(M/B)^3_{\e} =
	\Big[ M^3_{\psi,T};
	(S_{LM}(N_1) \cup S_{LR}(N_1) \cup S_{MR}(N_1)); \ldots;
	(S_{LM}(N_{\ell}) \cup S_{LR}(N_{\ell}) \cup S_{MR}(N_{\ell})) \Big].
\end{equation*}
We denote the blow-down map by $\beta_{(3)}:(M/B)^3_\e \lra M^3_{\psi}.$

We denote the collective boundary hypersurfaces obtained from these blow-ups by
\begin{equation*}
	T(N) \leftrightarrow \bhst{\phi\phi\phi}(N), \quad
	S_{LM}(N) \leftrightarrow \bhst{\phi\phi0}(N), \quad
	S_{LR}(N) \leftrightarrow \bhst{\phi0\phi}(N), \quad
	S_{MR}(N) \leftrightarrow \bhst{0\phi\phi}(N).
\end{equation*}
We denote the other collective boundary hypersurfaces by
\begin{equation*}
	\bhss{N} \times_{\psi} M^2_{\psi} \leftrightarrow \bhst{100}(N), \quad
	M \times_{\psi} \bhss{N} \times_{\psi} M \leftrightarrow \bhst{010}(N), \quad
	M^2_{\psi} \times_{\psi} \bhss{N} \leftrightarrow \bhst{001}(N).
\end{equation*}

\begin{proposition}\label{thm:trip-space-maps}
  The projections $\pi_{\bullet}$, $\bullet \in \{LM, LR, MR \}$
  lift to b-fibrations $\beta_{\bullet} \colon (M/B)^3_\e \lra
  (M/B)^2_\e$ 
\begin{equation*}
	\xymatrix{
	& (M/B)^3_\e 
		\ar[ld]_{\beta_{LM}} \ar[d]^{\beta_{LR}} \ar[rd]^{\beta_{MR}} & \\
	(M/B)^2_\e & (M/B)^2_\e & (M/B)^2_\e }.
\end{equation*}
\end{proposition}

The main general fact about radial blow ups we will use in the proof is:
\begin{lemma}\label{thm:lemma-p}
  Let $\pi^{(0)} \colon M_1 \to M_2$ be a b-fibration of manifolds with corners, $N_i
  \subset M_i$ p-submanifolds, and assume $N_1 \subset
  \pi^{-1}(N_2)$ such that $\pi^{(0)}\rvert_{N_1} \colon N_1 \lra N_2$
  is a diffeomorphism.  Then $\pi^{(0)}$ extends to a b-fibration
$$
\pi \colon [M_1 ; N_1 ; (\pi^{(0)})^{-1}(N_2)]\lra [M_2 ; N_2].
$$
If $\ff$ is the introduced front face of $[M_2; N_2]$ and $\ff'_1,
\ff'_2$ are the two introduced front faces of $[M_1 ; N_1 ;
(\pi^{(0)})^{-1}(N_2)]$, then the exponent matrix $e$ for $\pi$ satisfies
$$
e(\ff'_1, \ff) = e(\ff'_2, \ff) = 1.
$$
\end{lemma}
\begin{proof}
  This is straightforward from the local description of radial blow
  up.  Indeed, by assumption, there are coordinates $(x,\  y,\ 
  z_1,\  z_2)$
  in which the map $\pi^{(0)}$ has action $(x,\  y,\  z_1,\  z_2) \mapsto
  (x,\   y)$ and such that $N_1 = \{ x = 0 = z_1 \}$ while
  $(\pi^{(0)})^{-1}(N_2) = \{ x = 0 \}$.  Blowing up $N_1$ in $M_1$ gives
  polar coordinates $\rho = (|x|^2 + |z_1|^2)^{1/2},\  \phi = (x,\ 
  z_1)/\rho$ together with $y$ and $z$.  The intersection with
  $(\pi^{(0)})^{-1}(N_2)$ is contained in $\phi = (0,\  z_1/|z_1|)$ so nearby
  this intersection one can use coordinates $\xi = x/|z_1|,\ \  y,\
  |z_1|,\  \hat z_1 = z_1 / |z_1|
  z_2$, and blowing up $\xi = 0$ gives coordinates $r = |\xi|,\  \hat
  \xi = \xi / |\xi| = x / |x|,\ y,\  |z_1|, \ \hat z_1,  z_2$.
  The map to $[M_2; N_2]$ is locally (in polar coordinates) $(r,\ \hat
  \xi,\  y,\ |z_1|, \ \hat z_1,\ z_2) \mapsto
  (\wt{r},\  \hat x,\  y)$ where $\wt{r} = |x|,\  \hat x= x/|x|$.
  From this one can deduce that the map is a b-fibration with this
  given exponent matrix.
\end{proof}

\begin{proof}[Proof of Proposition \ref{thm:trip-space-maps}]
Since $B$ enters only parametrically we assume that $B = \pt$, so $X =
M$ and all the strata $N$ are in fact the strata $Y$ of $X$.  Letting
$X^3_{\e, T}(k + 1) = [X^3 ; T(Y_1), \dots, T(Y_j)]$ where the $Y_1,
\dots, Y_j$ are the strata of depth greater than $k$, let $Y$ be a
stratum of depth $k$.  In the interior $Y^{\circ}$ and near the
intersection of $Y$ with a lower depth stratum $\wt{Y}$ we have a
diagram as in \eqref{eq:IFSDetail'}, and work locally in the interior
of $Y^{\circ}$ and $W$ in
coordinates $(x, y, r, w, \wt{z})$ on $X$ as in
\eqref{eq:IFSDetail''}.  Recall that near the intersection
$\bhsd{\phi\phi}(Y) \cap \bhsd{\phi\phi}(\wt{Y})$ we have a local
diffeomorphism to
$$
[0, 1)_{x'} \times [0, 1)_{\wt{R}} \times \mathbb{S}^{2 + \dim Y +
  \dim W}_+ \times \wt{Z}^2
$$
with
$$
\wt{R} = (r^2 + (r')^2 + ((x/x') - 1)^2 + |(y - y')/x'|^2 + |w - w'|^2)^{1/2}
$$
and $\mathbb{S}^{2 + \dim Y +
  \dim W}_+$ parametrized by
$$
\wt{\phi} = (r,\  r',\ (x/x') - 1,\ (y - y')/x',\  w - w') / \wt{R}.
$$

The front face $\ff$ of $[X^3_{\e, T}(k + 1) ;
T(Y)]$ fibers $\pi \colon \ff \to Y$, and the inverse image
$\pi^{-1}(U)$ of an open
neighborhood $U \subset Y^{\circ}$ intersected with a neighborhood of
the interior lift of $T(\wt{Y})$ is diffeomorphic to
$$
[0, 1)_\rho \times \mathbb{S}^{2 + 2 \dim Y}_{++} \times U \times (\bhs{\wt{Y}Z}
\times [0, 1)_r)^3,
$$
where $\wt{Z} \fib  \bhs{ \wt{Y}Z} \lra W$.  Here we can take $\rho =
(x_1^2 + x_2^2 + x_3^2 + |y_1 - y_3|^2 + |y_2 - y_3|^2)^{1/2}$ and
parametrize $\mathbb{S}^{2 + 2 \dim Y}_{++}$ with $\phi = (x_1 , \
x_2 , \  x_3 , \  y_1 - y_3 , \  y_2 - y_3)/\rho$, where $x_1$ is the
pullback of $x$ from projection of $X^3$ onto the left factor, etc.\ .  From this it is easy to
check that the interior lifts of the $S_\bullet(Y)$ are pairwise
disjoint, since they are contained, respectively, in
$$
\phi \in  \{ (x_1, 0, 0, y_1 - y_3, 0)/\rho \},\ \phi \in  \{ (0, x_2,
0, 0, y_2 - y_3)/\rho \},\ \phi \in  \{ (0,0,x_3,y_1 - y_3, y_1 - y_3)/\rho \}.
$$
Moreover, the inverse image of $T(\wt{Y})$ lies in 
$$
\{ \phi = (1/\sqrt{3}, 1/\sqrt{3}, 1/\sqrt{3}, 0 ,
0) \}
$$
so the interior lifts of the $S_\bullet(Y)$ and $T(\wt{Y})$ are
disjoint p-submanifolds in $[X^3_{\e, T}(k + 1) ;
T(Y)]$.

The proposition then follows from application of Lemma
\ref{thm:lemma-p} and induction.  Indeed, $T(Y)$ and $S_\bullet(Y)$
satisfy the hypotheses of Lemma \ref{thm:lemma-p} with $N_1 = T(Y)$
and $N_2 = S_\bullet(Y)$, and since for varying $\bullet \in \{LM, LR, MR
\}$ these become
disjoint in the lift and the b-fibration property is local, the
proposition follows.
\end{proof}

Inspection of the proof and the lemma above show that the exponent matrices have only zeros and ones, so we specify them by listing the preimages of the collective boundary hypersurfaces.
Recall that the boundary hypersurfaces of $(M/B)^2_\e$ are collective boundary hypersurfaces $\bhsd{10}(N),$ $\bhsd{01}(N),$ and the front face over $Y,$ $\bhsd{\phi\phi}(N).$ We have, for each $\{ N \in \cS_{\psi}(M)\},$
\begin{equation*}
\begin{gathered}
	\beta_{LM}^*\bhsd{10}(N) = 
	\{\bhst{100}(N), \bhst{\phi0\phi}(N) \}, \quad
	\beta_{LM}^*\bhsd{01}(N) = \{ \bhst{010}(N), \bhst{0\phi\phi}(N) \}, \\
	\beta_{LM}^*\bhsd{\phi\phi}(N) 
	= \{ \bhst{\phi\phi0}(N), \bhst{\phi\phi\phi}(N) \}, 
\end{gathered}
\end{equation*}
\begin{equation*}
\begin{gathered}
	\beta_{LR}^*\bhsd{10}(N) = 
	\{\bhst{100}(N), \bhst{\phi\phi0}(N) \}, \quad
	\beta_{LR}^*\bhsd{01}(N) = \{ \bhst{001}(N), \bhst{0\phi\phi}(N) \}, \\
	\beta_{LR}^*\bhsd{\phi\phi}(N) 
	= \{ \bhst{\phi0\phi}(N), \bhst{\phi\phi\phi}(N) \}, 
\end{gathered}
\end{equation*}
\begin{equation*}
\begin{gathered}
	\beta_{MR}^*\bhsd{10}(N) = 
	\{\bhst{010}(N), \bhst{\phi\phi0}(N) \}, \quad
	\beta_{MR}^*\bhsd{01}(N) = \{ \bhst{001}(N), \bhst{\phi0\phi}(N) \}, \\
	\beta_{MR}^*\bhsd{\phi\phi}(N) 
	= \{ \bhst{0\phi\phi}(N), \bhst{\phi\phi\phi}(N) \}. 
\end{gathered}
\end{equation*}

Applying Melrose's push-forward and pull-back theorems this leads to a composition result for the large edge calculus.
The behavior with respect to the conormal singularity at the diagonal is standard, so we will focus on operators of order $-\infty.$ We will also simplify notation by not including vector bundles.

Thus, from \eqref{eq:DefLargeCalc} and Definition \ref{def:BddCalc}, we will establish composition results for conormal distributions in
\begin{equation*}
	\sA_{phg}^{\cE}((M/B)^2_\e; \Hom(E,F)\otimes \Omega_{\mf d,R})
\end{equation*}
where we recall that
\begin{multline*}
	\Omega_{\mf d,R} = \rho_{(M/B)^2_\e}^{\mf d}
	\beta_{(2),R}^*\Omega(M/B), \Mwith 
	\mf d:\cM_1((M/B)^2_\e) \lra \bbR, \\
	\mf d(H) =
	\begin{cases}
	-(\dim(N/B) + 1) & \Mif H \subseteq \bhsd{\phi\phi}(N) 
		\Mforsome N \in \cS_{\psi}(M) \\
	0 & \Motherwise
	\end{cases}
\end{multline*}

For an operator $A,$ let us write its integral kernel as
\begin{equation*}
	\cK_A \rho_{(M/B)^2_\e}^{\mf d} \mu_R.
\end{equation*}
Then if the composition $C = A \circ B$ is defined, its integral kernel is given by 
\begin{equation*}
	\cK_C\; \rho^{\mathfrak d} \mu_{R} 
	= (\beta_{LR})_*
	(\beta_{LM}^*(\cK_A \; \rho^{\mathfrak d} \mu_{R} ) 
	\cdot \beta_{MR}^*(\cK_B \; \rho^{\mathfrak d} \mu_{R} ) ).
\end{equation*}

\begin{theorem}
If $\cK_A \in \sA_{phg}^{\cE_A}((M/B)^2_\e)$ and $\cK_B \in \sA_{phg}^{\cE_B}((M/B)^2_\e)$ where the index sets satisfy
\begin{equation*}
	\Re(\cE_A(\bhsd{01}(N))) + \Re(\cE_B(\bhsd{10}(N))) > -1 
	\Mforall N \in \cS_{\psi}(M)
\end{equation*}
then $\cK_C \in \sA_{phg}^{\cE_C}((M/B)^2_\e)$ with, for each $N \in \cS_{\psi}(M),$
\begin{equation*}
\begin{aligned}
	\cE_C(\bhsd{10}(N)) &= \cE_A(\bhsd{10}(N)) \bar\cup 
		\lrpar{ \cE_A(\bhsd{\phi\phi}(N)) + \cE_B(\bhsd{10}(N))}, \\
	\cE_C(\bhsd{01}(N)) &= \cE_B(\bhsd{01}(N)) \bar\cup 
		\lrpar{ \cE_A(\bhsd{01}(N)) + \cE_B(\bhsd{\phi\phi}(N)) }, \\
	\cE_C(\bhsd{\phi\phi}(N)) &= 
		\lrpar{ \cE_A(\bhsd{\phi\phi}(N)) + \cE_B(\bhsd{\phi\phi}(N))} \\
		&\phantom{xxxxx} \bar\cup 
		\lrpar{\cE_A(\bhsd{10}(N)) + \cE_B(\bhsd{01}(N)) + \dim(N/B) + 1} 
\end{aligned}
\end{equation*}
\end{theorem}

\begin{proof}
Let $\mu(W)$ denote a nowhere vanishing section of $\Omega(W)$ and $\mu_b(W)$ a nowhere vanishing section of $\Omega_b(W).$

The behavior of the densities under pull-back is given by
\begin{equation*}
\begin{gathered}
	(\beta^{(2)})^*(\mu(M^2_{\psi}/B)) 
	= \prod_{N \in \cS_{\psi}(M)} \rho_{\bhsd{\phi\phi}(N)}^{\dim(N/B) + 1}
		\mu((M/B)^2_{\e}/B) \\
	(\beta^{(3)})^*(\mu(M^3_{\psi}/B)) = \prod_{N \in \cS_{\psi}(M)} 
		(\rho_{\bhst{\phi\phi0}(N)}\rho_{\bhst{\phi0\phi}(N)}
		\rho_{\bhst{0\phi\phi}(N)})^{\dim (M/B) +1}
		\rho_{\bhst{\phi\phi\phi}(N)}^{2\dim(M/B) +2} \mu((M/B)^3_{\e}/B).
\end{gathered}
\end{equation*}

Multiplying the integral kernel of $C$ by $\mu_L = (\beta^{(2)}_L)^*\mu$ yields
\begin{equation*}
\begin{gathered}
	\cK_C\; \rho^{\mathfrak d} (\beta^{(2)})^*\mu(M^2_{\psi}/B) 
	= (\beta_{LR})_*(\beta_{LM}^*(\cK_A  \rho^{\mathfrak d}) 
		\cdot \beta_{MR}^*(\cK_B  \rho^{\mathfrak d} )
	\cdot \beta_{LR}^*\mu_L\cdot \beta_{LM}^*\mu_R\cdot 
		\beta_{MR}^*\mu_{R} ) \\
	= (\beta_{LR})_*(\beta_{LM}^*(\cK_A  \rho^{\mathfrak d}) \cdot 
		\beta_{MR}^*(\cK_B  \rho^{\mathfrak d} )
	\cdot (\beta^{(3)})^*\mu(M^3_{\psi}/B)),
\end{gathered}
\end{equation*}
hence
\begin{equation*}
	\cK_C \mu((M/B)^2_{\e}/B)
	= (\beta_{LR})_*\Big(\beta_{LM}^*(\cK_A) \cdot \beta_{MR}^*(\cK_B) 
	\prod_{N \in \cS_{\psi}(M)} \rho_{\bhsd{\phi0\phi}(N)}^{\dim(N/B)+1} 
		\mu((M/B)^3_{\e}/B) \Big).
\end{equation*}
Now we write this in terms of $b$-densities
\begin{equation*}
	\cK_C \mu_b((M/B)^2_{\e}/B)
	= (\beta_{LR})_*\Big(\beta_{LM}^*(\cK_A) \cdot \beta_{MR}^*(\cK_B) 
	\prod_{N \in \cS_{\psi}(M)} \rho_{\bhsd{\phi0\phi}(N)}^{\dim(N/B)+1}
		\rho_{\bhsd{010}(N)} \mu_b((M/B)^3_{\e}/B) \Big)
\end{equation*}
and we can apply the pull-back and push-forward theorems.

\end{proof}

The action on polyhomogeneous functions is also easy to write down.
Given $A$ as above and $f \in \sA_{phg}^{\cF}(X)$ we define
\begin{equation*}
	Af = (\beta_L)_*(\cK_A \; \rho^{\mathfrak d} \mu_{R} \cdot \beta_R^*f).
\end{equation*}

\begin{proposition}
If $\cK_A \in \sA_{phg}^{\cE_A}((M/B)^2_\e)$ and $f \in \sA_{phg}^{\cG_f}(M/B)$ where the index sets satisfy
\begin{equation*}
	\Re(\cE_A(\bhsd{01}(N))) + \Re(\cG_f(\bhss{N})) > -1 \Mforall N \in \cS_{\psi}(M)
\end{equation*}
then
$Af \in \sA_{phg}^{\cG_{Af}}(N)$ with, for each $N \in \cS_{\psi}(M),$
\begin{equation*}
	\cG_{Af}(\bhss{N}) = \cE_A(\bhsd{10}(N)) \bar\cup 
	\lrpar{ \cE_A(\bhsd{\phi\phi}(N)) + \cG_f(\bhss{N}) }
\end{equation*}
\end{proposition}

\begin{proof}
Multiplying $Af$ by $\mu(M/B)$ yields
\begin{equation*}
	Af \mu(M/B) 
	= (\beta_L)_*\Big(\cK_A  \rho^{\mathfrak d} \cdot \beta_R^*f 
		\cdot (\beta^{(2)})^*\mu(M^2_{\psi}) \Big).
	= (\beta_L)_*\Big(\cK_A \cdot \beta_R^*f \mu((M/B)^2_{\e}/B)\Big).
\end{equation*}
and passing to $b$-densities
\begin{equation*}
	Af \mu_b(M/B)
	= (\beta_L)_*\Big(\cK_A \cdot \beta_R^*f 
	\prod_{N \in \cS_{\psi}(M)} \rho_{\bhsd{01}(N)} \mu_b((M/B)^2_{\e}/B)\Big).
\end{equation*}
and we can apply the pull-back and push-forward theorems.

\end{proof}

\begin{corollary}
If the integral kernel of $A$ satisfies $\cK_A \in \sA_{phg}^{\cE_A}((M/B)^2_\e)$ then $A$ defines a bounded map, for any $t \in \bbR,$
\begin{equation*}
	\rho^{\mathfrak s}H^t_\e(M/B) \lra \rho^{\mathfrak s'}H^{\infty}_\e(M/B)
\end{equation*}
as long as
\begin{equation*}
\begin{gathered}
	\Re(\cE_A(\bhsd{01}(N)) ) + \mathfrak s(N) > -\tfrac12 \\
	\Re(\cE_A(\bhsd{10}(N))) > \mathfrak s'(N) -\tfrac12\\
	\Re(\cE_A(\bhsd{\phi\phi}(N)) ) + \mathfrak s(N) \geq \mathfrak s'(N)
\end{gathered}
\end{equation*}
\end{corollary}

\section{Composition of wedge heat operators} \label{sec:HeatComp}

{\bf The wedge heat composition space.}
Composition is through convolution in time,
\begin{equation*}
	f(t) \; dt = \int_0^t g(t-t')h(t') \; dt' dt= \int_{t''+t' = t} (g(t'') \; dt'') (h(t') \; dt').
\end{equation*}
We prefer to work with $\tau = \sqrt t$ instead of $t,$ so this becomes
\begin{equation*}
	f(\tau) \; d\tau = \frac2\tau\int_{\sqrt{s^2+\wt s^2} = \tau} (s\wt s) 
	(g(s) \; ds) (h(\wt s)\; d\wt s).
\end{equation*}
In terms of the maps
\begin{equation*}
	\xymatrix{ 
	& \bbR^+_s \times \bbR^+_{\wt s} \ar[ld]^{\pi_L} \ar[d]^{\pi_C} \ar[rd]^{\pi_R} & \\
	\bbR^+ & \bbR^+ & \bbR^+ } \quad
	\xymatrix{ 
	& (s,\wt s) \ar@{|->}[ld]^{\pi_L} \ar@{|->}[d]^{\pi_C} \ar@{|->}[rd]^{\pi_R} & \\
	s & \sqrt{s^2+\wt s^2} & \wt s }
\end{equation*}
the composition is given by 
\begin{equation*}
	f(\tau) \; d\tau = \frac2\tau (\pi_C)_*
	\lrpar{ \pi_L^*(\tau g(\tau)\; d\tau) \pi_R^*(\tau h(\tau)\; d\tau)}.
\end{equation*}
This push-forward is not well-behaved because $\pi_C$ is not a b-fibration but we can fix this by replacing $(\bbR^+)^2$ with
\begin{equation*}
	\sT^2_0 = [\bbR^+_s \times \bbR^+_{\wt s}; \{ (0,0) \}]
\end{equation*}
and then the composition formula is well-behaved.
We denote the boundary hypersurfaces of this space by
\begin{equation*}
	\{0\} \times \bbR^+ \leftrightarrow \bhsT{10}, \quad
	\bbR^+ \times \{0\} \leftrightarrow \bhsT{01}, \quad
	\{0\} \times \{0\} \leftrightarrow \bhsT{11}.
\end{equation*}

Now bringing in the spatial variables, we want to construct a space $H(M/B)^3_\w$ so that the maps
\begin{equation*}
	\xymatrix{ 
	& M^3_{\psi} \times (\bbR^+)^2 \ar[ld]^{\pi_{LM,L}} \ar[d]^{\pi_{LR,C}} \ar[rd]^{\pi_{MR,R}} & \\
	M^2_{\psi} \times \bbR^+ 
	& M^2_{\psi} \times \bbR^+ 
	& M^2_{\psi} \times \bbR^+ } \quad
	\xymatrix{ 
	& (\zeta, \zeta', \zeta'', s,\wt s) \ar@{|->}[ld]^{\pi_{LM,L}} \ar@{|->}[d]^{\pi_{LR,C}} \ar@{|->}[rd]^{\pi_{MR,R}} & \\
	(\zeta, \zeta',s) & (\zeta, \zeta'', \sqrt{s^2+\wt s^2}) & (\zeta', \zeta'', \wt s) }
\end{equation*}
(where we are using the notation $M^2_{\psi},$ $M^3_{\psi}$ from Appendix \ref{sec:CompEdge})
lift to b-maps
\begin{equation*}
	\xymatrix{ 
	& H(M/B)^3_{\w} \ar[ld]^{\beta_{LM,L}} \ar[d]^{\beta_{LR,C}} \ar[rd]^{\beta_{MR,R}} & \\
	H(M/B)_{\w}  & H(M/B)_{\w} & H(M/B)_{\w} }.
\end{equation*}

We construct the space $H(M/B)^3_{\w}$ in steps starting from 
\begin{equation*}
	H_0(M/B)^3_{\w} = M^3_{\psi}\times \sT^2_0.
\end{equation*}
Let $\cS_{\psi}(M) = \{ N_1, \ldots, N_{\ell} \}$ be a non-decreasing list and recall the notation $T(N),$ $S_{LM}(N),$ etc. from \eqref{eq:TripDiag}.
Inductively, for $1 \leq i \leq \ell,$ let
\begin{multline*}
	H_i(M/B)^3_{\w} = \Big[H_{i-1}(M/B)^3_{\w};
	T(N_i) \times \bhsT{11};
	(S_{LM}(N_i) \cup S_{LR}(N_i) \cup S_{MR}(N_i)) \times \bhsT{11}; \\
	S_{LM}(N_i) \times \bhsT{10};
	S_{MR}(N_i) \times \bhsT{01}; 
	T(N_i);
	(S_{LM}(N_i) \cup S_{LR}(N_i) \cup S_{MR}(N_i))\Big].
\end{multline*}
Finally, we need to blow-up the (interior) lifts of the partial diagonals. So let
\begin{equation*}
	\diag_{LM} = \pi_{LM}^{-1}(\diag_M), \quad
	\diag_{LR} = \pi_{LR}^{-1}(\diag_M), \quad
	\diag_{MR} = \pi_{MR}^{-1}(\diag_M), 
\end{equation*}
and let $\diag_{LMR}$ be their intersection and then define
\begin{multline*}
	H(M/B)^3_{\w}
	= \Big[H_{\ell}(M/B)^3_{\w};
	\diag_{LMR} \times \bhsT{11}; \\
	(\diag_{LM} \cup \diag_{LR} \cup \diag_{MR}) \times \bhsT{11}; 
	\diag_{LM}\times \bhsT{10};
	\diag_{MR} \times \bhsT{01}
	\Big].
\end{multline*}

Our notation for the (collective) boundary hypersurfaces of $H(M/B)^3_\w$ is as follows. First we have,
\begin{equation*}
\begin{gathered}
	M^3_{\psi} \times \bhsT{10} \leftrightarrow \bhsC{000,10}, \quad
	M^3_{\psi} \times \bhsT{01} \leftrightarrow \bhsC{000,01}, \quad
	M_{\psi}^3 \times \bhsT{11} \leftrightarrow \bhsC{000,11}, \\
	\diag_{LMR} \times \bhsT{11} \leftrightarrow \bhsC{ddd,11}, \quad
	\diag_{LM} \times \bhsT{11} \leftrightarrow \bhsC{dd0,11}, \quad
	\diag_{LR} \times \bhsT{11} \leftrightarrow \bhsC{d0d,11}, \\
	\diag_{MR} \times \bhsT{11} \leftrightarrow \bhsC{0dd,11}, \quad
	\diag_{LM} \times \bhsT{10} \leftrightarrow \bhsC{dd0,10}, \quad
	\diag_{MR} \times \bhsT{01} \leftrightarrow \bhsC{0dd,01},
\end{gathered}
\end{equation*}
and then, for each $N \in \cS_{\psi}(M),$
\begin{equation*}
\begin{gathered}
	\bhss{N} \times_{\psi} M^2_{\psi} \times \sT^2_0 
		\leftrightarrow \bhsC{100,00}(N), \quad
	M \times_{\psi} \bhss{N} \times_{\psi} M \times \sT^2_0 
		\leftrightarrow \bhsC{010,00}(N), \\
	M^2_{\psi} \times_{\psi} \bhss{N} \times \sT^2_0 
		\leftrightarrow \bhsC{001,00}(N), \quad
	T(N) \times \bhsT{11} \leftrightarrow \bhsC{\phi\phi\phi,11}(N), \\
	S_{LM}(N) \times \bhsT{11} \leftrightarrow \bhsC{\phi\phi0,11}(N), \quad
	S_{LR}(N) \times \bhsT{11} \leftrightarrow \bhsC{\phi0\phi,11}(N), \\
	S_{MR}(N) \times \bhsT{11} \leftrightarrow \bhsC{0\phi\phi,11}(N), \quad
	S_{LM}(N) \times \bhsT{10} \leftrightarrow \bhsC{\phi\phi0,10}(N), \\
	S_{MR}(N) \times \bhsT{01} \leftrightarrow \bhsC{0\phi\phi,01}(N) \quad
	T(N) \times \sT^2_0 \leftrightarrow \bhsC{\phi\phi\phi,00}(N), \\
	S_{LM}(N) \times \sT^2_0 \leftrightarrow \bhsC{\phi\phi0,00}(N), \quad
	S_{LR}(N) \times \sT^2_0 \leftrightarrow \bhsC{\phi0\phi,00}(N), \\
	S_{MR}(N) \times \sT^2_0 \leftrightarrow \bhsC{0\phi\phi,00}(N).
\end{gathered}
\end{equation*}

$ $\\
{\bf The exponent matrices.}
We specify the exponent matrix of the maps $\beta_{\cdot,\cdot,\cdot}$ by specifying the pull-back of the collective boundary hypersurfaces.
Thus $\beta_{LM,L}$ maps $\bhsC{0dd,01},$ $\bhsC{000,01}$ and, for each $N \in \cS_{\psi}(M),$ $\bhsC{001,00}(N),$ into the interior of $H(M/B)_\w;$ otherwise
\begin{equation*}
\begin{aligned}
	\beta_{LM,L}^*\bhsh{00,1}
	&= \{ \bhsC{000,10}, \bhsC{000,11},  
		\bhsC{d0d, 11}, \bhsC{0dd, 11} \} \cup 
		\bigcup_{N \in \cS_{\psi}(M)} 
			\{\bhsC{\phi0\phi,11}(N), \bhsC{0\phi\phi,11}(N) \},\\
	\beta_{LM,L}^*\bhsh{dd,1}
	&= \{ \bhsC{ddd, 11}, \bhsC{dd0, 11}, \bhsC{dd0,10} \}
\end{aligned}
\end{equation*}
and, for each $N \in \cS_{\psi}(M),$
\begin{equation*}
\begin{aligned}
	\beta_{LM,L}^*\bhsh{10,0}(N)
	&= \{ \bhsC{100,00}(N), \bhsC{\phi0\phi, 11}(N), 
		\bhsC{\phi0\phi, 00}(N) \} \\
	\beta_{LM,L}^*\bhsh{01,0}(N)
	&= \{ \bhsC{010,00}(N), \bhsC{0\phi\phi,11}(N), \bhsC{0\phi\phi,01}(N),
		\bhsC{0\phi\phi,00}(N) \} \\
	\beta_{LM,L}^*\bhsh{\phi\phi,1}(N)
	&= \{ \bhsC{\phi\phi\phi,11}(N), \bhsC{\phi\phi0,11}(N), 
		\bhsC{\phi\phi0, 10}(N) \}, \\
	\beta_{LM,L}^*\bhsh{\phi\phi,0}(N)
	&= \{ \bhsC{\phi\phi\phi,00}(N), \bhsC{\phi\phi0,00}(N) \},
\end{aligned}
\end{equation*}
where we note that $\bhsC{\phi0\phi,11}(N)$ and $\bhsC{0\phi\phi,11}(N)$ are repeated.

Similarly, the map $\beta_{LR,C}$ maps $\bhsC{000,10},$ $\bhsC{000,01},$ $\bhsC{dd0,10},$ $\bhsC{0dd,01}$ and, for each $N \in \cS_{\psi}(M),$ $\bhsC{010,00}(N),$ into the interior of $H(M/B)_\w;$ otherwise
\begin{equation*}
\begin{aligned}
	\beta_{LR,C}^*\bhsh{00,1}
	&= \{ \bhsC{000,11}, \bhsC{dd0,11}, \bhsC{0dd, 11} \}
	\cup \bigcup_{N \in \cS_{\psi}(M)} \{ \bhsC{\phi\phi0,11}(N), 
		\bhsC{0\phi\phi,11}(N) \} \\
	\beta_{LR,C}^*\bhsh{dd,1}
	&= \{ \bhsC{ddd, 11}, \bhsC{d0d, 11} \}
\end{aligned}
\end{equation*}
and, for each $N \in \cS_{\psi}(M),$
\begin{equation*}
\begin{aligned}
	\beta_{LR,C}^*\bhsh{10,0}(N)
	&= \{ \bhsC{100,00}(N), \bhsC{\phi\phi0,11}(N), \bhsC{\phi\phi0,10}(N),
		\bhsC{\phi\phi0,00}(N) \} \\
	\beta_{LR,C}^*\bhsh{01,0}(N)
	&= \{ \bhsC{001, 00}(N), \bhsC{0\phi\phi,11}(N), \bhsC{0\phi\phi,01}(N),
		\bhsC{0\phi\phi,00}(N) \} \\
	\beta_{LR,C}^*\bhsh{\phi\phi,1}(N)
	&= \{ \bhsC{\phi\phi\phi,11}(N), \bhsC{\phi0\phi,11}(N) \} \\
	\beta_{LR,C}^*\bhsh{\phi\phi,0}(N)
	&= \{ \bhsC{\phi\phi\phi,00}(N), \bhsC{\phi0\phi,00}(N) \}
\end{aligned}
\end{equation*}
where we note that $\bhsC{\phi\phi0,11}(N)$ and $\bhsC{0\phi\phi,11}(N)$ are repeated.

Finally, the map $\beta_{MR,R}$ maps  $\bhsC{000,10},$ $\bhsC{dd0,10}$
and, for each $N \in \cS_{\psi}(M),$ $\bhsC{100,00}(N),$
into the interior of $H(M/B)_\w;$ otherwise
\begin{equation*}
\begin{aligned}
	\beta_{MR,R}^*\bhsh{00,1}
	&= \{ \bhsC{000,01}, \bhsC{000,11},  \bhsC{dd0,11}, \bhsC{d0d,11} \}
	\cup \bigcup_{N \in \cS_{\psi}(M)} \{ \bhsC{\phi\phi0,11}(N), 
		\bhsC{\phi0\phi,11}(N) \} \\
	\beta_{MR,R}^*\bhsh{dd,1}
	&= \{ \bhsC{ddd,11}, \bhsC{0dd,11}, \bhsC{0dd, 01} \}
\end{aligned}
\end{equation*}
and, for each $N \in \cS_{\psi}(M),$
\begin{equation*}
\begin{aligned}
	\beta_{MR,R}^*\bhsh{10,0}(N)
	&= \{ \bhsC{010,00}(N), \bhsC{\phi\phi0,11}(N), \bhsC{\phi\phi0,10}(N),
		\bhsC{\phi\phi0,00}(N) \} \\
	\beta_{MR,R}^*\bhsh{01,0}(N)
	&= \{ \bhsC{001,00}(N), \bhsC{\phi0\phi,11}(N), \bhsC{\phi0\phi,00}(N) \} \\
	\beta_{MR,R}^*\bhsh{\phi\phi,1}(N)
	&= \{ \bhsC{\phi\phi\phi,11}(N), \bhsC{0\phi\phi,11}(N), 
		\bhsC{0\phi\phi,01}(N) \} \\
	\beta_{MR,R}^*\bhsh{\phi\phi,0}(N)
	&= \{ \bhsC{\phi\phi\phi,00}(N), \bhsC{0\phi\phi,00}(N) \}
\end{aligned}
\end{equation*}
where we note that $\bhsC{\phi\phi0,11}(N)$ and $\bhsC{\phi0\phi,11}(N)$ are repeated.

$ $\\
{\bf Composition.}
Now let us discuss the composition law.
As before we are interested in integral kernels that are sections of a weighted density bundle. Let us start by recalling the weight from \eqref{eq:HeatWeight}, namely
\begin{equation*}
\begin{gathered}
	\mathfrak h: \cM_1(H(M/B)_\w) \lra \bbR, \\
	\mathfrak h(H) =
	\begin{cases}
	-(\dim(N/B) + 3) & \Mif H \subseteq \bhsh{\phi\phi,1}(N) 
		\Mforsome N \in \cS_{\psi}(M) \\
	-(\dim(N/B) + 1) & \Mif H \subseteq \bhsh{\phi\phi,0}(N) 
		\Mforsome N \in \cS_{\psi}(M) \\
	-(\dim(M/B) + 2) & \Mif H = \bhsh{dd,1} \\
	\infty & \Mif H = \bhsh{00,1} \\
	0 & \Motherwise
	\end{cases}
\end{gathered}
\end{equation*}
and let $\mu_R = \beta_{(H), R}^*\mu(M/B).$ We will determine the behavior under composition for integral kernels of the form $\cK_A \rho^{\mathfrak h} \mu_R$ with $\cK_A \in \cA_{phg}^{\cE_A}(H(M/B)_\w).$ Ultimately we are interested in kernels that are merely conormal with bounds acting on sections of a vector bundle, but the corresponding composition result follows easily from this one.

\begin{proposition} \label{prop:WedgeHeatComp}
Let $A$ have kernel $\cK_A \rho^{\mathfrak h} \mu_R$ with $\cK_A \in \cA_{phg}^{\cE_A}(H(M/B)_\w)$ and $B$ have kernel $\cK_B \rho^{\mathfrak h} \mu_R$ with $\cK_B \in \cA_{phg}^{\cE_B}(H(M/B)_\w).$ 
If 
\begin{equation*}
\begin{gathered}
	\Re(\cE_A(\bhsh{dd,1})) > 0, \quad
	\Re(\cE_B(\bhsh{dd,1})) > 0, \Mand \\
	\Re(\cE_A(\bhsh{01,0}(N))) + \Re(\cE_B(\bhsh{10,0}(N)))+1 > 0 
		\Mforall N \in \cS_{\psi}(M),
\end{gathered}
\end{equation*}
then we may define their composition $C = A \circ B$ by the formula
\begin{equation*}
	\wt\cK_C \mu_R \beta_{(H)}^*(\tau \; d\tau)
	= (\beta_{LR,C})_*( 
	\beta_{LM,L}^*(\cK_A \rho^{\mathfrak h} \mu_R \; \beta_{(H)}^*(\tau d\tau))
	\cdot
	\beta_{MR,R}^*(\cK_B \rho^{\mathfrak h} \mu_R \; 
		\beta_{(H)}^*(\tau d\tau)) )
\end{equation*}
and we have $\wt\cK_C \in \rho^{\mathfrak h} \cA_{phg}^{\cE_C}(H(M/B)_\w)$ with
\begin{equation*}
	\cE_C(\bhsh{dd,1}) = 
	\cE_A(\bhsh{dd,1}) + \cE_B(\bhsh{dd,1}) 
\end{equation*}
and, for each $N \in \cS_{\psi}(M),$
\begin{align*}
	\cE_C(\bhsh{10,0}(N)) &=
	\cE_A(\bhsh{10,0}(N)) \bar\cup
	\big( \cE_A(\bhsh{\phi\phi,1}(N)) + \cE_B(\bhsh{10,0}(N) )\big) \\
	\cE_C(\bhsh{01,0}(N)) &=
	\cE_B(\bhs{01,0}(N)) \bar\cup
	\big(\cE_A(\bhsh{01,0}(N)) + \cE_B(\bhsh{\phi\phi,1}(N)) \big) \\
	\cE_C(\bhsh{\phi\phi,1}(N)) &=
	\cE_A(\bhsh{\phi\phi,1}(N)) + \cE_B(\bhsh{\phi\phi,1}(N)) \\
	\cE_C(\bhsh{\phi\phi,0}(N)) &=
	(\cE_A(\bhsh{\phi\phi,0}(N)) + \cE_B(\bhsh{\phi\phi,0}(N)) ) \\
	& \phantom{xxxxx} \bar\cup
	\big(\cE_A(\bhsh{10,0}(N)) + \cE_B(\bhsh{01,0}(N)) + \dim(N/B)+1\big)
\end{align*}
\end{proposition}

\begin{proof}
Let us write, during the proof,
\begin{equation*}
\begin{gathered}
	\mathfrak h: \cM_1(H(M/B)_\w) \lra \bbR, \\
	\mathfrak h(H) =
	\begin{cases}
	-(\dim(N/B) + a) & \Mif H \subseteq \bhsh{\phi\phi,1}(N) 
		\Mforsome N \in \cS_{\psi}(M) \\
	-(\dim(N/B) + a') & \Mif H \subseteq \bhsh{\phi\phi,0}(N) 
		\Mforsome N \in \cS_{\psi}(M) \\
	-(\dim(M/B) + b) & \Mif H = \bhsh{dd,1} \\
	\infty & \Mif H = \bhsh{00,1} \\
	0 & \Motherwise
	\end{cases}
\end{gathered}
\end{equation*}
to motivate the choice ($a=3,$ $a'=1,$ $b=2$) made above.

Note that $\beta_{LR,C}$ is not a b-fibration so that push-forward along it does not preserve polyhomogeneous functions. However, the weight $\mathfrak h$ is such that the problem faces can be blown-down which is why we will end up with a polyhomogeneous function.
Indeed, 
the product $\beta_{LM,L}^*(\rho^{\mathfrak h})\beta_{MR,R}^*(\rho^{\mathfrak h})$ vanishes to infinite order at every face in
\begin{equation*}
\begin{multlined}
	\beta_{LM,L}^*\bhsh{00,1} \cup \beta_{MR,R}^*\bhsh{00,1}
	= \{ \bhsC{000,10}, \bhsC{000,01}, \bhsC{000,11},  
		\bhsC{dd0,11}, \bhsC{d0d, 11}, \bhsC{0dd, 11} \} \\ \cup 
		\bigcup_{N \in \cS_{\psi}(M)} \{\bhsC{\phi\phi0,11}(N), 
		\bhsC{\phi0\phi,11}(N), \bhsC{0\phi\phi,11}(N) \},
\end{multlined}
\end{equation*}
so we see that the push-forward along $\beta_{LR,C}$ will be polyhomogeneous and will vanish at $\bhs{00,1}$ to infinite order.

Thus let us write $\wt \cK_C = \cK_C \rho^{\mathfrak h}$ for some $\cK_C$ polyhomogeneous, which 
after multiplying both sides of the formula for $C$ by $\mu_L = \beta_{(H), L}^*\mu(M/B),$ satisfies 
\begin{multline*}
	\cK_C \rho^{\mathfrak h} \beta_{(H)}^*(\tau\; 
		\mu(M^2_{\psi}/B\times \bbR^+) ) \\
	= (\beta_{LR,C})_*( 
	\beta_{LM,L}^*(\cK_A \rho^{\mathfrak h}  \beta_{(H)}^*(\tau)) 
	\cdot
	\beta_{MR,R}^*(\cK_B \rho^{\mathfrak h} \beta_{(H)}^*(\tau))
	\cdot \beta_{(C)}^*(\mu(M^3_{\psi}/B\times (\bbR^+)^2) ) ).
\end{multline*}
We need to work out the density weight factors. Start by noting that
\begin{equation*}
	\beta_{(H)}^*\tau 
	=\rho_{\bhsh{00,1}}\rho_{\bhsh{dd,1}}
	\prod_{N \in \cS_{\psi}(M)}  \rho_{\bhsh{\phi\phi,1}(N)}
\end{equation*}
and then, ignoring the faces where we have seen infinite order decay,
\begin{multline*}
	\beta_{LR,C}^*(\rho^{\mathfrak h}\beta_{(H)}^*\tau)^{-1}
	\beta_{LM,L}^*(\rho^{\mathfrak h}\beta_{(H)}^*\tau)
	\beta_{MR,R}^*(\rho^{\mathfrak h}\beta_{(H)}^*\tau) \\
	=
	(\rho_{\bhsC{ddd,11}}\rho_{\bhsC{dd0,10}}
		\rho_{\bhsC{0dd,01}})^{-\dim(M/B)-b+1}
	\prod_{N \in \cS_{\psi}(M)} (\rho_{\bhsC{\phi\phi\phi,11}(N)}
		\rho_{\bhsC{\phi\phi0,10}(N)}
		\rho_{\bhsC{0\phi\phi,01}(N)})^{-\dim(N/B)-a+1} \\
	\prod_{N \in \cS_{\psi}(M)} (\rho_{\bhsC{\phi\phi\phi,00}(N)}
		\rho_{\bhsC{\phi\phi0,00}(N)}
		\rho_{\bhsC{0\phi\phi,00}(N)}
		\rho_{\bhsC{\phi0\phi,00}(N)}^{-1})^{-\dim(N/B)-a'}
\end{multline*}

Next, the lifts of the densities (continuing to ignore faces where we have infinite decay) are
\begin{multline*}
	\beta_{(H)}^*\mu(M^2_{\psi}/B\times \bbR^+) 
	= \Big[ \rho_{\bhsh{dd,1}}^{\dim(M/B)}
		\prod_{N \in \cS_{\psi}(M)} \rho_{\bhsh{\phi\phi,1}(N)}^{\dim(N/B)+2}
		\rho_{\bhsh{\phi\phi,0}(N)}^{\dim(N/B)+1} \Big] 
		\mu(H(M/B)_\w) \\
	=\rho_{(H)}^{\mathfrak j_1} \; \mu_b(H(M/B)_{\w}/B)
\end{multline*}
\begin{multline*}
	\beta_{(C)}^*\mu(M^3_{\psi}/B \times (\bbR^+)^2) \\
	= 
	\Big[\rho_{\bhsC{ddd,11}}^{2\dim (M/B)+1} 
	(\rho_{\bhsC{dd0,10}}\rho_{\bhsC{0dd,01}})^{\dim(M/B)} 
	\prod_{N \in \cS_{\psi}(M)}
	(\rho_{\bhsC{\phi\phi\phi,11}(N)}^2\rho_{\bhsC{\phi\phi0,10}(N)}\rho_{\bhsC{0\phi\phi,01}(N)})^{\dim (N/B)+2} \\
	(\rho_{\bhsC{\phi\phi\phi,00}(N)}^2\rho_{\bhsC{\phi\phi0,00}(N)}\rho_{\bhsC{0\phi\phi,00}(N)}
	\rho_{\bhsC{\phi0\phi,00}(N)})^{\dim (N/B)+1}
	 \Big] \mu(H(M/B)^3_\w) \\
	 = \rho_{(C)}^{\mathfrak j_2} \mu_b(H(M/B)^3_{\w}/B)
\end{multline*}
where $\rho_{(H)}$ and $\rho_{(C)}$ are, respectively, total boundary defining functions for $H(M/B)_\w$ and $H(M/B)^3_\w.$
From the exponent matrices of $\beta_{LR,C}$ we have
\begin{multline*}
	(\beta_{LR,C}^*\rho_{(H)}^{\mathfrak j_1})^{-1} \rho_{(C)}^{\mathfrak j_2}
	= 
	(\rho_{\bhsC{ddd,11}}\rho_{\bhsC{dd0,10}}\rho_{\bhsC{0dd,01}})^{
		\dim (M/B)+1} \\
	\prod_{N \in \cS_{\psi}(M)}
	(\rho_{\bhsC{\phi\phi\phi,11}(N)}\rho_{\bhsC{\phi\phi0,10}(N)}
		\rho_{\bhsC{0\phi\phi,01}(N)})^{\dim(N/B)+2} \\
	(\rho_{\bhsC{\phi\phi\phi,00}(N)}\rho_{\bhsC{\phi\phi0,00}(N)}
		\rho_{\bhsC{0\phi\phi,00}(N)})^{\dim(N/B)+1}
	\rho_{\bhsC{010,00}(N)}
\end{multline*}

So altogether
\begin{equation*}
	\cK_C \mu_b(H(M/B)^3_{\w}/B)
	= (\beta_{LR,C})_*( 
	\beta_{LM,L}^*(\cK_A) \cdot \beta_{MR,R}^*(\cK_B) \rho^{\mathfrak j_3} \mu_b(H(M/B)^3_{\w}/B) ).
\end{equation*}
where, ignoring faces with infinite decay, $\rho_{(C)}^{\mathfrak j_3}$ is equal to
\begin{multline*}
	(\rho_{\bhsC{ddd,11}}\rho_{\bhsC{dd0,10}}\rho_{\bhsC{0dd,01}})^{-b+2} 
	\prod_{N \in \cS_{\psi}(M)}
	(\rho_{\bhsC{\phi\phi\phi,11}(N)}\rho_{\bhsC{\phi\phi0,10}(N)}
		\rho_{\bhsC{0\phi\phi,01}(N)})^{-a+3}\\
	(\rho_{\bhsC{\phi\phi\phi,00}(N)}\rho_{\bhsC{\phi\phi0,00}(N)}
		\rho_{\bhsC{0\phi\phi,00}(N)})^{-a'+1}
	\rho_{\bhsC{\phi0\phi,00}(N)}^{a'+\dim(N/B)}
	\rho_{\bhsC{010,00}(N)}.
\end{multline*}
Now we can apply the pull-back and push-forward theorems to get the result.

\end{proof}

\section{Composition of b, wedge heat kernels}\label{sec:bwHeatComp}

In section \ref{sec:BWInd}, we prove a families index theorem on manifolds with iterated fibration structures endowed with b-wedge metrics. The locally trivial family is denoted
\begin{equation*}
	X' \fib M' \xlra{\psi'} B'
\end{equation*}
in order to distinguish it from the fiber bundle used in the bulk of the text. There is a minimal $N_0' \in \cS(M')$ such that $\dim(N_0'/B')=0,$ and the metric is of b-type near $\bhs{N_0'}$ and of wedge type at all other $N' \in \cS(M').$

In this section we establish a composition result for b-wedge heat operators.

{\bf The b, wedge heat composition space.}
Recall that the b,wedge heat space is given by
\begin{multline*}
	H(M'/B')_{b,\w} =
	\Big[ M'\times_{\psi'} M' \times \bbR^+_{\tau}; \bhs{N_0'}
        \times_{\phi_{N_0'}} \bhs{N_0'} \times \bbR^+_{\tau}; \dots \bhs{N_\ell'} \times_{\phi_{N_\ell'}} \bhs{N_\ell'} \times \bbR^+_{\tau}; \\	 
        \bhs{N_1'} \times_{\phi_{N_1'}} \bhs{N_1'} \times \{ 0\}; \ldots ; 
	\bhs{N_{\ell}'} \times_{\phi_{N_\ell'}} \bhs{N_\ell'} \times
        \{ 0\} \Big],
\end{multline*}
where $\{ N_0', N_1'. \ldots, N_{\ell}'\}$ is a non-decreasing list of $\cS(M').$
Thus this space is constructed by treating $N_0'$ as in the construction of the b-heat space \cite[Chapter 7]{tapsit}, and treating the other $N'$ as in the construction of the wedge heat space in \S\ref{sec:Heat}. We will follow the same pattern below.

We use the space $\sT^2_0$ defined in Appendix \ref{sec:HeatComp} as
\begin{equation*}
	\sT^2_0 = [\bbR^+_s \times \bbR^+_{\wt s}; \{ (0,0) \}]
\end{equation*}
together with the notation for its boundary hypersurfaces $\bhsT{10},$ $\bhsT{01},$ $\bhsT{11}.$
We define the b, wedge composition space using the notation from \eqref{eq:TripDiag} starting from 
\begin{equation*}
	H_0(M'/B')^3_{b,\w} = 
	[(M')^3_{\psi'}\times \sT^2_0;
	T(N_0'); 
	(S_{LM}(N'_0) \cup S_{LR}(N'_0) \cup S_{MR}(N'_0))].
\end{equation*}
Inductively, for $1 \leq i \leq \ell,$ let
\begin{multline*}
	H_i(M'/B')^3_{b,\w} = \Big[H_{i-1}(M'/B')^3_{b,\w};
	T(N_i') \times \bhsT{11};
	(S_{LM}(N_i') \cup S_{LR}(N_i') \cup S_{MR}(N_i')) \times \bhsT{11}; \\
	S_{LM}(N_i') \times \bhsT{10};
	S_{MR}(N_i') \times \bhsT{01}; 
	T(N_i);
	(S_{LM}(N_i') \cup S_{LR}(N_i') \cup S_{MR}(N_i'))\Big].
\end{multline*}
Finally, we need to blow-up the (interior) lifts of the partial diagonals. So let
\begin{equation*}
	\diag_{LM} = \pi_{LM}^{-1}(\diag_M), \quad
	\diag_{LR} = \pi_{LR}^{-1}(\diag_M), \quad
	\diag_{MR} = \pi_{MR}^{-1}(\diag_M), 
\end{equation*}
and let $\diag_{LMR}$ be their intersection and then define
\begin{multline*}
	H(M'/B')^3_{b,\w}
	= \Big[H_{\ell}(M'/B')^3_{b,\w};
	\diag_{LMR} \times \bhsT{11}; \\
	(\diag_{LM} \cup \diag_{LR} \cup \diag_{MR}) \times \bhsT{11}; 
	\diag_{LM}\times \bhsT{10};
	\diag_{MR} \times \bhsT{01}
	\Big].
\end{multline*}

As anticipated, $N_0'$ gives rise to the same blow-ups as in the composition space for the b-calculus \cite{Albin:RenInd} and the other $N'$ give rise to the same blow-ups as in the wedge heat composition space in Appendix \ref{sec:HeatComp}.

Our notation for the (collective) boundary hypersurfaces of $H(M'/B')^3_{b,\w}$ is as follows. First we have,
\begin{equation*}
\begin{gathered}
	(M')^3_{\psi'} \times \bhsT{10} \leftrightarrow \bhsC{000,10}, \quad
	(M')^3_{\psi'} \times \bhsT{01} \leftrightarrow \bhsC{000,01}, \quad
	(M')^3_{\psi'} \times \bhsT{11} \leftrightarrow \bhsC{000,11}, \\
	\diag_{LMR} \times \bhsT{11} \leftrightarrow \bhsC{ddd,11}, \quad
	\diag_{LM} \times \bhsT{11} \leftrightarrow \bhsC{dd0,11}, \quad
	\diag_{LR} \times \bhsT{11} \leftrightarrow \bhsC{d0d,11}, \\
	\diag_{MR} \times \bhsT{11} \leftrightarrow \bhsC{0dd,11}, \quad
	\diag_{LM} \times \bhsT{10} \leftrightarrow \bhsC{dd0,10}, \quad
	\diag_{MR} \times \bhsT{01} \leftrightarrow \bhsC{0dd,01},
\end{gathered}
\end{equation*}
and then, for each $N' \in \cS(M'),$
\begin{equation*}
\begin{gathered}
	\bhss{N'} \times_{\psi} M^2_{\psi} \times \sT^2_0 
		\leftrightarrow \bhsC{100,00}(N'), \quad
	M \times_{\psi} \bhss{N'} \times_{\psi} M \times \sT^2_0 
		\leftrightarrow \bhsC{010,00}(N'), \\
	M^2_{\psi} \times_{\psi} \bhss{N'} \times \sT^2_0 
		\leftrightarrow \bhsC{001,00}(N'), \quad
	T(N') \times \sT^2_0 \leftrightarrow \bhsC{\phi\phi\phi,00}(N'), \\
	S_{LM}(N') \times \sT^2_0 \leftrightarrow \bhsC{\phi\phi0,00}(N'), \quad
	S_{LR}(N') \times \sT^2_0 \leftrightarrow \bhsC{\phi0\phi,00}(N'), \\
	S_{MR}(N') \times \sT^2_0 \leftrightarrow \bhsC{0\phi\phi,00}(N'),
\end{gathered}
\end{equation*}
and, for each $N' \in \cS(M')\setminus\{N'_0\},$
\begin{equation*}
\begin{gathered}
	T(N') \times \bhsT{11} \leftrightarrow \bhsC{\phi\phi\phi,11}(N'), \quad
	S_{LM}(N') \times \bhsT{11} \leftrightarrow \bhsC{\phi\phi0,11}(N'), \\
	S_{LR}(N') \times \bhsT{11} \leftrightarrow \bhsC{\phi0\phi,11}(N'), \quad
	S_{MR}(N') \times \bhsT{11} \leftrightarrow \bhsC{0\phi\phi,11}(N'), \\
	S_{LM}(N') \times \bhsT{10} \leftrightarrow \bhsC{\phi\phi0,10}(N'), \quad
	S_{MR}(N') \times \bhsT{01} \leftrightarrow \bhsC{0\phi\phi,01}(N').
\end{gathered}
\end{equation*}

$ $\\
{\bf The exponent matrices.}
As in Appendix \ref{sec:HeatComp}, we have b-maps
\begin{equation*}
	\xymatrix{ 
	& H(M'/B')^3_{b,\w} \ar[ld]^{\beta_{LM,L}} \ar[d]^{\beta_{LR,C}} \ar[rd]^{\beta_{MR,R}} & \\
	H(M'/B')_{b,\w}  & H(M'/B')_{b,\w} & H(M'/B')_{b,\w} }
\end{equation*}
and we now specify their behavior with respect to the boundary hypersurfaces.

First, $\beta_{LM,L}$ maps $\bhsC{0dd,01},$ $\bhsC{000,01}$ and, for each $N' \in \cS(M'),$ $\bhsC{001,00}(N'),$ into the interior of $H(M'/B')_{b,\w};$ otherwise
\begin{equation*}
\begin{aligned}
	\beta_{LM,L}^*\bhsh{00,1}
	&= \{ \bhsC{000,10}, \bhsC{000,11},  
		\bhsC{d0d, 11}, \bhsC{0dd, 11} \} \cup 
		\bigcup_{N' \in \cS(M')\setminus\{N'_0\}} 
		\{\bhsC{\phi0\phi,11}(N'), \bhsC{0\phi\phi,11}(N') \},\\
	\beta_{LM,L}^*\bhsh{dd,1}
	&= \{ \bhsC{ddd, 11}, \bhsC{dd0, 10}, \bhsC{dd0,11} \},
\end{aligned}
\end{equation*}
for $N'_0,$
\begin{equation*}
\begin{aligned}
	\beta_{LM,L}^*\bhsh{10,0}(N'_0)
	&= \{ \bhsC{100,00}(N'_0), 
		\bhsC{\phi0\phi, 00}(N'_0) \} \\
	\beta_{LM,L}^*\bhsh{01,0}(N'_0)
	&= \{ \bhsC{010,00}(N'_0), 
		\bhsC{0\phi\phi,00}(N'_0) \} \\
	\beta_{LM,L}^*\bhsh{\phi\phi,0}(N'_0)
	&= \{ \bhsC{\phi\phi\phi,00}(N'_0), \bhsC{\phi\phi0,00}(N'_0) \},
\end{aligned}
\end{equation*}
and, for each $N' \in \cS(M')\setminus\{N_0'\},$
\begin{equation*}
\begin{aligned}
	\beta_{LM,L}^*\bhsh{10,0}(N')
	&= \{ \bhsC{100,00}(N'), \bhsC{\phi0\phi, 11}(N'), 
		\bhsC{\phi0\phi, 00}(N') \} \\
	\beta_{LM,L}^*\bhsh{01,0}(N')
	&= \{ \bhsC{010,00}(N'), \bhsC{0\phi\phi,11}(N'), \bhsC{0\phi\phi,01}(N'),
		\bhsC{0\phi\phi,00}(N') \} \\
	\beta_{LM,L}^*\bhsh{\phi\phi,1}(N')
	&= \{ \bhsC{\phi\phi\phi,11}(N'), \bhsC{\phi\phi0,11}(N'), 
		\bhsC{\phi\phi0, 10}(N') \}, \\
	\beta_{LM,L}^*\bhsh{\phi\phi,0}(N')
	&= \{ \bhsC{\phi\phi\phi,00}(N'), \bhsC{\phi\phi0,00}(N') \}.
\end{aligned}
\end{equation*}
We note that $\bhsC{\phi0\phi,11}(N')$ and $\bhsC{0\phi\phi,11}(N')$ are repeated, for $N' \in \cS(M')\setminus\{N'_0\}.$

Similarly, the map $\beta_{LR,C}$ maps $\bhsC{000,10},$ $\bhsC{000,01},$ $\bhsC{dd0,10},$ $\bhsC{0dd,01}$ and, for each $N' \in \cS(M'),$ $\bhsC{010,00}(N'),$ into the interior of $H(M'/B')_{b,\w};$ otherwise
\begin{equation*}
\begin{aligned}
	\beta_{LR,C}^*\bhsh{00,1}
	&= \{ \bhsC{000,11}, \bhsC{dd0,11}, \bhsC{0dd, 11} \}
	\cup \bigcup_{N' \in \cS(M')\setminus\{N'_0\}} \{ \bhsC{\phi\phi0,11}(N'), \bhsC{0\phi\phi,11}(N') \} \\
	\beta_{LR,C}^*\bhsh{dd,1}
	&= \{ \bhsC{ddd, 11}, \bhsC{d0d, 11} \}
\end{aligned}
\end{equation*}
for $N_0',$
\begin{equation*}
\begin{aligned}
	\beta_{LR,C}^*\bhsh{10,0}(N'_0)
	&= \{ \bhsC{100,00}(N'_0),
		\bhsC{\phi\phi0,00}(N'_0) \} \\
	\beta_{LR,C}^*\bhsh{01,0}(N'_0)
	&= \{ \bhsC{001, 00}(N'_0), 
		\bhsC{0\phi\phi,00}(N'_0) \} \\
	\beta_{LR,C}^*\bhsh{\phi\phi,0}(N'_0)
	&= \{ \bhsC{\phi\phi\phi,00}(N'_0), \bhsC{\phi0\phi,00}(N'_0) \}
\end{aligned}
\end{equation*}
and, for each $N' \in \cS(M')\setminus\{N_0'\},$
\begin{equation*}
\begin{aligned}
	\beta_{LR,C}^*\bhsh{10,0}(N')
	&= \{ \bhsC{100,00}(N'), \bhsC{\phi\phi0,11}(N'), \bhsC{\phi\phi0,10}(N'),
		\bhsC{\phi\phi0,00}(N') \} \\
	\beta_{LR,C}^*\bhsh{01,0}(N')
	&= \{ \bhsC{001, 00}(N'), \bhsC{0\phi\phi,11}(N'), \bhsC{0\phi\phi,01}(N'),
		\bhsC{0\phi\phi,00}(N') \} \\
	\beta_{LR,C}^*\bhsh{\phi\phi,1}(N')
	&= \{ \bhsC{\phi\phi\phi,11}(N'), \bhsC{\phi0\phi,11}(N') \} \\
	\beta_{LR,C}^*\bhsh{\phi\phi,0}(N')
	&= \{ \bhsC{\phi\phi\phi,00}(N'), \bhsC{\phi0\phi,00}(N') \}.
\end{aligned}
\end{equation*}
We note that $\bhsC{\phi\phi0,11}(N')$ and $\bhsC{0\phi\phi,11}(N')$ are repeated for $N' \in \cS(M')\setminus\{N_0'\}.$

Finally, the map $\beta_{MR,R}$ maps  $\bhsC{000,10},$ $\bhsC{dd0,10}$
and, for each $N' \in \cS(M'),$ $\bhsC{100,00}(N'),$
into the interior of $H(M'/B')_{b,\w};$ otherwise
\begin{equation*}
\begin{aligned}
	\beta_{MR,R}^*\bhsh{00,1}
	&= \{ \bhsC{000,01}, \bhsC{000,11},  \bhsC{dd0,11}, \bhsC{d0d,11} \}
	\cup \bigcup_{N' \in \cS(M')\setminus\{N'_0\}} \{ \bhsC{\phi\phi0,11}(N'), \bhsC{\phi0\phi,11}(N') \} \\
	\beta_{MR,R}^*\bhsh{dd,1}
	&= \{ \bhsC{ddd,11}, \bhsC{0dd,11}, \bhsC{0dd, 01} \}
\end{aligned}
\end{equation*}
for $N_0',$
\begin{equation*}
\begin{aligned}
	\beta_{MR,R}^*\bhsh{10,0}(N'_0)
	&= \{ \bhsC{010,00}(N'_0), 
		\bhsC{\phi\phi0,00}(N'_0) \} \\
	\beta_{MR,R}^*\bhsh{01,0}(N'_0)
	&= \{ \bhsC{001,00}(N'_0), \bhsC{\phi0\phi,00}(N'_0) \} \\
	\beta_{MR,R}^*\bhsh{\phi\phi,0}(N'_0)
	&= \{ \bhsC{\phi\phi\phi,00}(N'_0), \bhsC{0\phi\phi,00}(N'_0) \}.
\end{aligned}
\end{equation*}
and, for each $N' \in \cS(M')\setminus\{N_0'\},$
\begin{equation*}
\begin{aligned}
	\beta_{MR,R}^*\bhsh{10,0}(N')
	&= \{ \bhsC{010,00}(N'), \bhsC{\phi\phi0,11}(N'), \bhsC{\phi\phi0,10}(N'),
		\bhsC{\phi\phi0,00}(N') \} \\
	\beta_{MR,R}^*\bhsh{01,0}(N')
	&= \{ \bhsC{001,00}(N'), \bhsC{\phi0\phi,11}(N'), \bhsC{\phi0\phi,00}(N') \} \\
	\beta_{MR,R}^*\bhsh{\phi\phi,1}(N')
	&= \{ \bhsC{\phi\phi\phi,11}(N'), \bhsC{0\phi\phi,11}(N'), 
		\bhsC{0\phi\phi,01}(N') \} \\
	\beta_{MR,R}^*\bhsh{\phi\phi,0}(N')
	&= \{ \bhsC{\phi\phi\phi,00}(N'), \bhsC{0\phi\phi,00}(N') \}.
\end{aligned}
\end{equation*}
We note that $\bhsC{\phi\phi0,11}(N')$ and $\bhsC{\phi0\phi,11}(N')$ are repeated, for $N' \in \cS(M')\setminus\{N_0'\}.$

$ $\\
{\bf Composition.}
Now let us discuss the composition law.
As before we are interested in integral kernels that are sections of a weighted density bundle.
Let us start by recalling the weight from \S\ref{sec:BWInd}, namely
\begin{equation*}
\begin{gathered}
	\mathfrak h: \cM_1(H(M'/B')_{b,\w}) \lra \bbR, \\
	\mathfrak h(H) =
	\begin{cases}
	-(\dim(N'/B') + 3) & \Mif H \subseteq \bhsh{\phi\phi,1}(N') 
		\Mforsome N' \in \cS(M')\setminus\{N'_0\} \\
	-(\dim(N'/B') + 1) & \Mif H \subseteq \bhsh{\phi\phi,0}(N') 
		\Mforsome N' \in \cS(M') \\
	-(\dim(M'/B') + 2) & \Mif H = \bhsh{dd,1} \\
	\infty & \Mif H = \bhsh{00,1} \\
	0 & \Motherwise
	\end{cases}
\end{gathered}
\end{equation*}
and let $\mu_R = \beta_{(H), R}^*\mu(M'/B').$ We will determine the behavior under composition for integral kernels of the form $\cK_A \rho^{\mathfrak h} \mu_R$ with $\cK_A \in \cA_{phg}^{\cE_A}(H(M'/B')_{b,\w}).$ Ultimately we are interested in kernels that are merely conormal with bounds acting on sections of a vector bundle, but the corresponding composition result follows easily from this one.

\begin{proposition} \label{prop:BWedgeHeatComp}
Let $A$ have kernel $\cK_A \rho^{\mathfrak h} \mu_R$ with $\cK_A \in \cA_{\phg}^{\cE_A}(H(M'/B')_{b,\w})$ and $B$ have kernel $\cK_B \rho^{\mathfrak h} \mu_R$ with $\cK_B \in \cA_{phg}^{\cE_B}(H(M'/B')_{b,\w}).$ 
If 
\begin{equation*}
\begin{gathered}
	\Re(\cE_A(\bhsh{dd,1})) > 0, \quad
	\Re(\cE_B(\bhsh{dd,1})) > 0, \Mand \\
	\Re(\cE_A(\bhsh{01,0}(N'))) + \Re(\cE_B(\bhsh{10,0}(N')))+1 > 0 \Mforall N' \in \cS(M'),
\end{gathered}
\end{equation*}
then we may define their composition $C = A \circ B$ by the formula
\begin{equation*}
	\wt\cK_C \mu_R \beta_{(H)}^*(\tau \; d\tau)
	= (\beta_{LR,C})_*( 
	\beta_{LM,L}^*(\cK_A \rho^{\mathfrak h} \mu_R \; \beta_{(H)}^*(\tau d\tau))
	\cdot
	\beta_{MR,R}^*(\cK_B \rho^{\mathfrak h} \mu_R \; \beta_{(H)}^*(\tau d\tau)) )
\end{equation*}
and we have $\wt\cK_C \in \rho^{\mathfrak h} \cA_{phg}^{\cE_C}(H(M'/B')_{b,\w})$ with
\begin{equation*}
	\cE_C(\bhsh{dd,1}) = 
	\cE_A(\bhsh{dd,1}) + \cE_B(\bhsh{dd,1}) 
\end{equation*}
and, for $N'_0,$
\begin{align*}
	\cE_C(\bhsh{10,0}(N'_0)) &=
	\cE_A(\bhsh{10,0}(N'_0)) \bar\cup
	\cE_B(\bhsh{10,0}(N'_0) ) \\
	\cE_C(\bhsh{01,0}(N'_0)) &=
	\cE_B(\bhs{01,0}(N'_0)) \bar\cup
	\cE_A(\bhsh{01,0}(N'_0))\\
	\cE_C(\bhsh{\phi\phi,0}(N'_0)) &=
	(\cE_A(\bhsh{\phi\phi,0}(N'_0)) + \cE_B(\bhsh{\phi\phi,0}(N'_0)) ) \\
	& \phantom{xxxxx} \bar\cup
	\big(\cE_A(\bhsh{10,0}(N'_0)) + \cE_B(\bhsh{01,0}(N'_0)) + \dim(N'_0/B)+1\big)
\end{align*}
and, for each $N' \in \cS(M')\setminus\{N_0'\},$
\begin{align*}
	\cE_C(\bhsh{10,0}(N')) &=
	\cE_A(\bhsh{10,0}(N')) \bar\cup
	\big( \cE_A(\bhsh{\phi\phi,1}(N')) + \cE_B(\bhsh{10,0}(N') )\big) \\
	\cE_C(\bhsh{01,0}(N')) &=
	\cE_B(\bhs{01,0}(N')) \bar\cup
	\big(\cE_A(\bhsh{01,0}(N')) + \cE_B(\bhsh{\phi\phi,1}(N')) \big) \\
	\cE_C(\bhsh{\phi\phi,1}(N')) &=
	\cE_A(\bhsh{\phi\phi,1}(N')) + \cE_B(\bhsh{\phi\phi,1}(N')) \\
	\cE_C(\bhsh{\phi\phi,0}(N')) &=
	(\cE_A(\bhsh{\phi\phi,0}(N')) + \cE_B(\bhsh{\phi\phi,0}(N')) ) \\
	& \phantom{xxxxx} \bar\cup
	\big(\cE_A(\bhsh{10,0}(N')) + \cE_B(\bhsh{01,0}(N')) + \dim(N'/B')+1\big).
\end{align*}
\end{proposition}

\begin{proof}
The proof is essentially the same as that of Proposition \ref{prop:WedgeHeatComp}, using the exponent matrices computed above.
\end{proof}
  
\newpage

\printnomenclature

%

\newcommand{\etalchar}[1]{$^{#1}$}
\def\cprime{$'$}

\end{document}